\documentclass[11pt]{article}

\usepackage[margin=1in]{geometry}
\usepackage{amsmath,amssymb,amsthm}
\usepackage{graphicx}
\usepackage{float}
\usepackage{booktabs}
\usepackage{xcolor}
\usepackage{colortbl}
\usepackage[colorlinks=true,linkcolor=blue,citecolor=blue,urlcolor=blue]{hyperref}
\usepackage{etoolbox}
\pretocmd{\section}{\clearpage}{}{}

\usepackage{listings}
\definecolor{leanblue}{RGB}{0,0,180}
\definecolor{Nblue}{RGB}{31,119,180}
\definecolor{Ngreen}{RGB}{44,160,44}
\definecolor{leangreen}{RGB}{0,128,0}
\definecolor{leanred}{RGB}{180,0,0}
\definecolor{followerviolet}{RGB}{120,40,180}
\lstdefinelanguage{Lean}{
  morekeywords={
    import, open, variable, variables, section, end, axiom,
    def, theorem, lemma, example, abbrev, set, unfold, conv_lhs,
    noncomputable, namespace, where, by, intro, intros,
    apply, fun, match, with, if, then, else, do,
    let, in, exact, rfl, simp, dsimp, rw, have, from, ring, field_simp,
    push_cast, ring_nf, linear_combination, congr, exact_mod_cast, simpa, only
  },
  sensitive=true,
  morecomment=[l]{--},
  morestring=[b]"
}
\lstdefinestyle{leanstyle}{
  language=Lean,
  basicstyle=\ttfamily\footnotesize,
  keywordstyle=\color{leanblue}\bfseries,
  commentstyle=\color{leangreen},
  stringstyle=\color{purple},
  identifierstyle=\color{black},
  columns=fullflexible,
  keepspaces=true,
  showstringspaces=false,
  breaklines=true,
  breakindent=2em,
  numbers=left,
  numberstyle=\tiny\color{gray},
  stepnumber=1,
  literate={ℝ}{{$\mathbb{R}$}}1
           {ℂ}{{$\mathbb{C}$}}1
           {ℤ}{{$\mathbb{Z}$}}1
           {ω}{{$\omega$}}1
           {φ}{{$\varphi$}}1
           {ψ}{{$\psi$}}1
           {χ}{{$\chi$}}1
           {σ}{{$\sigma$}}1
           {θ}{{$\theta$}}1
           {→}{{$\to$}}1
           {←}{{$\leftarrow$}}1
           {‖}{{$\|$}}1
           {≠}{{$\neq$}}1
           {¬}{{$\neg$}}1
           {∃}{{$\exists$}}1
           {∀}{{$\forall$}}1
           {⌈}{{$\lceil$}}1
           {⌉}{{$\rceil$}}1
           {₀}{{$_0$}}1
           {₁}{{$_1$}}1
           {₂}{{$_2$}}1
           {∓}{{$\mp$}}1
           {·}{{$\cdot$}}1
           {ℕ}{{$\mathbb{N}$}}1
           {Σ}{{$\Sigma$}}1
           {∑}{{$\sum$}}1
           {∈}{{$\in$}}1
           {ζ}{{$\zeta$}}1
           {π}{{$\pi$}}1
           {⊢}{{$\vdash$}}1
}

\graphicspath{{figures/}}

\newtheorem{theorem}{Theorem}[section]
\newtheorem{proposition}[theorem]{Proposition}

\newtheorem{lemma}[theorem]{Lemma}
\newtheorem{corollary}[theorem]{Corollary}
\theoremstyle{definition}

\newtheorem{remark}[theorem]{Remark}

\newcommand{\linkmatd}[2]{%
\left(\begin{array}{@{}l@{\,}l@{\,}l@{}}
\cos #1 & -\sin #1           & \frac{\cos #1}{#2}\\[1pt]
\sin #1 & \phantom{-}\cos #1 & \frac{\sin #1}{#2}\\[1pt]
\multicolumn{1}{@{}c@{\,}}{0} & \multicolumn{1}{c@{\,}}{0} & \multicolumn{1}{c@{}}{1}
\end{array}\right)}
\newcommand{\linkmatx}[2]{%
\left(\begin{array}{@{}l@{\,}l@{\,}l@{}}
\cos #1 & -\sin #1           & -\frac{|\chi|\cos #1}{#2}\\[1pt]
\sin #1 & \phantom{-}\cos #1 & -\frac{|\chi|\sin #1}{#2}\\[1pt]
\multicolumn{1}{@{}c@{\,}}{0} & \multicolumn{1}{c@{\,}}{0} & \multicolumn{1}{c@{}}{1}
\end{array}\right)}
\newcommand{\linkmat}[2]{%
\left(\begin{array}{@{}l@{\ \,}l@{\ \,}l@{}}
\cos #1 & -\sin #1           & #2\cos #1\\
\sin #1 & \phantom{-}\cos #1 & #2\sin #1\\
\multicolumn{1}{@{}c@{\ \,}}{0} & \multicolumn{1}{c@{\ \,}}{0} & \multicolumn{1}{c@{}}{1}
\end{array}\right)}

\title{\textbf{Remainder Terms of the Riemann Zeta Function}}
\author{Paul Stahura\\[0.3em]\small\texttt{paul+zeta@stahura.net}}
\date{\today}

\begin{document}
\maketitle

\begin{abstract}
Starting from the Riemann--Siegel formula for $\zeta(s)$ given in Siegel's
1932 paper, we introduce a change of variable, $t=I(T)$, that replaces
Siegel's coupled pair ``imaginary part $t$ / summation cutoff $m$'' with one
continuous cutoff parameter $T$. Five results follow.

\noindent (1) We split Siegel's single remainder integral $R$ into two exact
pieces and prove $R=R_{1}+R_{2}$, so that
$\zeta=\Sigma_{1}+R_{1}+\Sigma_{2}+R_{2}$; the algebra is formally verified
in Lean in an appendix, as are the critical-line equalities that descend
from it.

\noindent (2) Each remainder is nothing more than \emph{one additional,
fractional partial summand} appended to its Dirichlet sum:
$R_{1}=\hat d_1\,(m+1)^{-s}$ and $R_{2}=\hat d_2\,\chi(s)\,(m+1)^{s-1}$
with $m=\lfloor T\rfloor$; we prove the weights $\hat d_1,\hat d_2$ are
real across the strip, equal and positive on the critical line, and
positive off it outside two narrow windows of width $O(1/T)$ per unit
interval of $T$.

\noindent (3) The weights are periodic in the index: $d_1$ repeats one
waveform in every unit interval of $T$, converging to the explicit tangent
waveform $\mathcal W_{\infty}$, and at fixed fractional part $\{T\}=x$ the
normalized weight $\sqrt{m+1}\,d_1$ converges to a closed-form limit profile
$d(x)$.

\noindent (4) We find thin loci in the strip, the equal-leg and folded-leg
ovals, that are the only places off the critical line where a zero of
$\zeta$ could possibly lie.

\noindent (5) We construct a counting function $N_{\ast}(T)$ from
critical-line information alone; unlike the Riemann--von Mangoldt count it
sees only the zeros on the line, and it is continuous and increasing (the
latter under the Riemann hypothesis), passing an integer exactly where $Z$
changes sign.
\end{abstract}

\tableofcontents

\section{Introduction}

Siegel published his study of Riemann's unpublished papers in
1932~\cite{Siegel1932}. Working through the Nachla\ss{} at
G\"ottingen\footnote{The chronology pins the reading to G\"ottingen,
where the Nachla\ss{} is held: Siegel studied it there in the early
1930s (though he didn't move to G\"ottingen until 1938) and published
in 1932, before his first visit to the Institute for Advanced Study
(January--June 1935) and his longer residence there (1940--1951).} he
recovered, from calculations Riemann never published, an asymptotic
evaluation of $\zeta$ by the saddle-point method. The result is what is
now called the Riemann--Siegel formula, and it is the starting point of
this paper.

This paper does seven things: the five results announced in the
abstract, together with the reparameterization they rest on and an
account of the geometry from which they came.

\begin{enumerate}
\item \textbf{A reparameterization.} We introduce a change of variable
$t=I(T)$ that collapses Siegel's two roles for $t$ (the imaginary part of the
input and the quantity determining the summation cutoff) into one continuous
cutoff parameter $T$ (\S\ref{sec:IT}). The map $t=I(T)$ is exact, while our
choice $m=\lfloor T\rfloor$ differs from Siegel's standard cutoff. The
Riemann--Siegel identity remains exact with this new cutoff.

\item \textbf{Two new remainders.} The Riemann--Siegel formula
states that
$\zeta = \Sigma_1 + \Sigma_2 + R$, where $\Sigma_1,\Sigma_2$ are finite Dirichlet
sums and $R$ is a remainder integral. We define two exact
remainder terms,
$R_{1}$ and $R_{2}$, and prove
\[
R \;=\; R_{1} + R_{2},
\qquad\text{hence}\qquad
\zeta \;=\; \Sigma_{1} + R_{1} + \Sigma_{2} + R_{2}.
\]
This is the paper's central identity, stated and proved as
Theorem~\ref{thm:decomp} (\S\ref{sec:decomp}). A formalization of its proof
in Lean is given in Appendix~\ref{app:lean-decomp}.

\item \textbf{Just a fractional summand.} The two new remainders are not
opaque integrals: each is exactly \emph{one more summand} of its Dirichlet
sum, namely the $(m+1)$st term, scaled by a real ``fractional'' weight
$\hat d_1$ or $\hat d_2$ (\S\ref{sec:summand}). The exact zeta function
thus reduces to the main sum and the dual sum, each carried one fractional
summand past its cutoff, with no other remainder term:
\[
\zeta(s)=\sum_{n=1}^{m} n^{-s} + \hat d_1\,(m+1)^{-s}
        + \chi(s)\sum_{n=1}^{m} n^{s-1} + \hat d_2\,\chi(s)\,(m+1)^{s-1},
\]
which is \eqref{eq:zeta-clean} of \S\ref{sec:summand}. The weights are
positive on the critical line, where they also coincide, $d_1=d_2$, and
positive off it outside narrow windows (Proposition~\ref{prop:weights},
proved in Appendix~\ref{sec:positivity}).

\item \textbf{Periodicity.} In the variable $\{T\}$ the weights repeat one
fixed waveform in every unit interval of $T$ (\S\ref{sec:periodicity}):
$d_1$ converges to the explicit tangent waveform $\mathcal W_{\infty}$ of
\eqref{eq:W-infty}, and at fixed fractional part $\{T\}=x$ the normalized
weight $\sqrt{m+1}\,d_1$ converges to a closed-form limit profile $d(x)$.
The waveform and its limit were first observed geometrically, as the yin
and yang curves of \S\ref{sec:yinyang-inf}; only afterwards did we show
they can be reached symbolically from a function in Siegel's paper
(\S\ref{sec:tangent-derivation}).

\begin{center}
{\small Periodicity of the weights\quad(detail of
Figure~\ref{fig:d1-periodicity})}\\[2pt]
\includegraphics[width=0.96\textwidth]{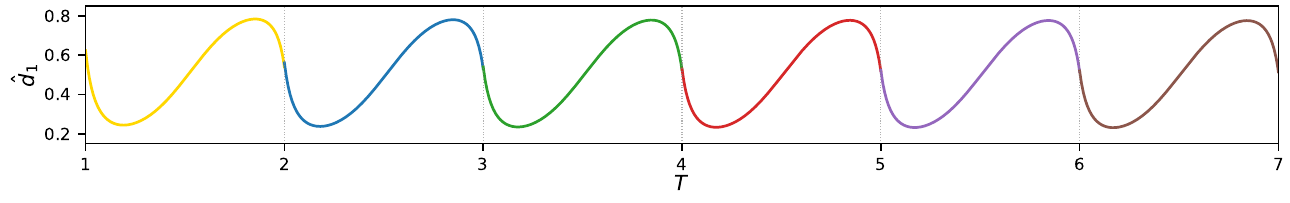}
\end{center}

\item \textbf{Where the insights came from: geometry.} None of this was
first derived; it was first \emph{seen}. Drawing the partial-sum chains of
$\zeta$ as links and joints in the complex plane, and sweeping $\sigma$
and $T$ continuously, is how the mapping $I(T)$, the yin-yang curves, the
split of $R$, and the
weights were found (\S\ref{sec:geometry}, \S\ref{sec:geometric-origins}).
The bisector point, the two legs it makes with the origin and $\zeta$, and
the angles $\vartheta_1,\vartheta_2$ organize everything that follows.

\item \textbf{Loci that confine the zeros.} A zero of $\zeta$ requires the
two legs to be equal in length \emph{and} folded back onto one another.
Off the critical line the equal-leg locus consists of thin ovals, and
those ovals are the only places off the line where a zero could possibly
lie. A zero also needs the fold angle to reach $\pi$ there, so showing
the two conditions never meet on an oval is yet another possible
route to the Riemann hypothesis (\S\ref{sec:toward-proof},
\S\ref{sec:eqleg-density}).

\item \textbf{A new zero-counting function.} From critical-line information alone
we build
\[
N_{\ast}(T)
=\frac{1}{\pi}\arg\Bigl(Z-\frac{iZ'}{\vartheta'}\Bigr)+\frac32,
\]
which is \eqref{eq:N-star-prufer} of \S\ref{sec:counting}: unlike the
Riemann--von
Mangoldt count $N(t)$, which counts all zeros in the strip, $N_{\ast}$
counts only the zeros on the critical line; it is continuous, increasing under the
Riemann hypothesis, and passes an integer exactly where $Z$ changes sign
(\S\ref{sec:counting}). $N(t)$ is a step function and tells nothing about
$|Z|$, while the vector behind $N_{\ast}$ carries $|Z|$ along with the
count: its winding gives the zeros and its modulus gives the envelope of
$|Z|$.
\end{enumerate}

\begin{center}
\textbf{Riemann--von Mangoldt count \textcolor{Nblue}{$N(t)$}, and \textcolor{Ngreen}{$N_{\ast}(T)$}}\\[0.4em]
\includegraphics[width=\linewidth]{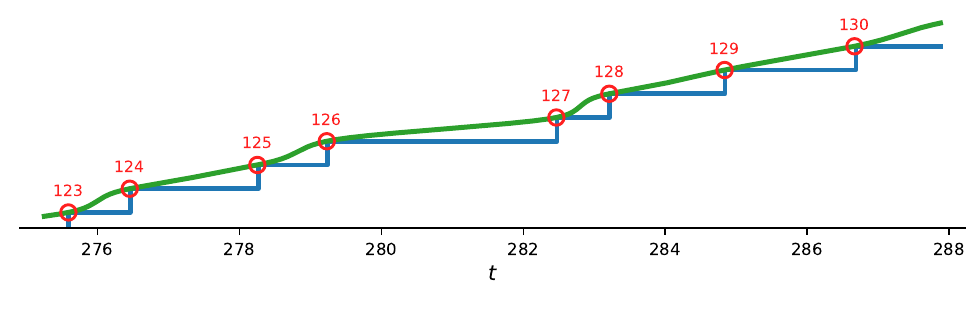}
\end{center}

\section{Siegel's 1932 formula}\label{sec:siegel}

We begin with the Riemann--Siegel formula, which expresses $\zeta(s)$ as two
finite Dirichlet sums plus a remainder integral, as it appears in Siegel's
1932 paper~\cite{Siegel1932} (his equation 13):
\begin{equation}\label{eq:siegel13}
\zeta(s)
=
\sum_{n=1}^{m} n^{-s}
\;+\;
\frac{(2\pi)^s}{2\,\Gamma(s)\,\cos\!\bigl(\tfrac{\pi s}{2}\bigr)}
\sum_{n=1}^{m} n^{s-1}
\;+\;
\frac{(2\pi)^s e^{\tfrac{\pi i s}{2}}}{\Gamma(s)\,\bigl(e^{2\pi i s}-1\bigr)}
\int_{C_2} g(x)\,\mathrm{d}x,
\end{equation}
where $m$ is an integer cutoff, fixed below in \eqref{eq:siegel-m}, and
in Siegel's notation $\epsilon=e^{-\pi i/4}=(1-i)/\sqrt2$. The oriented
path $C_2$ consists of two half-lines meeting at
$\eta\bigl(1-\tfrac{\epsilon}{2}\bigr)$, with
$\eta=+\sqrt{t/2\pi}$ the saddle point; one half-line passes through
$\eta$, the other through $-(m+\tfrac12)$, and $g$ is defined below.

It is convenient to name the three pieces. Recall the functional-equation
factor
\begin{equation}
\chi(s) \;:=\; 2^{s}\pi^{s-1}\sin\!\Bigl(\tfrac{\pi s}{2}\Bigr)\Gamma(1-s)
\;=\;
\frac{(2\pi)^s}{2\,\Gamma(s)\,\cos\!\bigl(\tfrac{\pi s}{2}\bigr)},
\end{equation}
so that, writing $s=\sigma+it$ and taking the first $m$ terms,
\begin{align}
\Sigma_{1}(s) &\;:=\; \sum_{n=1}^{m} n^{-s}
\quad\text{(the main sum)}, \\[2pt]
\Sigma_{2}(s) &\;:=\; \frac{(2\pi)^s}{2\,\Gamma(s)\,\cos\!\bigl(\tfrac{\pi s}{2}\bigr)}
\sum_{n=1}^{m} n^{s-1}
= \chi(s)\sum_{n=1}^{m} n^{s-1}
\quad\text{(the dual sum)}, \\[2pt]
R(s) &\;:=\; \frac{(2\pi)^s e^{\tfrac{\pi i s}{2}}}{\Gamma(s)\,\bigl(e^{2\pi i s}-1\bigr)}
\int_{C_2} g(x)\,\mathrm{d}x,
\end{align}
where
\begin{equation}
g(x)\;:=\;x^{s-1}\,\frac{e^{-2\pi i m x}}{e^{2\pi i x}-1}.
\end{equation}
With these names, \eqref{eq:siegel13} reads
\begin{equation}\label{eq:siegel-named}
\zeta(s) = \Sigma_{1} + \Sigma_{2} + R.
\end{equation}

In Siegel's formulation $m$ is the summation cutoff, and to minimize the
remainder he takes
\begin{equation}\label{eq:siegel-m}
m=\left\lfloor\sqrt{\tfrac{t}{2\pi}}\right\rfloor .
\end{equation}
Thus $t$ plays two roles at once: it is the imaginary part of the input $s$,
and, through \eqref{eq:siegel-m}, it also fixes the number of terms $m$. In
this way $m$ depends on $t$. The next section reverses the dependency.

\section{Reparameterization and cutoff choice: The \texorpdfstring{$I(T)$}{I(T)} mapping}\label{sec:IT}

Unless a statement explicitly gives a larger domain, assume throughout that
$T>0$, $0<\sigma<1$, $t=I(T)$, and $s=\sigma+it$.

We introduce a change-of-variable mapping for $t$, and with it fix the
variables used for the rest of the paper:
\begin{equation}\label{eq:IT}
t
\;=\;
\frac{\pi\,\bigl(2\,T + 1\bigr)}{\log\!\bigl(\tfrac{1}{T} + 1\bigr)}
\;=\; I(T),
\qquad
s=\sigma+it,
\qquad
m \;=\; \lfloor T\rfloor,
\qquad
M \;=\; \lceil T\rceil,
\end{equation}
For non-integer $T$, $M=m+1$ and is the index of the next omitted summand.
At integer $T$, $M=m$ and refers instead to the last included summand. All
``next summand'' statements below are therefore restricted to non-integer
$T$ or interpreted through one-sided limits.
Here $T$ is a real number, while $m=\lfloor T\rfloor$ is an integer. So now
$m$ and $t$ depend on $T$, the reverse of Siegel's \eqref{eq:siegel-m}. The single
function $I$ therefore determines both the imaginary part of the input and the
integer summation cutoff; for this reason we call $T$ ``the index'' even
though it is real.

\begin{remark}\label{rem:T-vs-a}
The index is similar to, but not exactly, the classical
Riemann--Siegel cutoff $\sqrt{t/2\pi}$. Expanding the logarithm
in \eqref{eq:IT} gives
\[
\frac{I(T)}{2\pi}
\;=\;T^{2}+T+\tfrac16+O\!\bigl(\tfrac1{T^{2}}\bigr)
\;=\;\Bigl(T+\tfrac12\Bigr)^{2}-\tfrac1{12}+O\!\bigl(\tfrac1{T^{2}}\bigr),
\]
and taking square roots,
\[
\sqrt{\tfrac t{2\pi}}
\;=\;T+\tfrac12-\frac{1}{24\,(T+\tfrac12)}+O\!\bigl(\tfrac1{T^{3}}\bigr).
\]
So $T$ is approximately, but not exactly, $\sqrt{t/2\pi}-\tfrac12$, and
the approximation improves as $T$ grows: the difference between $T+\tfrac12$ and
$\sqrt{t/2\pi}$ is $O(1/T)$, with explicit leading term
$\tfrac1{24}(T+\tfrac12)^{-1}$. (At $T=6$, for instance,
$\sqrt{t/2\pi}=6.4936$ while $T+\tfrac12-\tfrac1{24\cdot6.5}=6.4936$ as
well.)
\end{remark}

The abbreviations of \eqref{eq:IT} are in force for the rest of
the paper: every object below depends on $\sigma$ and $T$ through them,
and we do not carry the arguments. In this notation Siegel's formula
\eqref{eq:siegel-named} becomes
\begin{equation}\label{eq:zeta-T}
\zeta
=
\Sigma_{1} + \Sigma_{2} + R,
\end{equation}
where
\begin{align}
\Sigma_{1} &= \sum_{n=1}^{m} n^{-s}
\quad\text{(the main sum)}, \\[4pt]
\Sigma_{2} &= \chi(s)
\sum_{n=1}^{m} n^{\,s-1}
\quad\text{(the dual sum)},
\end{align}
where $s=\sigma+i\,I(T)$ and $m=\lfloor T\rfloor$. Using the identity
\begin{equation}
\frac{(2\pi)^s e^{\tfrac{\pi i s}{2}}}{\Gamma(s)\bigl(e^{2\pi i s}-1\bigr)}
\;=\;
\frac{\chi(s)}{\,e^{i\pi s}-1\,},
\end{equation}
the remainder becomes
\begin{equation}\label{eq:R-T}
R
=
\frac{\chi(s)}{\,e^{\,i\pi s}-1\,}
\int_{C_2} g(x)\,\mathrm{d}x ,
\end{equation}
where $s=\sigma+i\,I(T)$.

For any integer $q\geq0$, let
\[
R_q(s):=\zeta(s)-\sum_{n=1}^{q}n^{-s}
-\chi(s)\sum_{n=1}^{q}n^{s-1}.
\]
Siegel's standard choice is $q=\lfloor\sqrt{t/(2\pi)}\rfloor$, whereas here
$t=I(T)$ and $q=m=\lfloor T\rfloor$; below we write $R:=R_m$. Thus the map
$t=I(T)$ is exact, but the construction combines that exact
reparameterization with a change from Siegel's standard cutoff.

\section{Decomposing the remainder: \texorpdfstring{$R=R_{1}+R_{2}$}{R = R1 + R2}}\label{sec:decomp}

We now split Siegel's single remainder $R$ into two exact pieces, named
$R_{1}$ and $R_{2}$, stated in the variables $t$, $s$, $m$, and $M$ fixed
in \eqref{eq:IT}. Define
\begin{align}
R_{1} &\;:=\; d_{1}\, M^{-it},
\label{eq:R1-def}\\[4pt]
R_{2} &\;:=\; d_{2}\, M^{\,it}\,
\frac{\chi}{|\chi|},
\label{eq:R2-def}\\
\intertext{with real weights}
d_{1} &\;:=\; |R|\,
\frac{\sin\!\bigl(\omega - \arg R + \arg\chi\bigr)}{\sin\!\bigl(2\omega + \arg\chi\bigr)},
\label{eq:d1}\\[4pt]
d_{2} &\;:=\; |R|\,
\frac{\sin\!\bigl(\omega + \arg R\bigr)}{\sin\!\bigl(2\omega + \arg\chi\bigr)},
\label{eq:d2}\\
\intertext{where $R$ is Siegel's remainder \eqref{eq:R-T}, $\chi=\chi(s)$, and}
\omega &\;:=\; t\,\log M .
\label{eq:omega-def}
\end{align}
For non-integer $T$ the angle $\omega=t\log(m+1)$ carries a
geometric meaning: it is the total \emph{clockwise} rotation accumulated
by the summands of the first Dirichlet sum, so that $-\omega$ is the
absolute direction angle, measured from the positive real axis, of the
next summand: the summand $n^{-s}=n^{-\sigma}e^{-it\ln n}$
rotates by $\theta_n=-t\ln\tfrac{n+1}{n}$ relative to its predecessor, and these
relative rotations telescope,
\begin{equation}\label{eq:omega-telescope}
\theta_1+\theta_2+\cdots+\theta_m
\;=\;
-t\ln\tfrac{2}{1}-t\ln\tfrac{3}{2}-\cdots-t\ln\tfrac{m+1}{m}
\;=\;
-t\ln(m+1)
\;=\;
-\omega,
\end{equation}
so the $(m+1)$st summand points along $e^{-i\omega}$. (At integer $T$,
where $M=m$, the direction $e^{-i\omega}$ is instead that of
the last \emph{included} summand; the two directions are exactly
antiparallel there, and the proof of Theorem~\ref{thm:decomp} shows the
ceiling is the orientation under which the decomposition holds at the
integers as well.)

Writing $\psi=\arg\chi(s)$, the
definitions \eqref{eq:R1-def}--\eqref{eq:R2-def} express $R$ as a combination
of two fixed \emph{unit} directions,
\begin{equation}\label{eq:R-two-dirs}
R \;=\; d_1\,e^{-i\omega} \;+\; d_2\,e^{\,i(\omega+\psi)},
\end{equation}
namely the phases of the next summands of the two Dirichlet sums. The following
proposition collects the basic facts about the two weights.

\begin{proposition}[Reality and positivity of the weights $d_1$ and $d_2$]\label{prop:weights}
Let $\psi=\arg\chi(s)$ and $x=\{T\}$, and let $T$ be
non-integer.
\begin{enumerate}
\item[(i)] Each remainder lies along the \emph{next} summand of its
Dirichlet sum.
\item[(ii)] For every $\sigma$ and every $T$ with
$\sin(2\omega+\psi)\neq0$, the weights $d_1$ and $d_2$ are real.
\item[(iii)] On the critical line $\sigma=\tfrac12$ they are equal and
positive.
\item[(iv)] Fix $\sigma\in(0,1)$, $\sigma\neq\tfrac12$. For all
$T\geq T_0(\sigma)$, both weights are positive except on two windows of
width $O_\sigma(1/T)$, centered $O(1/T)$ from $x=\tfrac14$ and
$\tfrac34$. Inside each window one weight is negative, with
$d_1/d_2\to-1$ at the pole between them, so the windows cannot be
removed.
\end{enumerate}
\end{proposition}

The proof, which occupies Appendix~\ref{sec:positivity}, reduces
everything to a sign criterion on the two numerators of
\eqref{eq:d1}--\eqref{eq:d2}.
Geometrically, positivity is the statement that each
fractional summand points \emph{forward} along its
partial sum, a positive fraction of the next term.

\begin{theorem}\label{thm:decomp}
Suppose $T>0$ and $\sin(2\omega+\psi)\neq0$. Then
\[
R_{1} + R_{2} = R.
\]
Consequently,
\[
\zeta = \Sigma_{1} + R_{1} + \Sigma_{2} + R_{2}.
\]
At a zero of $\sin(2\omega+\psi)$, the combined expression on a punctured
neighborhood has limiting value $R$. At an off-line pole the individual
terms $R_1$ and $R_2$ diverge and are not assigned finite values there.
\end{theorem}

The qualifier in the last sentence costs nothing. On the critical line
$\sigma=\tfrac12$ the vanishing of the denominator $\sin(2\omega+\psi)$ is
always cancelled by the numerators of
\eqref{eq:d1}--\eqref{eq:d2}, so $d_1$ and $d_2$ have no poles there at
all (Appendix~\ref{sec:positivity}). Off the line the two weights do have genuine
poles at those points, but the divergences are equal and opposite and
cancel in the sum, which stays finite and equal to $R$
(Appendix~\ref{sec:pole-locations}).

\begin{proof}
Write $\phi = \arg R$, $\psi = \arg\chi$, and $\omega=t\log M$, so that
$R=|R|e^{i\phi}$ and $\chi/|\chi|=e^{i\psi}$. Since $M$ is a
positive real base and $\omega=t\log M$ by
\eqref{eq:omega-def}, the substitution
$M^{\,\pm it}=e^{\pm i\omega}$ is valid for \emph{every}
$T>0$, integer or not.

\smallskip
\noindent\textbf{Step 1 (substitute).} From \eqref{eq:R1-def}--\eqref{eq:d2},
\[
R_{1} = |R|\,\frac{\sin(\omega - \phi + \psi)}{\sin(2\omega + \psi)}\, e^{-i\omega},
\qquad
R_{2} = |R|\,\frac{\sin(\omega + \phi)}{\sin(2\omega + \psi)}\, e^{i\omega}\, e^{i\psi}.
\]

\noindent\textbf{Step 2 (add).}
\[
R_{1} + R_{2}
= \frac{|R|}{\sin(2\omega + \psi)}
\Bigl[\, e^{-i\omega}\sin(\omega - \phi + \psi) + e^{i(\omega + \psi)}\sin(\omega + \phi)\,\Bigr].
\]

\noindent\textbf{Step 3 (expand via $\sin\theta = \tfrac{e^{i\theta}-e^{-i\theta}}{2i}$).}
\begin{align*}
e^{-i\omega}\sin(\omega - \phi + \psi)
&= \frac{e^{i(\psi - \phi)} - e^{i(-2\omega + \phi - \psi)}}{2i}, \\
e^{i(\omega + \psi)}\sin(\omega + \phi)
&= \frac{e^{i(2\omega + \psi + \phi)} - e^{i(\psi - \phi)}}{2i}.
\end{align*}

\noindent\textbf{Step 4 (add the expansions).} The $e^{i(\psi-\phi)}$ terms
cancel, leaving
\[
\frac{1}{2i}\Bigl[\, e^{i(2\omega + \psi + \phi)} - e^{i(\phi - \psi - 2\omega)}\,\Bigr].
\]

\noindent\textbf{Step 5 (factor $e^{i\phi}$).}
\[
= \frac{e^{i\phi}}{2i}\Bigl[\, e^{i(2\omega + \psi)} - e^{-i(2\omega + \psi)}\,\Bigr]
= e^{i\phi}\sin(2\omega + \psi).
\]

\noindent\textbf{Step 6 (substitute back).}
\[
R_{1} + R_{2}
= \frac{|R|}{\sin(2\omega + \psi)}\cdot e^{i\phi}\sin(2\omega + \psi)
= |R|e^{i\phi} = R.
\]

\smallskip
\noindent\textbf{Integer $T$.} No case split is needed: Steps 1--6 used
only trigonometric algebra and the substitutions recorded at the start of
the proof, all valid for every $T>0$. What deserves spelling out is the
bookkeeping that the ceiling in \eqref{eq:omega-def} handles. At integer
$T$ one has $M=\lceil T\rceil=m$ while $\lfloor T+1\rfloor=m+1$, and by the
definition \eqref{eq:IT} of $I$ the two candidate angles differ by
\[
t\log\tfrac{m+1}{m}
\;=\;
t\log\bigl(1+\tfrac1T\bigr)
\;=\;
\pi(2T+1),
\]
exactly, which for integer $T$ is an \emph{odd} multiple of $\pi$. Passing
from $\omega$ to $\omega'=\omega+\pi(2T+1)$ therefore flips the sign of
both unit directions, $e^{\mp i\omega'}=-e^{\mp i\omega}$, and
simultaneously flips the sign of both weights: each numerator in
\eqref{eq:d1}--\eqref{eq:d2} shifts by an odd multiple of $\pi$ and
changes sign, while the common denominator shifts by the even multiple
$2\pi(2T+1)$ and does not. The two flips cancel in the products, so the
partial summands $R_{1}$ and $R_{2}$, and with them the identity
$R_{1}+R_{2}=R$, come out the same under either convention,
\emph{provided} the weights and the directions are computed from the same
angle. The definitions \eqref{eq:R1-def}--\eqref{eq:d2} do exactly that:
both use $M$, through \eqref{eq:omega-def}. (Mixing the
conventions, with weights from $\lfloor T+1\rfloor$ against directions
from $M$, would flip only one of the two signs and produce
$-R$ at the integers.)

\smallskip
\noindent\textbf{Exceptional points.} The algebra above divides by
$\sin(2\omega+\psi)$ and so applies wherever that denominator is nonzero.
Its zeros are isolated in $T$, and $R$ is continuous, so on a punctured
neighborhood of each zero the identity $R_1+R_2=R$ holds and forces the
limit of $R_1+R_2$ to exist and equal $R$ there, which is the limiting
sense adopted in the statement. At an off-line pole this does not assign
finite values to the two individual terms.
\end{proof}

A formalization of the algebra of Theorem~\ref{thm:decomp} in Lean is given in
Appendix~\ref{app:lean-decomp}. It assumes no axioms of its own; the appendix
records what its three hypotheses are and what is left outside its scope.
A second Lean file, reproduced in Appendix~\ref{app:lean-critical}, formalizes
the critical-line statements that descend from
Proposition~\ref{prop:weights}(iii): the reflection \eqref{eq:R2-conj-R1} and
the equal-leg corollaries of \S\ref{sec:legs}.
Since $R=R_{1}+R_{2}$, we have in particular proved the exact formula
\begin{equation}\label{eq:zeta-ps}
\zeta(s)
=
\Sigma_{1} + R_{1} + R_{2} + \Sigma_{2}.
\end{equation}
\section{The remainders are ``one more partial summand''}\label{sec:summand}

Theorem~\ref{thm:decomp} is more than an algebraic identity. Unlike Siegel's
integral $R$, which bears no
resemblance to the terms of the Dirichlet sums, each of $R_{1}$ and $R_{2}$
is \emph{exactly the next term} of its partial sum, scaled by a real
number (positive in the sense of Proposition~\ref{prop:weights}). We call
such a term a \emph{fractional summand}.

Concretely, the last term actually included in $\Sigma_1$ is the $m$-th
term. For non-integer $T$, the very next term is the $M$-th, and
$R_{1}$ lies along its direction, scaled by a real signed coefficient;
similarly for $R_{2}$ along the next term of $\Sigma_2$. It is literally
shortened only when the corresponding hatted coefficient lies in $[0,1]$.
Table~\ref{tab:frac} uses this non-integer convention.

\begin{table}[htbp]
\centering
\begin{tabular}{c>{\columncolor{gray!15}}llclcl}
\toprule
 & \multicolumn{1}{>{\columncolor{gray!15}}c}{\shortstack[c]{Last\\ summand}}
 & \multicolumn{1}{c}{\shortstack[c]{Next\\ summand}}
 &
 & \multicolumn{1}{c}{\shortstack[c]{Fractional\\ amount}}
 &
 & \multicolumn{1}{c}{\shortstack[c]{Partial\\ summand}} \\
\midrule
$R_{1}$
& $\hphantom{\chi(s)\,}m^{-s}$
& $\hphantom{\chi(s)\,}M^{-s}$
& $\times$
& $\bigl(d_{1}\,M^{\sigma}=\hat d_1\bigr)$
& $=$
& $d_{1}\,M^{-it}$ \\[10pt]
$R_{2}$
& $\chi(s)\,m^{\,s-1}$
& $\chi(s)\,M^{\,s-1}$
& $\times$
& $\Bigl(d_{2}\,\dfrac{M^{\,1-\sigma}}{|\chi(s)|}=\hat d_2\Bigr)$
& $=$
& $d_{2}\,M^{\,it}\,\dfrac{\chi(s)}{|\chi(s)|}$ \\
\bottomrule
\end{tabular}
\caption{Each remainder is the $M$-th (``next'') summand of its Dirichlet sum,
scaled by a real fractional amount; here $s=\sigma+it$ with $t=I(T)$,
$m=\lfloor T\rfloor$, and $M=\lceil T\rceil$. The shaded column is the last
summand actually included in the sum, shown for reference; the remaining
columns display the identity (next summand) $\times$ (fractional amount)
$=$ (partial summand), the fractional amount being exactly the hatted
weight $\hat d_1$ or $\hat d_2$ of \eqref{eq:frac-vs-weight}, and the
product simplifying to $R_{1}$ and $R_{2}$ respectively.}
\label{tab:frac}
\end{table}

Reading \eqref{eq:zeta-ps} in this light, the zeta function takes the
clean ``one more summand'' form
\begin{equation}\label{eq:zeta-clean}
\boxed{\;
\zeta(s)=\sum_{n=1}^{m} n^{-s} + \hat d_1\,M^{-s}
        + \chi(s)\sum_{n=1}^{m} n^{s-1} + \hat d_2\,\chi(s)\,M^{s-1}
\;}
\end{equation}
with $\hat d_1,\hat d_2$ real numbers, positive on the critical line (and off it
outside the pole windows of Proposition~\ref{prop:weights}). In words: the two finite Dirichlet sums,
each carried exactly one fractional term past its cutoff, together
reproduce $\zeta$ with no other remainder term.

The coefficients here are the \emph{fractional amounts} of
Table~\ref{tab:frac}, not the weights $d_1,d_2$ of
\eqref{eq:R1-def}--\eqref{eq:R2-def}. The hat is what distinguishes them,
and the two differ by the length of the summand they scale:
\begin{equation}\label{eq:frac-vs-weight}
\hat d_1=M^{\sigma}\,d_1,
\qquad
\hat d_2=\frac{M^{\,1-\sigma}}{|\chi|}\,d_2.
\end{equation}
The unhatted weights are signed coefficients satisfying
$|R_{1}|=|d_1|$ and $|R_{2}|=|d_2|$. The hatted quantities are signed
normalized coefficients; they are literal fractions of their summands only
where $0\leq\hat d_i\leq1$. At
$\sigma=\tfrac12$, $T=6.18$ for instance, $d_1=0.0875$ while
$\hat d_1=\sqrt7\,d_1=0.2315$.

\section{Periodicity in \texorpdfstring{$T$}{T}}\label{sec:periodicity}

Exact periodicity is rare in the study of $\zeta$. As a function of $t$,
each summand $n^{-it}=e^{-it\log n}$ has the fixed frequency $\log n$.
What grows with $t$ is the cutoff and hence the largest active frequency,
$\log m=\tfrac12\log(t/(2\pi))+o(1)$. The closest
classical notions are Bohr's almost periodicity in the half-plane
$\sigma>1$ \cite[Ch.~XI]{Titchmarsh1986} and the fixed coefficient functions $C_k(p)$ of the
Riemann--Siegel expansion, which depend only on the fractional part
$p=\bigl\{\sqrt{t/2\pi}\bigr\}$. The index $T$ exposes structure of the second kind:
nothing here is exactly periodic. The unnormalized weights decay with $T$,
while their normalized profiles converge to a common unit-cell limit. This
section studies that asymptotic repetition and attempts to explain its origin.

\subsection{\texorpdfstring{$d_1$ and $d_2$}{d1 and d2} on the critical line}\label{sec:periodicity-line}

Figure~\ref{fig:d1-periodicity} plots the weight $d_1$ and the fractional
amount $\hat d_1=\sqrt{m+1}\,d_1$ on the critical line for
$1<T<7$; on the line the two weights coincide, $d_2=d_1$
(Proposition~\ref{prop:weights}(iii)), so the figure shows both at once.
A pattern is visible: $\hat d_1$ traces nearly the same waveform
in every unit interval of $T$, while $d_1$ repeats the same shape with
slowly decaying amplitude and a small jump at each integer, where the
summand it scales hands off from one term to the next.

\begin{figure}[H]
\centering
\includegraphics[width=\textwidth]{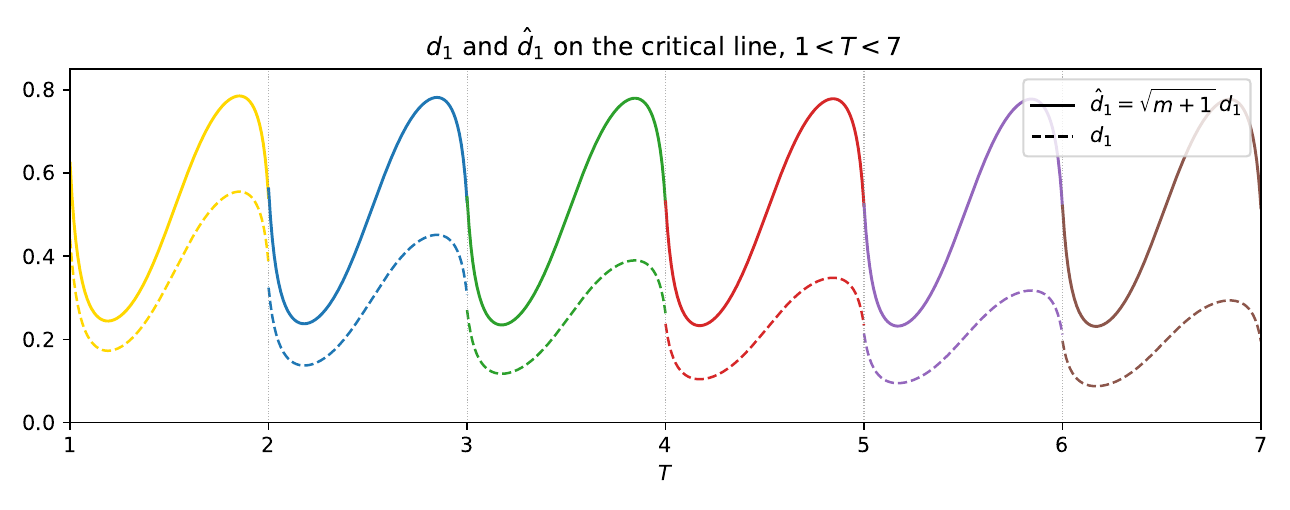}
\caption{$d_1$ (dashed) and $\hat d_1=\sqrt{m+1}\,d_1$ (solid) on the
critical line for $1<T<7$, one color per integer interval of $T$; this
color scheme is held throughout the paper, a separate color for each
interval $[m,m+1]$, to show the periodicity in $T$. Computed from
the exact reduction
$d_1=r/(2\cos u)$ of \eqref{eq:exact-line}. In each unit interval
$\hat d_1$ sweeps nearly the same waveform, while $d_1$ decays like
$T^{-1/2}$ and jumps at the integers. Generated by
\texttt{fig\_d1\_periodicity.py}.}
\label{fig:d1-periodicity}
\end{figure}

\subsection{A continuous and periodic normalization}\label{sec:periodicity-pinned}

The near-repetition has an exact limiting shape.
Figure~\ref{fig:pinned-waves} overlays the six unit intervals of
Figure~\ref{fig:d1-periodicity} after two normalizations that put the
waves on a common footing.

The first normalization rescales and centers: on
$[m,m+1]$ the centered coordinate
\[
P(T)=T^{1/2}\Bigl(d_1(T)-\Phi\bigl(-1,\tfrac12,m+1\bigr)\Bigr),
\qquad
\Phi\bigl(-1,\tfrac12,m+1\bigr)=\sum_{k\ge m+1}\frac{(-1)^{k-m-1}}{\sqrt{k}},
\]
undoes the $T^{-1/2}$ decay and measures $d_1$ against the alternating sum
of the remaining summand lengths ($\Phi$ is the Lerch transcendent).

The second
normalization pins both ends of the cycle (which fall exactly at
integer $T$) so they begin and end at zero:
the endpoint values of $P$ on $[m,m+1]$ are
$+\varepsilon_m$ and $-\varepsilon_{m+1}$, where
$\varepsilon_n=n^{1/2}\bigl(\Phi(-1,\tfrac12,n)-d_1(n^{-})\bigr)$.\footnote{The
right endpoint is the definition of $\varepsilon_{m+1}$. The left endpoint
uses the exact handoff relation $d_1(m^{+})+d_1(m^{-})=m^{-1/2}$, visible
as the jump at each integer in Figure~\ref{fig:d1-periodicity}, together
with $\Phi(-1,\tfrac12,m)+\Phi(-1,\tfrac12,m+1)=m^{-1/2}$. The handoff
relation follows from the reduction $d_1=r/(2\cos u)$ of
\eqref{eq:exact-line}: at $T=m$ the angle $u$ jumps by
$I(m)\log\tfrac{m+1}{m}=\pi(2m+1)$, an odd multiple of $\pi$, so $\cos u$
changes sign, while $r$ drops the newly included summand's contribution
$2m^{-1/2}\cos u$; dividing gives $d_1(m^{+})=m^{-1/2}-d_1(m^{-})$.}
Subtracting the chord through them,
\[
\mathcal W(T)=P(T)-(1-x)\,\varepsilon_m+x\,\varepsilon_{m+1},
\qquad x=T-m,
\]
pins each wave to zero at the two integer endpoints of its unit
interval. Plotted against the
fractional part $x=\{T\}$, the six waves nearly coincide, and as
$T\to\infty$ they converge to the tangent waveform
\begin{equation}\label{eq:W-infty}
\mathcal W_{\infty}(x)
=-\tfrac12\tan(2\pi x)\,
\tan\!\bigl(2\pi(x-\tfrac14)(x-\tfrac34)\bigr),
\end{equation}
at rate $O(1/T)$; the convergence is proved in Appendix~\ref{sec:limits}
(Corollary~\ref{cor:W-limit}).

\begin{figure}[htbp]
\centering
\includegraphics[width=0.9\textwidth]{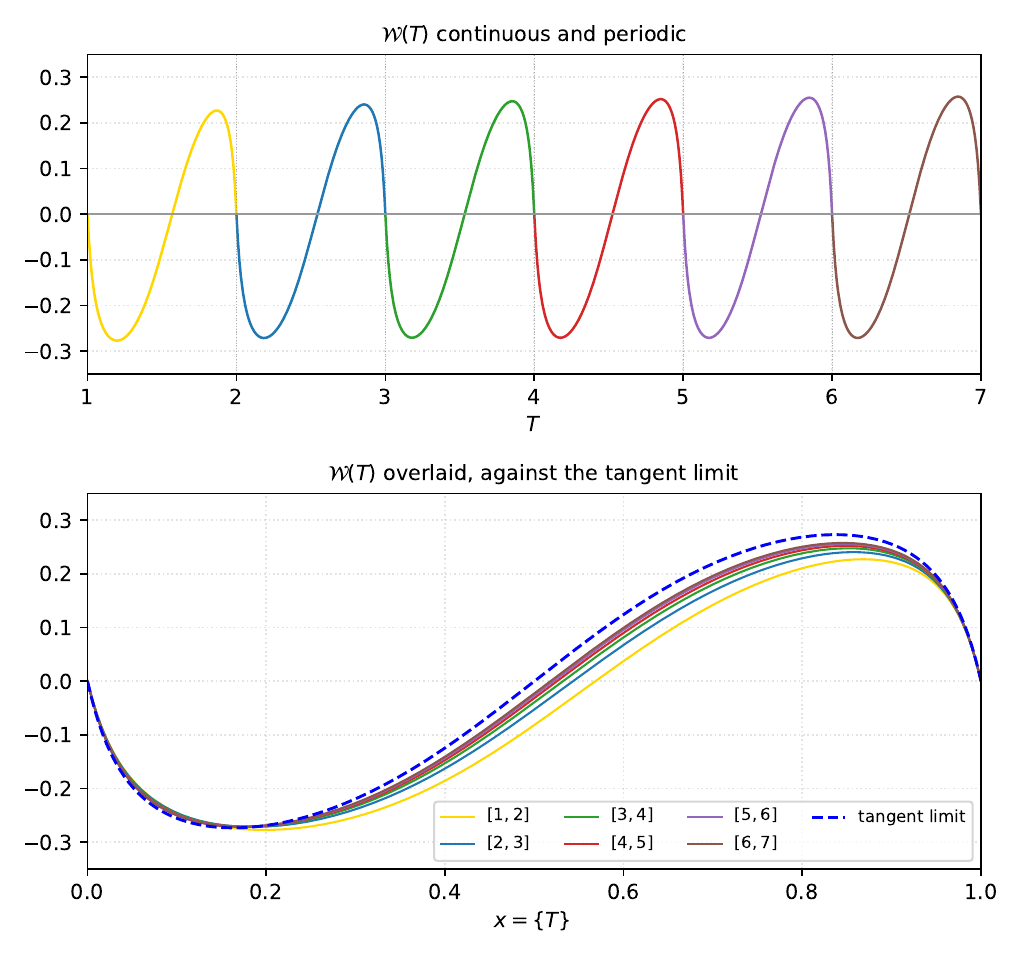}
\caption{Top: the pinned waves $\mathcal W(T)$ for $1<T<7$, one color per
unit interval, the same range as Figure~\ref{fig:d1-periodicity}; the
normalizations have removed the decay and the jumps, leaving a curve that
is zero at every integer. Bottom: the six unit intervals overlaid, each
pinned wave on $[m,m+1]$, $m=1,\dots,6$, plotted against the fractional
part $x=\{T\}$ in the color of the top panel, with the tangent limit
$\mathcal W_{\infty}$ dashed. The deviation from the limit falls from
$0.087$ on $[1,2]$ to $0.025$ on $[6,7]$. Generated by
\texttt{fig\_pinned\_waves.py}.}
\label{fig:pinned-waves}
\end{figure}

Figure~\ref{fig:d-limit} records the same convergence for the fractional
amount itself: at fixed fractional part $x=\{T\}$ the normalized weight
$\hat d_1=\sqrt{m+1}\,d_1$ tends to a closed-form profile
$d(x)=\tfrac12+\mathcal W_{\infty}(x)$, derived in the next subsection.

\begin{figure}[H]
\centering
\includegraphics[width=0.9\textwidth]{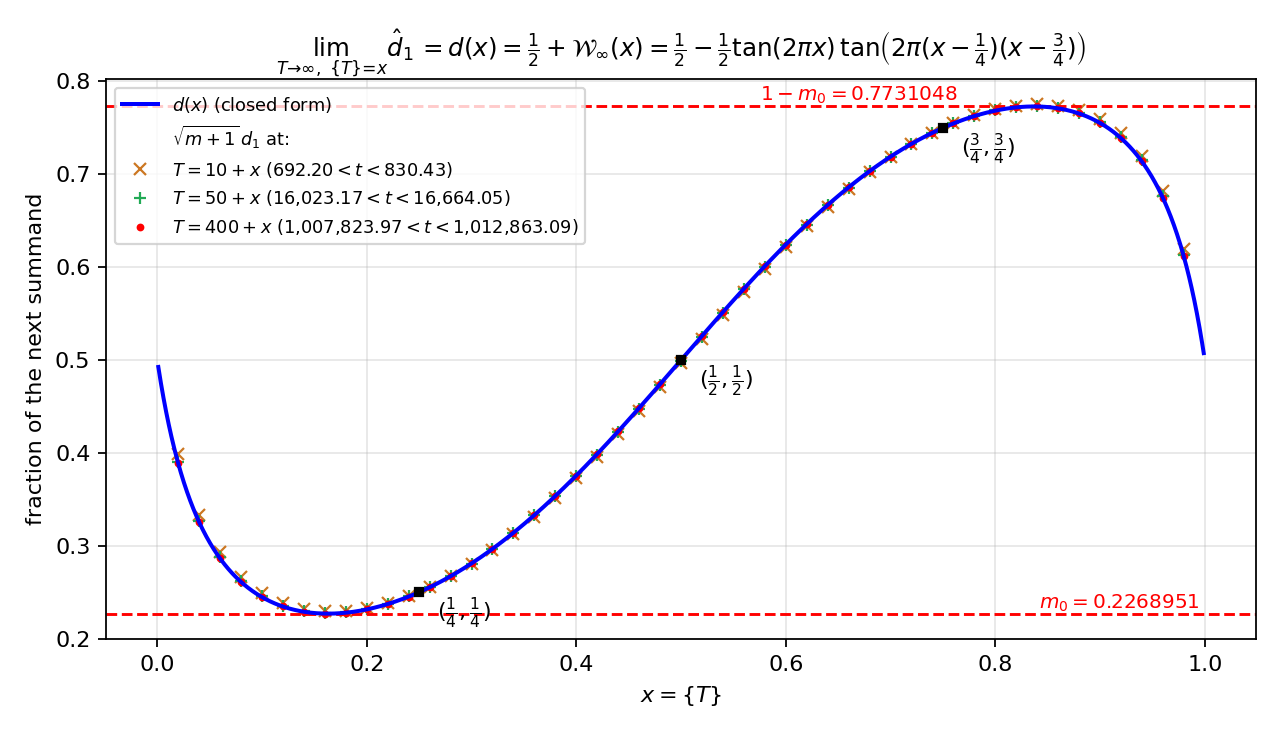}
\caption{The closed-form profile $d(x)$ (blue) against the exact
$\sqrt{m+1}\,d_1$ sampled at $T=10+x$, $50+x$, $400+x$ on the critical
line. Already at $T=10+x$ the samples sit on the curve; the marked squares
are the exact values $(\tfrac14,\tfrac14)$, $(\tfrac12,\tfrac12)$,
$(\tfrac34,\tfrac34)$, and the dashed red lines are the extremes
$m_0=0.2268951\ldots$ and $1-m_0=0.7731048\ldots$ of the profile.
Generated by \texttt{fig\_d\_limit.py}.}
\label{fig:d-limit}
\end{figure}

\subsection{Derivation of the tangent limit from Siegel's \texorpdfstring{$F(\tau)$}{F(tau)}}\label{sec:tangent-derivation}

The tangent waveform \eqref{eq:W-infty} is not an empirical fit: it can
be derived from a function as old as Riemann's unpublished notes. Evaluating the elementary integral in equation~5 of Siegel's 1932
paper~\cite{Siegel1932}, Siegel arrives at a closed form he calls the
``desired result'',
\begin{equation}\label{eq:siegel-F}
F(\tau)\;=\;
\frac{1}{1-e^{-2\pi i \tau}}
\;-\;
\frac{e^{i\pi \tau^{2}}}{e^{i\pi \tau}-e^{-i\pi \tau}} .
\end{equation}
For real $\tau$ its complex conjugate has the compact form
\begin{equation}\label{eq:siegel-Fbar}
\overline{F(\tau)}\;=\;\frac{e^{-i\pi\tau^{2}}-e^{-i\pi\tau}}{2i\sin(\pi\tau)}
\;=\;\frac{e^{-i\pi\tau^{2}}-e^{-i\pi\tau}}{e^{i\pi\tau}-e^{-i\pi\tau}},
\end{equation}
and it is the conjugate we work with below.
The apparent poles at integer $\tau$ are removable: $\tau^{2}-\tau$ is
even there, so the numerator of \eqref{eq:siegel-Fbar} vanishes with the
denominator, and
$\overline{F(\tau)}\to\tfrac12-n$ as $\tau\to n$. Figure~\ref{fig:F-curve}
traces the curve.

\begin{figure}[H]
\centering
\includegraphics[width=0.62\textwidth]{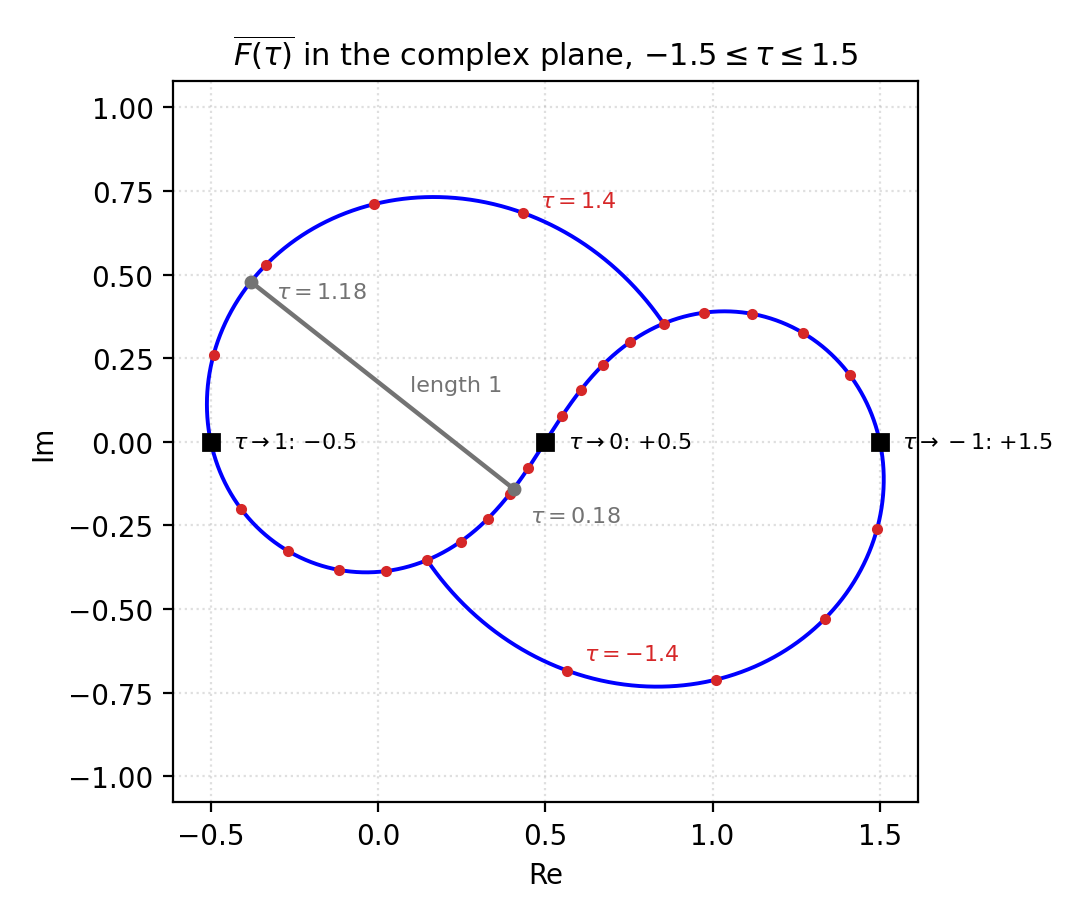}
\caption{The trace of $\overline{F(\tau)}$ of \eqref{eq:siegel-Fbar} in
the complex
plane for $-1.5\le\tau\le1.5$, with red dots every $0.1$ in $\tau$. The
apparent poles at the integers are removable (black squares), with
$\overline{F(\tau)}\to\tfrac12-n$ as $\tau\to n$: the values $-\tfrac12$,
$+\tfrac12$, $+\tfrac32$ at $n=1,0,-1$. Points of the curve one apart in
$\tau$ are exactly one unit apart in the plane:
$\overline{F(\tau+1)}-\overline{F(\tau)}=-e^{-i\pi(\tau^{2}+\tau)}$, a
unit vector; the gray
chord joins $\overline{F(0.18)}$ to $\overline{F(1.18)}$ and has length
exactly $1$. Generated
by \texttt{fig\_F\_curve.py}.}
\label{fig:F-curve}
\end{figure}

\begin{remark}[regarding $F$ as the source of Siegel's $f(s)$]\label{rem:F-to-f}
$F$ is more than an intermediate step in Siegel's derivation: it is the
source of his function $f$, the function that carries the first half of
zeta, $f(s)\approx\Sigma_1+R_{1ak}$ (\S\ref{sec:kuznetsov} records the
details as \eqref{eq:siegel-f}). In his \S3 Siegel multiplies his
equation~5, that is \eqref{eq:siegel-F}, by $u^{-s}$ and integrates along
the bisector of the first quadrant, and under that operation the two
terms of $F$ separate into the two halves of zeta. The pole term
$1/(1-e^{-2\pi iu})$ yields
$-(2\pi)^{s-1}e^{\,i\pi(1-s)/2}\,\Gamma(1-s)\,\zeta(1-s)$, while the
Gaussian term yields $f$ itself:
\begin{equation}\label{eq:f-from-F}
f(s)\;=\;\bigl(e^{\pi i s}-1\bigr)
\int_{0}^{e^{i\pi/4}\infty}
\frac{u^{-s}\,e^{\pi i u^{2}}}{e^{\pi i u}-e^{-\pi i u}}\,\mathrm{d}u
\qquad (\sigma<0),
\end{equation}
the restriction $\sigma<0$ being needed only for convergence at the
endpoint $u=0$; $f$ is continued analytically from there, and the contour
form \eqref{eq:siegel-58} of \S\ref{sec:kuznetsov} converges for every
$s$. The same separation can be read off \eqref{eq:siegel-Fbar}: since
$1/(1-e^{2\pi i\tau})=-e^{-\pi i\tau}/(2i\sin\pi\tau)$, the numerator
term $e^{-\pi i\tau^{2}}$ of $\overline{F}$ is the $f(s)$ side and
$-e^{-\pi i\tau}$ is the $\zeta(1-s)$ side, so the two lobes of
Figure~\ref{fig:F-curve} carry the two halves in their numerator. So $F$
is the common source of both $\Sigma_1+R_{1ak}$ and its reflected
counterpart (the yin-yang connection is drawn in
\S\ref{sec:yinyang-inf}), and this common origin is one route into the
convergence proved by other means in Theorem~\ref{thm:yin-convergence}.
\end{remark}

Everything we need is contained in the polar form of $\overline F$.
Factoring the
mean exponent $e^{-i\pi(\tau^{2}+\tau)/2}$ out of the numerator turns
the difference of exponentials into a sine:
\begin{equation}\label{eq:F-polar}
\overline{F(\tau)}
=-\,\frac{\sin\!\bigl(\tfrac{\pi}{2}(\tau^{2}-\tau)\bigr)}{\sin(\pi\tau)}\;
e^{-\tfrac{i\pi}{2}(\tau^{2}+\tau)} .
\end{equation}
Now rescale. Let $a=\sqrt{t/2\pi}$, $N=\lfloor a\rfloor$, and
$\hat p=a-N$ be the classical Riemann--Siegel quantities (the same ones
used throughout Appendix~\ref{sec:positivity}), and substitute
$\tau=2\hat p-\tfrac12$ into \eqref{eq:F-polar}. The three pieces
collapse: $\sin(\pi\tau)=-\cos(2\pi\hat p)$,
$\tfrac{\pi}{2}(\tau^{2}-\tau)=2\pi\bigl(\hat p^{2}-\hat p-\tfrac1{16}\bigr)
+\tfrac{\pi}{2}$, and
$\tfrac{\pi}{2}(\tau^{2}+\tau)=2\pi\hat p^{2}-\tfrac{\pi}{8}$, so that
\begin{equation}\label{eq:F-C0}
\overline{F}\bigl(2\hat p-\tfrac12\bigr)
= C_0(\hat p)\,e^{-iA},
\qquad
A=2\pi\hat p^{2}-\tfrac\pi8,
\qquad
C_0(\hat p)=\frac{\cos\!\bigl(2\pi(\hat p^{2}-\hat p-\tfrac1{16})\bigr)}
{\cos(2\pi\hat p)} .
\end{equation}
In words: evaluated at $\tau=2\hat p-\tfrac12$, twice the fractional
part less one half, $\overline F$ has amplitude exactly $C_0(\hat p)$, the
leading coefficient of the Riemann--Siegel correction, and phase $-A$;
both reappear, together with $B=2\pi\hat p$, in
Lemma~\ref{lem:C0-trig}.

The limit now assembles from the critical-line machinery of
Appendix~\ref{sec:positivity}. On the line the exact reduction
\eqref{eq:exact-line} reads $d_1=r/(2\cos u)$; Gabcke's form of the
Riemann--Siegel expansion \eqref{eq:gabcke} gives $r$ the leading term
$(-1)^{N-1}a^{-1/2}\,C_0(\hat p)$, the amplitude of \eqref{eq:F-C0}; and
the phase expansion \eqref{eq:cosu} gives
$\cos u=(-1)^{m+1}\cos\beta+O(1/T)$, where
$\beta=2\pi(x-\tfrac14)(x-\tfrac34)$ and $x=\{T\}$. In each half of the
unit interval the offset is $\hat p=x\pm\tfrac12$ \eqref{eq:cut-cases},
and the sign factors combine as in Steps~3 and~4 of the proof of
Theorem~\ref{thm:line-bound}, so that
\[
\hat d_1=\sqrt{m+1}\,d_1
\;\longrightarrow\;
\frac{C_0(x+\tfrac12)}{2\cos\beta}
\quad(x<\tfrac12),
\qquad
\hat d_1
\;\longrightarrow\;
1-\frac{C_0(x-\tfrac12)}{2\cos\beta}
\quad(x>\tfrac12),
\]
the extra $1$ in the second case being the newly included summand. Two
lines of the product formula
$\cos(A\mp B)=\cos A\cos B\pm\sin A\sin B$ turn both branches into one and
the same expression,
\[
\hat d_1
\;\longrightarrow\;
\tfrac12\bigl(1-\tan(2\pi x)\tan\beta\bigr)
\;=\;\tfrac12+\mathcal W_{\infty}(x)
\;=:\;d(x),
\]
the tangent waveform of \eqref{eq:W-infty} lifted by $\tfrac12$.

Written in branch form, with $g$ recording the $x<\tfrac12$ branch,
\begin{equation}\label{eq:d-profile}
d(x)=
\begin{cases}
g(x), & x<\tfrac12,\\[2pt]
1-g(1-x), & x\ge\tfrac12,
\end{cases}
\qquad
g(y)=\frac{-\cos\bigl(2\pi y^{2}-\tfrac{5\pi}{8}\bigr)}
{2\cos(2\pi y)\,\cos\bigl(2\pi y(1-y)-\tfrac{3\pi}{8}\bigr)}.
\end{equation}
In \eqref{eq:d-profile} the point $y=\tfrac14$ is a
removable $0/0$ singularity of $g$ (the numerator vanishes together with
$\cos(2\pi y)$), filled in by continuity with $g(\tfrac14)=\tfrac14$;
this supplies the values of $d$ at $x=\tfrac14$ and, through the second
branch, at $x=\tfrac34$. The $d(x)$ limit profile passes
through $(\tfrac14,\tfrac14)$, $(\tfrac12,\tfrac12)$,
$(\tfrac34,\tfrac34)$ and stays inside the band $[m_0,\,1-m_0]$ with
$m_0=0.2268951\ldots$ (at finite $T$ the excursions are slightly
larger but still bounded between $\tfrac15$ and $\tfrac45$;
Theorem~\ref{thm:d1-bounds} in Appendix~\ref{sec:d1-bounds}); in particular
the fractional amount never
approaches either end of its summand.

Figure~\ref{fig:d-limit} plots the $d(x)$ limit profile against the
exact normalized weight $\hat d_1=\sqrt{m+1}\,d_1$: the convergence is
at rate $O(1/T)$, and the samples at $T=10+x$
already sit on the curve. Theorem~\ref{thm:d1-limit} in
Appendix~\ref{sec:limits} makes the derivation of this subsection rigorous,
controlling every error term including those at the transition points
$x=0,\tfrac12,1$.

\subsection{Separating \texorpdfstring{$\overline F$}{F bar} into two
curves: yin and yang at infinity}\label{sec:yinyang-inf}

Because the curve in Figure~\ref{fig:F-curve} resembles a
\emph{taijitu}\,(\raisebox{-0.35ex}{\includegraphics[height=1.9ex]{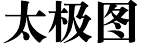}}),
we call its two component curves yin and yang. The pair of curves turns the limit profile
\eqref{eq:d-profile} into a geometric picture. Write $x=\{T\}$ for the
fractional part, so that $x$ runs from
$0$ to $1$ across each unit interval of $T$, and write
\begin{equation}\label{eq:psi-fn}
\Psi(x)\;:=\;
\frac{\cos\!\bigl(2\pi(x^{2}-x-\tfrac1{16})\bigr)}{\cos(2\pi x)}
\end{equation}
for the Riemann--Siegel correction function: it is the same function as
$C_0$ of \eqref{eq:F-C0} (we keep Gabcke's symbol $C_0$ for the analytic
input and $\Psi$ for the limit curve), and it appears in \S15.3 of
Titchmarsh's book~\cite{Titchmarsh1986} and in \S5.4.1 of Pugh's
thesis~\cite{Pugh1998}. Define the two curves
\begin{equation}\label{eq:yin-inf}
\mathrm{Yin}_{\infty}(x)\;:=\;
\Psi(x)\Bigl(-\cos\bigl(2\pi(x^{2}-\tfrac1{16})\bigr)
+i\sin\bigl(2\pi(x^{2}-\tfrac1{16})\bigr)\Bigr)+1
\;=\;1-\Psi(x)\,e^{-2\pi i(x^{2}-\frac1{16})}
\end{equation}
and
\begin{equation}\label{eq:yang-inf}
\mathrm{Yang}_{\infty}(x)\;:=\;
-e^{4\pi ix}\,\mathrm{Yin}_{\infty}(x)+e^{4\pi ix}
\;=\;e^{4\pi ix}\bigl(1-\mathrm{Yin}_{\infty}(x)\bigr).
\end{equation}
Both are built from $\overline F$ alone: the exponent
$2\pi(x^{2}-\tfrac1{16})$ is the phase $A$ of \eqref{eq:F-C0}, so
$\Psi(x)\,e^{-iA}=\overline F(2x-\tfrac12)$ and
\begin{equation}\label{eq:yinyang-from-F}
\mathrm{Yin}_{\infty}(x)=1-\overline F\bigl(2x-\tfrac12\bigr),
\qquad
\mathrm{Yang}_{\infty}(x)=e^{4\pi ix}\,\overline F\bigl(2x-\tfrac12\bigr).
\end{equation}
As $x$ runs over the unit interval the argument $\tau=2x-\tfrac12$ runs
from $-\tfrac12$ to $\tfrac32$, from one lobe-meeting point of
Figure~\ref{fig:F-curve} to the other, covering both lobes once. The yin
curve is the point reflection of that arc through $\tfrac12$; the yang
curve is the same arc spun by the unit factor $e^{4\pi ix}$, two full
turns per unit interval. Figure~\ref{fig:yinyang-inf} draws the pair.

\begin{figure}[H]
\centering
\includegraphics[width=0.7\textwidth]{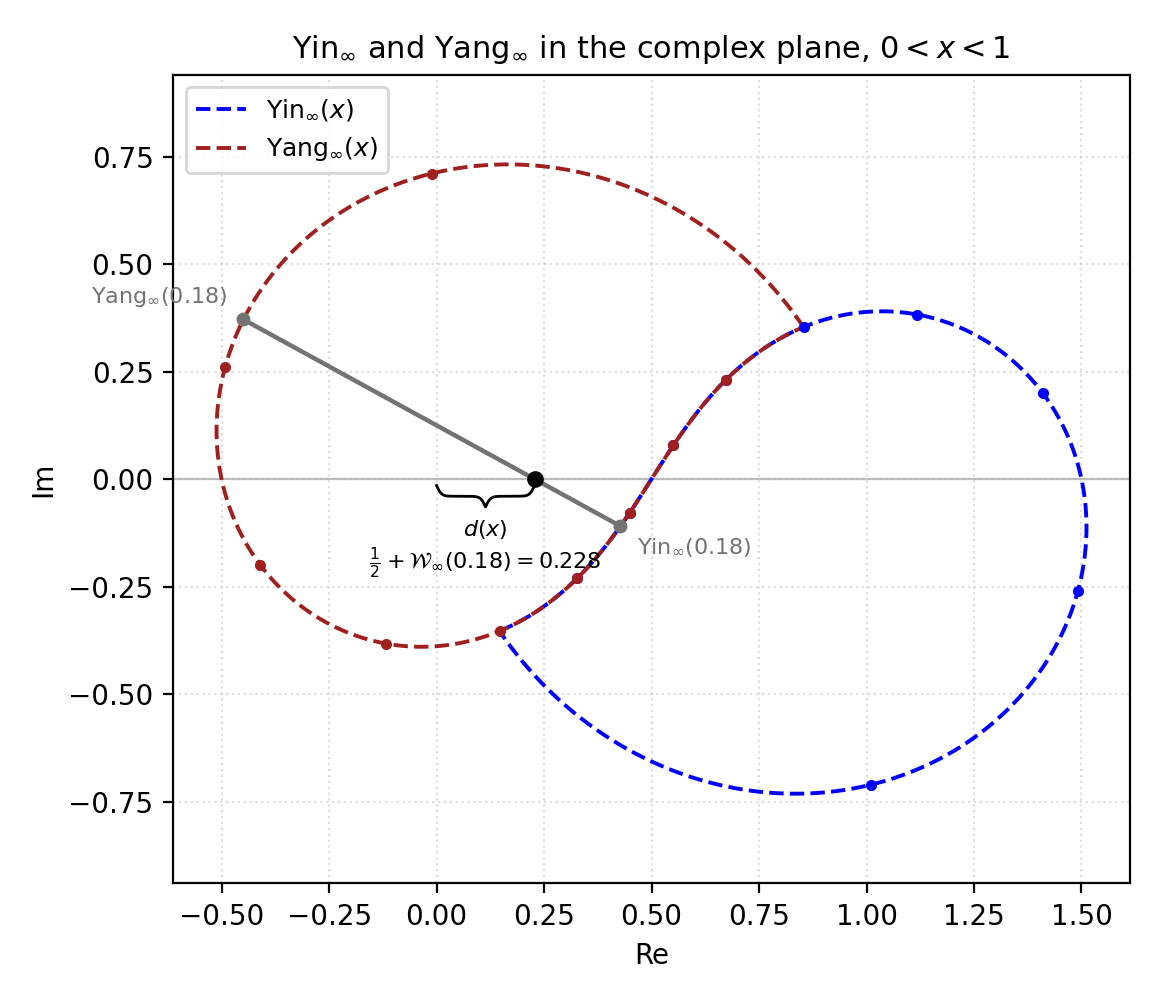}
\caption{The curves $\mathrm{Yin}_{\infty}(x)$ (blue, dashed) and
$\mathrm{Yang}_{\infty}(x)$ (dark red, dashed) for $0<x<1$, with dots every $0.1$
in $x$; both are the single curve $\overline F$ of
Figure~\ref{fig:F-curve} repackaged as in \eqref{eq:yinyang-from-F}. The
gray chord joins the two points at $x=0.18$; it crosses the real axis at
$d(0.18)=\tfrac12+\mathcal W_{\infty}(0.18)\approx0.228$ (black dot), as
Proposition~\ref{prop:yinyang-chord} asserts for every $x$. Generated by
\texttt{fig\_yinyang\_infinity.py}.}
\label{fig:yinyang-inf}
\end{figure}

The pair of curves is more than a repackaging: the chord between them
crosses the real axis at exactly the limit profile $d(x)$.

\begin{proposition}[the chord recovers the limit profile]
\label{prop:yinyang-chord}
For $x\in(0,1)$ with $x\neq\tfrac14,\tfrac34$, the line through
$\mathrm{Yin}_{\infty}(x)$ and $\mathrm{Yang}_{\infty}(x)$ meets the
real axis at the single point
\[
\tfrac12+\mathcal W_{\infty}(x)\;=\;d(x),
\]
the limit profile of \eqref{eq:d-profile}.
\end{proposition}

\begin{proof}
Abbreviate $A=2\pi x^{2}-\tfrac\pi8$, $B=2\pi x$, and $\Psi=\Psi(x)$, so
that $\mathrm{Yin}_{\infty}=1-\Psi e^{-iA}$ and
$\mathrm{Yang}_{\infty}=\Psi e^{i(2B-A)}$. The imaginary parts are
$\Psi\sin A$ and $\Psi\sin(2B-A)$, the real parts $1-\Psi\cos A$ and
$\Psi\cos(2B-A)$. The line through points $z_{1},z_{2}$ with
$\operatorname{Im}z_{1}\neq\operatorname{Im}z_{2}$ crosses the real
axis at
\[
x_{0}
=\frac{\operatorname{Re}z_{1}\operatorname{Im}z_{2}
-\operatorname{Re}z_{2}\operatorname{Im}z_{1}}
{\operatorname{Im}z_{2}-\operatorname{Im}z_{1}} .
\]
Substituting, the denominator is
$\Psi\bigl(\sin(2B-A)-\sin A\bigr)=2\Psi\cos B\sin(B-A)$, and the
numerator is $\Psi\bigl(\sin(2B-A)-\Psi\sin 2B\bigr)$. Since
$\Psi=\cos(A-B)/\cos B$ (the form of $C_0$ used in
Lemma~\ref{lem:C0-trig}) and $2\sin B\cos(A-B)=\sin A+\sin(2B-A)$, the
numerator collapses to $-\Psi\sin A$, so
\[
x_{0}=\frac{-\sin A}{2\cos B\,\sin(B-A)} .
\]
Now let $\beta=2\pi(x-\tfrac14)(x-\tfrac34)$ as in
\S\ref{sec:tangent-derivation}. Then $A-B=\beta-\tfrac\pi2$, so
$\sin(B-A)=\cos\beta$; and $B+\beta=A+\tfrac\pi2$, so
$-\sin A=\cos(B+\beta)$. Therefore
\[
x_{0}
=\frac{\cos(B+\beta)}{2\cos B\cos\beta}
=\tfrac12\bigl(1-\tan(2\pi x)\tan\beta\bigr)
=\tfrac12+\mathcal W_{\infty}(x)
=d(x).
\]
Since $|\beta|\le\tfrac{3\pi}8<\tfrac\pi2$ on $[0,1]$, the factor
$\cos\beta$ never vanishes; the excluded points $x=\tfrac14,\tfrac34$
are exactly where $\cos B=0$, and there $\sin A=\sin(2B-A)=0$: both ends
of the chord land on the real axis itself, and the crossing point is
recovered instead by continuity, $d(\tfrac14)=\tfrac14$ and
$d(\tfrac34)=\tfrac34$.
\end{proof}

\subsection{General yin and yang curves}\label{sec:yinyang-general}

What are the yin and yang curves when $\sigma$ is not $\tfrac12$ and $T$
is not infinite? Nothing in their construction requires either limit.
Return to the exact
decomposition $\zeta=\Sigma_1+\Sigma_2+R$ of \eqref{eq:zeta-T} and, as
in \S\ref{sec:decomp}, write $t=I(T)$, $s=\sigma+it$, and
$M=\lceil T\rceil$. The natural frame in which to watch the remainder is
the one that makes the next summand of the main sum the unit segment:
translate the end of the main sum, $\Sigma_1$, to the origin, and divide
by the summand vector $M^{-s}$, that is,
\begin{equation}\label{eq:frame-map}
z\;\longmapsto\;(z-\Sigma_1)\,M^{s}.
\end{equation}
In this coordinate frame, the next summand of the main sum runs from $0$ to $1$.
Define two points in this frame,
\begin{equation}\label{eq:yin1}
\mathrm{Yin}(s)\;:=\;R\,M^{s}
\end{equation}
and
\begin{equation}\label{eq:yang1}
\mathrm{Yang}(s)\;:=\;\mathrm{Yin}(s)\;-\;\chi(s)\,M^{2s-1}.
\end{equation}
The first point is where the local coordinate frame map
\eqref{eq:frame-map} sends $\Sigma_1+R=\zeta-\Sigma_2$, the end of the
remainder. The second point is where the transformation to the local
frame sends $\Sigma_1+R-\chi M^{s-1}$, the
same point after the dual sum takes one more step $n=M$: subtracting the
next dual summand $\chi M^{s-1}$ and applying \eqref{eq:frame-map}
contributes the $\chi M^{2s-1}$ term. As $T$ sweeps with $\sigma$ fixed,
the two points trace the general yin and yang
curves.\footnote{The construction has an exact mirror. Sending the next
summand of the \emph{dual} sum to the unit segment instead (translate
$\Sigma_1+R$ to the origin and divide by the dual summand vector
$-\chi M^{s-1}$, that is, $z\mapsto(\Sigma_1+R-z)\,M^{1-s}/\chi$), the
two ends of the next \emph{main} summand land at $R\,M^{1-s}/\chi$ (the
image of $\Sigma_1$) and $R\,M^{1-s}/\chi-M^{1-2s}/\chi$ (the image of
$\Sigma_1+M^{-s}$), and the mirrored chord crosses the real axis at
$\hat d_2=d_2M^{1-\sigma}/|\chi|$, the fractional amount of the dual
sum's next summand.}

The chord property of Proposition~\ref{prop:yinyang-chord} is not
special to the limit: it is the $T\to\infty$ limit of an identity that
holds exactly at every finite $T$, which we state next. Figure~\ref{fig:yinyang-general} draws
the pair for $1\leq T\leq2$ on the critical line over the limit curves
of \S\ref{sec:yinyang-inf}, with the chord at $T=1.18$.

\begin{figure}[H]
\centering
\includegraphics[width=\textwidth]{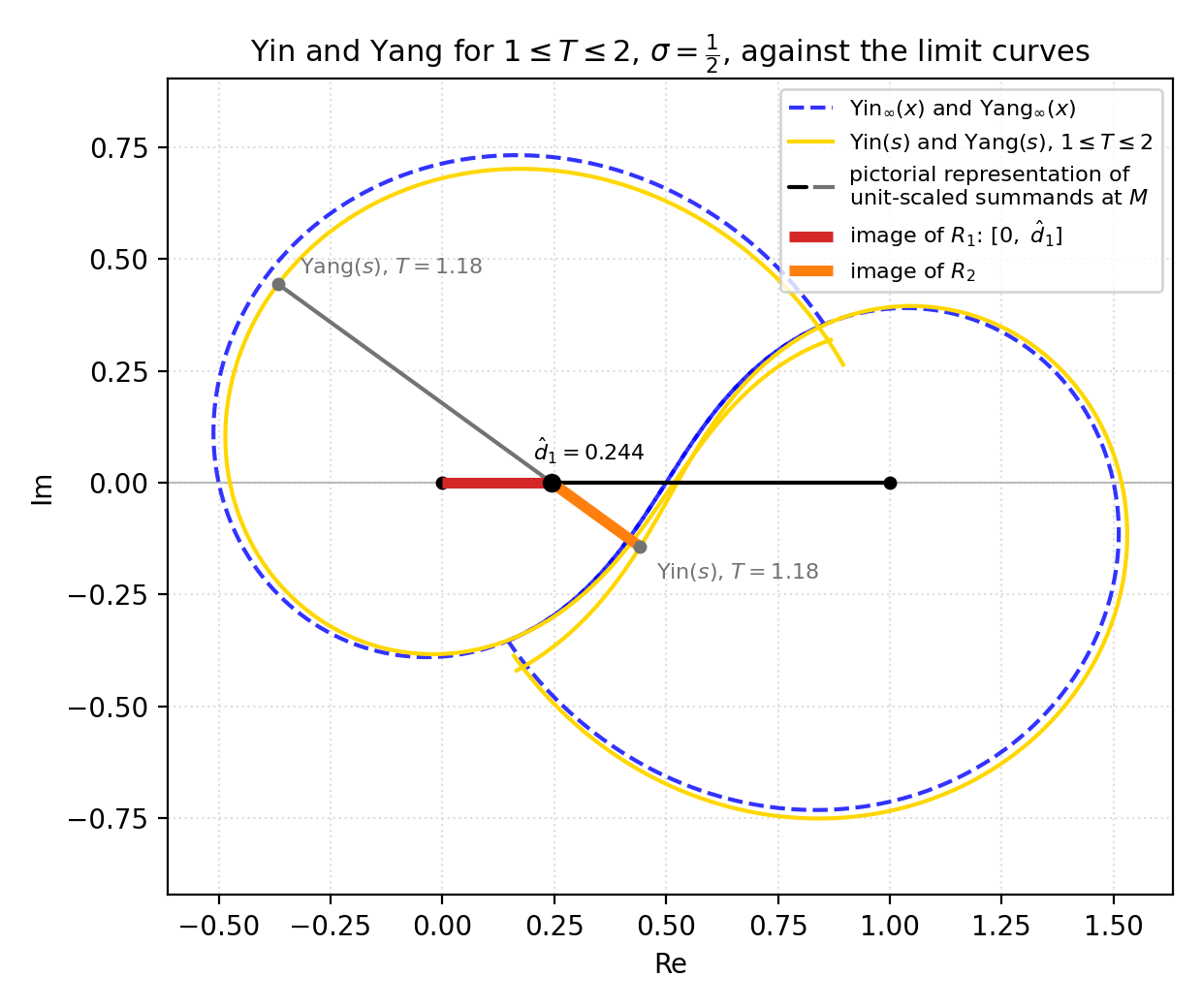}
\caption{The general yin and yang curves for $1\leq T\leq2$ at
$\sigma=\tfrac12$ (solid, the color of the $[1,2]$ interval of
Figure~\ref{fig:pinned-waves}), drawn over
the limit curves $\mathrm{Yin}_{\infty}$ and $\mathrm{Yang}_{\infty}$ of
\S\ref{sec:yinyang-inf} (dashed, both in the blue of
Figure~\ref{fig:F-curve}): even on this very
first interval the finite curves track their limits closely. The gray
chord joins the two points at
$T=1.18$ and crosses the real axis at the black dot
$\hat d_1=0.244$. The red piece $[0,\hat d_1]$ of the axis
is the in-frame image of $R_1$, and the orange piece from the crossing
to the yin point is the image of $R_2$. The black segment from $0$ to
$1$ and the gray chord are together a pictorial representation of the
unit-scaled summands at $M$: the black segment is the next summand of
the main sum carried to $[0,1]$ by the local coordinate frame
transformation, and the chord is the next
summand of the dual sum in the same frame (on the critical line both
have length exactly $1$). Generated by
\texttt{fig\_yinyang\_general.py}.}
\label{fig:yinyang-general}
\end{figure}

\begin{proposition}[the chord recovers the fractional amount of the next main summand]
\label{prop:yinyang-chord-finite}
Fix $\sigma$ and non-integer $T$ with $\sin(2\omega+\psi)\neq0$, where
$\omega=t\log M$ and $\psi=\arg\chi(s)$ as in \S\ref{sec:decomp}. Then
$\mathrm{Yin}(s)$ and $\mathrm{Yang}(s)$ both lie on the line
through the real point $\hat d_1=d_1M^{\sigma}$ with direction
$e^{i(2\omega+\psi)}$:
\begin{equation}\label{eq:yinyang-line}
\mathrm{Yin}=\hat d_1+\bigl(d_2M^{\sigma}\bigr)e^{i(2\omega+\psi)},
\qquad
\mathrm{Yang}=\hat d_1+\bigl(d_2M^{\sigma}-|\chi|M^{2\sigma-1}\bigr)
e^{i(2\omega+\psi)},
\end{equation}
with real coefficients. In particular the line through the two points
meets the real axis at the single point
\[
\hat d_1\;=\;d_1M^{\sigma}\;=\;\frac{R_1}{M^{-s}},
\]
the fractional amount of Table~\ref{tab:frac}: the ratio is real
because $R_1$ lies along the next main summand $M^{-s}$, and it equals
$|R_1|/|M^{-s}|$ whenever $d_1>0$.
\end{proposition}

\begin{proof}
Since $M$ is a positive real base, $M^{s}=M^{\sigma}e^{i\omega}$.
Substituting the two-direction form \eqref{eq:R-two-dirs} of the
remainder, $R=d_1e^{-i\omega}+d_2e^{i(\omega+\psi)}$, into
\eqref{eq:yin1},
\[
\mathrm{Yin}
=R\,M^{\sigma}e^{i\omega}
=d_1M^{\sigma}+d_2M^{\sigma}e^{i(2\omega+\psi)}.
\]
For the yang point, $\chi M^{2s-1}=|\chi|M^{2\sigma-1}e^{i(2\omega+\psi)}$,
and subtracting it from $\mathrm{Yin}$ gives the second formula of
\eqref{eq:yinyang-line}. The coefficients $d_1M^{\sigma}$,
$d_2M^{\sigma}$, and $d_2M^{\sigma}-|\chi|M^{2\sigma-1}$ are real by
Proposition~\ref{prop:weights}(ii), and the direction
$e^{i(2\omega+\psi)}$ is not horizontal exactly when
$\sin(2\omega+\psi)\neq0$, so the line meets the real axis at the single
point $d_1M^{\sigma}=\hat d_1$.
\end{proof}

Whenever $0<d_2<|\chi|M^{\sigma-1}$, that is, whenever the dual
fractional amount is strictly between $0$ and $1$, the two coefficients
in \eqref{eq:yinyang-line} have opposite signs: the yin and yang points
then lie on opposite sides of the real axis and the segment itself, not
merely its extension, crosses at $\hat d_1$.

Nothing in the construction is tied to the critical line, and the curves
converge to the limit curves of \S\ref{sec:yinyang-inf} at every $\sigma$.
Figure~\ref{fig:yinyang-asym} draws the yin curve against the
\emph{reflected} yang curve $1-\mathrm{Yang}(s)$ over $0\leq T\leq7$ at
three values of $\sigma$. The point reflection $z\mapsto1-z$ exchanges
the two ends of the unit segment, and under it the yang family lands on
top of the yin family; the two limit traces coincide exactly, since a
direct computation from \eqref{eq:yin-inf}--\eqref{eq:yang-inf}, using
$\Psi(1-x)=\Psi(x)$, gives
\begin{equation}\label{eq:refl-yang}
1-\mathrm{Yang}_{\infty}(x)\;=\;\mathrm{Yin}_{\infty}(1-x).
\end{equation}
So a single dashed teardrop, the trace of $\mathrm{Yin}_{\infty}$,
is the limit of both families.

\begin{figure}[H]
\centering
\includegraphics[width=0.92\textwidth]{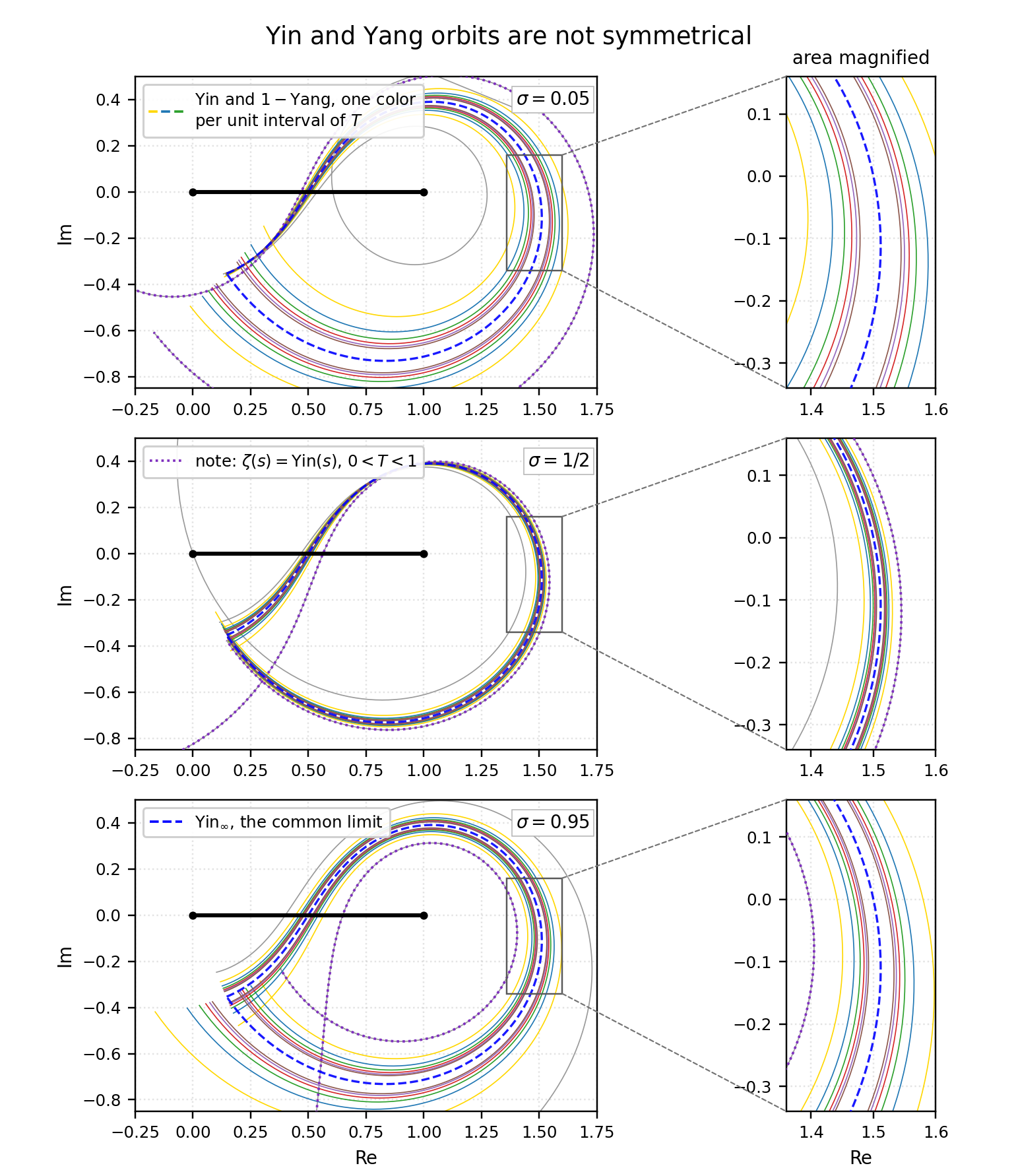}
\caption{The yin curve against the reflected yang curve
$1-\mathrm{Yang}(s)$, to make comparison easier. The figure is for
$0\leq T\leq7$, one loop of each per unit
interval of $T$, out to $t\approx353$. One row per $\sigma$, all on the
same scale: $\sigma=0.05$ on top, the critical line in the middle,
$\sigma=0.95$ below. Both curves of each unit interval wear that
interval's color, the colors of Figures~\ref{fig:d1-periodicity}
and~\ref{fig:pinned-waves}; the first period
$0\leq T\leq1$ is gray, and there the yin curve is $\zeta(\sigma+it)$
itself (dotted purple). The
black segment from $0$ to $1$ is the next summand of the main sum
carried to the unit interval by the frame map \eqref{eq:frame-map}, as
in Figure~\ref{fig:yinyang-general}. The
dashed blue curve, appearing between them, is the limit teardrop
$\mathrm{Yin}_{\infty}$ of \S\ref{sec:yinyang-inf}, which by
\eqref{eq:refl-yang} is also the trace of the reflected limit
$1-\mathrm{Yang}_{\infty}$. \emph{Left:} all seven periods, the view
cropped to the cluster; the stray gray arcs crossing it belong to the
first period, where $t<13.6$ and the frame is still far from its limit.
\emph{Right:} the boxed window magnified: at every $\sigma$ the later
periods hug the dashed limit more closely, in interval order. Generated
by \texttt{fig\_yinyang\_asym.py}.}
\label{fig:yinyang-asym}
\end{figure}

The convergence visible in Figure~\ref{fig:yinyang-asym} is a theorem,
and the machinery used in Appendix~\ref{sec:positivity} already
contains most of its proof:
by \eqref{eq:yin1} the yin point is $R$ carried through the frame map
\eqref{eq:frame-map}, and Appendix~\ref{sec:positivity} knows the size, the
sign, and the phase of $R$ to exactly the accuracy needed. The one
genuinely new ingredient is a polar form of the limit curves.

\begin{lemma}[the two limit curves are really a single curve, described
explicitly: $\mathrm{Yang}_{\infty}(x)=\mathrm{Yin}_{\infty}(x-\tfrac12)$]
\label{lem:yinf-polar}
Write $A(x)=2\pi\bigl(x^{2}-\tfrac1{16}\bigr)$ as in \eqref{eq:F-C0} and
$\theta(x)=2\pi x(1-x)-\tfrac{3\pi}8$, the phase of \eqref{eq:cosu}.
Then for all real $x$, with removable points filled in by continuity,
\begin{equation}\label{eq:yinf-polar}
\mathrm{Yin}_{\infty}(x)\;=\;1-\Psi(x)\,e^{-iA(x)}
\;=\;\Psi\bigl(x+\tfrac12\bigr)\,e^{\,i\theta(x)},
\end{equation}
and the yang limit is the yin limit half a unit behind in the argument:
\begin{equation}\label{eq:yinf-halflag}
\mathrm{Yang}_{\infty}(x)
\;=\;\mathrm{Yin}_{\infty}\bigl(x-\tfrac12\bigr)
\;=\;\mathrm{Yin}_{\infty}(x)-e^{\,2i\theta(x)}.
\end{equation}
\end{lemma}

\begin{proof}
For \eqref{eq:yinf-polar}, multiply through by $\cos(2\pi x)$ and write
$\varphi=2\pi x$. On the left, $\Psi(x)\cos(2\pi x)=\cos(A-\varphi)$,
since $2\pi(x^{2}-x-\tfrac1{16})=A(x)-\varphi$. On the right,
$\cos\bigl(2\pi(x+\tfrac12)\bigr)=-\cos(2\pi x)$ and
$(x+\tfrac12)^{2}-(x+\tfrac12)-\tfrac1{16}=x^{2}-\tfrac5{16}$, so
$\Psi(x+\tfrac12)\cos(2\pi x)=-\cos\bigl(A(x)-\tfrac\pi2\bigr)=-\sin A$;
also $\theta(x)=\varphi-A-\tfrac\pi2$, so
$e^{i\theta}=-ie^{\,i(\varphi-A)}$. The claim therefore reads
\[
\cos\varphi-\cos(A-\varphi)\,e^{-iA}
\;=\;(-\sin A)\,(-i)\,e^{\,i(\varphi-A)}
\;=\;i\sin A\,e^{\,i(\varphi-A)},
\]
and expanding the left side settles it: the real part is
$\cos\varphi-\cos(A-\varphi)\cos A=\sin A\,\sin(A-\varphi)$ (product
formula), the imaginary part is $\sin A\,\cos(A-\varphi)$, and
$\sin A\bigl(\sin(A-\varphi)+i\cos(A-\varphi)\bigr)
=i\sin A\,e^{-i(A-\varphi)}$, which is the right side.

For \eqref{eq:yinf-halflag}, the definition \eqref{eq:yang-inf} and the
first equality of \eqref{eq:yinf-polar} give
$\mathrm{Yang}_{\infty}(x)=e^{\,2i\varphi}\bigl(1-\mathrm{Yin}_{\infty}(x)\bigr)
=\Psi(x)\,e^{\,i(2\varphi-A)}$. Applying \eqref{eq:yinf-polar} at
$x-\tfrac12$ instead,
$\mathrm{Yin}_{\infty}(x-\tfrac12)=\Psi(x)\,e^{\,i\theta(x-\frac12)}$,
and $\theta(x-\tfrac12)=2\pi(x-\tfrac12)(\tfrac32-x)-\tfrac{3\pi}8
\equiv2\varphi-A\pmod{2\pi}$, so the first equality holds. For the
second, $e^{\,2i\theta}=-e^{\,2i(\varphi-A)}$, and
\begin{align*}
\mathrm{Yin}_{\infty}(x)-\mathrm{Yang}_{\infty}(x)
&=1-\Psi\bigl(e^{-iA}+e^{\,i(2\varphi-A)}\bigr)
=1-2\Psi\cos\varphi\,e^{\,i(\varphi-A)}\\
&=1-2\cos(A-\varphi)\,e^{\,i(\varphi-A)}
=-e^{\,2i(\varphi-A)}
=e^{\,2i\theta(x)},
\end{align*}
using $\Psi\cos\varphi=\cos(A-\varphi)$ in the middle step.
\end{proof}

\begin{theorem}[the Yin and Yang curves converge to the limit curves,
$\mathrm{Yin}_{\infty}$ and $\mathrm{Yang}_{\infty}$]
\label{thm:yin-convergence}
Fix $x\in(0,1)$ and $\sigma\in(0,1)$. Then
\[
\lim_{\substack{T\to\infty\\ \{T\}=x}}\mathrm{Yin}(s)
\;=\;\mathrm{Yin}_{\infty}(x),
\qquad
\lim_{\substack{T\to\infty\\ \{T\}=x}}\mathrm{Yang}(s)
\;=\;\mathrm{Yang}_{\infty}(x),
\]
at rate $O(1/T)$, uniformly for $x$ and $\sigma$ in compact subsets of
$(0,1)$; the limit does not depend on $\sigma$.
\end{theorem}

The proof needs only Gabcke's bound \eqref{eq:gabcke} on the critical
line, and off it additionally Siegel's first-order asymptotics
\eqref{eq:R-general}, the same input as in the proof of
Proposition~\ref{prop:weights}(iv).

\begin{proof}
Throughout, $C_0$ and $\Psi$ are the same function
(\S\ref{sec:yinyang-inf}).

\emph{Step 1 (the yin point in asymptotic closed form).} By
\eqref{eq:yin1}, $\mathrm{Yin}=R\,M^{s}=R\,(m+1)^{\sigma}e^{\,i\omega}$,
with $M=m+1$ for non-integer $T$. On the critical line $R=e^{-i\vartheta}r$
with $r$ real (Step~1 of the proof of
Proposition~\ref{prop:weights}(iii)), so
$\mathrm{Yin}=\sqrt{m+1}\;r\,e^{\,iu}$, and inserting
\eqref{eq:r-two-cases}--\eqref{eq:gabcke},
\begin{equation}\label{eq:yin-asymp}
\mathrm{Yin}
=\Bigl(\tfrac{m+1}{a}\Bigr)^{\!1/2}(-1)^{N-1}\bigl(C_0(\hat p)+E\bigr)
e^{\,iu}
\;+\;\bigl[N=m{+}1\bigr]\,\bigl(1+e^{\,2iu}\bigr),
\end{equation}
where the extra summand of \eqref{eq:r-two-cases} became the closed pair
$2\cos u\,e^{\,iu}=1+e^{\,2iu}$. Off the line, the first term of
Siegel's asymptotic expansion of the remainder at fixed $\sigma$
\cite{Siegel1932}, the same input used in the proof of
Proposition~\ref{prop:weights}(iv), gives
\begin{equation}\label{eq:R-general}
\begin{split}
R(s)={}&(-1)^{N-1}\,\frac{C_0(\hat p)}{2}
\Bigl(a^{-\sigma}e^{-i\tilde\vartheta(t)}
+\chi(s)\,a^{\sigma-1}e^{+i\tilde\vartheta(t)}\Bigr)\\
&+\bigl[N=m{+}1\bigr]\bigl((m{+}1)^{-s}+\chi\,(m{+}1)^{s-1}\bigr)
\;+\;O_\sigma\bigl(a^{-\sigma-1}\bigr),
\end{split}
\end{equation}
with the elementary phase
$\tilde\vartheta(t)=\tfrac t2\log\tfrac t{2\pi}-\tfrac t2-\tfrac\pi8
=\vartheta(t)+O(1/t)$. The error is additive because Siegel's expansion
advances in powers of $a^{-1}$, one step of which is all we keep. Against
the displayed term it is a relative $O(1/a)=O(1/T)$, the order of Gabcke's
$E$ in \eqref{eq:gabcke}: the two pieces of the parenthesis align rather
than cancel, since $|\chi(s)|\,a^{\sigma-1}=a^{-\sigma}\bigl(1+O(1/t)\bigr)$
and their phases differ by $\arg\chi+2\tilde\vartheta=O_\sigma(1/t)$, so the
parenthesis has modulus $2a^{-\sigma}\bigl(1+o(1)\bigr)$, while $C_0$ stays
between $\cos\tfrac{3\pi}{8}=0.3826\ldots$ and $\cos\tfrac\pi8=0.9238\ldots$
on $[0,1]$. The frame factor $(m+1)^{\sigma}e^{\,i\omega}$
carries the two pieces of the parenthesis to
$\bigl(\tfrac{m+1}{a}\bigr)^{\sigma}e^{\,i(\omega-\tilde\vartheta)}$ and
$|\chi|\,a^{\sigma-1}(m+1)^{\sigma}e^{\,i(\omega+\psi+\tilde\vartheta)}$,
and since $|\chi(\sigma+it)|=a^{1-2\sigma}\bigl(1+O(1/t)\bigr)$ (by
Stirling's formula) and $\psi=-2\vartheta+O_\sigma(1/t)$, both tend to
$e^{\,iu}$; likewise the Iverson pair maps to
$1+|\chi|(m+1)^{2\sigma-1}e^{\,i(2\omega+\psi)}\to1+e^{\,2iu}$. In every
case
\[
\mathrm{Yin}=(-1)^{N-1}C_0(\hat p)\,e^{\,iu}\bigl(1+O_\sigma(1/T)\bigr)
+\bigl[N=m{+}1\bigr]\bigl(1+e^{\,2iu}\bigr)\bigl(1+O_\sigma(1/T)\bigr).
\]

\emph{Step 2 (the phase, kept as an angle).} The two case expansions of
$u$ displayed before \eqref{eq:cosu}, kept as angles rather than used
only through their cosines, combine with \eqref{eq:cut-cases} into the
single statement
\[
u\;\equiv\;\pi(m+1)+\theta(x)+O(1/T)\pmod{2\pi},
\]
while $\hat p\to x+\tfrac12$ (with $N=m$) for $x<\tfrac12$ and
$\hat p\to x-\tfrac12$ (with $N=m+1$) for $x>\tfrac12$.

\emph{Step 3 (case $x<\tfrac12$).} Here $N=m$, the Iverson term is
absent, and the sign $(-1)^{N-1}=(-1)^{m-1}$ cancels against
$e^{\,i\pi(m+1)}=(-1)^{m+1}$. With $\tfrac{m+1}a\to1$ (by
\eqref{eq:a-vs-T}) and $C_0=\Psi$,
\[
\mathrm{Yin}\;\longrightarrow\;
\Psi\bigl(x+\tfrac12\bigr)e^{\,i\theta(x)}
\;=\;\mathrm{Yin}_{\infty}(x)
\]
by the polar form \eqref{eq:yinf-polar}.

\emph{Step 4 (case $x>\tfrac12$).} Here $N=m+1$, the sign product is
$-1$, and the Iverson pair is present:
\[
\mathrm{Yin}\;\longrightarrow\;
1-\Psi\bigl(x-\tfrac12\bigr)e^{\,i\theta(x)}+e^{\,2i\theta(x)} .
\]
Since $\theta(x)=-A(x-\tfrac12)$ exactly, the first two terms are
$\mathrm{Yin}_{\infty}(x-\tfrac12)$, and the single-curve relation
\eqref{eq:yinf-halflag} finishes:
$\mathrm{Yin}\to\mathrm{Yin}_{\infty}(x-\tfrac12)+e^{\,2i\theta(x)}
=\mathrm{Yin}_{\infty}(x)$.

\emph{Step 5 (the yang point).} By \eqref{eq:yang1},
$\mathrm{Yang}-\mathrm{Yin}=-\chi\,M^{2s-1}$, whose modulus
$|\chi|(m+1)^{2\sigma-1}=\bigl(\tfrac{m+1}a\bigr)^{2\sigma-1}
\bigl(1+O(1/t)\bigr)\to1$ and phase
$2\omega+\psi=2u+O_\sigma(1/t)\equiv2\theta(x)+O(1/T)$, so the
subtracted vector tends to $e^{\,2i\theta(x)}$ and
$\mathrm{Yang}\to\mathrm{Yin}_{\infty}(x)-e^{\,2i\theta(x)}
=\mathrm{Yang}_{\infty}(x)$, again by \eqref{eq:yinf-halflag}.

Every ingredient contributes $O(1/T)$: the amplitude ratio
$\tfrac{m+1}a=1+O(1/T)$, Gabcke's $E=O(1/a)$, the phase remainder of
Step~2, and off the line the $O_\sigma(1/t)$ terms of $\psi$ and
$|\chi|$ together with the $O_\sigma(a^{-\sigma-1})$ of
\eqref{eq:R-general}, which against a leading term of size $a^{-\sigma}$ is
again $O(1/T)$. At
the seam $x=\tfrac12$ the two case formulas agree:
$\Psi(1)=\Psi(0)=\cos\tfrac\pi8$ and
$1+e^{\,i\pi/4}=2\cos\tfrac\pi8\,e^{\,i\pi/8}$, so both give
$\cos\tfrac\pi8\,e^{\,i\pi/8}$, which is
$\mathrm{Yin}_{\infty}(\tfrac12)$; the narrow boundary intervals of
\eqref{eq:cut-cases} near $x=0,1$ are excluded by keeping $x$ in a
compact subset of $(0,1)$.
\end{proof}

Three attributes of the yin and yang curves stand out.
\begin{enumerate}
\item At any $\sigma$, the yin and yang curves approach
$\mathrm{Yin}_{\infty}$ and $\mathrm{Yang}_{\infty}$ as $T\to\infty$:
this is Theorem~\ref{thm:yin-convergence}, at rate $O(1/T)$, with a
limit that does not depend on $\sigma$. Figure~\ref{fig:yinyang-asym}
shows the successive orbits settling onto the limit curves as
$T\to\infty$ for each of $\sigma=0.05$, $\tfrac12$, and $0.95$. The chord joining the pair is
a \emph{unit needle}: its length is $|\chi|M^{2\sigma-1}$, exactly $1$
at every $T$ on the critical line and $1+O(1/T)$ at any fixed $\sigma$.
On the critical line, and at $T\to\infty$, the two curves travel at
constant unit separation.
\item The curves are not symmetrical (when not at infinity), and this
becomes obvious once the
yang curve is reflected to $1-\mathrm{Yang}(s)$ and compared with the
yin curve: by \eqref{eq:refl-yang} the two limit traces coincide
exactly, but at finite $T$ the two families only interleave, closing on
one another as $T$ grows and coinciding at no finite
$T$.\footnote{The asymmetry is provable in the strongest sense: over
one period the two curves have different arc lengths (at
$\sigma=\tfrac12$ the yin loop of $[6,7]$ is longer by $0.029$; at
$\sigma=0.05$ by $0.32$; at $\sigma=0.95$ the yang loop is longer by
$0.26$; in each case whichever family runs outside, in the sense of the
next attribute, is the longer one), so no isometry of the
plane (reflection, rotation, or point reflection) carries one curve
exactly onto the other. What does hold exactly, at each instant, is a
\emph{moving} mirror: at any $\sigma$ the two points are the ends of
the in-frame next dual summand, hence point reflections of each other
through its midpoint, and on the critical line
$\mathrm{Yin}=1+e^{\,2iu}\,\overline{\mathrm{Yang}}$ exactly, a glide
reflection whose axis direction $e^{\,iu}$ turns with $T$ (both inputs
are exact there: with $u=\omega-\vartheta(t)$ as in
Appendix~\ref{sec:positivity}, $\mathrm{Yin}$ is a real multiple of
$e^{\,iu}$, and
$\mathrm{Yang}=\mathrm{Yin}-e^{\,2iu}$ since $|\chi M^{2s-1}|=1$). The
nearest fixed map, the time-reversed point reflection
$\mathrm{Yang}(T)\approx1-\mathrm{Yin}(2m{+}1{-}T)$ at the same
$\sigma$, becomes exact only as $T\to\infty$.} The global asymmetry coexists
with an instantaneous balance: on the critical line
\eqref{eq:yinyang-line} with $d_1=d_2$ makes the crossing point
equidistant from $0$ and the yin point (both distances $\hat d_1$) and
from $1$ and the yang point (both $1-\hat d_1$), so the pair rides two
concentric circles centered at the moving crossing, with radii summing
to $1$.
\item Which family runs outside depends on $\sigma$. For $\sigma$ close
to $1$ the yin curve runs inside the reflected yang curve; for $\sigma$
close to $0$ it runs outside; and on the critical line the two swap
sides near $\{T\}=\tfrac12$, one crossing outward as the other crosses
inward: at matched positions on the teardrop the two offsets from the
limit trade signs. The swap is forced by an exact duality: since
$\chi(s)\chi(1-s)=1$ exchanges the main and dual sums,
$R(1-s)=\chi(1-s)\,R(s)$, and with $\lambda=\chi(1-s)M^{1-2s}$,
\[
\mathrm{Yin}(1-s)=\lambda\,\mathrm{Yin}(s),\qquad
\mathrm{Yang}(1-s)=\lambda\bigl(\mathrm{Yin}(s)-1\bigr),\qquad
\lambda\,\bigl(\mathrm{Yin}-\mathrm{Yang}\bigr)(s)=1,
\]
so reflecting $\sigma\mapsto1-\sigma$ (at conjugate height)
interchanges yin and yang and swaps the needle with the unit segment.
\end{enumerate}

\clearpage

\section{Summands as links and joints}\label{sec:matrix-product}

Equation~\eqref{eq:zeta-ps} presents $\zeta$ as the sum of the four terms
$\Sigma_1+R_{1}+\Sigma_2+R_{2}$, two of which are partial sums. Read
geometrically, each summand of a partial sum is one \emph{link} of a
planar chain, and each remainder $R_{1},R_{2}$ is represented by one
additional signed partial-summand link. When $0\leq\hat d_i\leq1$, the
coefficient is the fraction of that link used; a negative coefficient
reverses its direction, and a coefficient exceeding one extends it past
the full summand.
With $t$, $s$, $m$, and $M$ as in \eqref{eq:IT} and $\omega=t\log M$ as in
\eqref{eq:omega-def}, denote two complex numbers by
\begin{equation}
B_1=\Sigma_1+R_{1},\qquad B_2=\Sigma_2+R_{2},\qquad \zeta=B_1+B_2 .
\end{equation}
The pair $B_1,B_2$ will be used heavily in what follows: $B_1$ is the
bisector point of \S\ref{sec:bisector}, the two legs of
\S\ref{sec:legs} run from the origin to $B_1$ and from $B_1$ to $\zeta$,
and the angles read off those legs drive the zero-counting functions of
\S\ref{sec:counting}.

For non-integer $T$, each remainder is a single extra signed link appended
to its partial sum, pointing along the direction of the $(m+1)$st summand
with real coefficient $d_1$ or $d_2$, which makes each remainder the
partial summand of
\S\ref{sec:summand}:
\begin{equation}\label{eq:R-phasors}
R_{1}=d_1\,M^{-it}=d_1\,e^{-i\omega},\qquad
R_{2}=d_2\,M^{\,it}\,\frac{\chi}{|\chi|}=d_2\,e^{\,i(\omega+\arg\chi)},
\end{equation}
with $d_1,d_2$ as in \eqref{eq:d1}--\eqref{eq:d2}; the phasor form
restates \eqref{eq:R-two-dirs}.

\subsection{Sum form}\label{sec:sum-form}

We now write zeta in different forms, first as a sum and then, in the
next subsection, as a product, so the two forms can be compared side by
side.
Splitting $\zeta$ into its partial sum plus this last
fractional link,
\begin{alignat}{2}
B_1 &= \phantom{\chi(s)}\sum_{n=1}^{m} n^{-s}    &&{}+ d_1\,e^{-i\omega}, \label{eq:B1-sum}\\
B_2 &= \chi(s)\sum_{n=1}^{m} n^{\,s-1}           &&{}+ d_2\,e^{\,i(\omega+\arg\chi)}. \label{eq:B2-sum}
\end{alignat}

\begin{figure}[H]
\centering
\includegraphics[width=0.86\textwidth]{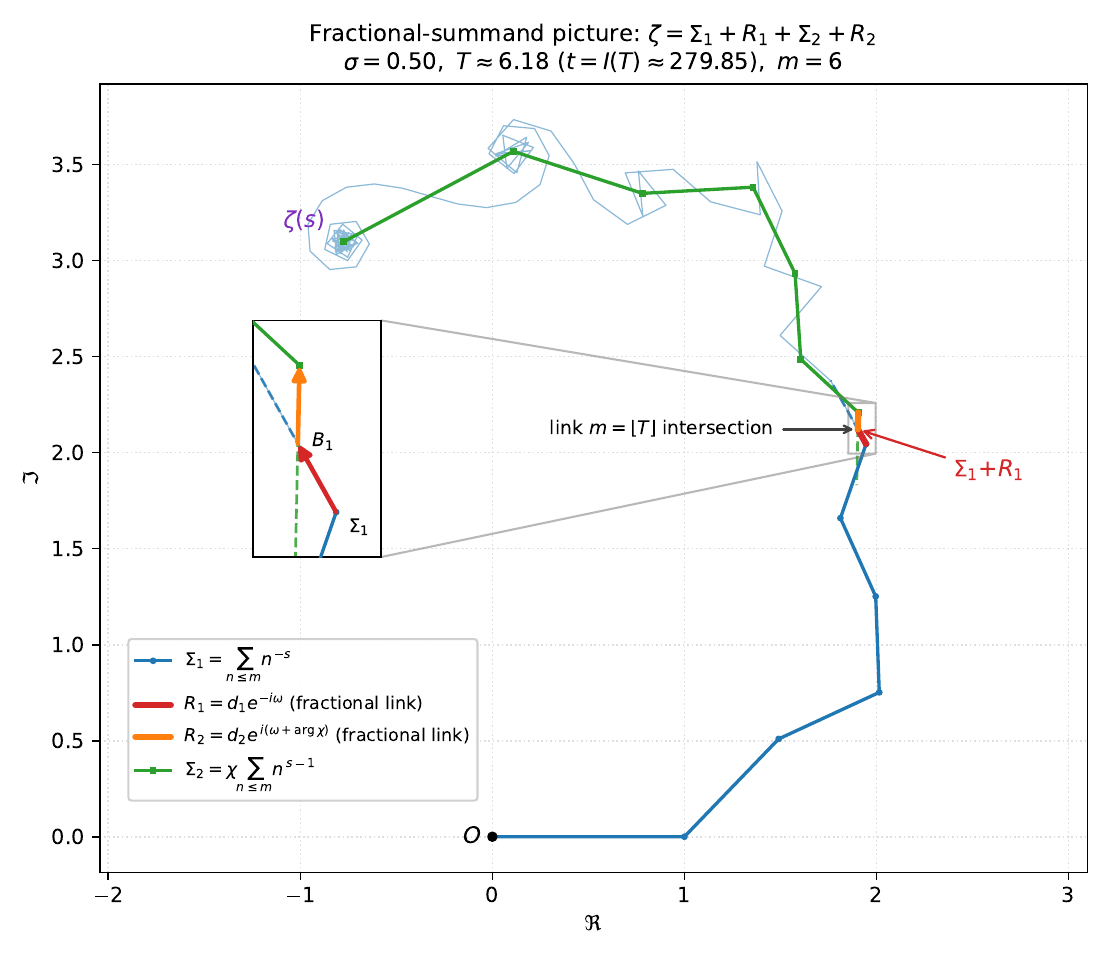}
\caption{The fractional-summand picture: pictorial (spiral) representation of the
four pieces of \eqref{eq:zeta-clean} in the complex plane for $\sigma=0.5$ and
$T\approx 6.18$ (equivalently $t=I(T)\approx 279.85$, $m=6$). The blue path is the
partial-sum spiral $\Sigma_1=\sum_{n\le m}n^{-s}$; the green path,
laid tip-to-tail from the point $\Sigma_1+R_{1}$, is
$\Sigma_2=\chi\sum_{n\le m}n^{\,s-1}$. The short red and orange segments $R_{1}=d_1e^{-i\omega}$ and
$R_{2}=d_2e^{\,i(\omega+\arg\chi)}$ lie along link $m$ of each spiral
(the faint dashed vectors show the full next summand, of which each remainder is
a ``fractional'' part), so that the four pieces add up to the point $\zeta(s)$.
Underneath, the thin faint blue path continues the $\Sigma_1$ chain past $m$
out to link $\lfloor I(T)/\pi\rfloor=89$, winding into its spiral near
$\zeta$.
On the critical line $d_1=d_2$, so the two remainders have equal length.
The fractional links are short here ($d_1=d_2\approx0.087$), so the inset
zooms in on them. Generated
by \texttt{fig1\_spiral\_summands.py}.}
\label{fig:spiral-summands}
\end{figure}

\begin{figure}[H]
\centering
\includegraphics[width=0.72\textwidth]{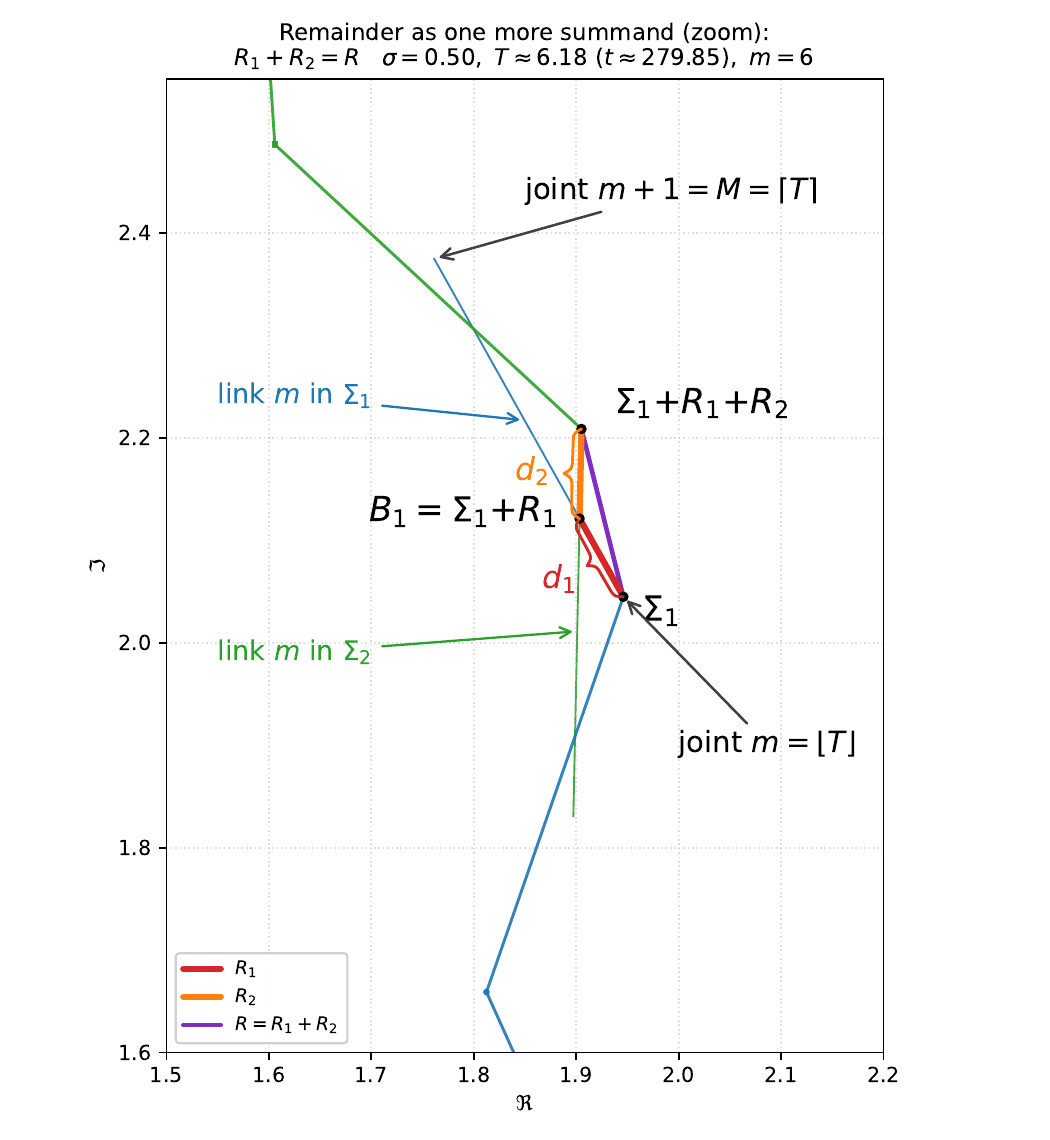}
\caption{Zoom into the joint region of Figure~\ref{fig:spiral-summands}
(imaginary part $\approx 1.5$--$2.6$), showing the theorem
$R_{1}+R_{2}=R$ as a small triangle. Starting from the joint $\Sigma_1$
(end of the sum-1 partial sum), the red link $R_{1}$ reaches
$\Sigma_1+R_{1}$ and the orange link $R_{2}$ reaches
$\Sigma_1+R_{1}+R_{2}$; the straight purple vector from $\Sigma_1$ to that
last joint is Siegel's remainder $R=R_{1}+R_{2}$ (see the legend). The
curly braces labeled $d_1$ and $d_2$ span the two fractional links,
marking their lengths, i.e.\ how far each reaches from its own joint.
The thin blue and green lines show the entire extent each fractional summand
would have if it were not shortened to the lengths $d_1$ and $d_2$: each
extends its fractional link to the full $(m+1)$st summand. The faint blue and
green chains are the neighboring links of the two partial sums, arriving at
and departing from the joint region. Generated by
\texttt{fig2\_remainder\_zoom.py}.}
\label{fig:remainder-zoom}
\end{figure}

\subsection{Product form}\label{sec:product-form}

This subsection is a tangent, and a reader in a hurry can skip it:
nothing later rests on the matrix form, which is motivated by an analogy
with multi-jointed planar robot manipulators, and a desire for a product
form of the zeta function but over the integers, not over the primes.

The usual product form of the zeta function is the Euler product over
the primes, one multiplication for each prime,
\begin{equation}\label{eq:euler-product}
\zeta(s)=\prod_{p\ \mathrm{prime}}\frac{1}{1-p^{-s}}
\qquad(\Re s>1).
\end{equation}
The goal of this subsection is to write each link
as a homogeneous coordinate transformation matrix, and then to group the
links into a single transformation matrix, so that $\zeta$ takes the form
of a matrix product with four components: the two sums $\Sigma_1,\Sigma_2$
and the two partial-summand remainders $R_{1},R_{2}$.

Each summand is a link in the plane. We number the
links of each chain from zero: link $n$ runs from joint $n$ to joint $n+1$
(joint $n$ being the value of the partial sum after $n$ summands), so link $n$
represents the $(n+1)$st summand. The $m$ summands of a partial sum are thus
links $0,\dots,m-1$, and each remainder is the \emph{fractional link} $m$. We
encode one link by the $3\times3$ matrix with link length $l$ and rotation
$\theta$ from the previous link,\footnote{This is the planar case of the
Denavit--Hartenberg convention of robot kinematics~\cite{DenavitHartenberg1955},
in which each link of a serial chain is encoded by the $4\times4$ homogeneous
transformation
$\mathrm{Rot}_z(\theta_i)\,\mathrm{Trans}_z(d_i)\,\mathrm{Trans}_x(a_i)\,
\mathrm{Rot}_x(\alpha_i)$
built from four parameters: the joint angle $\theta_i$, the joint offset $d_i$,
the link length $a_i$ and the link twist $\alpha_i$. A planar chain has
$d_i=\alpha_i=0$, which leaves the two parameters $(\theta_i,a_i)$ and the
$3\times3$ matrix \eqref{eq:link-matrix}; \eqref{eq:zeta-serial} is then the
forward-kinematic product of the chain. One collision of names to watch: the
Denavit--Hartenberg offset $d_i$ has nothing to do with the weights $d_1,d_2$ of
\eqref{eq:d1}--\eqref{eq:d2}.}
\begin{equation}\label{eq:link-matrix}
\mathsf L(\theta,l)=\linkmat{\theta}{l},
\end{equation}
where the accumulated position is the top two entries of the last column.
The link matrix rotates by $\theta$ and then translates forward $l$
along the new heading. Its inverse is again a link-shaped matrix, of a
second type that translates \emph{back} $l$ along the current heading and
then undoes the rotation:
\begin{equation}\label{eq:link-matrix-2}
\mathsf L(\theta,l)^{-1}=
\begin{pmatrix}
\cos\theta & \phantom{-}\sin\theta & -l\\
-\sin\theta & \phantom{-}\cos\theta & \phantom{-}0\\
\phantom{-}0 & \phantom{-}0 & \phantom{-}1
\end{pmatrix}.
\end{equation}
Forward links build chain 1; inverses are what let chain 2 be walked
backwards.

\emph{Chain 1.} Collect the $m$ summand links of $\Sigma_1$ into a single
product and give the fractional link its own name:
\begin{equation}\label{eq:P1-TR1-defs}
P_{\Sigma_1}=\prod_{n=0}^{m-1}\mathsf L\bigl(\theta_n,\,(n+1)^{-\sigma}\bigr),
\qquad
\mathsf L_{R_{1}}=\mathsf L\bigl(-t\ln\tfrac{m+1}{m},\,d_1\bigr),
\end{equation}
where link $n$ carries the two parameters
$\theta_n$, its rotation from the previous link, and
$l_n=(n+1)^{-\sigma}$, its length: as in \eqref{eq:omega-telescope},
$\theta_0=0$ and $\theta_n=-t\ln\tfrac{n+1}{n}$ for $n\ge1$. Written out, with
the first two factors and the last two shown, the chain that
reaches $B_1$ is
\begin{equation}\label{eq:P1-written}
\footnotesize
P_{\Sigma_1}\mathsf L_{R_{1}}=
\underbrace{\left(\begin{array}{@{}c@{\ \,}c@{\ \,}c@{}}
1 & 0 & 1\\
0 & 1 & 0\\
0 & 0 & 1
\end{array}\right)}_{\text{link }0}
\underbrace{\linkmatd{\theta_1}{2^{\sigma}}}_{\text{link }1}\;\cdots\;
\underbrace{\linkmatd{\theta_{m-1}}{m^{\sigma}}}_{\text{link }m-1}
\underbrace{\linkmat{\theta_m}{d_1}}_{\text{fractional link }m},
\end{equation}
the leading factor being the unit step $1^{-s}=1$ along the real axis, where
$\theta_0=0$ and $l_0=1$ collapse \eqref{eq:link-matrix} to a pure translation,
and the trailing factor being $\mathsf L_{R_{1}}$, whose rotation
$\theta_m=-t\ln\tfrac{m+1}{m}$ continues the same pattern while its length is
$d_1$ rather than a power of $m+1$. The product carries the frame from the origin
to the point $B_1=\Sigma_1+R_{1}$, whose coordinates are the top two
entries of the last column, and it arrives with accumulated rotation
$-t\ln(m+1)=-\omega$ by the telescoping \eqref{eq:omega-telescope}.

\emph{Chain 2.} The second chain is traversed tip-first, exactly as
Figure~\ref{fig:spiral-summands} draws the green chain: the fractional
link $R_{2}$ leaves $B_1$, and the summand links follow in descending
order.
Walking a link backwards is exactly its inverse \eqref{eq:link-matrix-2};
regrouping each rotation with the backward translation that follows it
returns everything to the forward type \eqref{eq:link-matrix}, with
negated parameters and each rotation paired with the preceding link's
length. Two bookkeeping factors appear at the ends of the backward walk: a
leading pure translation $\mathsf L(0,-d_2)$, and a trailing pure rotation
$\mathsf L(-\arg\chi,0)$, with the phase rotation associated with $\chi$
peeled off on its own; the modulus $|\chi|$ remains in the link lengths.
The leading translation absorbs the zero-length joint
$\mathsf L(2\omega+\arg\chi+\pi,\,0)$ that must sit between the two chains (a
rotation followed by a translation is a single link), defining the
backward fractional link
\begin{equation}\label{eq:TR2-def}
\mathsf L_{R_{2}}
\;:=\;
\mathsf L\bigl(2\omega+\arg\chi+\pi,\,0\bigr)\,
\mathsf L(0,\,-d_2)
\;=\;
\mathsf L\bigl(2\omega+\arg\chi+\pi,\;-d_2\bigr);
\end{equation}
the backward summand links collect into
\begin{equation}\label{eq:P2-def}
P_{\Sigma_2}
\;:=\;
\prod_{n=m}^{1}\mathsf L\bigl(\theta_n,\;-|\chi|\,n^{\sigma-1}\bigr)
\quad\text{($n$ descending, $\theta_n=-t\ln\tfrac{n+1}{n}$)};
\end{equation}
and the trailing rotation is simply dropped, since a post-multiplied
zero-length joint never moves the position column, only the final
orientation.

Written out as \eqref{eq:P1-written} was, with the first two factors and the last
two shown and with $\psi=\arg\chi$ as before, the second chain is
\begin{equation}\label{eq:P2-written}
\footnotesize
\begin{aligned}
\mathsf L_{R_{2}}P_{\Sigma_2}
&=\underbrace{\linkmat{(2\omega+\psi+\pi)}{-d_2}}_{\text{fractional link }m}
\underbrace{\linkmatx{\theta_m}{m^{1-\sigma}}}_{\text{link }m-1}\;\cdots\\[3pt]
&\hphantom{{}={}}\cdots\;
\underbrace{\linkmatx{\theta_2}{2^{1-\sigma}}}_{\text{link }1}
\underbrace{\linkmat{\theta_1}{-|\chi|}}_{\text{link }0},
\end{aligned}
\end{equation}
the leading rotation being the merged joint of \eqref{eq:TR2-def}. The links run
backwards, from the fractional link $m$ down to link $0$, each brace naming the
link whose length that factor carries; the rotation it carries is the $\theta_n$
of the link after it, which is the regrouping. Its
last column is not $B_2$ but $e^{i\omega}B_2$, chain 2 read in the frame chain 1
leaves behind, whose orientation is $-\omega$; multiplying on the left by
$P_{\Sigma_1}\mathsf L_{R_{1}}$ turns that back onto world axes and adds $B_1$, which
is \eqref{eq:zeta-serial}. All three columns agree to $28$ digits at
$\sigma=\tfrac12$, $T=6.18$ (\texttt{check\_chain2\_matrix.py}).

\paragraph{The full $3\times3$ equation for $\zeta$.}
With these four names the two chains combine into one serial product, in a
single link type:
\begin{equation}\label{eq:zeta-serial}
\mathsf T_{\zeta}
\;=\;
P_{\Sigma_1}\,\mathsf L_{R_{1}}\,\mathsf L_{R_{2}}\,P_{\Sigma_2}
\;=\;
\begin{pmatrix}
\cos\alpha_{\zeta} & -\sin\alpha_{\zeta} & \Re\,\zeta\\
\sin\alpha_{\zeta} & \phantom{-}\cos\alpha_{\zeta} & \Im\,\zeta\\
0 & 0 & 1
\end{pmatrix},
\qquad \alpha_{\zeta}=\pi+\arg\chi.
\end{equation}
The last column is $\zeta$ and the rotation block records $\alpha_{\zeta}$,
discussed below.
The factors line up one to one with the terms of \eqref{eq:zeta-ps},
repeated here for comparison:
\begin{equation*}\tag{\ref{eq:zeta-ps}}
\zeta(s)
=
\Sigma_{1} + R_{1} + R_{2} + \Sigma_{2}.
\end{equation*}
Figure~\ref{fig:zeta-chain} draws \eqref{eq:zeta-serial} with one arrow
per factor.

\begin{figure}[H]
\centering
\includegraphics[width=0.8\textwidth]{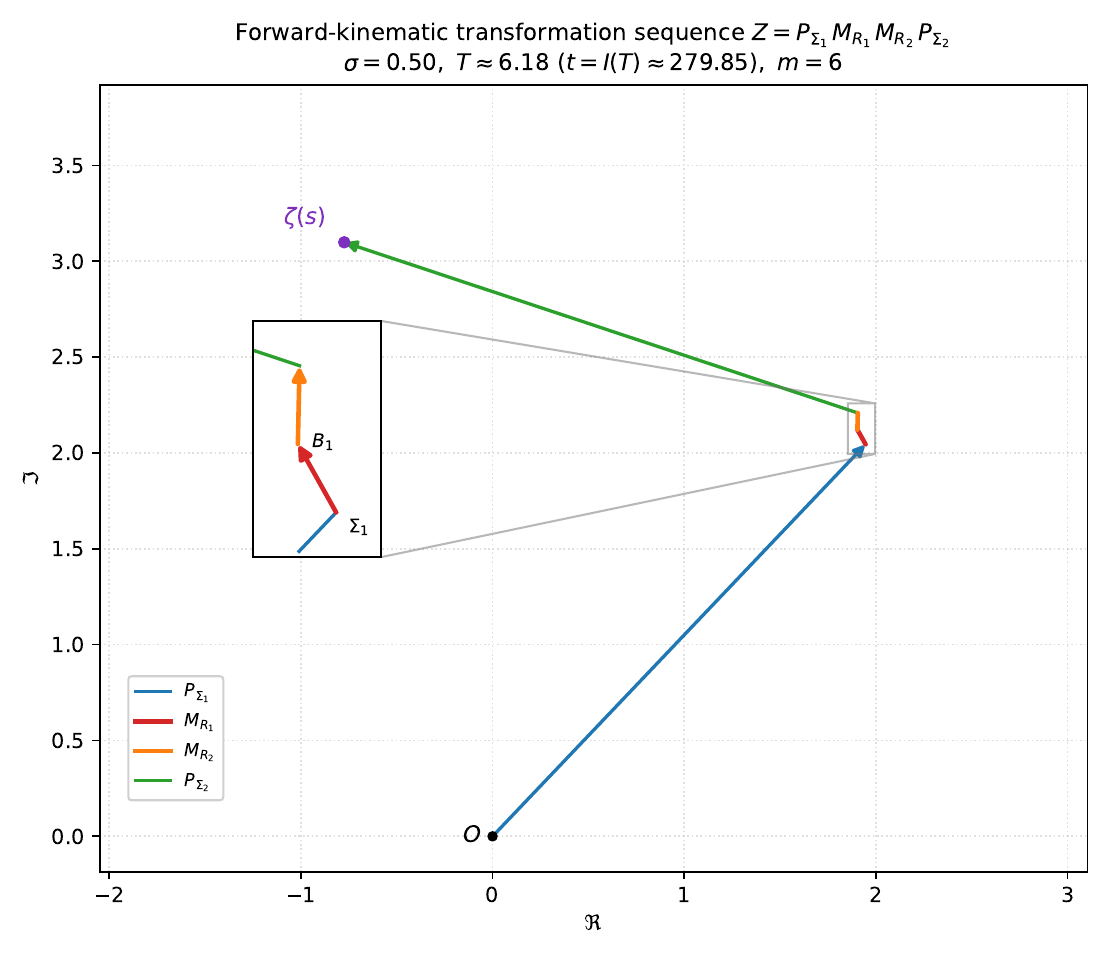}
\caption{The serial chain of \eqref{eq:zeta-serial} at $\sigma=\tfrac12$,
$T=6.18$ ($t=I(T)\approx279.85$, $m=6$), one arrow for the net displacement
of each matrix factor: $P_{\Sigma_1}$ (blue, $O$ to $\Sigma_1$),
$\mathsf L_{R_{1}}$ (red, to
$B_1$), $\mathsf L_{R_{2}}$ (orange, to $B_1+R_{2}$; its rotation parameter
carries the merged joint $2\omega+\arg\chi+\pi$), and $P_{\Sigma_2}$ (green,
to $\zeta$). The four arrows are the four terms of
$\zeta=\Sigma_1+R_{1}+R_{2}+\Sigma_2$; the two fractional links are
short here ($d_1=d_2\approx0.087$), so the inset zooms in on them. Colors
match Figure~\ref{fig:spiral-summands}.
Generated by
\texttt{fig\_zeta\_chain.py}.}
\label{fig:zeta-chain}
\end{figure}

Three features of \eqref{eq:zeta-serial}
deserve comment. First, the rotation $2\omega+\arg\chi+\pi$ carried by
$\mathsf L_{R_{2}}$ has three jobs at once: $+\omega$ cancels chain 1's
accumulated rotation $-t\ln(m+1)=-\omega$, a
further $+(\omega+\arg\chi)$ pre-rotates against the counter-rotation of
the backward-walked chain 2, and the final $+\pi$ rotates the backward
translations to point forward again; together
$\omega+(\omega+\arg\chi)+\pi=2\omega+\arg\chi+\pi$. Second, the links of
$P_{\Sigma_2}$ carry exactly the rotation parameters $\theta_n$ of chain 1:
in this form the two chains use identical rotations and differ only in
their lengths, $(n+1)^{-\sigma}$ forward versus $-|\chi|\,n^{\sigma-1}$
backward. Third, the net displacements of the four factors $P_{\Sigma_1}$,
$\mathsf L_{R_{1}}$, $\mathsf L_{R_{2}}$, $P_{\Sigma_2}$ are exactly the four terms
$\Sigma_1$, $R_{1}=d_1e^{-i\omega}$,
$R_{2}=d_2e^{\,i(\omega+\arg\chi)}$, $\Sigma_2$ of \eqref{eq:zeta-ps},
and the final orientation $\alpha_{\zeta}=\pi+\arg\chi$ is the direction of
$-\chi/|\chi|$: the trailing rotation $\mathsf L(-\arg\chi,0)$ (the phase
rotation associated with $\chi$, peeled off automatically when the base
link of chain 2 is regrouped; the modulus $|\chi|$ remains in the link
lengths) was discarded, so $\chi$'s direction is
left showing in the orientation block, and $\mathsf T_{\zeta}$ packages the pair
$(\zeta,\arg\chi)$ in one matrix. (Appending the discarded
$\mathsf L(-\arg\chi,0)$ would leave the position column untouched and make the
orientation the constant $\pi$.)

\section{Other remainders}\label{sec:other-remainders}

The split $R=R_{1}+R_{2}$ of \S\ref{sec:decomp} is not the only way to
dissect Siegel's remainder in two. This section introduces two others, so
that three splits are in play, all reassembling the same $R$: exactly in
two of the three cases and to high accuracy in the third. We first name
the three, then drill down on the newest of them, Kuznetsov's; a closing
subsection (\S\ref{sec:remainder-scaling}) records a scaling observation.
The remainder figures at different heights are nearly scaled copies, with
the residual amplitude discrepancy tending toward zero as $T$ grows.

\subsection{The three splits and their names}\label{sec:remainders-summary}

To keep the remainders straight we name them
with subscripts. Siegel's remainder integral is written $R$ when unsplit; each
named split below recovers that same $R$ (exactly, or approximately for the
Kuznetsov form):
\begin{enumerate}
\item \textbf{Riemann--Siegel half-split.}
\[
R \;=\; R_{rs} \;=\; R_{1rs}+R_{2rs},
\qquad
R_{1rs}=R_{2rs}=\tfrac12 R.
\]
Here $R_{1rs}$ is simply half the remainder integral.
\item \textbf{Partial-summand split} (\S\ref{sec:decomp}--\S\ref{sec:summand}),
the split of this paper.
\[
R \;=\; R_{ps} \;=\; R_{1ps}+R_{2ps}.
\]
Here $R_{1ps}$ is a partial next summand: the next Dirichlet term after
$\Sigma_1$ stops, the $M$-th, shortened to length $d_1$. The unsubscripted
$R_{1},R_{2}$ used everywhere else in this paper are exactly
$R_{1ps},R_{2ps}$; we attach the $ps$ subscript only where several splits
stand side by side, as in this section and \S\ref{sec:geometry}.
\item \textbf{Kuznetsov / Siegel $f$ split} (\S\ref{sec:kuznetsov}).
\[
R \;\approx\; R_{ak} \;=\; R_{1ak}+R_{2ak}.
\]
Here $R_{1ak}$ is, to high accuracy, Siegel's $f(s)$ minus the first
partial sum, $R_{1ak}\approx f(s)-\Sigma_1$, with equality if this is
taken as the definition (\S\ref{sec:kuznetsov}).
\end{enumerate}

All three are measured from the same pair of joints, both at index
$m=\lfloor T\rfloor$: each $R_{1\bullet}$ leaves joint $m$ of the $\Sigma_1$
chain of \S\ref{sec:matrix-product}, which is the partial sum $\Sigma_1$
itself, and each $R_{2\bullet}$
arrives at joint $m$ of the $\Sigma_2$ chain, which is the point $\Sigma_1+R$.
A split therefore chooses nothing but the point $\Sigma_1+R_{1\bullet}$ at which
the two halves meet (for the $ps$ split this meeting point is the $B_1$ of
\S\ref{sec:matrix-product}); the anchoring joints are common to all three.
Named for their splits, the three such points are
\begin{equation}\label{eq:B1-three}
B_{1rs}=\Sigma_1+\tfrac12R,
\qquad
B_{1ps}=\Sigma_1+R_{1ps}=B_1,
\qquad
B_{1ak}=\Sigma_1+R_{1ak}\approx f(s).
\end{equation}

\subsection{Kuznetsov's remainders}\label{sec:kuznetsov}

It is worth contrasting our remainders $R_{1ps},R_{2ps}$ with the
remainder of Alexey
Kuznetsov~\cite{Kuznetsov2025}, whose paper appeared on the arXiv in March 2025
and in print in 2026. It gives a simple and accurate algorithm for
computing $\zeta$, which we have found useful in our own numerical work.
There are two differences. The first is one of \emph{kind}: Kuznetsov's
remainder is an \emph{approximation} to $R$, whereas $R_{1ps}+R_{2ps}=R$ is an
\emph{exact} identity (Theorem~\ref{thm:decomp}), though, as we shall see
below, if $R_{1ak}$ is read as $f(s)-\Sigma_1$ the $ak$ split is exact too,
the approximation moving into the numerical evaluation of $f$. The second is one of
\emph{direction}: neither of his two halves lies along the next summand of its
Dirichlet sum, while each of ours is exactly that summand, shortened.

Where his pair comes from is worth a word, since Siegel has one remainder and
Kuznetsov has two. Siegel's \S2 opens at his equation~8, one Dirichlet sum plus
one contour integral, and the residue theorem carries that into his
equation~13, our \eqref{eq:siegel13}: still one integral, now standing beside
both sums. One integral, one remainder, and that single $R$ is what we split in
two. Kuznetsov opens instead at the symmetric formula of Siegel's \S3, Siegel's
equation~56, the one he says makes the functional equation evident. It
carries two integrals, one in $s$ and its reflection in $1-s$ weighted by
$\chi$, so his two halves arrive with the starting point rather than from a
later choice. Gabcke's rederivation~\cite{Gabcke1979} begins in the same place,
and Kuznetsov reaches it through the quadrature forms of
Galway~\cite{Galway2001} and Arias de Reyna~\cite{AriasDeReyna2011} rather than
from Siegel directly.

Writing Kuznetsov's approximate remainder as $R_{ak}$ (the subscript ``$ak$''
for Alexey Kuznetsov, and to distinguish it from Siegel's exact, unsubscripted
$R$), his method gives
\begin{equation}\label{eq:Rak}
R \;\approx\; R_{ak}
\;=\; -\tfrac12(-1)^{m}
\Bigl(I_{1} + \chi(s)\,I_{2}\Bigr),
\end{equation}
where $I_1,I_2$ are the two integrals of Kuznetsov's construction, each
evaluated as a short finite series in precomputed coefficients
$\omega_0,\omega_1[n],\lambda[n]$ (his notation; this $\omega_0$ is
unrelated to the direction angle $\omega$ of \eqref{eq:omega-def}); in
some of our calculations we use $8$ such
coefficients, which already gives a dozen or more digits of accuracy across
the range of interest. Just as we split $R$ into two pieces, the two terms of \eqref{eq:Rak}
define Kuznetsov's two half-remainders
\begin{equation}\label{eq:Rak-split}
R_{ak}=R_{1ak}+R_{2ak},\qquad
R_{1ak}=-\tfrac12(-1)^{m}I_{1},\qquad
R_{2ak}=-\tfrac12(-1)^{m}\chi\,I_{2},
\end{equation}
so that, in parallel with \eqref{eq:zeta-ps},
\[
\zeta \;\approx\; \Sigma_1 + R_{1ak} + \Sigma_2 + R_{2ak}.
\]
The split \eqref{eq:Rak-split} mirrors the functional equation: $R_{1ak}$ pairs
with the first Dirichlet sum $\Sigma_1$ and $R_{2ak}$ carries the
functional-equation factor $\chi$ alongside $\Sigma_2$.
Figure~\ref{fig:kuznetsov-zoom} compares the two routes from $\Sigma_1$ to
$\Sigma_1+R$: the exact fractional links on one side and Kuznetsov's dashed
pair on the other.

\begin{figure}[!htb]
\centering
\includegraphics[width=0.72\textwidth]{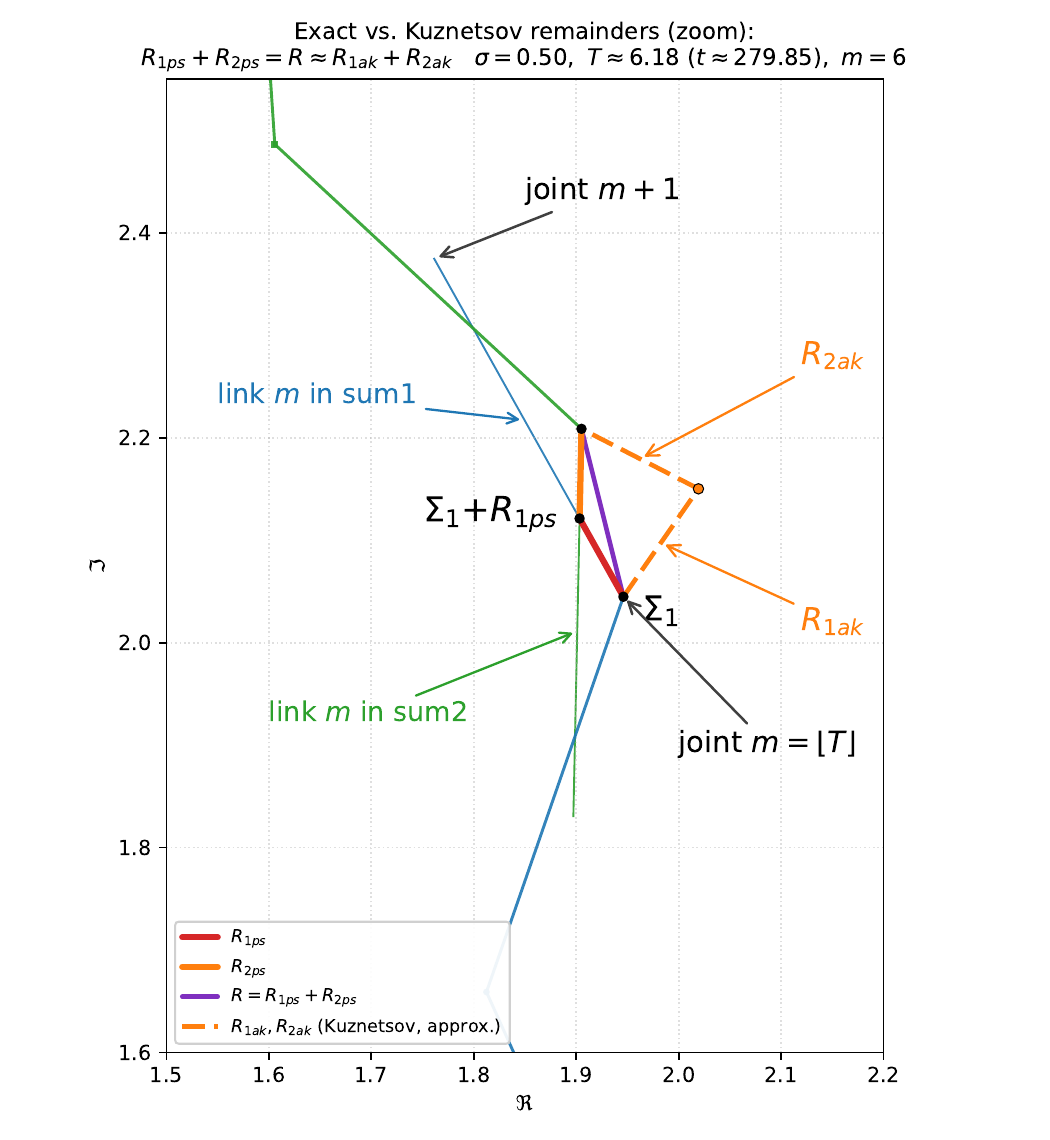}
\caption{Exact versus Kuznetsov remainders, in the same zoom region as
Figure~\ref{fig:remainder-zoom}. The exact fractional summands
$R_{1ps}$ (red) and $R_{2ps}$ (orange) run from $\Sigma_1$ to $\Sigma_1+R$ on
one side of the resultant $R$ (purple); Kuznetsov's approximate half-remainders
$R_{1ak},R_{2ak}$ are drawn as the orange dashed pair. Since $R_{1ak}+R_{2ak}\approx R$
(here to within $10^{-14}$), the dashed path lands on $\Sigma_1+R$; unlike the
exact links, the $ak$ links do not lie along a summand direction. Computed with
$8$ Kuznetsov coefficients and generated by \texttt{fig4\_kuznetsov\_zoom.py}.}
\label{fig:kuznetsov-zoom}
\end{figure}

This pairing is exactly Siegel's: in his 1932 paper~\cite{Siegel1932} Siegel
writes $\zeta$ as
$f(s)$ plus its reflected counterpart, and the first side is
\begin{equation}\label{eq:siegel-f}
f(s) \;\approx\; \Sigma_1 + R_{1ak},
\end{equation}
i.e.\ the first partial sum completed by an approximation to the first half of
the remainder. The sign is $\approx$ rather than $=$ because $R_{1ak}$ is the
quantity $-\tfrac12(-1)^{m}I_1$ of \eqref{eq:Rak-split}, with
$I_1$ evaluated as a truncated series; read the other way round, as the
\emph{definition} $R_{1ak}:=f(s)-\Sigma_1$, the relation is an identity and the
approximation moves into the numerical evaluation of $R_{1ak}$.

The form of \eqref{eq:siegel-f} is Siegel's own, and it is worth recording where
in his paper to look for it. He defines $f$ in his \S3, equation~58, as the
contour integral
\begin{equation}\label{eq:siegel-58}
f(s)
\;=\;
\int_{0\swarrow 1}
\frac{x^{-s}\,e^{\pi i x^{2}}}{e^{\pi i x}-e^{-\pi i x}}\,\mathrm{d}x ,
\end{equation}
where his symbol $0\swarrow 1$ marks a straight path crossing the real axis
between $0$ and $1$, running from upper right to lower left parallel to the
bisector of the first and third quadrants; it is the reflection in the real axis
of the path $0\nwarrow 1$ carrying the integral $\Phi$ of his \S1. The Gaussian
factor $e^{\pi i x^{2}}$ decays in both directions along that line, so
\eqref{eq:siegel-58} converges. Siegel then develops $f$ asymptotically in his \S4,
where his equation~84 reads
\begin{equation}\label{eq:siegel-84}
f(s)
\;=\;
\sum_{n=1}^{m_{1}} n^{-s}
\;+\;
O\!\bigl((\lvert s\rvert/2\pi e)^{-\sigma/2}\bigr)
\qquad (\sigma\ge 0,\; t>0),
\end{equation}
with $m_{1}=[\Re\eta_{1}-\Im\eta_{1}]$ and $\eta_{1}=\sqrt{s/2\pi i}$, which to
leading order is Siegel's cutoff $\lfloor\sqrt{t/2\pi}\rfloor$
of~\eqref{eq:siegel-m}. Earlier in the same section he states it in words as
well: developing $f$ asymptotically yields ``as a principal term a sum of
$[\sqrt{t/2\pi}\,]$ summands, that is $\sum_{n=1}^{m}n^{-s}$''. So ``first
partial sum plus a remainder'' is Siegel's own description of $f$, and it
shows the remainder is genuinely needed: at $\sigma=\tfrac12$ the bare
partial sum misses $f$ by $O(t^{-1/4})$, the gap that $R_{1ak}$ closes.
One caution on the cutoff: because
$\sqrt{t/2\pi}=T+\tfrac12-\tfrac1{24T}+O(T^{-2})$ by
Remark~\ref{rem:T-vs-a}, while
$\Re\eta_{1}-\Im\eta_{1}=\sqrt{t/2\pi}+\tfrac{\sigma}{4\pi T}+O(T^{-2})$,
Siegel's $m_{1}$ is one more than our $m=\lfloor T\rfloor$ whenever
$\{T\}>\tfrac12+\tfrac1{24T}-\tfrac{\sigma}{4\pi T}$; at $\sigma=\tfrac12$
the two corrections nearly cancel, the coefficient
$\tfrac1{24}-\tfrac1{8\pi}$ being about $0.0019$. This does not disturb
\eqref{eq:siegel-f}, since $R_{1ak}$ carries the cutoff as a parameter and
re-centers with it. Siegel's choice aligns the cutoff with the saddle and
yields the asymptotic remainder stated above. Equation numbers cited here are
those of Siegel's paper as reprinted
in his collected works, the numbering carried over by the Barkan--Sklar
translation~\cite{BarkanSklar2018}.\footnote{One incidental finding while
checking \eqref{eq:siegel-f} against the source: the translation's
equation~83 prints the sum as $\sum n^{s-1}$, while equation~84, derived from
it, has $\sum n^{-s}$. The numerics say $n^{-s}$ is correct, so equation~83
carries a typo in the Barkan--Sklar translation.}

Neither Siegel
nor Kuznetsov names the other side, but it is $\Sigma_2+R_{2ak}$, and adding the
two recovers $\zeta$ to the same accuracy. The contrast with our result is that
$\Sigma_1+R_{1ps}$ and $\Sigma_2+R_{2ps}$ are exact however they are read, and
each remainder there is a single fractional summand
(\S\ref{sec:summand}).

\subsection{Scaling the remainders with the fractional part of \texorpdfstring{$T$}{T} unchanged}\label{sec:remainder-scaling}

This subsection is a diversion into an observation we made about the
three remainders just discussed. At a fixed fractional part $\{T\}$,
the remainder figures at different heights are nearly scaled copies, with
the residual amplitude discrepancy tending toward zero as $T$ grows.
This scaling is the unit-$T$
periodicity of the yin and yang curves (\S\ref{sec:yinyang-general})
expressed in the raw remainders: the frame map there multiplies by $M^{s}$, which
removes exactly the $M^{-\sigma}$ size factor and the $\omega$ rotation
observed here, so a figure that repeats with $\{T\}$ in that frame must
reappear here as a nearly scaled copy at each $\{T\}$; and since all three
splits assemble the same $R$, the copy carries all three. Nothing in the
rest of the paper uses this observation.

Write $T=\lfloor T\rfloor+\{T\}$ for the integer and fractional parts of the
index. Empirically, the \emph{shape} of the remainder figure depends almost
entirely on the fractional part $\{T\}$. By shape we mean the relative angles
among $R$, $R_{1ps}$, $R_{2ps}$, $R_{1ak}$, and $R_{2ak}$ once the
configuration is rotated into the frame in which $R$ lies along the positive
real axis. The integer part $\lfloor T\rfloor$ mainly retunes the overall
size: the lengths scale like $M^{-\sigma}$ times a
slowly varying amplitude. Figure~\ref{fig:remainder-scale-grid} makes this
visible. Each panel shows only the remainders, drawn in the $R$-frame
(rotated so $R$ lies on the real axis and translated so the midpoint of $R$
sits at the origin), with a common axis scale. The top row holds
$\{T\}=0.18$ at $T=6.18$, $50.18$, and $100.18$; the bottom row holds
$\{T\}=0.72$ at $T=6.72$, $50.72$, and $100.72$. Across a row the figure
shrinks but keeps its shape; across a column (same $\lfloor T\rfloor$,
different $\{T\}$) the shape changes.

\begin{figure}[H]
\centering
\includegraphics[width=\textwidth]{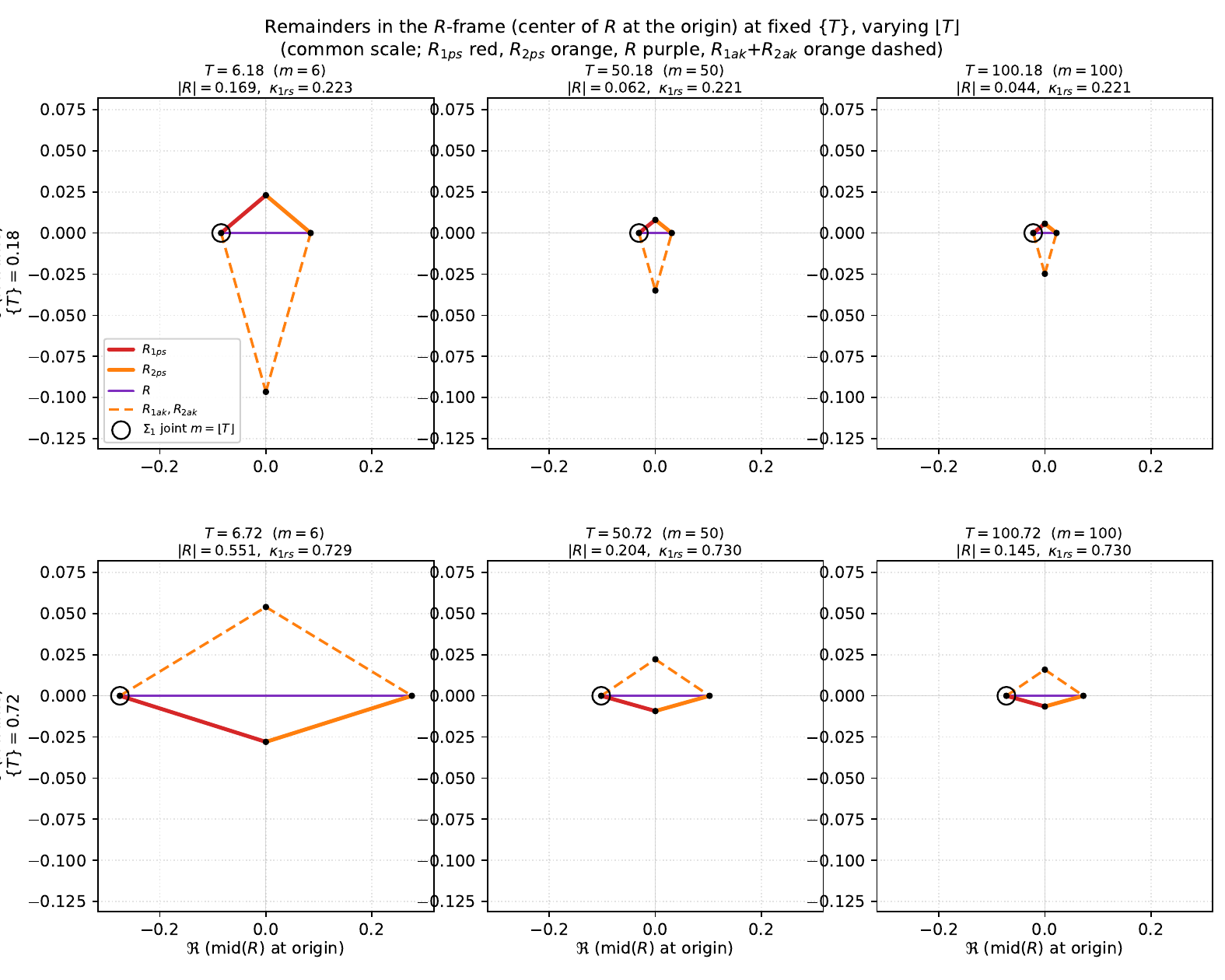}
\caption{Remainders in the $R$-frame at fixed fractional part $\{T\}$,
varying $\lfloor T\rfloor$. The frame is centered on the midpoint of $R$
(so $R$ runs from $-\lvert R\rvert/2$ to $+\lvert R\rvert/2$ on the real
axis). Top row: $\{T\}=0.18$ at $T=6.18$, $50.18$, $100.18$. Bottom row:
$\{T\}=0.72$ at $T=6.72$, $50.72$, $100.72$. Colors match
Figure~\ref{fig:kuznetsov-zoom}: $R_{1ps}$ red, $R_{2ps}$ orange, $R$
purple, $R_{1ak}$ and $R_{2ak}$ orange dashed. The large open circle marks the
left endpoint of $R$, which is the anchoring joint $\Sigma_1$ at $m=\lfloor
T\rfloor$ that all three splits leave from. All six panels share the
same axis scale. Generated by \texttt{fig\_remainder\_scale\_panels.py}.}
\label{fig:remainder-scale-grid}
\end{figure}

That figure holds $\{T\}$ fixed and climbs the strip. The complementary view
holds the chain fixed and lets $T$ sweep a whole unit, which is
Figure~\ref{fig:remainder-frame-sweep}: the same $R$-frame in sixteen snapshots
from $T=4$ to $T=5$ at $m=4$. Two things
show up there that a fixed $\{T\}$ cannot show. First, as $T$ sweeps, both
apexes stay pinned to the perpendicular bisector of $R$: on the critical line
$|R_{1ps}|=|R_{2ps}|$ exactly (by
Proposition~\ref{prop:weights}(iii), since $d_1=d_2$ and the two directions
are unit), likewise for the $ak$ pair, and this holds at every instant of the
motion, so the state of a split is at all times the single number giving the
height of its apex above the chord.
Second, the two splits move quite differently: the $ak$ apex marches steadily
from below $R$ to well above it, while the $ps$ apex swings down onto the chord,
crosses to the far side, and comes back, touching the chord exactly at the two
heights where $\sin(2\omega+\arg\chi)$ vanishes
(Appendix~\ref{sec:pole-locations}). There the two unit directions of
\eqref{eq:R-two-dirs} have become parallel and no longer span the plane, and
on the critical line the split collapses into
the plain halving $R_{1ps}=R_{2ps}=\tfrac12R$: the $d_1$ pole of
\eqref{eq:d1} is removable at $\sigma=\tfrac12$, its numerator vanishing with
its denominator, which is why the sweep passes through those heights smoothly.
Off the line it is a genuine pole: at $\sigma=0.6$ and $\sigma=0.75$, $d_1$
runs through infinity and changes sign there, and that is what the windows of
Appendix~\ref{sec:pole-locations} are.

\begin{figure}[H]
\centering
\includegraphics[width=0.82\textwidth]{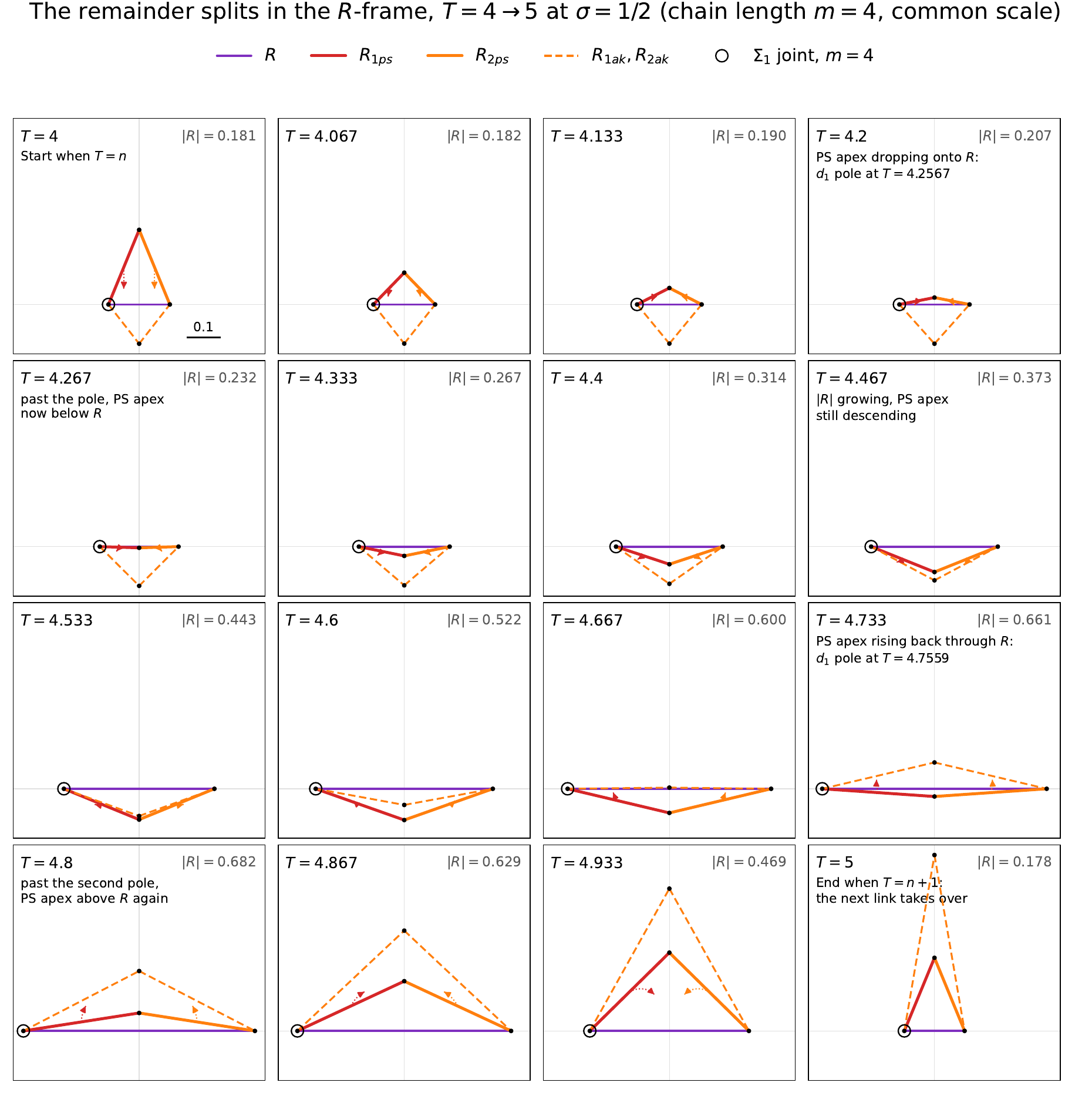}
\caption{The remainder splits in the $R$-frame as $T$ sweeps one unit,
$T=4\to5$ at $\sigma=\tfrac12$. Frame and colors as in
Figure~\ref{fig:remainder-scale-grid}: $R$ purple along the real axis,
midpoint at the origin; $R_{1ps}$ red and $R_{2ps}$ orange from the $\Sigma_1$
joint (open circle) to the far end of $R$; $R_{1ak}$, $R_{2ak}$ orange
dashed; all panels to one scale, $|R|$ given in each. The chain
length is held at $m=4$ throughout, so $T=5$ is the handoff. Dotted arrows
trace the midpoints of $R_{1ps}$ and $R_{2ps}$ from each panel to the next.
The $ps$ apex crosses the chord between panels four and five and again
between twelve and thirteen, the two heights of
\eqref{eq:pole-condition} where the split is the plain halving; the $ak$ apex
crosses once, between ten and eleven.
Generated by \texttt{fig\_remainder\_frame\_sweep.py}.}
\label{fig:remainder-frame-sweep}
\end{figure}

To make the scaling precise, factor out the leading $M^{-\sigma}$ from each
half-remainder length by writing
\[
|R_{1\bullet}| \;=\; \kappa_{1\bullet}\,M^{-\sigma}
\]
for the three first halves
$R_{1\bullet}\in\{R_{1rs},\,R_{1ps},\,R_{1ak}\}$ of
\S\ref{sec:remainders-summary}. Each split is then on the same footing:
$R=R_{1rs}+R_{2rs}=R_{1ps}+R_{2ps}\approx R_{1ak}+R_{2ak}$, and for the
$rs$ half-split one has $R_{1rs}=R_{2rs}=\tfrac12 R$ so
$\kappa_{2rs}=\kappa_{1rs}$. The three amplitudes admit the forms
\begin{equation}\label{eq:kappa}
\begin{aligned}
\kappa_{1rs}
&=
\tfrac12\,M^{\sigma}\,|R|
\;\approx\;
\tfrac14\,M^{\sigma}\,\bigl|I_{1}+\chi(s)\,I_{2}\bigr|,
\\[4pt]
\kappa_{1ps}
&=
|d_{1}|\,M^{\sigma}
=
2\kappa_{1rs}\,
\left|
\frac{\sin(\omega-\arg R+\psi)}{\sin(2\omega+\psi)}
\right|,
\\[4pt]
\kappa_{1ak}
&=
\tfrac12\,M^{\sigma}\,|I_{1}|,
\end{aligned}
\end{equation}
with $\psi=\arg\chi(s)$ and
$I_1,I_2$ the Kuznetsov integrals of \S\ref{sec:kuznetsov}; the
``$\approx$'' in the first line holds when $R$ is evaluated by Kuznetsov's
method \eqref{eq:Rak}, and the second equality in the $\kappa_{1ps}$ line is
exact with $\kappa_{1rs}$ read as $\tfrac12 M^{\sigma}|R|$. Note that
$\kappa_{1ps}=|\hat d_1|$ is a nonnegative normalized length. The signed
normalized coefficient remains $\hat d_1=d_1M^{\sigma}$, as in
\eqref{eq:frac-vs-weight}. Consequently, at
fixed $\sigma$ and fixed $\{T\}=f$, the passage from $T_1=N_1+f$ to
$T_2=N_2+f$ (write $M_i=N_i+1$) multiplies every half-remainder length by
\begin{equation}\label{eq:remainder-scale}
\frac{|R_{1\bullet}|_{T_2}}{|R_{1\bullet}|_{T_1}}
=
\Bigl(\frac{M_1}{M_2}\Bigr)^{\sigma}
\frac{\kappa_{1\bullet}(T_2)}{\kappa_{1\bullet}(T_1)}.
\end{equation}
When the fractional part is held fixed the ratio of $\kappa$'s is nearly $1$,
so the leading factor $\bigl(M_1/M_2\bigr)^{\sigma}$ already accounts
for almost all of the size change. Figure~\ref{fig:remainder-scale-match}
checks this in the $R$-frame: the left panel is $T=6.18$; the right panel is
$T=50.18$ with all remainder vectors scaled by the single factor
$\lambda=|R|_{6.18}/|R|_{50.18}$, the $rs$ instance
$\bigl(M_2/M_1\bigr)^{\sigma}\,\kappa_{1rs}(T_1)/\kappa_{1rs}(T_2)$ of
\eqref{eq:remainder-scale}. Scaling by $\lambda$ makes the two $R$'s match by
construction; the test is that the split halves nearly coincide too, and they
do, $|R_{1ps}|$ agreeing to about $0.3\%$.

\begin{figure}[H]
\centering
\includegraphics[width=0.92\textwidth]{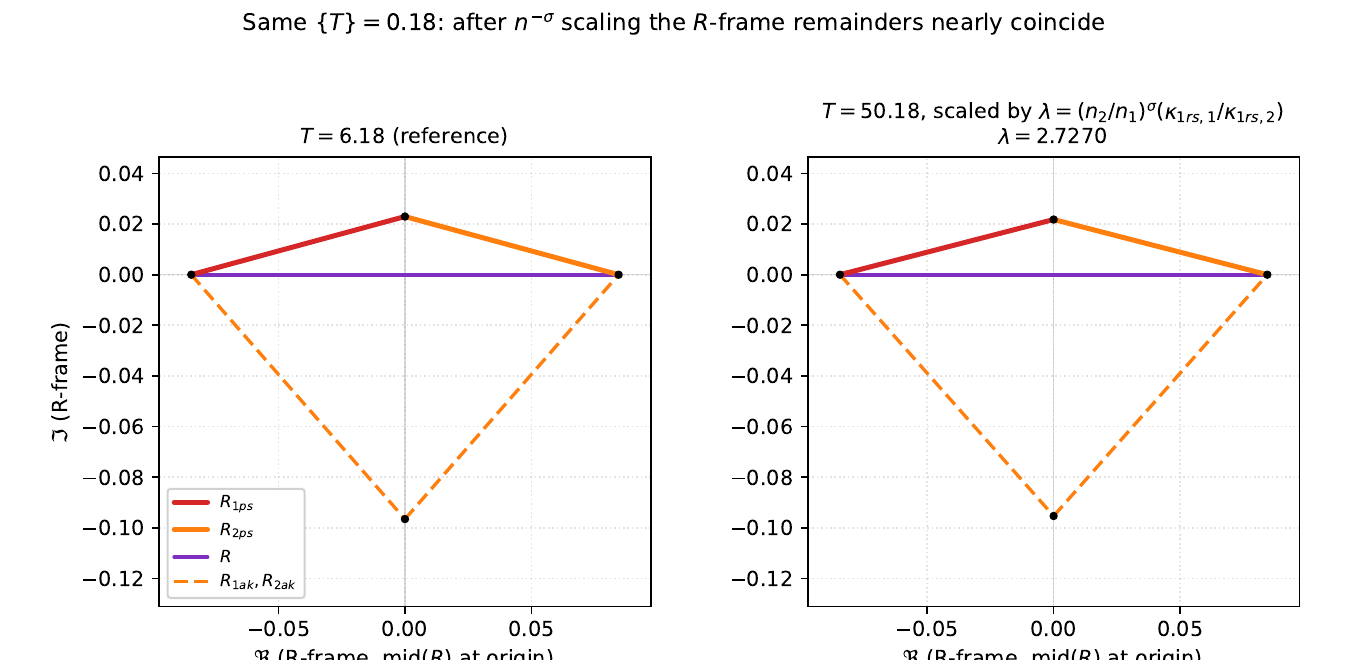}
\caption{Same fractional part $\{T\}=0.18$: $R$-frame remainders at
$T=6.18$ (left) and at $T=50.18$ after scaling by
$\lambda=\bigl(M_2/M_1\bigr)^{\sigma}\,\kappa_{1rs}(T_1)/\kappa_{1rs}(T_2)$
(right). The scaled figure matches the reference to high accuracy
($|R_{1ps}|$ agrees to about $0.3\%$). Generated by
\texttt{fig\_remainder\_scale\_panels.py}.}
\label{fig:remainder-scale-match}
\end{figure}

If one wants exactness rather than a close empirical match, divide by $R$
itself. By \eqref{eq:d1} the modulus of $R$ cancels identically, leaving
\begin{equation}\label{eq:R1ps-over-R}
\frac{R_{1ps}}{R}
\;=\;
\frac{\sin\bigl(\omega-\arg R+\arg\chi\bigr)}
     {\sin\bigl(2\omega+\arg\chi\bigr)}\;
e^{-i(\omega+\arg R)},
\end{equation}
a function of the angles alone, exact at every $T$ and every $\sigma$ (off
the line $\arg R$ must itself be computed; on the line it is $-\vartheta$
modulo $\pi$): the size is then gone exactly, not almost. On the critical line the equal weights
of Proposition~\ref{prop:weights}(iii) force
$\operatorname{Re}(R_{1ps}/R)=\tfrac12$,
so the shape is carried by the single number $\operatorname{Im}(R_{1ps}/R)$,
and that number is a function of $\{T\}$ up to $O(1/T)$, which is the content
of Theorem~\ref{thm:d1-limit}.

\begin{figure}[H]
\centering
\includegraphics[width=0.95\textwidth]{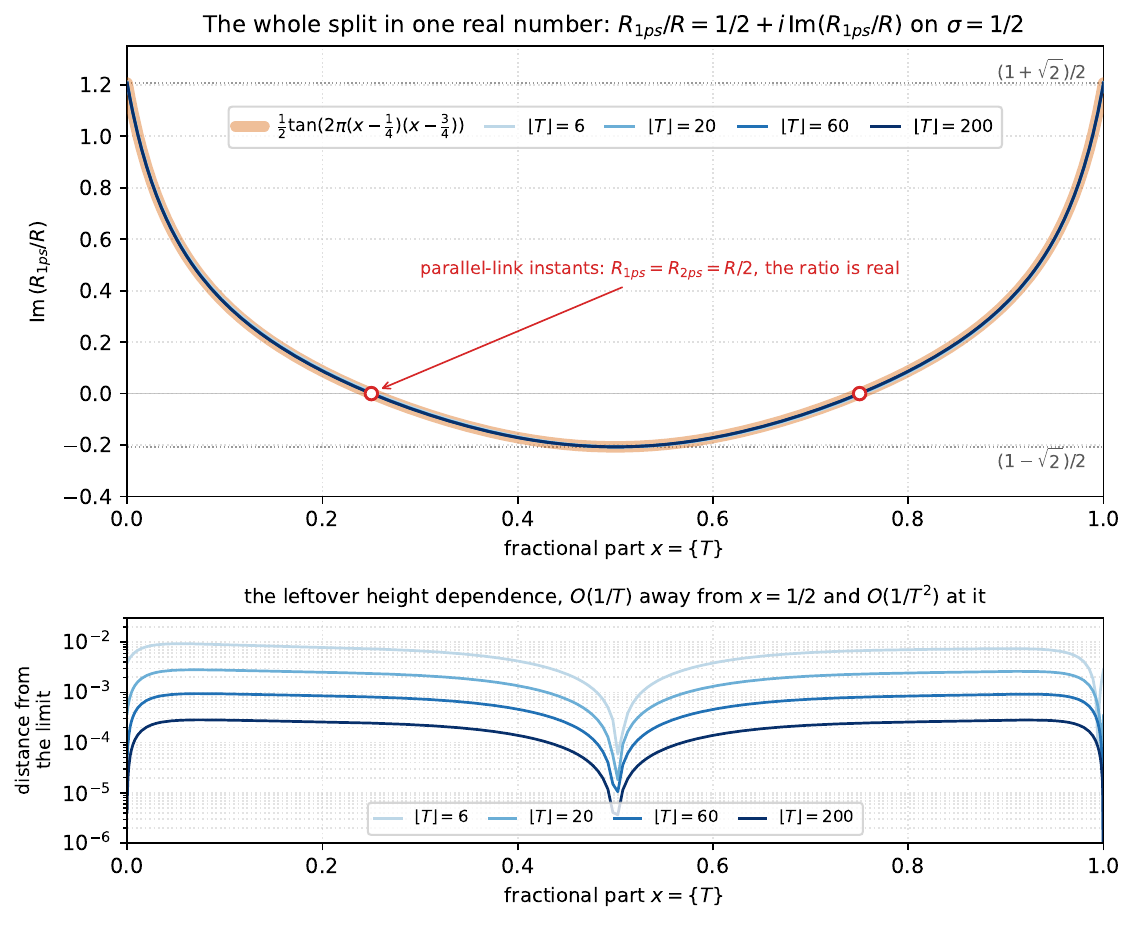}
\caption{The split after division by $R$, on the critical line. Top:
$\operatorname{Im}(R_{1ps}/R)$ against $x=\{T\}$ at
$\lfloor T\rfloor=6,20,60,200$, each computed through $\zeta$ and the
Cramer solution \eqref{eq:d1}, laid over the limit
$\tfrac12\tan\bigl(2\pi(x-\tfrac14)(x-\tfrac34)\bigr)$ (thick orange). The
limit falls from $(1+\sqrt2)/2$ at the integers to $(1-\sqrt2)/2$ at
$x=\tfrac12$, a swing of exactly $\sqrt2$, and vanishes at the two
parallel-link instants (circles; $x=0.250156$ and $0.750155$ at
$\lfloor T\rfloor=200$), where $R_{1ps}=R_{2ps}=\tfrac12R$. Bottom: distance
from the limit, $O(1/T)$ in general and $O(1/T^{2})$ at $x=\tfrac12$.
Generated by \texttt{fig\_ratio\_shape.py}.}
\label{fig:ratio-shape}
\end{figure}

Figure~\ref{fig:ratio-shape} dissects that single number. On the critical
line $\arg R=\tfrac12\arg\chi$ and $\chi=e^{-2i\vartheta}$, with
$\vartheta$ the Riemann--Siegel theta function of
Appendix~\ref{sec:phase-bound}, so
$\omega-\arg R+\arg\chi=\omega-\vartheta$ and
$2\omega+\arg\chi=2(\omega-\vartheta)$, the $\pi$ ambiguity in $\arg R$
flipping the sine and the exponential of \eqref{eq:R1ps-over-R} together.
That equation therefore collapses to
\begin{equation}\label{eq:ratio-tangent}
\frac{R_{1ps}}{R}
\;=\;
\frac{e^{-i(\omega-\vartheta)}}{2\cos(\omega-\vartheta)}
\;=\;
\frac12-\frac{i}{2}\tan(\omega-\vartheta),
\end{equation}
which is the exact critical-line form $d_1=r/(2\cos u)$ of
\eqref{eq:exact-line} divided by $R$. Three things follow, and
Figure~\ref{fig:ratio-shape} shows all three. First, the real part is $\tfrac12$ identically,
which is Proposition~\ref{prop:weights}(iii) once more; the plotted values were
computed the long way, through $\zeta$ and \eqref{eq:d1}, and at all $964$
samples their real parts agree with $\tfrac12$ to thirty digits. Second,
everything that is left is the one angle $\omega-\vartheta$ modulo $\pi$,
and that angle tends to $-\beta$, where
$\beta=2\pi(x-\tfrac14)(x-\tfrac34)$ and $x=\{T\}$; this is the same
reduction that turns $2\cos(\omega-\vartheta)$ into $2\cos\beta$ in the
proof of Theorem~\ref{thm:d1-limit}. So
$\operatorname{Im}(R_{1ps}/R)\to\tfrac12\tan\beta$, which is the second of
the two tangent factors in the limit waveform \eqref{eq:W-infty}: that
waveform is $-\tan(2\pi x)$ times this very quantity. The four
curves lie within $9.3\times10^{-3}$,
$2.8\times10^{-3}$, $9.4\times10^{-4}$ and $2.8\times10^{-4}$ of the limit at
$\lfloor T\rfloor=6,20,60,200$, shrinking in proportion to the index, which
is the $O(1/T)$ rate of the theorem seen directly. Third, the range is
fixed by $\beta$, which runs from $\tfrac{3\pi}{8}$ at the integers to
$-\tfrac{\pi}{8}$ at $x=\tfrac12$: the curve therefore lies between
$(1+\sqrt2)/2$ and $(1-\sqrt2)/2$, swinging by exactly $\sqrt2$, and the
tangent stays $\tfrac{\pi}{8}$ clear of its pole at $\tfrac{\pi}{2}$,
which is the boundedness of $d_1$ on the line read as an angle. It crosses
zero where $\beta$ does, at the parallel-link instants, where the split is
the plain halving as above.

\section{The geometry behind the result (experimental mathematics)}\label{sec:geometry}

The decomposition $R=R_{1ps}+R_{2ps}$ was not first derived; it was first
\emph{seen}. Over roughly nine years we built and used an interactive
visualization tool, Zest,\footnote{Source code and standalone executables (Mac
and PC) and a web app at \url{https://github.com/paul-stahura/zest}.} which draws the
partial-sum chains of $\zeta$ in the complex plane and lets us sweep $\sigma$
and $T$ continuously while measuring what appears on the screen. This section
reconstructs the series of observations made with that tool. It led, in order,
to a distinguished link and point of the spiral (the \emph{bisector link} and
\emph{bisector point}), and from their motion to the mapping $I(T)$ of
\S\ref{sec:IT}. The final observation in that sequence, the \emph{yin} and
\emph{yang} curves from which the weights $d_1,d_2$ of \S\ref{sec:decomp}
were read off, was developed in \S\ref{sec:yinyang-inf} and
\S\ref{sec:yinyang-general}; \S\ref{sec:geometric-origins} reconstructs how
both the mapping and the curves were found.

Throughout we use the link and joint numbering of
\S\ref{sec:matrix-product}: joint $k$ of a chain is the value of its partial
sum after $k$ summands, and link $k$ runs from joint $k$ to joint $k+1$, so
that link $k$ represents the $(k+1)$st summand and each remainder is the
fractional link $m$. The \emph{forward} link chain ($\Sigma_1$) is
anchored at joint $0$, the origin, and it represents the main sum. The
\emph{reverse} chain ($\Sigma_2$), representing the dual sum, is
anchored at $\zeta$ and traversed backwards, so its joint $m$ lands at
$\zeta-\Sigma_2=\Sigma_1+R$. In the pictures the chains are drawn well past
joint $m=\lfloor T\rfloor$: the sum, though ultimately divergent, first winds
into an Euler-like spiral whose center lies near $\zeta$; it seemingly
converges before it diverges. Figure~\ref{fig:full-spirals} shows the picture:
the forward chain winds from the origin into a spiral centered near $\zeta$,
while the reverse chain winds from $\zeta$ into a mirror spiral centered near
the origin, and the two link chains cross. If more links were drawn, the links would
continue to spiral outward around $\zeta$ and cover the entire figure, because
the zeta sum diverges on the critical strip, unlike the Dirichlet
$L$-functions of \S\ref{sec:IT-Lfunctions}, whose non-principal characters have bounded
partial sums and
so those chains close in on their limit instead of covering the plane.\footnote{At $t=176.13$ on the critical line, the partial
sums of $L(s,\chi_3)$ come within $0.021$, $0.0021$ and $0.00019$ of the value
by $N=10^{3}$, $10^{5}$ and $3\times10^{6}$, a drift inward like $N^{-1/2}$,
while zeta's are off by $0.18$, $1.80$ and $9.83$, a drift outward at
$\sqrt{N}/t$. For a principal character modulo $q$,
\[
L(s,\chi_0)=\zeta(s)\prod_{p\mid q}(1-p^{-s}),
\]
rather than $\zeta(s)$ itself. The convergence behavior of higher-degree
$L$-function Dirichlet series depends on the family and normalization, so
no blanket critical-strip claim is made here.} We stop drawing links at link number
\begin{equation}\label{eq:LN}
L_N(T,S_n)=\left\lfloor\frac{I(T)}{\pi\,(2S_n+1)}\right\rfloor,
\end{equation}
where $S_n$ is the spiral number, spiral number $0$ being the last one (the
biggest, at $\zeta$). Note that $\zeta$ itself is not typically \emph{on} this
last link; the last link is the link closest to $\zeta$. The same formula
locates the other spirals and pseudo-spirals too: with $S_n=1$, for example, it yields the link
number of the link most in the center of the second-to-last spiral.

\begin{figure}[p]
\centering
\includegraphics[height=0.82\textheight,keepaspectratio]{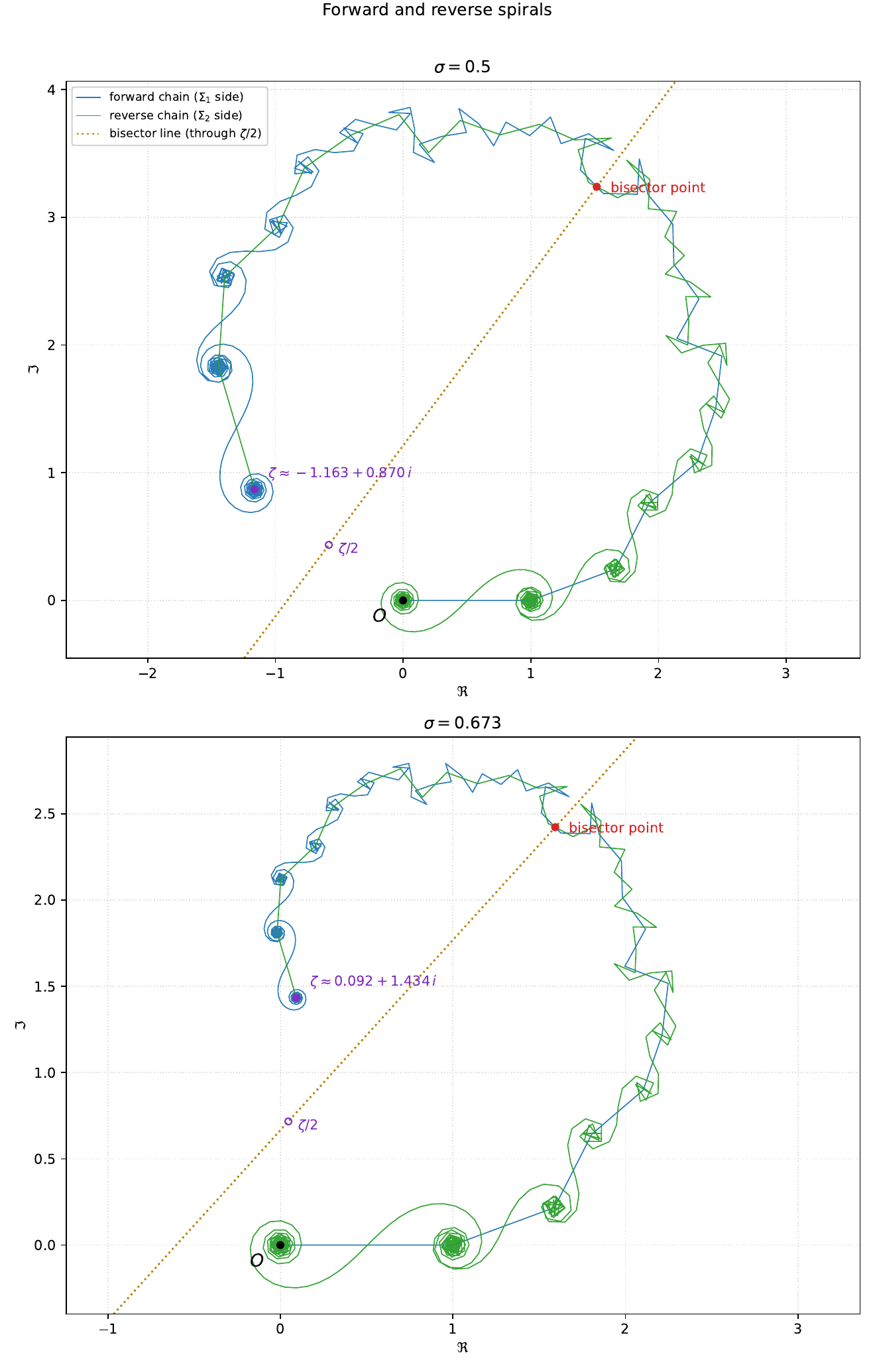}
\caption{The forward spiral (blue, anchored at the origin $O$) and the reverse
spiral (green, anchored at $\zeta$) for $T=14.3085$
(so $t=I(T)\approx 1377.33$ and $m=14$), on the critical line
$\sigma=\tfrac12$ (top, $\zeta\approx-1.163+0.870i$) and off it at
$\sigma=0.673$ (bottom, $\zeta\approx 0.092+1.434i$). Each spiral is drawn out
to link $\lfloor I(T)/\pi\rfloor=438$, the link nearest its spiral center, for
a total of $439$ links per spiral (the count is independent of $\sigma$). The
red dot is the bisector point $B_1=\Sigma_1+R_{1ps}$, where the two chains
cross at their bisector links (\S\ref{sec:bisector}); the dotted line is the
bisector line through the bisector point and $\zeta/2$ (open circle). On the
critical line (top) the whole configuration is symmetric about the bisector
line; at $\sigma=0.673$ (bottom) the symmetry is lost. Generated by
\texttt{fig\_full\_spirals.py}.}
\label{fig:full-spirals}
\end{figure}

\begin{figure}[p]
\centering
\includegraphics[height=0.8\textheight,keepaspectratio]{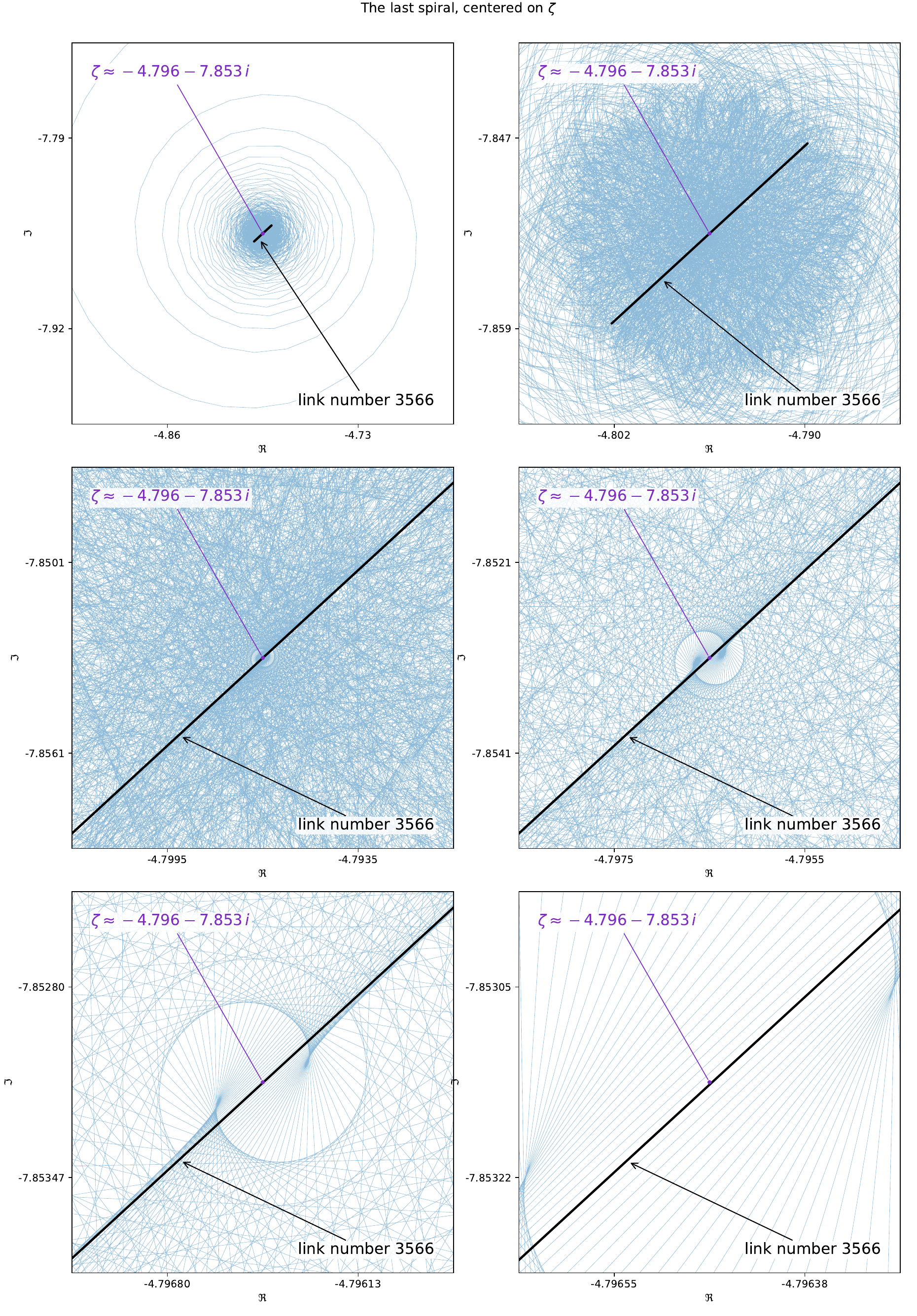}
\caption{Zooming in on the last spiral, the spiral centered near
$\zeta$, of the forward chain at $T=41.73$ (so
$t=I(T)\approx 11204.74$, $\sigma=\tfrac12$,
$\zeta\approx-4.796-7.853i$). The chain carries
$\lfloor I(T)/\pi\rfloor+1=3567$ links; the last link is number $3566$
by \eqref{eq:LN} with $S_n=0$, and it is highlighted in black in every
panel.
Each panel is more zoomed in than the last (panels $1$--$6$, left to
right from the top). Near the
spiral center each link turns by nearly $\pi$, so the links sweep back
and forth across $\zeta$ and the joints settle onto a tight ring of
radius about half a link length: at coarse scale the spiral renders as
a dense annulus (panels $1$--$3$), and the near-diameter links envelope
a lens-shaped clearing around $\zeta$ (panels $4$--$5$). Observe in
panel $6$ that $\zeta$ is \textbf{\emph{not on}} the last link: the last link
is only the link closest to $\zeta$. Generated by
\texttt{fig\_last\_spiral\_zoom.py}.}
\label{fig:last-spiral-zoom}
\end{figure}

\clearpage
\subsection{The bisector link and the bisector point}\label{sec:bisector}

Drawn at height $T$, each link chain exhibits spirals and pseudo-spirals, and (as
others have also observed~\cite{Erickson2005,Nickel2013,Kapitonets2019}) there is a one-to-one correspondence between the
joints and the spiral centers, in a type of symmetry; Nickel puts it as the
center-to-center distances and angles of the large linked spirals reproducing
the initial steps exactly. The link
at index $m=\lfloor T\rfloor$ sits at the boundary between the two
regimes, and we call it the \emph{bisector link}. Each chain has one: the \emph{forward bisector link} (link $m$ of the
$\Sigma_1$ chain, running from $\Sigma_1$ along the $(m+1)$st summand
direction) and the \emph{reverse bisector link} (link $m$ of the $\Sigma_2$
chain, running from $\Sigma_1+R$ along its $(m+1)$st summand direction).

The two spirals intersect at their bisector links, and we call the point where
the forward and reverse bisector links cross the \emph{bisector point}. One
qualification: the crossing can momentarily fail. Twice per unit of $T$, near
fractional parts ${\approx}\,0.25$ and ${\approx}\,0.75$, the two bisector
links become parallel and, off the critical line, do not intersect; these are
the isolated instants at which $\sin(2\omega+\arg\chi)$, the denominator of
\eqref{eq:d1}--\eqref{eq:d2}, vanishes. Those are exactly the places where
$d_1$ and $d_2$ have poles: the parallel-link condition and the pole
condition are the same equation, which is why the precise locations were
worked out in Appendix~\ref{sec:pole-locations}. On the critical line the two links are then collinear rather than
merely parallel, so the bisector point survives there.
In the notation of the earlier sections the bisector point is exactly
\begin{equation}\label{eq:bisector-point}
B_1 \;=\; \Sigma_1 + R_{1ps} \;=\; \Sigma_1 + R - R_{2ps};
\end{equation}
it is visible as the joint between the red and orange fractional
links in Figures~\ref{fig:spiral-summands} and~\ref{fig:remainder-zoom}.

\subsection{Legs}\label{sec:legs}

When $\sigma=\tfrac12$ the spiral object
has a type of symmetry about the line through the bisector point and
$\tfrac{\zeta}{2}$, so the bisector point sits at the apex of an isosceles
triangle whose base runs from the origin to $\zeta$ and whose equal sides are
the two legs $B_1$ and $\zeta-B_1$ (this is Corollary~\ref{cor:equal-legs}
below, proved from the reflection
$B_2=\chi\overline{B_1}$ and verified in Lean in
Appendix~\ref{app:lean-critical}; equivalently, the perpendicular projection of the
bisector point onto the base $\mathbb{R}e^{\,i\psi/2}$ carrying $\zeta$, where
$\psi=\arg\chi$, is exactly $\tfrac{\zeta}{2}$, which is
Corollary~\ref{cor:bisector-proj}). Off the line we keep the same
terminology even though the triangle's legs are then (usually) unequal.

\begin{figure}[H]
\centering
\includegraphics[width=\textwidth]{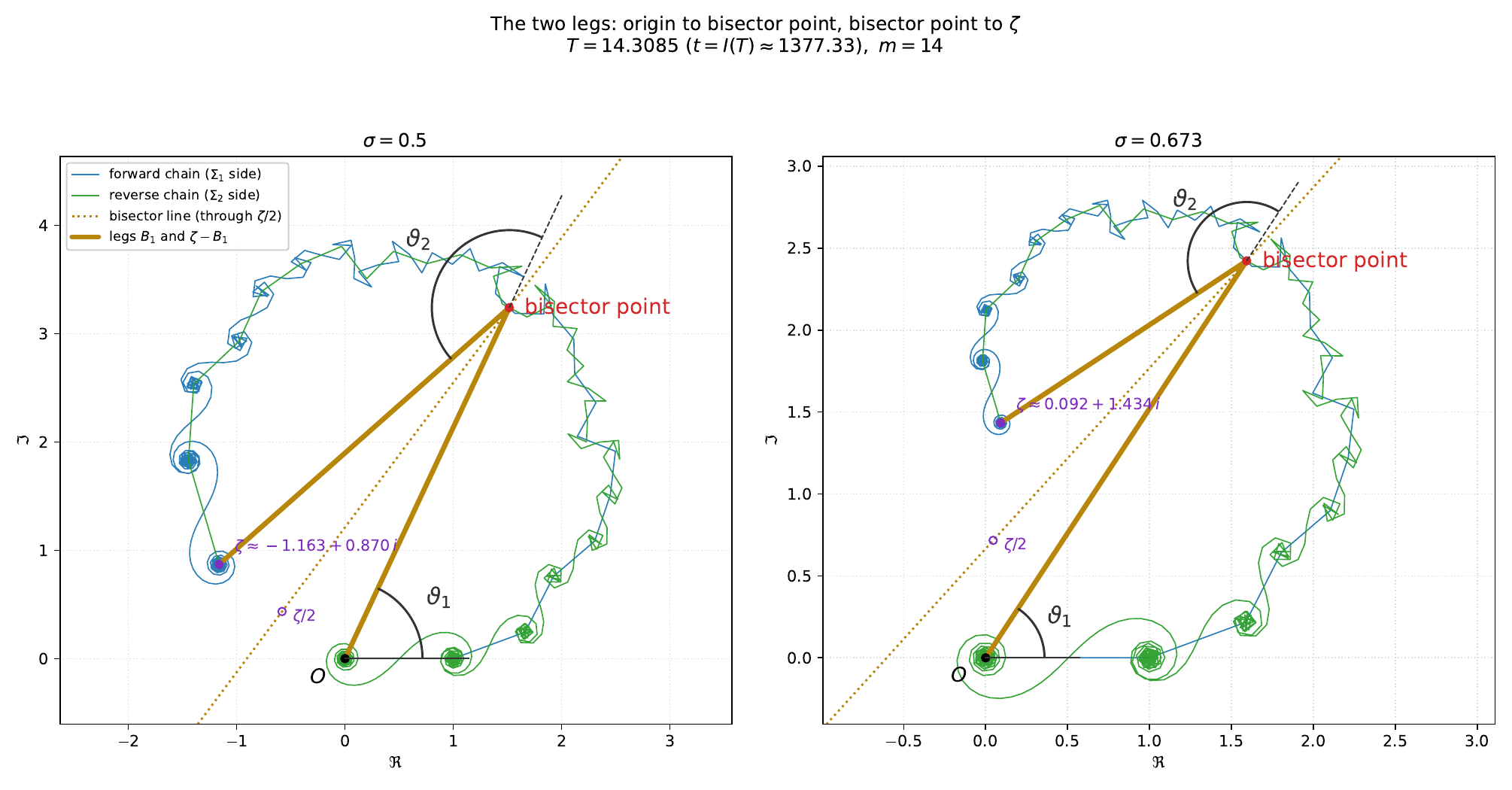}
\caption{The two legs, drawn as thick dark-yellow segments over the spirals of
Figure~\ref{fig:full-spirals}: one from the origin $O$ to the bisector point
($B_1$), the other from the bisector point to $\zeta$ ($\zeta-B_1$). On the
critical line (left, $\sigma=\tfrac12$) the legs have equal length
($|B_1|=|\zeta-B_1|\approx 3.577$): the bisector point is the apex of an
isosceles triangle whose base runs from $O$ to $\zeta$, and the dotted
bisector line is its axis of symmetry. Off the critical line (right,
$\sigma=0.673$) the legs are unequal ($\approx 2.899$ and $\approx 1.797$) at
this $T$. Both panels also mark the two angles defined below: $\vartheta_1$,
at $O$ from the positive real axis to Leg~1, and $\vartheta_2$, at the
bisector point from the extension of Leg~1 (dashed) to Leg~2.
Generated by \texttt{fig\_legs.py}.}
\label{fig:legs}
\end{figure}

On the critical line the two fractional summands are not independent: each is
the other's complex conjugate, rotated by the functional-equation factor.
Writing $R_{1ps}=d_1e^{-i\omega}$ and $R_{2ps}=d_2e^{\,i(\omega+\arg\chi)}$ as in
\S\ref{sec:matrix-product} and using that $|\chi|=1$ when $\sigma=\tfrac12$, the
equality $d_1=d_2$ of Proposition~\ref{prop:weights}(iii) is exactly
equivalent to the compact relation
\begin{equation}\label{eq:R2-conj-R1}
R_{2ps}=\chi\,\overline{R_{1ps}}\qquad\bigl(\sigma=\tfrac12\bigr).
\end{equation}
Geometrically, \eqref{eq:R2-conj-R1} says the second fractional summand is the
mirror image of the first (conjugation across the real axis) followed by
the rotation $\chi$. Because $\chi$ has unit modulus on the line this rotation
is an isometry, so $|R_{1ps}|=|R_{2ps}|$ immediately. This is the
functional-equation reflection $s\leftrightarrow 1-s$ acting as a symmetry of
the two remainders, and it is the source of the equal-leg phenomenon we now
develop.

The two legs $B_1=\Sigma_1+R_{1ps}$ and $B_2=\Sigma_2+R_{2ps}$ of
\S\ref{sec:matrix-product} meet at a ``bisector point'': Leg~1 runs from the
origin to $\Sigma_1+R_{1ps}$, and Leg~2 from there to $\zeta$. Their lengths are
\begin{align}
L_{1} &= \bigl|\,\Sigma_{1}+R_{1ps}\,\bigr|,\\
L_{2} &= \bigl|\,\Sigma_{2}+R_{2ps}\,\bigr|,
\end{align}
while $\vartheta_1$ is the angle Leg~1 makes with the real axis and $\vartheta_2$
is the angle from Leg~1 to Leg~2 (oriented so that a positive $\vartheta_2$
rotates Leg~1 toward Leg~2):\footnote{The subscripted leg angles
$\vartheta_1,\vartheta_2$ are not to be confused with the Riemann--Siegel
theta function $\vartheta(t)$ of Appendix~\ref{sec:phase-bound}.}
\begin{align}
\vartheta_{1} &= \arg\!\bigl(\Sigma_{1}+R_{1ps}\bigr),\\[2pt]
\vartheta_{2} &= \arg\!\left(\frac{\Sigma_{2}+R_{2ps}}{\Sigma_{1}+R_{1ps}}\right).
\end{align}
With these, $\zeta$ is simply the sum of the two legs,
\begin{equation}\label{eq:zeta-legs-polar}
\zeta
= L_{1}\,e^{\,i\vartheta_{1}}
\;+\;
L_{2}\,e^{\,i(\vartheta_{1}+\vartheta_{2})},
\end{equation}
or equivalently, in matrix form,
\begin{equation}\label{eq:zeta-matrix}
\zeta
=
\begin{pmatrix}1 & i\end{pmatrix}
\begin{pmatrix}
\cos\vartheta_{1} & \cos(\vartheta_{1}+\vartheta_{2})\\[2pt]
\sin\vartheta_{1} & \sin(\vartheta_{1}+\vartheta_{2})
\end{pmatrix}
\begin{pmatrix}
L_{1}\\[2pt]
L_{2}
\end{pmatrix},
\end{equation}
where the row $\begin{pmatrix}1 & i\end{pmatrix}$ turns the resulting real
$(x,y)$ components into the complex value $x+iy$.

\begin{corollary}\label{cor:equal-legs}
When $\sigma=\tfrac12$ the two legs are always equal in length: $L_1=L_2$.
\end{corollary}

This is immediate from Proposition~\ref{prop:weights}(iii):
on the critical line every summand of $\Sigma_2$ is $\chi$ times the
conjugate of the corresponding summand of $\Sigma_1$, and by
\eqref{eq:R2-conj-R1}, which is equivalent to $d_1=d_2$, the same holds for the
fractional summands. Hence
$\Sigma_2+R_{2ps} = \chi\,\overline{\Sigma_1+R_{1ps}}$, and since $|\chi|=1$
on the line, the two legs have the same length. A Lean proof is recorded in
Appendix~\ref{app:lean-critical}.

The same reflection locates the apex exactly, which is what the dotted bisector
line of Figures~\ref{fig:full-spirals} and~\ref{fig:legs} draws.

\begin{corollary}\label{cor:bisector-proj}
When $\sigma=\tfrac12$ the perpendicular projection of the bisector point $B_1$
onto the line through the origin carrying $\zeta$ is exactly $\tfrac{\zeta}{2}$:
writing $\psi=\arg\chi$,
\begin{equation}\label{eq:bisector-proj}
\Re\!\bigl(e^{-i\psi/2}B_1\bigr)\,e^{\,i\psi/2}\;=\;\frac{\zeta}{2}.
\end{equation}
\end{corollary}

Indeed $B_2=\chi\overline{B_1}$ gives
$e^{-i\psi/2}\zeta=e^{-i\psi/2}B_1+\overline{e^{-i\psi/2}B_1}
=2\Re(e^{-i\psi/2}B_1)$, and $e^{\,i\psi/2}$ is the direction of $\zeta$ (this is
the reality of the Riemann--Siegel $Z$ function). So the isosceles triangle with
base from $O$ to $\zeta$ has its apex over the midpoint of that base, and the
equal-leg statement $L_1=L_2$ is the same fact read as a distance. This is
also proved in Appendix~\ref{app:lean-critical}.

\paragraph{Consequences for zeta zeros.} A zero of $\zeta$ in the critical strip
requires \emph{both} of the following:
\[
\frac{L_{1}}{L_{2}} = 1
\qquad\text{and}\qquad
\vartheta_2 = \pi,
\]
i.e.\ the two legs must be equal in length \emph{and} Leg~2 must fold back onto
Leg~1.\footnote{Tolkien's subtitle for \emph{The Hobbit}, \emph{There and Back
Again}, is the whole of the requirement. A hobbit leaves the Shire (the origin)
for Mordor (the bisector point) and his journey must end where it began, and no
choice of direction for the return trip from Mordor can bring the hobbit back to
the Shire unless the way back is exactly as long as the way out. Equal legs is that necessity; $\vartheta_2=\pi$ is the further
demand that the way back be the way he came.} By Corollary~\ref{cor:equal-legs} the length condition holds automatically on
the critical line; there, a zero occurs exactly when $\vartheta_2 =
\pi$. Away from the critical line, the two legs are unequal except at
exceptional points, so most of the strip is immediately ruled out for zeros. The leg-length ratio is symmetric across the
line,
\begin{equation}
\frac{L_{1}(\sigma,T)}{L_{2}(\sigma,T)}
=
\frac{L_{2}(1-\sigma,T)}{L_{1}(1-\sigma,T)} .
\end{equation}
Numerically we have never found a point with $\sigma\neq\tfrac12$ where the
ratio is $1$ and $\vartheta_2=\pi$ simultaneously; such a point would disprove
the Riemann hypothesis.

\begin{figure}[p]
\centering
\includegraphics[width=0.9\textwidth]{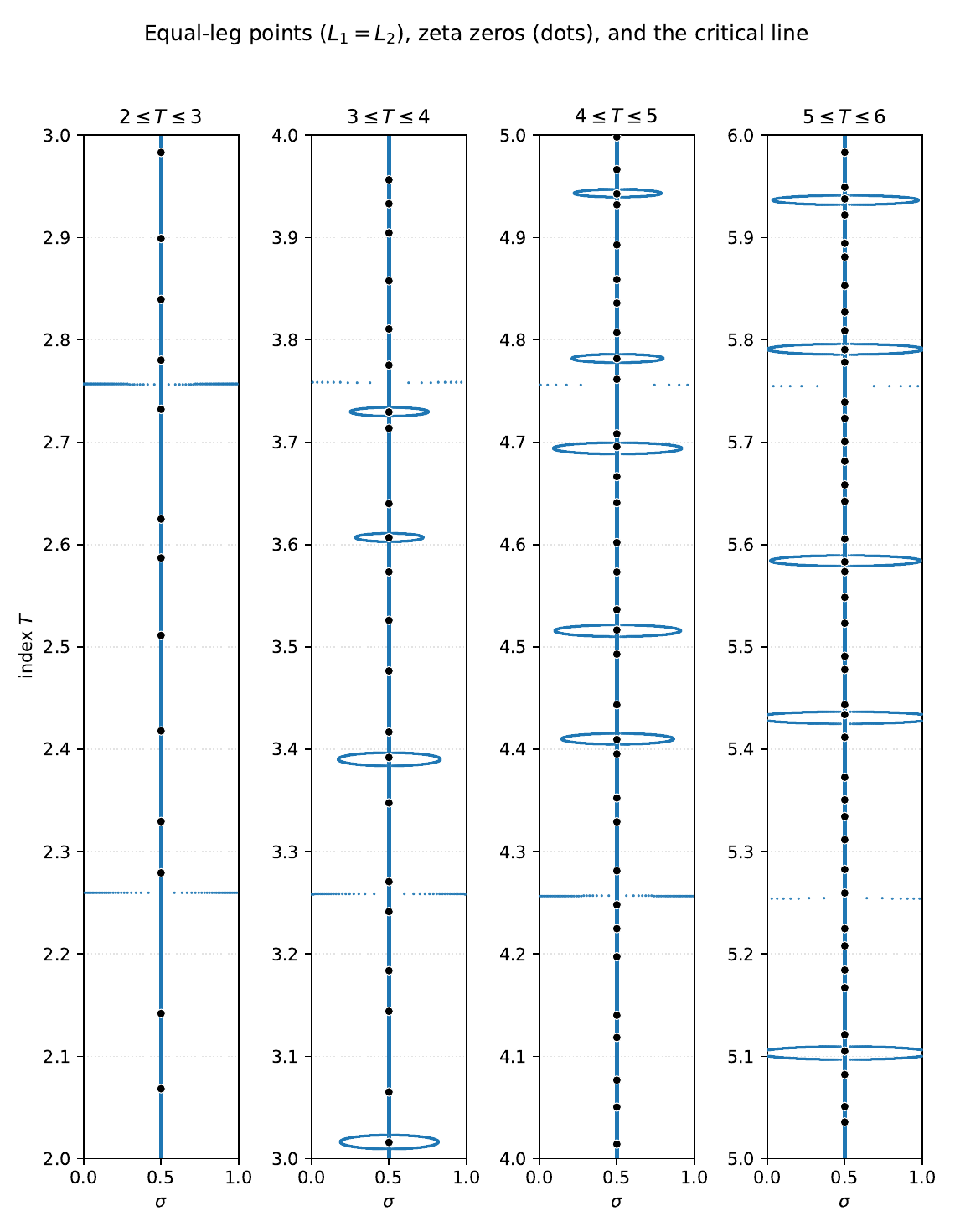}
\caption{Four windows of the critical strip ($0\le\sigma\le1$, index $T$
vertical): $2\le T\le3$, $3\le T\le4$, $4\le T\le5$, and $5\le T\le6$. The
blue vertical line is the critical line $\sigma=\tfrac12$, the black dots are
the zeta zeros, and the small blue points are the equal-leg points: points
$(\sigma,T)$ where $L_1=L_2$, i.e.\ where the bisector point
$B_1=\Sigma_1+R_{1ps}$ is equidistant from the origin and from $\zeta$. By
Corollary~\ref{cor:equal-legs} the entire critical line consists of equal-leg
points; off the line the equal-leg locus forms isolated ovals and thin
horizontal bands, and every zero sits on the critical line where the length
condition is automatic. The horizontal bands, near fractional parts
${\approx}\,0.25$ and ${\approx}\,0.75$ of $T$, are the instants where the two
bisector links become parallel and fail to intersect (\S\ref{sec:bisector}):
the bisector point has a pole there ($d_1$ and $d_2$ blow up), and as $T$
passes through it the equal-leg point sweeps across the full width of the
strip. The sweep gets faster as $T$ grows, so the fixed-step sampling leaves
fewer dots and the bands appear fainter. Generated by
\texttt{fig\_equal\_legs\_strips.py}.}
\label{fig:equal-legs-strips}
\end{figure}

\begin{figure}[p]
\centering
\includegraphics[width=0.9\textwidth]{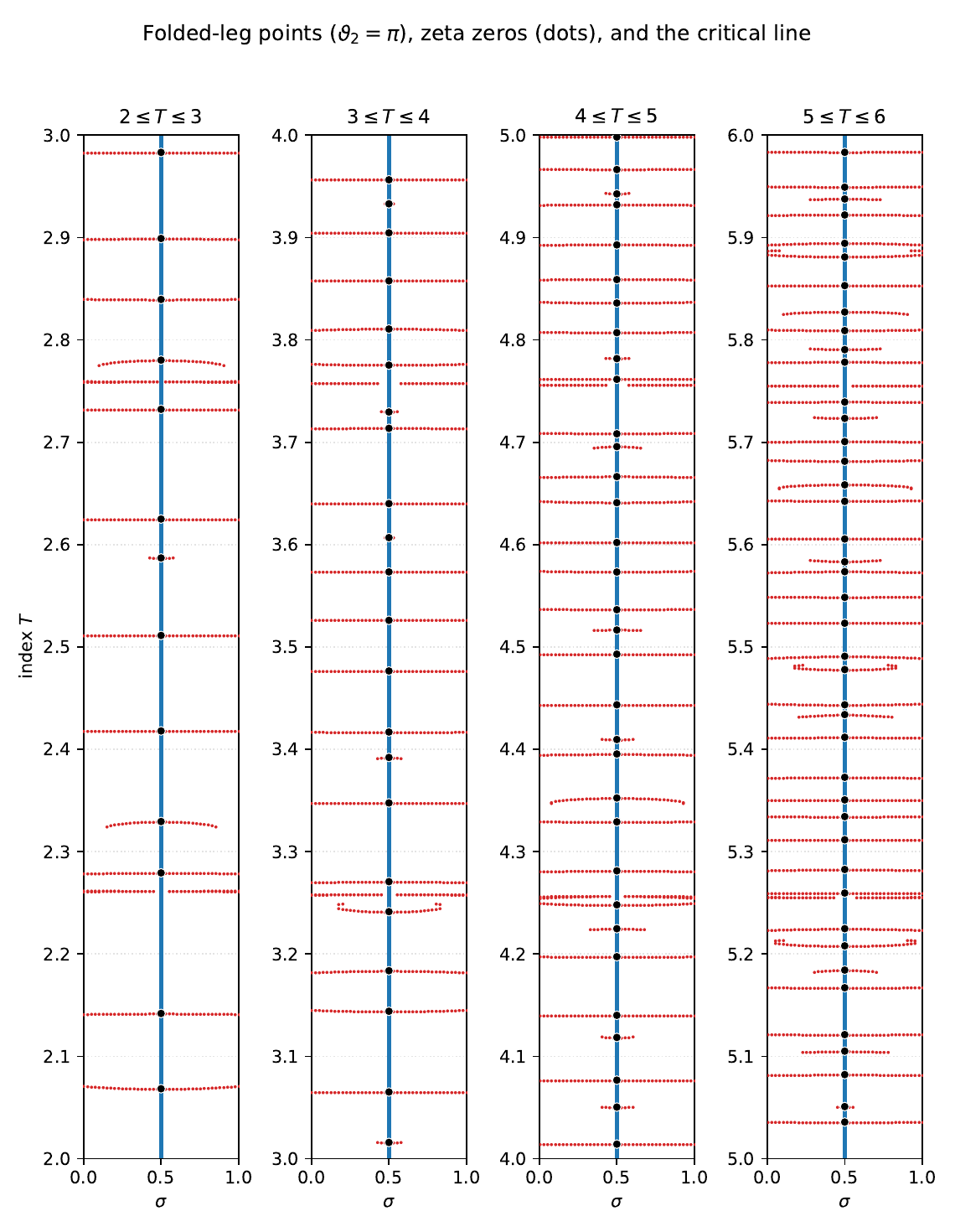}
\caption{The same four windows of the critical strip as in
Figure~\ref{fig:equal-legs-strips}, now showing the \emph{folded-leg} points
(red): points $(\sigma,T)$ where $\vartheta_2=\pi$, i.e.\ where Leg~2 folds
straight back onto Leg~1 and a zeta zero is therefore possible. The blue
vertical line is the critical line and the black dots are the zeta
zeros. The folded-leg locus forms nearly horizontal curves sweeping the
strip; a zero requires \emph{both} $\vartheta_2=\pi$ and $L_1=L_2$, and since
the length condition holds automatically on the critical line
(Corollary~\ref{cor:equal-legs}), every crossing of a red curve with the
critical line is a zeta zero. Generated by \texttt{fig\_theta2\_strips.py}.}
\label{fig:theta2-strips}
\end{figure}

\begin{figure}[p]
\centering
\includegraphics[width=0.9\textwidth]{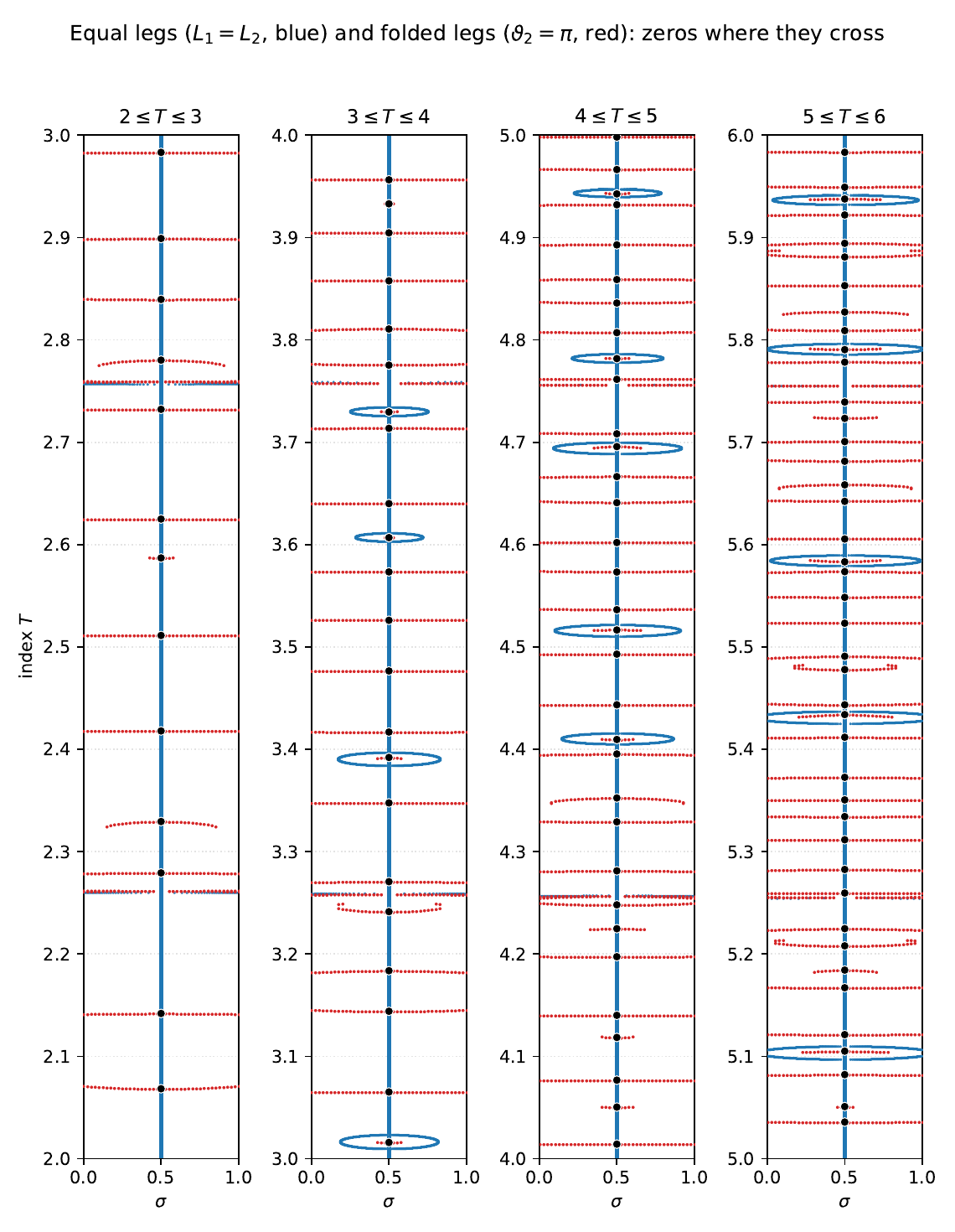}
\caption{Figures~\ref{fig:equal-legs-strips} and~\ref{fig:theta2-strips}
overlaid: the equal-leg locus $L_1=L_2$ (blue), the folded-leg locus
$\vartheta_2=\pi$ (red), the critical line (blue, since it belongs to the
equal-leg locus), and the zeta zeros (black dots). A zero requires both
conditions at once, so zeros occur exactly where
red meets blue. On the critical line the blue locus is the line itself
(Corollary~\ref{cor:equal-legs}), and every red curve crossing it produces a
zero;
away from the line, the red curves and the blue ovals interleave without
touching. A red/blue crossing off the critical line would be a
counterexample to the Riemann hypothesis. Generated by
\texttt{fig\_combined\_strips.py}.}
\label{fig:combined-strips}
\end{figure}

\clearpage

Figure~\ref{fig:combined-strips} makes the stakes plain: wherever a region of
the critical strip contains no red and no blue dots, a zero of $\zeta$ is
impossible. On the critical line the two legs are always equal in length, so
whenever $\vartheta_2=\pi$ there, a zero occurs. Off the line the legs are
seldom equal in length, but sometimes they are (the blue ovals and bands).
If one could prove that, off the critical line, in the rare times when
the legs are equal, that $\vartheta_2$ is simultaneously never $\pi$, one would have proven
the Riemann hypothesis true.

\clearpage
\subsection{PS, AK, and RS legs and angles}\label{sec:ps-ak-r2}

The equal-leg and folded-leg loci above were drawn with the exact
bisector point $B_1=\Sigma_1+R_{1ps}$. The same two conditions
$L_1=L_2$ and $\vartheta_2=\pi$ can be restated for any other choice of
bisector point. Two natural alternatives come from Kuznetsov's split
\eqref{eq:Rak-split}: take $B_1=\Sigma_1+R_{1ak}$, or take the midpoint
remainder $B_1=\Sigma_1+(R_{1ak}+R_{2ak})/2$, which is the geometric
analogue of Siegel's half-remainder $\tfrac12 R$ once
$R_{1ak}+R_{2ak}\approx R$.

Figure~\ref{fig:ps-ak-r2} magnifies a short window of the strip,
$4.65\le T\le4.80$, and overlays the three choices side by side. On the
left, equal-leg loci: PS in blue, AK in green, and $R/2$ in purple. The
three families form nested ovals about the critical line, but the
nesting order is not fixed: already in this short window the lower
oval has PS outermost and AK innermost, while the upper oval has AK
outermost and $R/2$ innermost, and across $3\le T\le20$ all six
orderings of PS, AK and $R/2$ occur with no simple alternating
pattern. On the right, folded-leg loci $\vartheta_2=\pi$: PS in red,
AK in orange, and $R/2$ in purple (hollow). The zeros (black) still sit
where a folded curve meets the critical line; the three remainder
choices shift the off-line curves but agree on those critical-line
crossings in this window.

\begin{figure}[p]
\centering
\includegraphics[width=\textwidth,height=0.78\textheight,keepaspectratio]{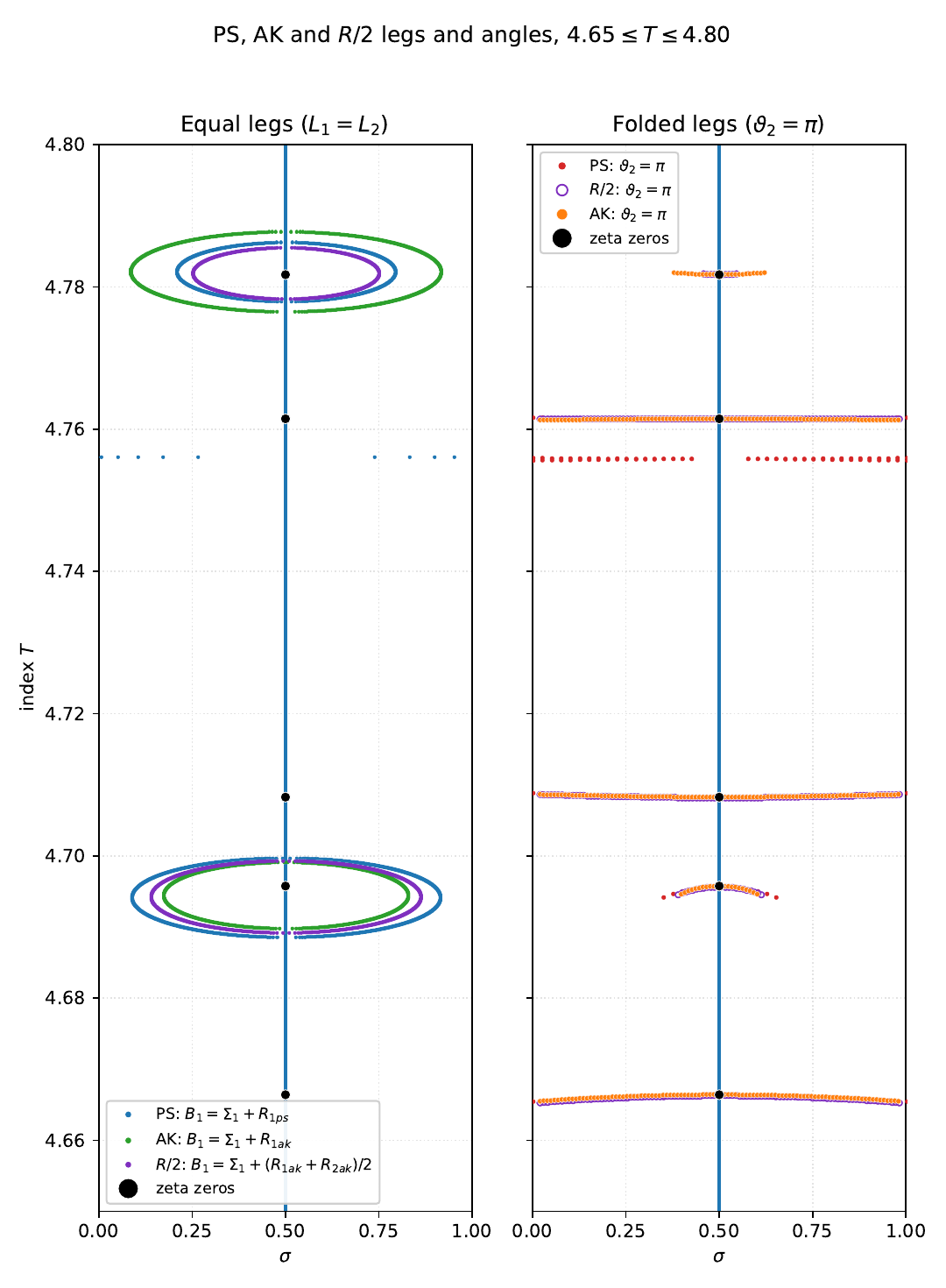}
\caption{PS, AK and $R/2$ legs and angles in the window
$4.65\le T\le4.80$. Left: equal-leg loci $L_1=L_2$ for three choices of
bisector point: PS $B_1=\Sigma_1+R_{1ps}$ (blue), AK
$B_1=\Sigma_1+R_{1ak}$ (green), and
$R/2$ $B_1=\Sigma_1+(R_{1ak}+R_{2ak})/2$ (purple). Right: folded-leg
loci $\vartheta_2=\pi$ for the same three choices (red, orange, and
hollow purple).
The blue vertical is the critical line; black dots are zeta zeros.
Equal-leg and folded-leg loci computed from the definitions
$L_1=L_2$ and $\vartheta_2=\arg((\zeta-B_1)/B_1)$ for each choice of
$B_1$. Generated by \texttt{fig\_ps\_ak\_r2\_legs\_angles.py}.}
\label{fig:ps-ak-r2}
\end{figure}

Figure~\ref{fig:ps-ak-r2-oval-9441} is the same three-way overlay for a
pair of neighboring equal-leg ovals near $T\approx9.44$, drawn in the
window $9.415\le T\le9.465$. The lower PS oval sits around the zero at
$T\approx9.431$ and the upper around $T\approx9.451$; a third zero at
$T\approx9.440$ lies in the gap between them (not inside either oval).
The AK and $R/2$ ovals nest with the PS locus as in
Figure~\ref{fig:ps-ak-r2}.

\begin{figure}[p]
\centering
\includegraphics[width=\textwidth,height=0.78\textheight,keepaspectratio]{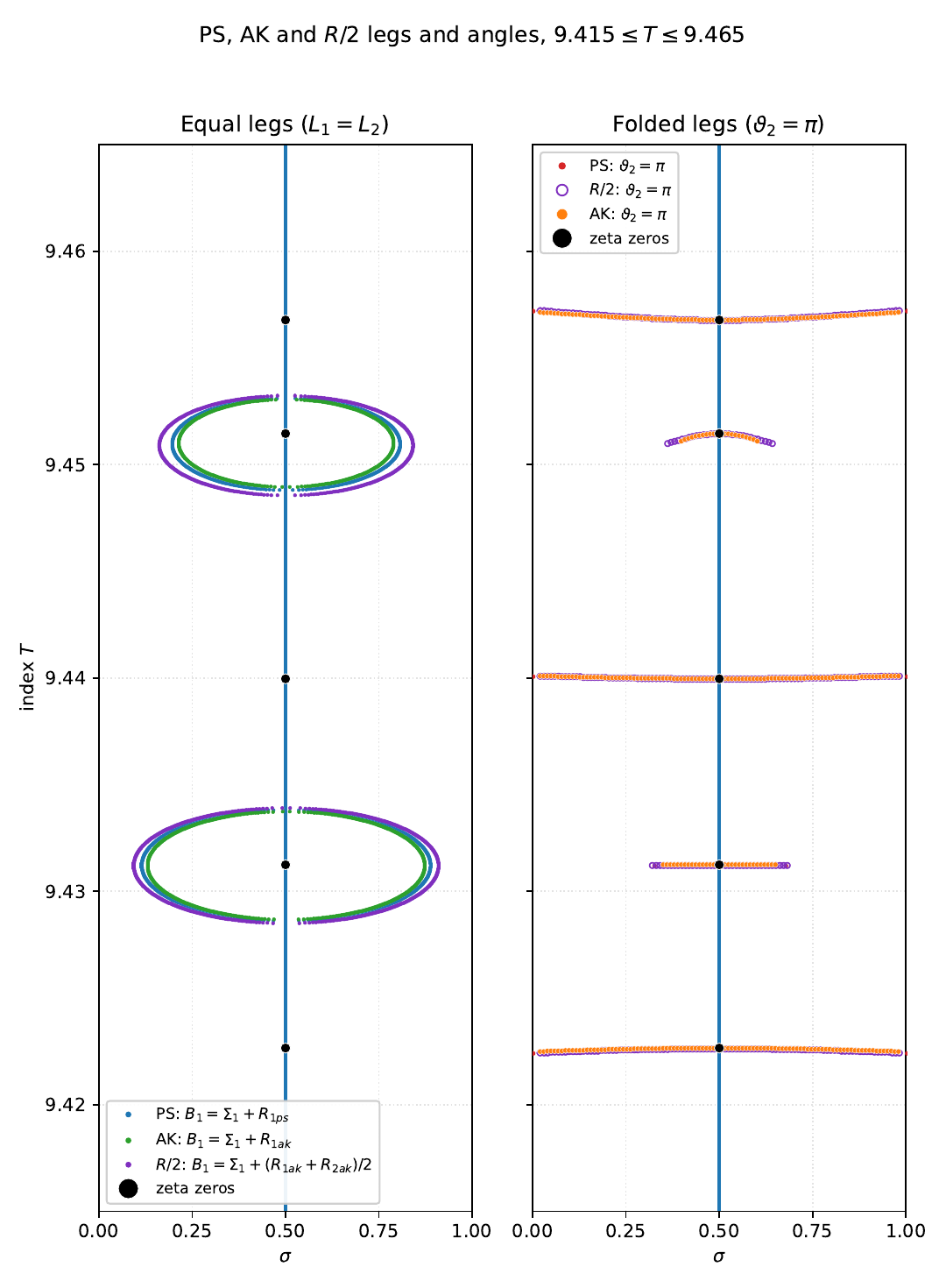}
\caption{PS, AK and $R/2$ legs and angles for two neighboring equal-leg
ovals near $T\approx9.44$, in the window $9.415\le T\le9.465$.
Left: equal-leg loci $L_1=L_2$ for PS (blue), AK (green), and $R/2$
(purple). Right: folded-leg loci $\vartheta_2=\pi$ for the same three
choices (red, orange, and hollow purple). The blue vertical is the
critical line; black dots are zeta zeros. The lower oval surrounds the
zero at $T\approx9.431$ and the upper the zero at $T\approx9.451$; the
zero at $T\approx9.440$ sits in the gap between them. Loci computed as in
Figure~\ref{fig:ps-ak-r2}. Generated by
\texttt{fig\_ps\_ak\_r2\_legs\_angles.py}
(\texttt{--out fig\_ps\_ak\_r2\_oval\_9441 --t-lo 9.415 --t-hi 9.465}).}
\label{fig:ps-ak-r2-oval-9441}
\end{figure}

\paragraph{Intersecting equal-leg ovals and the Riemann hypothesis.}
The nested ovals in Figures~\ref{fig:ps-ak-r2}
and~\ref{fig:ps-ak-r2-oval-9441} suggest a comparison that uses only exact
remainders. Write
\[
B_{1ps}
=
\Sigma_1+R_{1ps},
\qquad
B_{1rs}
=
\Sigma_1+R_{1rs},
\qquad
R_{1rs}
:=
\tfrac12 R,
\]
where $R=R_{1ps}+R_{2ps}=\zeta-\Sigma_1-\Sigma_2$ is Siegel's remainder, so
$R_{1rs}$ is the exact half-remainder (``$rs$'' for Riemann--Siegel, parallel to
the subscripts ``$ps$'' and ``$ak$'') and $B_{1rs}$ is the midpoint of
the remainder segment from $\Sigma_1$ to $\Sigma_1+R$. Let
$\mathcal{E}_{\mathrm{ps}}$ and $\mathcal{E}_{\mathrm{rs}}$ be the corresponding
equal-leg loci in the strip:
the sets of $(\sigma,T)$ at which $|B_1|=|\zeta-B_1|$ for
$B_1=B_{1ps}$ and $B_1=B_{1rs}$ respectively.

If $\zeta(\sigma+iI(T))=0$, then for either choice of bisector point one has
$L_2=|\zeta-B_1|=|B_1|=L_1$, so
\[
\frac{L_1}{L_2}=1
\]
holds simultaneously for both $B_{1ps}$ and $B_{1rs}$. Every
zero therefore lies in
$\mathcal{E}_{\mathrm{ps}}\cap\mathcal{E}_{\mathrm{rs}}$. On the critical line
this is automatic: Corollary~\ref{cor:equal-legs} puts the whole line in
$\mathcal{E}_{\mathrm{ps}}$, and the $rs$ split is reflection-symmetric there
as well: $\Sigma_2=\chi\,\overline{\Sigma_1}$ and $R=\chi\,\overline{R}$ (the
latter because $e^{-i\psi/2}R$ is real on the line: $R=e^{-i\vartheta}r$ with
$r$ real and $\psi=-2\vartheta$, Appendix~\ref{sec:phase-bound}), so
\[
\chi\,\overline{B_{1rs}}
=\Sigma_2+\tfrac12 R
=\zeta-B_{1rs},
\]
and since $|\chi|=1$ this gives
$\bigl|B_{1rs}\bigr|=\bigl|\zeta-B_{1rs}\bigr|$: the whole
line lies in $\mathcal{E}_{\mathrm{rs}}$ as well. (Note that
$B_{1rs}\neq\zeta/2$; only its perpendicular projection onto the
base is $\zeta/2$, by Corollary~\ref{cor:bisector-proj}.) The
two loci therefore intersect everywhere
along $\sigma=\tfrac12$; that common set is expected and does not constrain zeros
beyond the folded-leg condition already discussed.

Off the critical line the situation is different. An off-line zero would be an
off-line point of $\mathcal{E}_{\mathrm{ps}}\cap\mathcal{E}_{\mathrm{rs}}$: a
point where the PS and $rs$ equal-leg ovals meet. Consequently, an off-line
intersection of $\mathcal{E}_{\mathrm{ps}}$ and $\mathcal{E}_{\mathrm{rs}}$ is
\emph{necessary} for an off-line zero (it is not sufficient: the folded-leg
condition $\vartheta_2=\pi$, or equivalently $\zeta=0$, is still required). The
contrapositive is the statement of interest:

\begin{quote}
\emph{If $\mathcal{E}_{\mathrm{ps}}$ and $\mathcal{E}_{\mathrm{rs}}$ have no
point in common with $\sigma\neq\tfrac12$, then $\zeta$ has no zero off the
critical line.}
\end{quote}

In particular, if the two loci are disjoint off the line throughout the strip,
the Riemann hypothesis follows; if they are disjoint off the line for all
$T\le T_{\max}$, then there is no off-line zero with index at most $T_{\max}$.
The figures above are consistent with a strict nesting (the ovals approach one
another but do not cross off the line). A proof that this nesting persists for
all $T$, with due care at the pole heights of $d_1$ in
Appendix~\ref{sec:pole-locations} where $B_{1ps}$ ceases to be defined in
the usual way, would therefore be a proof of the Riemann hypothesis.

In short: on the critical line the leg-length ratio is always one
(Corollary~\ref{cor:equal-legs}, and likewise for $B_{1rs}$), so the
zeros there are determined by the legs folding back,
$\vartheta_2=\pi$, for either the $ps$ or the $rs$ choice of bisector point.
Off the line the length condition is no longer free: a zero can occur only
where $L_1/L_2=1$ for the legs, and since that must hold for both choices it
is a necessary condition that
\begin{equation}\label{eq:both-ratios-one}
\frac{L_1}{L_2}=1
\end{equation}
simultaneously for the $ps$ legs and for the $rs$ legs.

\paragraph{The mean leg-imbalance metric.}
The equal-leg loci above are the zero sets of a continuous imbalance. For each
of the three named remainder-splits of \S\ref{sec:remainders-summary}, write the two
leg lengths
\begin{equation}\label{eq:leg-lengths}
L_1 \;=\; |\Sigma_1+R_1|,
\qquad
L_2 \;=\; |\Sigma_2+R_2|,
\end{equation}
where $(R_1,R_2)$ is $(R_{1ps},R_{2ps})$, $(R_{1rs},R_{2rs})=(\tfrac12 R,\tfrac12 R)$,
or $(R_{1ak},R_{2ak})$ respectively. The normalized length imbalance of a
split is
\begin{equation}\label{eq:delta-leg}
\delta
\;=\;
\frac{|L_1-L_2|}{L_1+L_2+\varepsilon},
\qquad \varepsilon=10^{-18},
\end{equation}
so $\delta=0$ if and only if $L_1=L_2$ (the bisector condition for that
choice of bisector point). Averaging the three splits gives the plotted color map
\begin{equation}\label{eq:delta-bar}
\bar\delta
\;=\;
\frac{\delta_{\mathrm{ps}}+\delta_{\mathrm{rs}}+\delta_{\mathrm{ak}}}{3}.
\end{equation}
On the critical line each split has $L_1=L_2$ (Corollary~\ref{cor:equal-legs}
and the parallel identities for $rs$ and $ak$), so $\bar\delta\approx0$ and
the strip is dark purple along $\sigma=\tfrac12$. Off the line
$\bar\delta$ rises, and the equal-leg ovals of
Figures~\ref{fig:ps-ak-r2}--\ref{fig:ps-ak-r2-oval-9441} appear as the dark
cellular structures hugging the critical line in
Figure~\ref{fig:leg-imbalance}.

\begin{figure}[H]
\centering
\includegraphics[width=\textwidth]{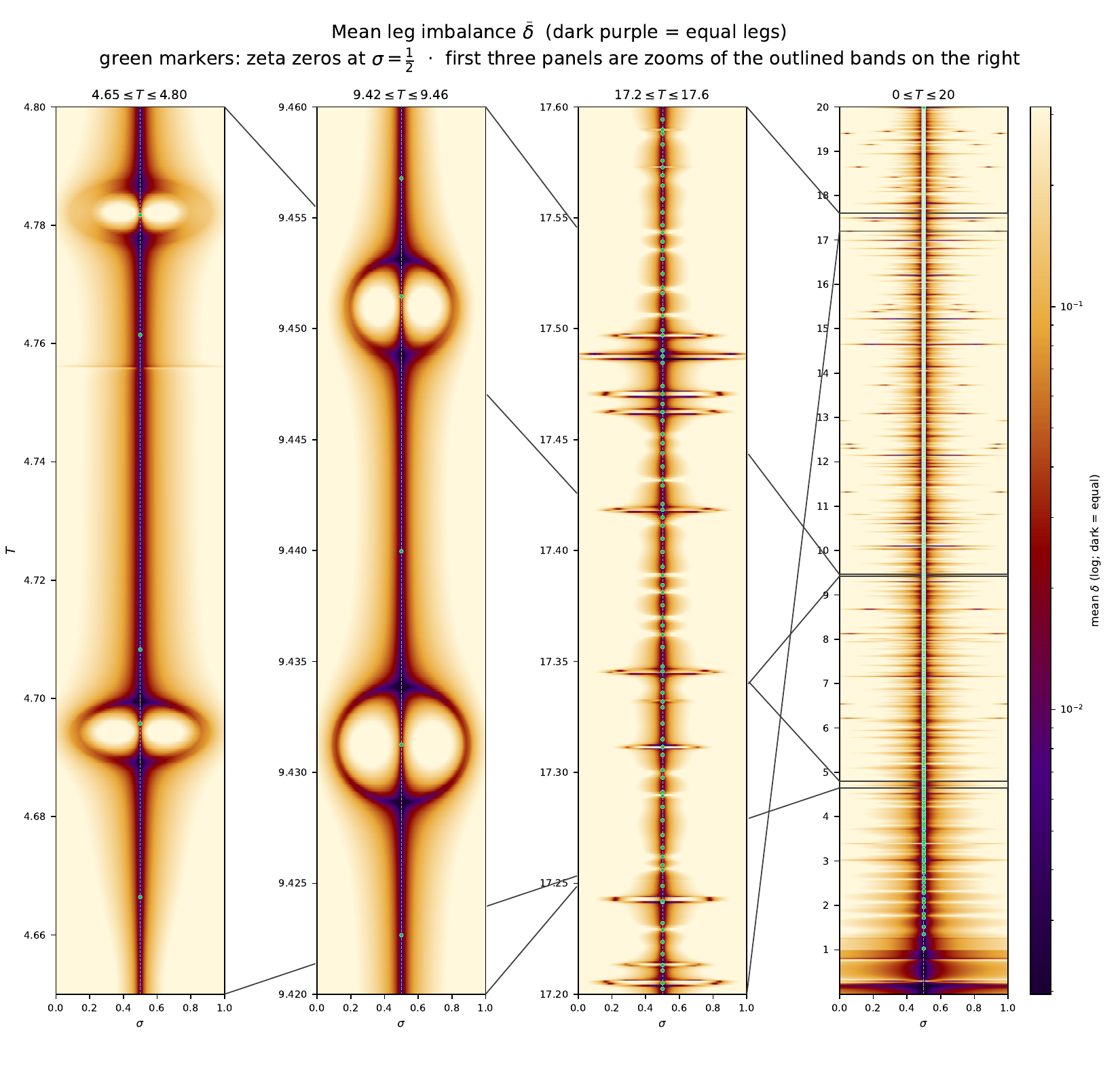}
\caption{Mean leg imbalance $\bar\delta$ of \eqref{eq:delta-bar}
(dark purple $=$ equal leg lengths). Horizontal axis $\sigma\in[0,1]$;
vertical axis the spiral index $T$. Color is $\bar\delta$ on a log scale.
Green markers are zeta zeros at $\sigma=\tfrac12$. The rightmost panel is
the full strip $0\le T\le20$; the first three panels are zooms of the
outlined bands
($4.65\le T\le4.80$, $9.42\le T\le9.46$, $17.2\le T\le17.6$).
The dark vertical band on $\sigma=\tfrac12$ is the simultaneous
equal-leg condition for the $ps$, $rs$, and $ak$ splits. Generated by
\texttt{plot\_leg\_equality\_four\_panel.py}.}
\label{fig:leg-imbalance}
\end{figure}

\paragraph{The mean fold-deviation metric.}
We treat the folded-leg loci in a similar fashion to the equal-leg loci. For each of the three remainder splits
write the two leg vectors $v_1=\Sigma_1+R_1$ and $v_2=\Sigma_2+R_2$, and let
$\vartheta_2=\arg(v_2/v_1)\in(-\pi,\pi]$ be the fold angle between them. The
normalized fold deviation of a split is the scale-free quantity
\begin{equation}\label{eq:tau-fold}
\tau
\;=\;
\frac{\pi-|\vartheta_2|}{\pi},
\end{equation}
so $\tau=0$ if and only if $\vartheta_2=\pi$ (leg~2 folded straight back onto
leg~1) and $\tau=1$ where the legs run parallel. Averaging the three splits
gives the angle companion of $\bar\delta$,
\begin{equation}\label{eq:tau-bar}
\bar\tau
\;=\;
\frac{\tau_{\mathrm{ps}}+\tau_{\mathrm{rs}}+\tau_{\mathrm{ak}}}{3},
\end{equation}
plotted in Figure~\ref{fig:theta2-pi-field} over the same four windows as
Figure~\ref{fig:leg-imbalance}. The two color maps organize themselves in
opposite directions. The length condition is automatic on the critical line
and rare off it, so $\bar\delta$ concentrates its dark values in a vertical
band. The fold condition instead is met along the nearly horizontal
folded-leg curves that sweep the strip at the mean zero density, on and off
the line alike (Figure~\ref{fig:theta2-strips}), so $\bar\tau$ concentrates
its dark values in horizontal filaments crossing the full width of the
strip. A zeta zero requires both: it is a point where a dark filament of
Figure~\ref{fig:theta2-pi-field} crosses the dark vertical band of
Figure~\ref{fig:leg-imbalance}, and every green marker sits on a filament at
$\sigma=\tfrac12$.

\begin{figure}[H]
\centering
\includegraphics[width=\textwidth]{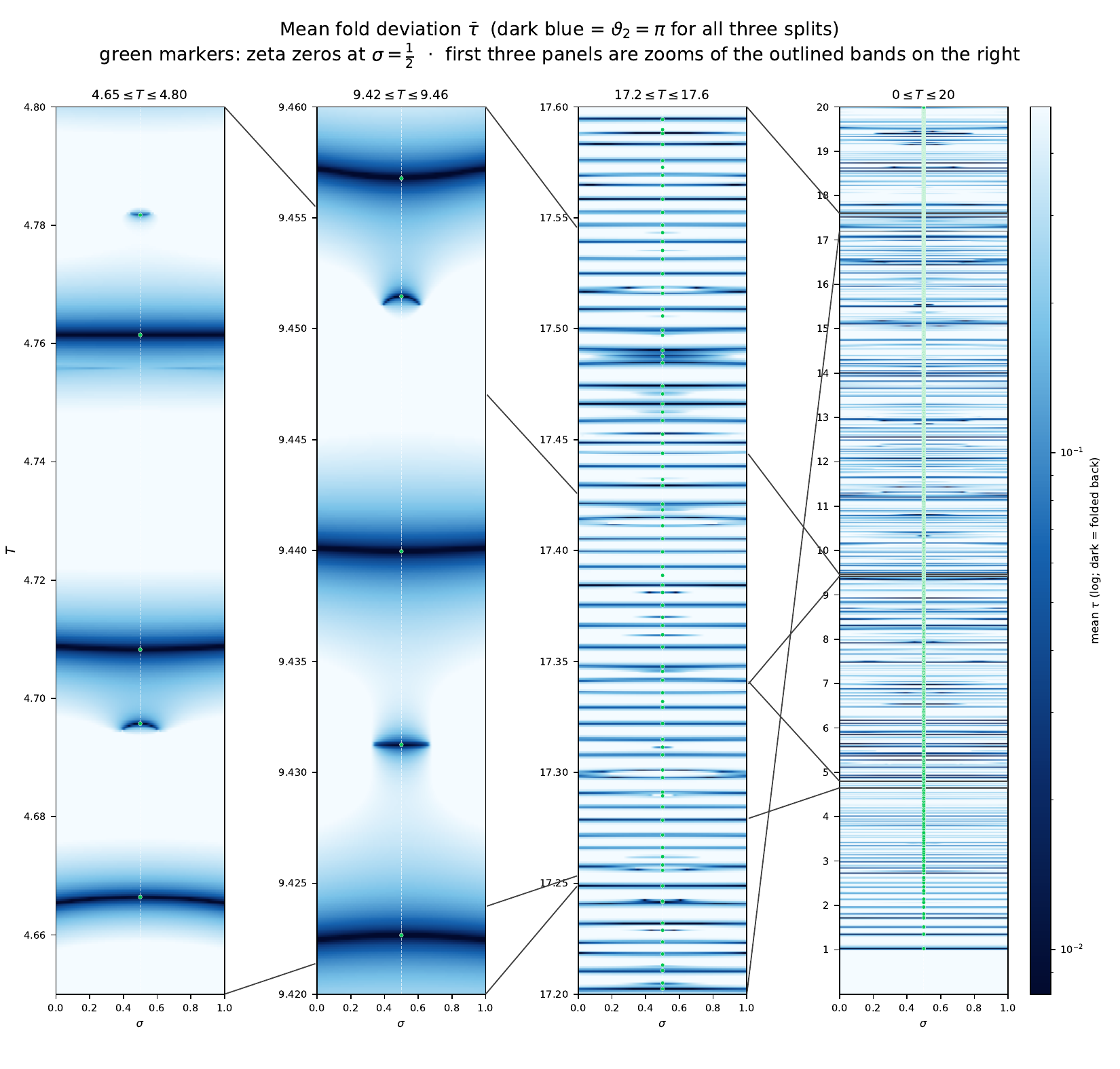}
\caption{Mean fold deviation $\bar\tau$ of \eqref{eq:tau-bar} (dark blue $=$
$\vartheta_2=\pi$ simultaneously for the $ps$, $rs$, and $ak$ splits).
Horizontal axis $\sigma\in[0,1]$; vertical axis the spiral index $T$; color
is $\bar\tau$ on a log scale, in the four-panel layout of
Figure~\ref{fig:leg-imbalance}: the rightmost panel is the full strip
$0\le T\le20$ and the first three panels are zooms of the outlined bands
($4.65\le T\le4.80$, $9.42\le T\le9.46$, $17.2\le T\le17.6$). Green markers
are zeta zeros at $\sigma=\tfrac12$. Where the equal-leg color map of
Figure~\ref{fig:leg-imbalance} is dark along the critical line, the fold
color map is dark along nearly horizontal filaments sweeping the strip at the
mean zero density, and every green marker sits on one. Generated by
\texttt{plot\_theta2\_pi\_four\_panel.py}.}
\label{fig:theta2-pi-field}
\end{figure}

Both color maps are exactly symmetric across the critical line: the leg ratio
obeys $L_1/L_2(\sigma,T)=L_2/L_1(1-\sigma,T)$, which leaves
\eqref{eq:delta-leg} unchanged, and the fold angle satisfies the
corresponding reflection, so $\bar\delta$ and $\bar\tau$ take the same value
at $(\sigma,T)$ and $(1-\sigma,T)$; on the plotted grids the two halves agree
to a relative $10^{-10}$. Half of each panel therefore carries the whole
color map, and the two conditions can share one strip.
Figure~\ref{fig:leg-theta2-spliced} splices them in the first two zoom
windows: the left half of each strip is the leg imbalance $\bar\delta$ and
the right half is the fold deviation $\bar\tau$, meeting at the seam
$\sigma=\tfrac12$. The two halves state the two conditions
of a zero together: the length condition holds along the half-line itself and on
the equal-leg ovals hugging it, the angle condition holds on the dark
filaments arriving from both sides, and each filament reaches the seam
exactly at a green zeta zero marker.

\begin{figure}[H]
\centering
\includegraphics[width=0.95\textwidth]{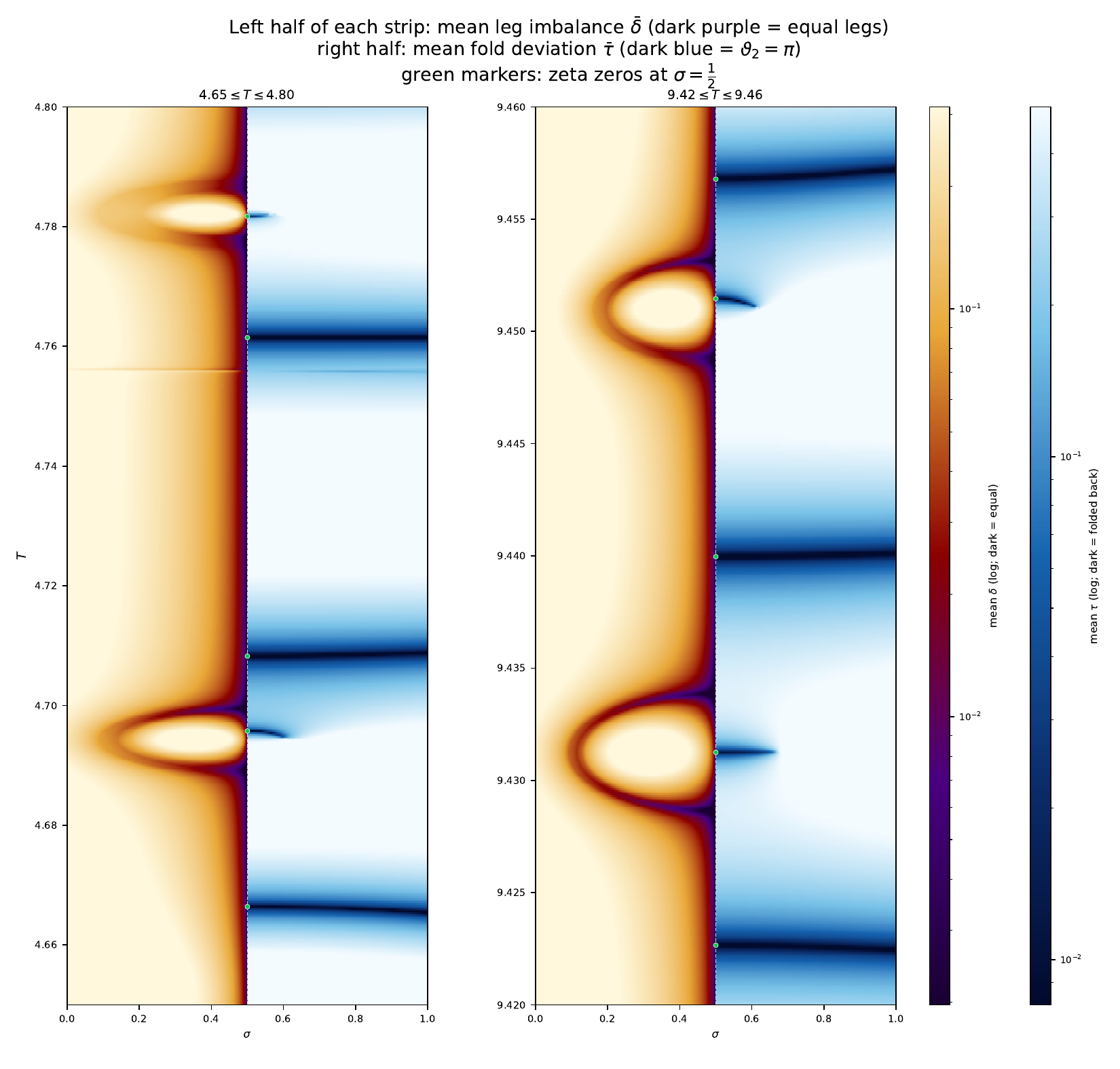}
\caption{The two zoom windows of Figures~\ref{fig:leg-imbalance}
and~\ref{fig:theta2-pi-field} spliced at the critical line: in each strip the
left half ($\sigma<\tfrac12$) shows the mean leg imbalance $\bar\delta$ of
\eqref{eq:delta-bar} on the color scale of Figure~\ref{fig:leg-imbalance}
(dark purple $=$ equal legs), and the right half ($\sigma>\tfrac12$) shows
the mean fold deviation $\bar\tau$ of \eqref{eq:tau-bar} on the color scale
of Figure~\ref{fig:theta2-pi-field} (dark blue $=$ folded back). Both color maps
are symmetric under $\sigma\leftrightarrow1-\sigma$ (to a relative $10^{-10}$
on these grids), so each half carries its whole color map and the splice loses
nothing. Green markers are zeta zeros. The length condition is dark along
the seam and on the ovals flanking it; the angle condition is dark on the
horizontal filaments; a zero is a filament reaching the seam. Generated by
\texttt{plot\_leg\_theta2\_spliced.py}.}
\label{fig:leg-theta2-spliced}
\end{figure}

\subsection{Toward a proof}\label{sec:toward-proof}

The geometry assembled in this section suggests two routes to the Riemann
Hypothesis, listed here in increasing order of how many remainder splits
participate.

\begin{enumerate}
\item \emph{Equal legs and fold-back together, for a single split.}
A zero at $(\sigma,T)$ forces, for any single choice of bisector point, both
leg conditions of \S\ref{sec:legs} at the same input: $L_1=L_2$ \emph{and}
$\vartheta_2=\pi$. On the critical line the length condition is automatic
(Corollary~\ref{cor:equal-legs}), so zeros there are cut out by the fold-back
alone. Off the line the two conditions trace independent families of curves
(the equal-leg ovals and the nearly horizontal folded-leg curves of
Figure~\ref{fig:combined-strips}), and a proof that the two families never
pass through the same off-line point, for even one split, would prove the
Riemann hypothesis.

\item \emph{Equal legs for two or more splits at once.}
A zero makes the legs equal for \emph{every} remainder split simultaneously: when
$\zeta=0$, $L_2=|\zeta-B_1|=|B_1|=L_1$ for any $B_1$. An off-half-line zero
would therefore be an off-half-line point of
$\mathcal{E}_{\mathrm{ps}}\cap\mathcal{E}_{\mathrm{rs}}$ (and of
$\mathcal{E}_{\mathrm{ak}}$ as well). This route drops the angle condition
entirely and asks only that two equal-leg ovals meet off the line; the
figures of \S\ref{sec:ps-ak-r2} show the ovals nesting without crossing.
This route is developed in detail below.
\end{enumerate}

\paragraph{Route 2 in detail.}
An off-half-line zero at $(\sigma,T)$ would require both equal-leg
identities at the same input:
\begin{subequations}\label{eq:equal-legs-abs}
\begin{equation}\label{eq:ps-equal-legs-abs}
\bigl|\Sigma_1+R_{1ps}\bigr|
=
\bigl|\Sigma_2+R_{2ps}\bigr|,
\end{equation}
\begin{equation}\label{eq:rs-equal-legs-abs}
\bigl|\Sigma_1+\tfrac12 R\bigr|
=
\bigl|\Sigma_2+\tfrac12 R\bigr|.
\end{equation}
\end{subequations}
Here \eqref{eq:ps-equal-legs-abs} is the $ps$ condition and
\eqref{eq:rs-equal-legs-abs} the $rs$ condition with $R_{1rs}=\tfrac12 R$.
Each is one real constraint on $(\sigma,T)$: an equality of two moduli is
one real equation, so the pair is two real conditions on the two real
coordinates of the strip, the right number of conditions to cut out isolated points (or
nothing).
On $\sigma=\tfrac12$ they are not independent, so they do not cut the
half-line down to points. Off it, they are.

On $\sigma=\tfrac12$ both hold for every $T$, by
Corollary~\ref{cor:equal-legs}. They are identities along the whole
half-line and do not cut anything out there. Every point of
$\sigma=\tfrac12$ is already a common root; that common set is expected
and does not constrain zeros beyond the folded-leg condition, which this
route does not use.

Off the half-line that automatic identity is gone. Then
\eqref{eq:ps-equal-legs-abs} and \eqref{eq:rs-equal-legs-abs} are
independent: they are not the same equation written twice, and satisfying
one does not force the other. Each traces its own family of curves, the
$ps$ ovals $\mathcal{E}_{\mathrm{ps}}$ and the $rs$ ovals
$\mathcal{E}_{\mathrm{rs}}$. A common root is a pair $(\sigma,T)$ that
solves both at once, hence a point of
$\mathcal{E}_{\mathrm{ps}}\cap\mathcal{E}_{\mathrm{rs}}$. On the half-line
that intersection is the whole line; off it, a meeting of the two oval
families.

An off-half-line zero would be such a root: a $(\sigma,T)$ that
satisfies both \eqref{eq:ps-equal-legs-abs} and
\eqref{eq:rs-equal-legs-abs} at once. At a zero the chain (the whole Dirichlet chain formed by the links of
$\Sigma_1$, the remainder, and $\Sigma_2$) closes,
$\zeta=0$, so for any $B_1$
\[
L_2=|\zeta-B_1|=|-B_1|=|B_1|=L_1.
\]
That does not depend on how $B_1$ was chosen, so it holds for $B_{1ps}$
and $B_{1rs}$ at once. A zero therefore sits on both loci, on the
half-line or off it; off the half-line that is exactly a common root of
\eqref{eq:equal-legs-abs}. The implication is one-way:
\[
\text{off-half-line zero}
\quad\Longrightarrow\quad
\text{off-half-line point of }\mathcal{E}_{\mathrm{ps}}\cap\mathcal{E}_{\mathrm{rs}}.
\]
A meeting of the ovals is necessary for an off-half-line zero, not
sufficient: the fold-back $\vartheta_2=\pi$ is still required. For
proving there are no off-half-line zeros, necessity is what is needed.

A proof that \eqref{eq:equal-legs-abs} has no solution for
$\sigma\neq\tfrac12$ (away from the poles of $d_1$, where the $ps$
split is undefined) is a proof that those two loci never meet off the
half-line: the off-half-line points of the intersection
\emph{are} the off-half-line common roots of \eqref{eq:equal-legs-abs}.
The contrapositive of the implication then gives the Riemann hypothesis.
\emph{If there is no off-half-line meeting, there cannot be an off-half-line
zero, because such a zero would force a meeting.} In one line,
\begin{align*}
\text{no off-half-line solution of \eqref{eq:equal-legs-abs}}
&\;\Longleftrightarrow\;
\text{no off-half-line point of }\mathcal{E}_{\mathrm{ps}}\cap\mathcal{E}_{\mathrm{rs}}
\\
&\;\Longrightarrow\;
\text{no off-half-line zero}
\;\Longleftrightarrow\;
\text{RH}.
\end{align*}
The figures of \S\ref{sec:ps-ak-r2} already suggest the ovals nest and do
not cross. Proving what those pictures represent, namely that
\eqref{eq:equal-legs-abs} has no off-half-line solution, would be a
proof of RH.

\subsection{Equal legs density}\label{sec:eqleg-density}

If a zero were to exist off the critical line, it would have to lie on an
equal-leg locus. The reason is the one used in \S\ref{sec:ps-ak-r2}: at a zero
the chain closes, $\zeta=0$, so the bisector point is equidistant from the two
ends of the chain,
\[
L_2=|\zeta-B_1|=|{-B_1}|=|B_1|=L_1 ,
\]
whichever split is used for $B_1$. Every zero, on the line or off it, is
therefore an equal-leg point. On the line that says nothing, the whole line
being equal-leg by Corollary~\ref{cor:equal-legs}; off the line it is a genuine
restriction, because the off-line part of the locus is not the whole strip but
the isolated ovals of Figure~\ref{fig:equal-legs-strips}, symmetric about
$\sigma=\tfrac12$. An off-line zero has to sit on one of those ovals. Their
number and size are therefore worth measuring: they are the only places in the
strip where an off-line zero can be. This subsection reports a numerical census
of them for $0\le T\le401$ (which is $t$ of $1.0129\times10^{6}$).

\paragraph{The imbalance and the scan.} The census counts sign changes of the
\emph{signed} imbalance
\begin{equation}\label{eq:g-imbalance}
g(\sigma,T)\;=\;\frac{L_1-L_2}{L_1+L_2},
\end{equation}
which is \eqref{eq:delta-leg} without the absolute value. The modulus is dropped
on purpose: $g$ changes sign across an oval, and it is the sign change that can
be counted. Numerically $g$ is odd under $\sigma\mapsto1-\sigma$, the reflection
that exchanges the two legs, so it vanishes identically on the critical line
(Corollary~\ref{cor:equal-legs} again), and the ovals are symmetric about the
line, as the figures show.

The scan fixes $\sigma=\tfrac12+\varepsilon$ and steps in $T$, recording sign
changes of $g$. Each oval contributes two, a bottom and a top, and only bottoms
are counted, so an oval straddling an integer $T$ is counted once, to the interval
holding its bottom. The pole bands are excluded: at the heights of
Appendix~\ref{sec:pole-locations} the PS bisector point runs off to infinity and the
equal-leg point sweeps the full width of the strip, which is not an oval. Those
bands are a PS artifact, $599$ of them against $1$ for AK, whose remainder has
no pole to escape through. Roots are refined by bisection to $10^{-13}$. The
step is
\begin{equation}\label{eq:oval-step}
\Delta T\;=\;10^{-4}/m ,
\end{equation}
which holds the number of samples across an oval fixed, the median oval height
falling like $1/T$: $1.3\times10^{-3}$ in the interval at $T=25$ and
$5.2\times10^{-5}$ in the one at $T=400$. Ovals thinner than the step are still caught, by minimizing $g$
wherever it has a positive local minimum below a twentieth of the interval's
median and testing whether the minimum dips below zero.

Four checks bound the error. First, $\zeta$ evaluated as $\Sigma_1+\Sigma_2+R_{ak}$
agrees with \texttt{mpmath} to $2\times10^{-10}$ at $T=150$ and
$3\times10^{-8}$ at $T=1000$, and $g$ vanishes on the critical line to machine
precision. Second, recounting the intervals at $m=250$, $275$, $299$, $300$, $350$, $398$,
$399$ and $400$ with $\Delta T/2$ and $\Delta T/4$ returns the same count every
time, for both splits, and finds the same thinnest cross-section in each. Third, the window $3\le T\le4$ reproduces
Figure~\ref{fig:equal-legs-strips}, the same four ovals at $T=3.0105$,
$3.3845$, $3.6031$ and $3.7257$. Fourth, the $\sigma$-extent of each oval was
measured by bisection, which is what decides whether a scan at a given
$\varepsilon$ can see it: at $\varepsilon=0.001$ every oval in the sample is
reachable and at $\varepsilon=0.01$ some $99.95\%$ are, so the count quoted
below, taken at $\varepsilon=0.01$, is not limited by the scan's reach.

\paragraph{Are they ovals?} Drawn against the index they are hairlines. The
component at $T=300$ in Figure~\ref{fig:eqleg-shape} is $0.39$ wide in $\sigma$
and $6.9\times10^{-5}$ tall in $T$, a ratio of $5600$ to one, and every critical strip
figure here stretches the vertical axis by some such factor to make the locus
visible at all. That shape is an artifact of the coordinate. The index is a
compressed ruler: at $T=300$, one step in $T$ is a step of $I'(300)=3789$ in
$t$. Undo that and the component is $6.9\times10^{-5}\times3789=0.26$ tall
against $0.39$ wide, which is its shape in the $s$-plane, where $\zeta$ is
evaluated. It is an oval, and close to an ellipse: comparing its half-height at
eight fractions of its half-width $a$ against $\sqrt{1-(u/a)^2}$, the one drawn
departs by $1.8\%$ rms. Over a sample of $300$ of them at nine heights from
$T=5$ to $T=400$, drawn at random within each interval, the median departure is
$0.5\%$, with $80\%$ inside $2\%$ and $94\%$ inside $5\%$; the ratio of width to
height has median $1.22$ and ranges from $0.96$ to $2.83$, with no trend in
height, the interval medians sitting between $1.13$ and $1.27$ throughout.
In the sample of $300$ closed one-zero components, the components are
numerically smooth, convex, and nearly elliptical, with a median
width-to-height ratio of $1.22$. This description excludes the $23$ two-zero
components, which are hourglasses: two oval lobes joined at a neck, with one
zero in each lobe, as in Figure~\ref{fig:two-zero-hourglass}. It also excludes
the $72$ components cut off by the edge of the strip instead of closing inside
it, all of which lie below $T=100$.

\begin{figure}[H]
\centering
\includegraphics[width=\textwidth]{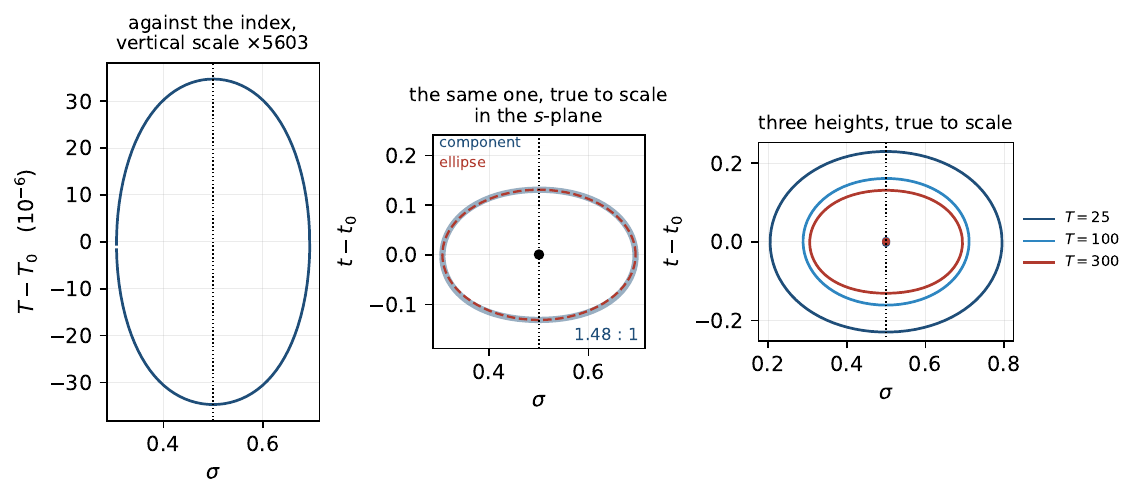}
\caption{Why the components of the equal-leg locus are called ovals. Left: one of them at $T=300$
plotted against the index, as the strip figures plot them, the vertical axis
stretched by $5600$ to make it visible. Middle: the same component in the
$s$-plane with the axes to a common scale, and the ellipse through its extremes
dashed over it; the dot is the zero of $\zeta$ it encloses. Right: components at
$T=25$, $100$ and $300$ to a common scale, each drawn about its zero. They shrink
with height and hold their shape. All three are the median-height component of
their interval, taken among those whose zero sits nearest the middle, since the
zero is central only on average. Generated by \texttt{fig\_oval\_shape.py} of
the equal-leg density companion note in the Zest repository.}
\label{fig:eqleg-shape}
\end{figure}

\paragraph{The oval count.} Over $0\le T\le401$, scanning the line
$\sigma=\tfrac12+\varepsilon$ with $\varepsilon=0.01$, there are $328\,511$
ovals for the AK remainder split $B_1=\Sigma_1+R_{1ak}$ and $328\,191$ for the
PS split $B_1=\Sigma_1+R_{1ps}$, against $1\,771\,661$ zeros of $\zeta$ in the
same range. The two splits converge on each other with
height: the mean disagreement per unit interval grows from $1.5$ ovals to
$5.0$, but falls from six percent of the count to $0.32\%$. The oval density of
both splits tracks the density of the zeros. Writing the interval's zeros per oval as a function of height, and
fitting over $3\le T\le401$ weighted by the oval count,
\begin{equation}\label{eq:oval-ratio}
\frac{\#\{\text{zeros in }(m,m+1)\}}{\#\{\text{ovals in }(m,m+1)\}}
\;\approx\;4.38+\frac{11.14}{\ln(t/2\pi)}
\qquad(3\leq T\leq401).
\end{equation}
This is a weighted empirical fit; above $T=100$, its binned residuals are
within $\pm0.01$. The competing fit $6.74\,T^{-0.040}$ is indistinguishable
on the measured data. If extrapolated, the displayed logarithmic model tends
to $4.38$, and with
$dN/dT=I'(T)\ln(t/2\pi)/2\pi\approx4T\ln T$ would give an oval density near
$0.9\,T\ln T$. This extrapolation is not an established asymptotic law.
Figure~\ref{fig:eqleg-density} plots both fits.

\begin{figure}[H]
\centering
\includegraphics[width=\textwidth]{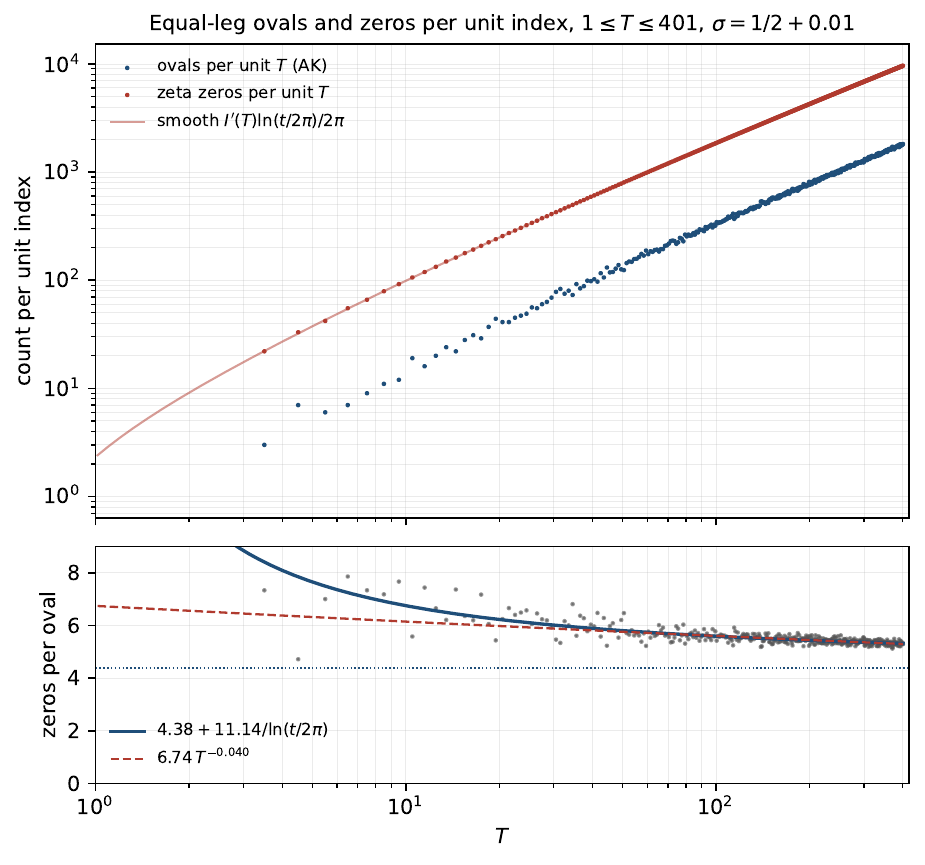}
\caption{Above: equal-leg ovals and zeros of $\zeta$ per unit interval of the
index, $1\le T\le401$, at $\sigma=\tfrac12+0.01$, log-log. The two clouds are
parallel. Below: the interval's zeros per oval, with \eqref{eq:oval-ratio} solid
and its extrapolated limit $4.38$ dotted, against the power law dashed. The two are
indistinguishable over $50\le T\le400$ and part company on either side.
Generated by \texttt{census\_analyze.py} of the equal-leg density companion
note in the Zest repository.}
\label{fig:eqleg-density}
\end{figure}

\paragraph{Where an off-line zero could be.} The observed oval extent
identifies a narrow region for further searching, but the sampled
cross-section does not provide a rigorous exclusion bound. In height: each
oval is about $0.49$ of a mean
zero spacing tall, a ratio steady to within $0.01$ above $T=100$, and there are about a fifth as many ovals as zeros, so the
ovals cover $8.8\%$ of the index axis over $350\le T<401$, up from $6.4\%$ over
$3\le T<25$, a share growing like $0.66\ln T+5.0$ per cent. On the sampled
cross-section $\sigma=\tfrac12+0.01$, detected oval intervals therefore cover
$8.8\%$ of the $T$-axis over $350\leq T<401$. This is sampled cross-sectional
coverage, not a bound on all possible off-line zero heights, because the
components may extend farther in height nearer the critical line. In $\sigma$:
the ovals draw in toward the critical line as $T$ grows. The
median oval reaches $\sigma=1.01$ at $T=5$ but only $\sigma=0.67$ at $T=400$;
the ninetieth percentile falls from $\sigma=1.11$ to $\sigma=0.77$;
and of the $258\,814$ ovals above $T=200$, just four reach $\sigma=0.95$.
Figure~\ref{fig:eqleg-reach} shows both measurements.

The tail does not vanish, though, and the distinction matters. What retreats is
the bulk of the ovals, not the widest of them: the widest oval in each interval
still reaches $\sigma\approx0.9$ at every height measured (never below
$\sigma=0.88$ for $T\ge50$), and the widest above
$T=200$ gets to $\sigma=0.962$. What thins is their number, from about eight
ovals per unit interval reaching $\sigma\ge0.9$ over $50\le T<100$ to four over
$100\le T<200$, two over $200\le T<301$ and one over $301\le T<401$, while the
total count per interval grows five-fold. A zero well off the line at large $T$
is therefore not excluded by any
of this; it would have to lie on the one or two such ovals per unit interval
rather than on any of the eighteen hundred, which is a sharper statement than the
count alone gives.

\begin{figure}[H]
\centering
\includegraphics[width=\textwidth]{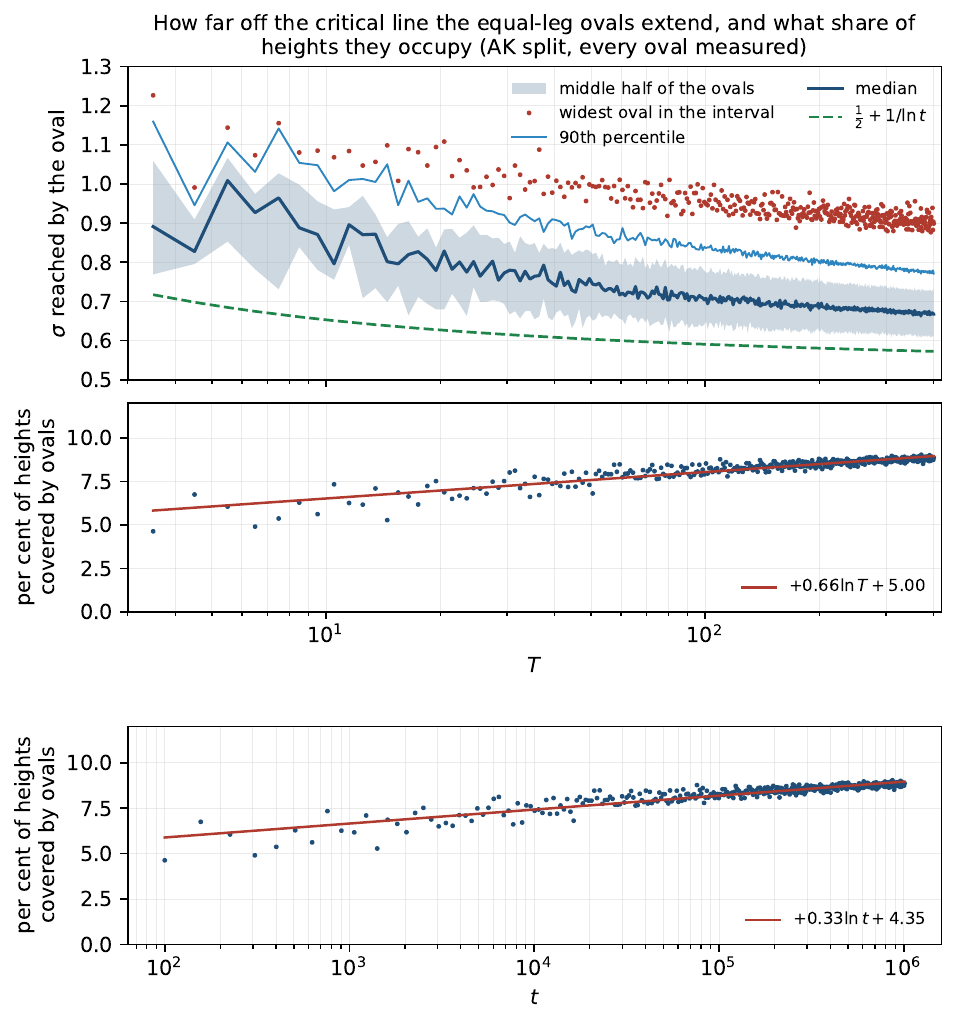}
\caption{Above: how far off the critical line the ovals extend, from
half-widths measured by bisection for every one of the $328\,511$ ovals in the
census, one point per unit interval, from $T=3$ where the first ovals are. The
band is the middle half of the ovals in that interval and the dots are the widest
oval in it. The dashed green curve is $\sigma=\tfrac12+1/\ln t$, drawn on the
index axis through $t=I(T)$. The bulk draws in toward
the line as $T$ grows while the widest ovals do not. Middle: the percentage of the
index axis covered by ovals, which is the share of heights at which an off-line
zero could sit, rising like $0.66\ln T+5.0$. Below: the same share read along
$t$, where $\zeta$ is evaluated. This is the middle panel with its horizontal
axis relabeled in $t$ instead of $T$, and its law reads $0.33\ln t+4.35$, the
law above rewritten through $\ln t=2\ln T+\ln2\pi+O(1/T)$. Measured directly in
$t$, rather than relabeled, the share moves by under $0.01$ of a percentage
point above $T=25$. Generated by
\texttt{reach\_census.py} and \texttt{fig\_reach.py} of the equal-leg density
companion note in the Zest repository.}
\label{fig:eqleg-reach}
\end{figure}

\paragraph{Usually, one zero per oval.} Almost every detected component
brackets exactly one critical-line zero: $328\,467$ of the $328\,511$ AK
components do so, while $23$ contain two zeros and $21$ fail the stated
one-zero bracket test. Not every zero has a detected component, and that
direction fails badly rather than by exception.
The zero sits at the oval's center on average rather than individually.
Its position within the oval has median $0.5000$ and mean $0.5000$ of the
height, so there is no systematic offset, but the spread about that is $0.14$.
While $91\%$ of zeros are in the middle half, one oval in eleven holds its
zero outside it and one in fifty outside the range $0.16$ to $0.84$.
The two-zero ovals are all unusually close pairs, separated by $0.25$ to
$0.65$ of the local mean spacing, which is what an oval half a spacing tall
would have to straddle to catch two. Those components are the hourglasses of
the previous paragraph: two oval lobes and a neck, one zero in each lobe.
Figure~\ref{fig:two-zero-hourglass} draws the first of them, at
$T\approx201.72$, and another at $T\approx320.57$, in the three-split style of
Figure~\ref{fig:ps-ak-r2}.

\begin{figure}[H]
\centering
\includegraphics[width=0.83\textwidth]{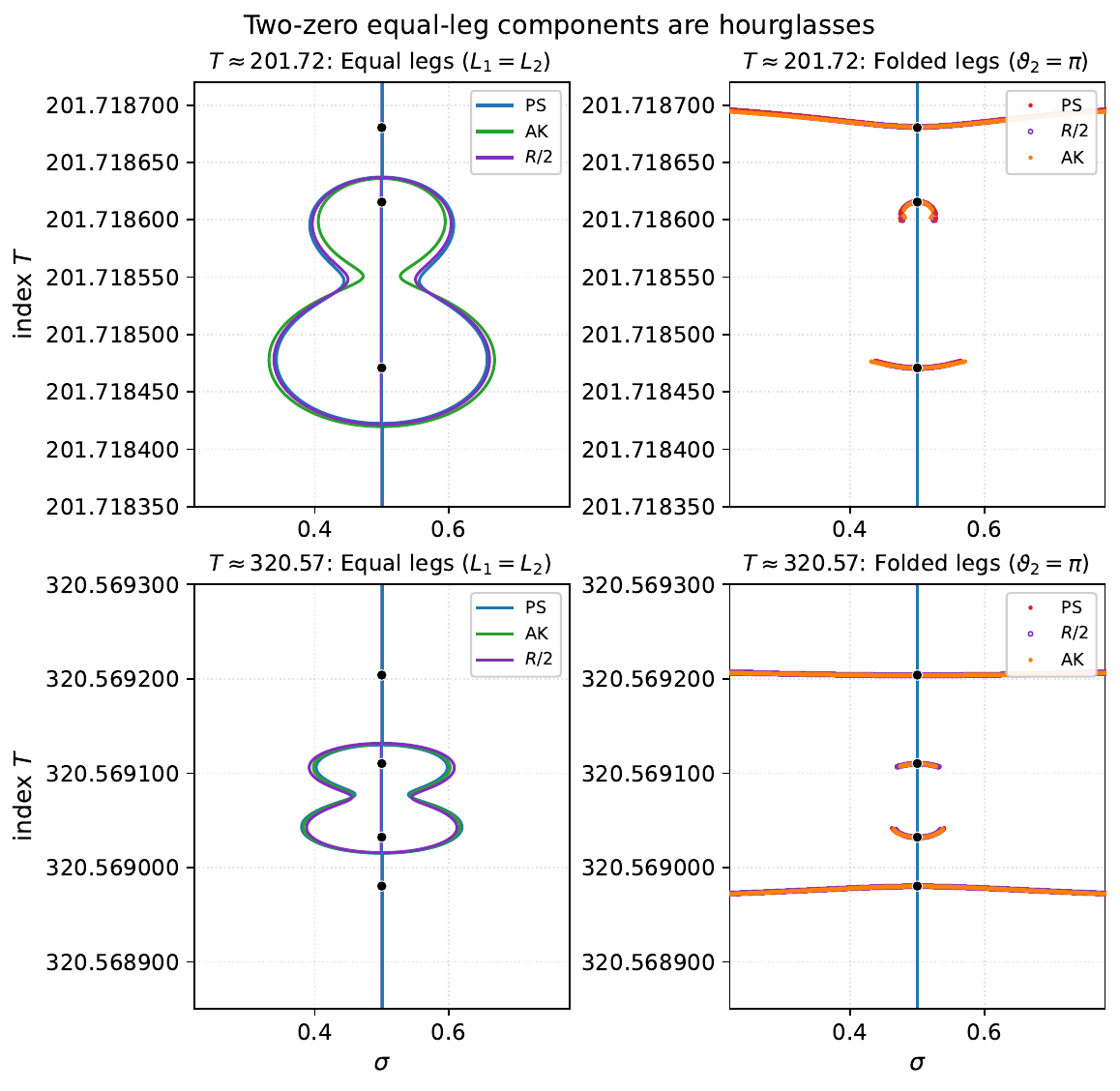}
\caption{Two AK ovals that each enclose two zeros, drawn as in
Figure~\ref{fig:ps-ak-r2}. Top: the first such oval, at
$T\approx201.72$. Bottom: another at $T\approx320.57$. Left: equal-leg
loci $L_1=L_2$ for PS (blue), AK (green), and $R/2$ (purple). Right:
folded-leg loci $\vartheta_2=\pi$ for the same three choices (red, orange,
and hollow purple). The blue vertical is the critical line; black dots are
zeta zeros. Each equal-leg component is an hourglass, two oval lobes joined
at a neck, with one zero in each lobe. The three splits nest and keep that
shape. Generated by \texttt{fig\_two\_zero\_hourglass.py}.}
\label{fig:two-zero-hourglass}
\end{figure}
Of the twenty-one empties, eighteen are ovals that span an integer, where the
chain gains a link, with their zero just across the integer-$T$ boundary; one is the
degenerate interval below $T=1$, where $\Sigma_1$ is empty; and two, at
$T=393$ and $T=398$, are ovals whose $\sigma$-tip stops a hair past the scan
line, at $\sigma=0.51004$ and $\sigma=0.51003$, so that the cross-section the
scan cuts is a sliver $3\times10^{-7}$ tall at the very tip, lying just to
one side of the zero's height. The first of those two misses below the zero
and the second above it, by under $10^{-7}$ either way. Both ovals do hold
their zero: at the critical line each is some $4\times10^{-6}$ tall with the
zero inside, and the cross-sections hold it out to $\sigma=0.50997$. The miss
is in the bracket test at $\varepsilon=0.01$, not in the oval. The exception
rate falls from $0.104\%$ over $3\le T<50$ to $0.007\%$ over $100\le T<200$
and then edges back up to $0.014\%$ over $301\le T<401$, the rise being the
close pairs, whose number grows with the zero count while staying under one
oval in seven thousand.

The other direction is not close to one for one, and it is the more important
half of the statement. There are $1\,771\,661$ zeros in the range against
$328\,511$ ovals, so approximately
$1-328\,511/1\,771\,661=81.46\%$, or about four zeros in five, have no
detected component around them. The
relationship between detected components and zeros is an incidence relation,
not a single-valued map: most components contain one zero, $23$ contain two,
and $21$ fail the stated one-zero bracket test, while most zeros lie in no
detected component. Which zeros acquire components is a question the census
does not answer.

Three things follow. First, the off-line locus is tied to the zeros rather
than being a structure independent of them. Almost every oval wraps one zero, centered on
it in the mean, about half a mean spacing tall and a fifth wider than that; the ovals are a
selection from the zeros rather than a phenomenon of their own, which is why
their count obeys \eqref{eq:oval-ratio} instead of a law of its own.

Second, it makes the disjointness program of \S\ref{sec:ps-ak-r2} and
\S\ref{sec:toward-proof} local. What is needed for an off-line zero is that
$\mathcal{E}_{\mathrm{ps}}$ and $\mathcal{E}_{\mathrm{rs}}$ meet off the line.
Since an oval almost always encloses exactly one zero, the two ovals that have to be
compared are the PS oval and the $rs$ oval belonging to the \emph{same} zero,
and the global statement ``the two loci have no common point off the line''
becomes one nesting statement per bracketed zero, of which there are about
$0.9\,T\ln T$ per unit index. That is the right shape for an argument: a local
inequality attached to each zero, to be proved uniformly, rather than a
statement about two curves wandering the strip.

Third, the caution. None of this is evidence for or against the Riemann
Hypothesis, and the one-zero-per-oval law in particular is not. The zero counted
inside each oval is found \emph{on} the critical line, by sign changes of
Hardy's $Z$. The census evaluates an off-line numerical approximation to
$\zeta$ when computing the imbalance $g$ on $\sigma=0.51$, but it does not
test the equation $\zeta=0$ along the detected components and does not search
them for off-line zeros. A component already holding an on-line zero at its
center is not forbidden from passing through one. No orphan component was
detected among the components intersecting the scan line $\sigma=0.51$ over
the stated range; this does not exclude components that fail to intersect that
line. The census also points at the search that
would test the hypothesis directly, since evaluating $\zeta$ along the ovals is
a one-dimensional problem (about $0.9\,T\ln T$ closed curves per unit index),
where searching the strip for off-line zeros is a two-dimensional one.

There is a converse worth stating. Almost every detected component brackets a
zero, but a zero usually has no detected component: in this sample only about
$18.54\%$ of zeros acquire one, and what distinguishes those from the rest is
open. The bracketing requires the
transverse imbalance $\partial g/\partial\sigma$ on the line to change sign on
both sides of the zero, and nothing here explains why it does so for two zeros
in nine and not for the others.

\section{Geometric origins of the \texorpdfstring{$I(T)$}{I(T)} and Yin and Yang functions}\label{sec:geometric-origins}

The two devices this paper leans on hardest, the index map $t=I(T)$ of
\S\ref{sec:IT} and the yin and yang curves of
\S\ref{sec:yinyang-inf}--\ref{sec:yinyang-general}, were both found in
the bisector frame, by watching the link chains of \S\ref{sec:geometry}
move. This section reconstructs the two discoveries.

\subsection{Geometric origin of \texorpdfstring{$I(T)$}{I(T)} functions}\label{sec:IT-origin}

Change the point of view: hold the forward bisector link still. Translate
joint $m$ to the origin and rotate and scale so that the bisector link lies
along the real axis from $0$ to $1$; everything else moves in this \emph{bisector
frame} as $t$ increases. Watching the two links on either side of the bisector
link, the link at $m-1$ behind it and the link at $m+1$ ahead of it, one
sees a periodic picture: each neighboring link revolves around the
stationary bisector link approximately \emph{twice} (approaching exactly two
revolutions as $T$ grows), and then the whole pattern repeats. The instant the
pattern repeats is precisely the handoff instant of \S\ref{sec:bisector}: the
neighboring link \textbf{\emph{folds back}} onto the bisector link, the bisector point moves
onto that next link and that next link becomes the new bisector link with its own two
neighbors.

\begin{figure}[H]
\centering
\includegraphics[width=\textwidth]{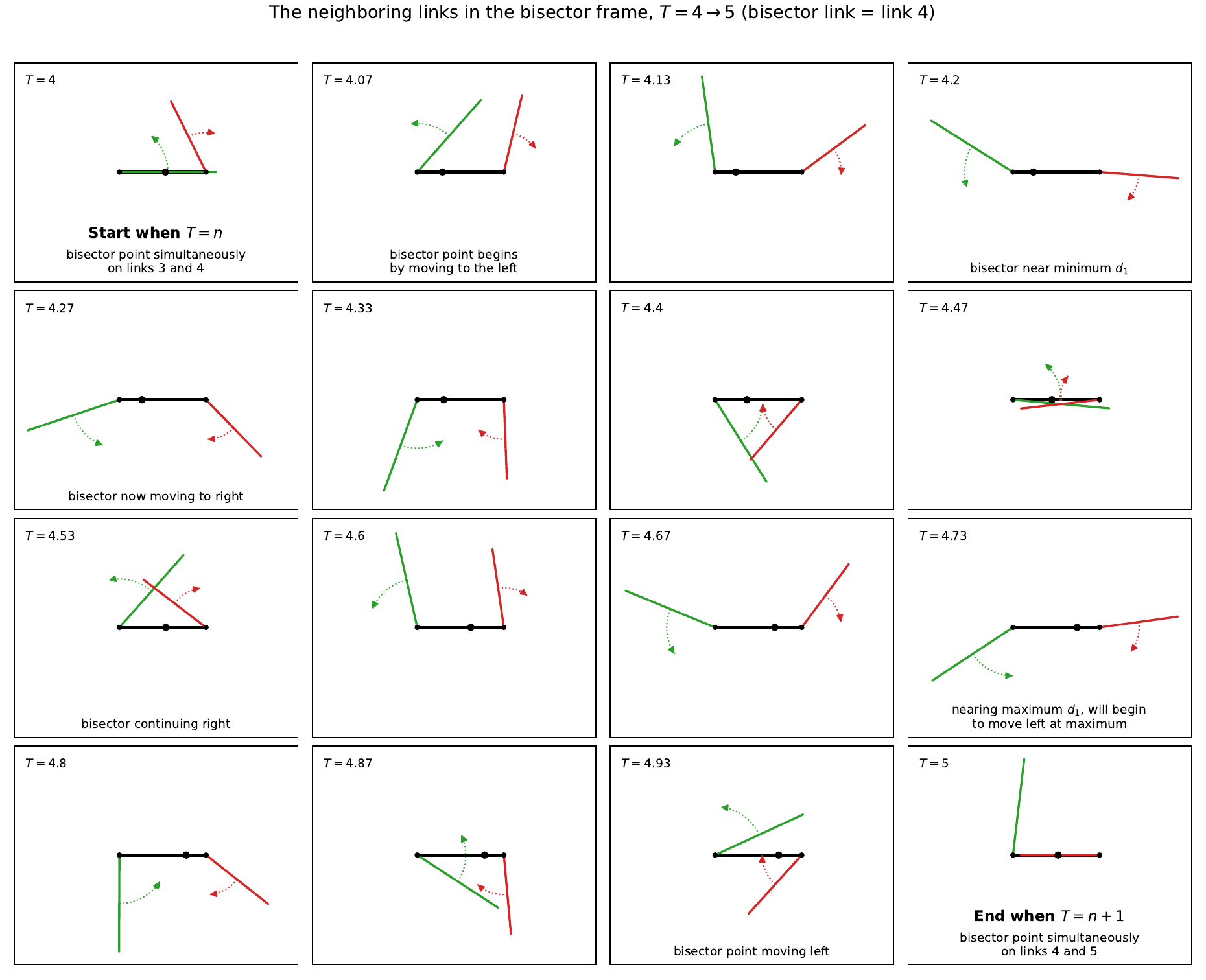}
\caption{The two neighboring links in the bisector frame as $T$ sweeps
from $4$ to $5$ ($\sigma=\tfrac12$), sixteen snapshots read left to
right from the top left ($T=4$) to the bottom right ($T=5$). In every
panel the bisector link, link $m=4$ of the forward chain, is held
fixed along the real axis from $0$ to $1$ (black); link $3$ (green)
hangs off joint $4$ at the origin and link $5$ (red) off joint $5$ at
the point $1$, both scaled relative to the unit bisector link. The
dotted curved arrow on each neighbor runs from the middle of the link
to the middle of that link as the next panel depicts it, tracing the
motion the link takes between panels. The black dot on the bisector
link marks the bisector point, at the crossing fraction
$\hat d_1=M^{\sigma}d_1$ along the unit link. As $T$
advances, each neighbor revolves around the stationary bisector link
approximately twice. At $T=4$ (Start) the green link is folded back
onto the bisector link (unwrapped phase increment $9\pi$, hence geometric
joint angle $\pi\pmod{2\pi}$),
so the bisector point lies on links $3$ and $4$ simultaneously: that is
the handoff instant at which link $4$ became the bisector link. At
$T=5$ (End) it is the red link that folds back (unwrapped phase increment
$11\pi$, hence geometric joint angle $\pi\pmod{2\pi}$),
the bisector point lies on links $4$ and $5$ simultaneously, and
link $5$ takes over as the new bisector link. Generated by
\texttt{fig\_bisector\_frame.py}.}
\label{fig:bisector-frame}
\end{figure}

Besides the fact that the two links resemble the motion of the two different
length hands of a clock but with one going clockwise and the other going
counterclockwise, the \emph{periodicity} begs for a clock that ticks once per
bisector point handoff, and that is
where $I(T)$ came from: we sought a formula $t=I(T)$ whose input $T$ passes
through an integer at exactly the moments the bisector point moves from one
link to the next. Two facts about the joint angle between the outgoing and
incoming bisector links pin the formula down.

\begin{enumerate}
\item \emph{From the zeta function.} Consecutive summands of $\Sigma_1$ differ
in phase by $t\ln\frac{n+1}{n}$, so the turn at the joint where link $T-1$
meets link $T$ (interpolating the index continuously) has unwrapped phase
increment
\begin{equation}\label{eq:jangle}
J_{\mathrm{unw}}(T) \;=\; t\,\log\!\Bigl(\tfrac{1}{T}+1\Bigr).
\end{equation}
\item \emph{From the observed motion.} At a handoff the two links fold back
onto one another, so the unwrapped phase increment there is an
\emph{odd multiple of $\pi$} (hence geometric joint angle $\pi$); and the neighboring link revolves around the
stationary bisector link \emph{twice} per unit of the index, so the odd
multiple advances by two with each unit of $T$ (hence the factor $2$).
Interpolating between handoffs, define the unwrapped phase increment by
\begin{equation}\label{eq:jangle-count}
J_{\mathrm{unw}}(T) \;=\; \pi\,(2T+1).
\end{equation}
The corresponding geometric joint angle is
\[
J_{\mathrm{geom}}(T)\equiv J_{\mathrm{unw}}(T)\pmod{2\pi};
\]
at every integer $T$, $J_{\mathrm{geom}}(T)\equiv\pi\pmod{2\pi}$.
\end{enumerate}

Setting \eqref{eq:jangle} equal to \eqref{eq:jangle-count} and solving for $t$
yields
\begin{equation}\label{eq:IT-derived}
\pi\,(2T+1) \;=\; t\,\log\!\Bigl(\tfrac{1}{T}+1\Bigr)
\qquad\Longrightarrow\qquad
t \;=\; \frac{\pi\,(2T+1)}{\log\!\bigl(\tfrac{1}{T}+1\bigr)} \;=\; I(T),
\end{equation}
which is the mapping adopted in \S\ref{sec:IT}: under $t=I(T)$ the bisector
point resides on link $\lfloor T\rfloor$, and $T$ is an integer exactly when
the bisector point moves from one link to the next.

\begin{remark}
Taylor-expanding the reciprocal logarithm in \eqref{eq:IT-derived},
$1/\log(1+\tfrac1T)=T+\tfrac12-\tfrac1{12T}+\tfrac1{24T^2}-\cdots$, and
collecting orders of $T$ gives a family of invertible polynomial
approximations to $I(T)$; keeping successively more terms:

\begin{center}
\begin{tabular}{c|l|l}
terms & $I(T)\;\approx$ & solved for $T$\\
\hline
$2$ & $2\pi\bigl(T^2+T\bigr)$
    & $T=-\tfrac12+\sqrt{\tfrac{t}{2\pi}+\tfrac14}$\\[4pt]
$3$ & $2\pi\bigl(T^2+T+\tfrac16\bigr)$
    & $T=-\tfrac12+\sqrt{\tfrac{t}{2\pi}+\tfrac1{12}}$\\[4pt]
$4$ & $2\pi\bigl(T^2+T+\tfrac16-\tfrac{1}{180T^2}\bigr)$
    & $T\approx-\tfrac12+\sqrt{\tfrac{t}{2\pi}+\tfrac1{12}+\tfrac{\pi}{90t}}$
\end{tabular}
\end{center}

\noindent
(The $1/T$ term of the expansion cancels, so the fourth term is already
$O(T^{-2})$.) The $2$-term form misses $I(T)$ by a constant
${\approx}\,\pi/3$; the $3$-term form is accurate to $O(T^{-2})$ and the
$4$-term form to $O(T^{-3})$. At $T=20$ their errors in $T$ after
inversion are $4\times10^{-3}$, $3\times10^{-7}$ and $4\times10^{-11}$
respectively. The first two rows invert exactly (quadratics); the
$4$-term row is inverted by substituting the $3$-term inverse into the
small $T^{-2}$ correction. Figure~\ref{fig:IT} draws $I(T)$ against the
two-term form; the difference is invisible at scale, about one against
values in the thousands.

The table also places Siegel's cutoff \eqref{eq:siegel-m} on the index. The
$3$-term row is $t/2\pi=(T+\tfrac12)^{2}-\tfrac1{12}+O(T^{-2})$, so
\begin{equation}\label{eq:siegel-cutoff-index}
\sqrt{\frac{t}{2\pi}}\;=\;T+\frac12-\frac{1}{24T}+O(T^{-2}),
\end{equation}
which is right to $5\times10^{-4}$ at $T=6.18$, $5\times10^{-5}$ at $T=20$ and
$2\times10^{-6}$ at $T=100$. Siegel's $m$ therefore steps \emph{near} the
\emph{half}-integers of $T$, a shade past them by $1/24T$, while ours steps
\emph{exactly at} the integers.
\end{remark}

\begin{figure}[H]
\centering
\includegraphics[width=0.88\textwidth]{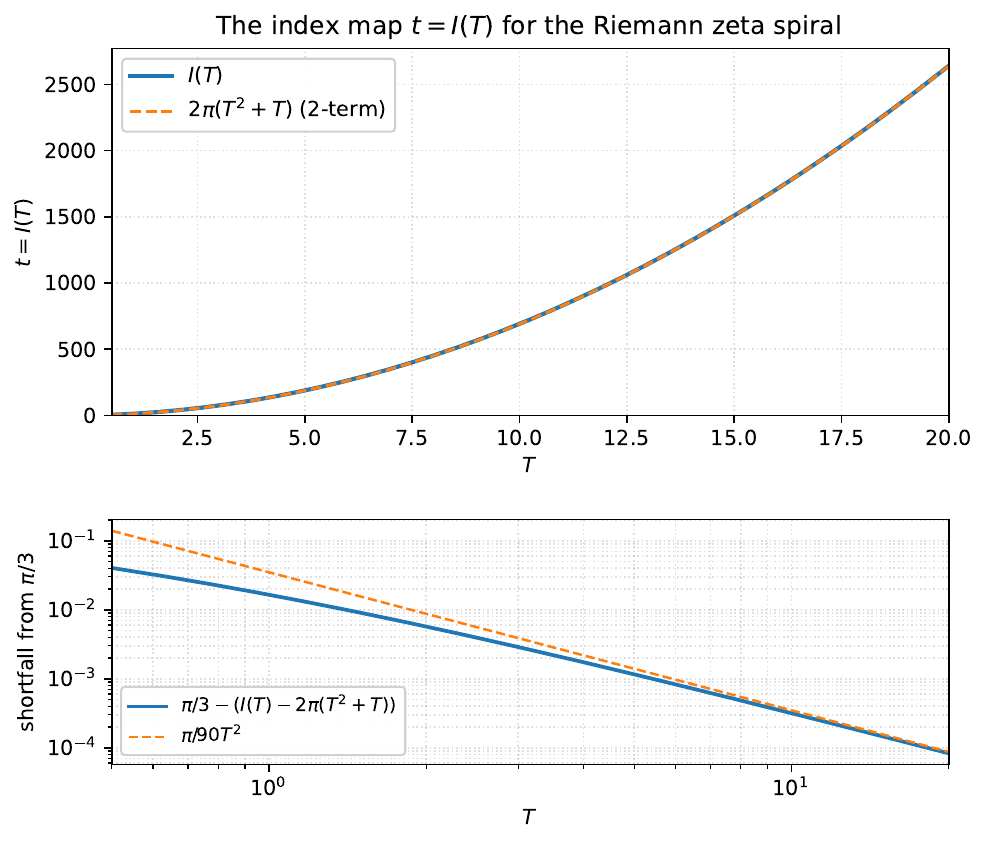}
\caption{Above: the index map $t=I(T)$ of \eqref{eq:IT} (solid), together with
the two-term approximation $2\pi(T^{2}+T)$ from the Taylor table above
(dashed). Below: the difference between the two, which the panel above cannot
show, the curves differing by about one against values in the thousands. That
difference is the constant $\pi/3$ of the table, the $\tfrac16$ the third term
restores, so what is plotted log-log is its distance from that constant, which
is the fourth term: it follows $\pi/90T^{2}$ (dashed) from $T=5$ on, and is
$8\times10^{-5}$ by $T=20$.
Generated by \texttt{fig\_IT.py}.}
\label{fig:IT}
\end{figure}

We stress that the polynomial forms of the table serve only to compare
with the familiar $\sqrt{t/2\pi}$:
everywhere in this paper $t=I(T)$ is used \emph{exactly}. The defining
property of the mapping, that the neighboring link folds back onto the
bisector link \emph{exactly} as $T$ passes an integer and the bisector
point transfers from one link to the next exactly at that instant,
holds only for the exact $I(T)$. Under any approximation, including the one whose
inverse is essentially Siegel's $\sqrt{t/2\pi}$, the fold instants drift
off the integers by the corresponding error in the table: the clock
still ticks once per handoff, but never precisely on the hour.

\begin{remark}[comparing the classical $t$ rescaling so that zeros have unit
mean spacing, and $I(T)$]\label{rem:unfolding}
In random-matrix and zeta-statistics work, zero ordinates are unfolded using
the smooth counting function
\[
\overline N(t)=\frac{t}{2\pi}\log\frac{t}{2\pi}
-\frac{t}{2\pi}+\frac78,
\qquad
\widetilde\gamma_n=\overline N(\gamma_n).
\]
The ordinates $\widetilde\gamma_n$ have asymptotically unit mean spacing,
which lets one compare zero statistics at different heights. The exact
staircase instead satisfies $N(\gamma_n)=n$ under the simple-zero convention
and therefore removes the fluctuations by construction. Our $I(T)$ is a change of
variable in the same spirit, a reparameterization of the height $t$, but it
runs in the opposite direction, compressing ever more zeros into each unit
interval of the index $T$. For example, the interval $[5,6]$ carries $42$
zeros, $[10,11]$ carries $106$, $[20,21]$ carries $256$, and $[50,51]$ carries
$802$. Rescaling, or what the random-matrix literature calls unfolding, the
zeros flattens the count to unit spacing; the index does the reverse.
Unfolding views the zeros one at a time, each in its own unit of the new
variable; the index views them in groups of ever increasing size, one
crowd of zeros per unit of $T$.
\end{remark}

\subsubsection{\texorpdfstring{$I(T)$}{I(T)} functions for other \texorpdfstring{$L$}{L}-functions besides the zeta function}\label{sec:IT-Lfunctions}

The same change-of-variable idea extends from $\zeta$ to Dirichlet
$L$-functions. This subsection is an aside recording another observation
we made along the way; nothing in it is used elsewhere in this paper.
Let $p$ be an odd prime. On the critical line the partial sums of
\begin{equation}\label{eq:Lchi}
L(s,\chi_p)
\;=\;
\sum_{n=1}^{\infty}\chi_p(n)\,n^{-s}
\end{equation}
form a spiral analogous to zeta's, but now the character
$\chi_p(n)=\bigl(\tfrac{n}{p}\bigr)$ inserts periodic sign changes and, when
$p\mid n$, zero-length \emph{phantom} links. The series is conditionally,
not absolutely, convergent on $\Re s=\tfrac12$. Counting only the nonzero links as \emph{links},
there are $p-1$ links and one phantom in every block of $p$ summands, and
the natural index $T$ is the link index (phantoms receive no $T$).

Watching these link chains as $t$ varies, we noticed that the bisector point
still moves from one link to the next, just as for zeta, but not always when
links fold back on one another: sometimes the bisector point moves from link to link at a
\emph{joint}, as the bisector line sweeps through it. So $L$-functions have
other event types that zeta does not. An \emph{event} is the instant the
bisector point moves from one link to the next. There are two ways this can
happen: either the two
links fold back onto one another (the turning angle at their common joint is
an odd multiple of $\pi$), so the bisector point slides across the joint from
the old link onto the new one; or the bisector line sweeps through a joint,
carrying the crossing from one link to its neighbor. For zeta only the first
kind occurs: every handoff of \S\ref{sec:bisector} is a folded event, once
per unit of $T$, and there are no bisector events. For $L(s,\chi_p)$ both
kinds occur, recurring with period $\varphi(p)=p-1$ in the integer index $T$:
\begin{itemize}
\item \emph{folded} events, where the turning angle at the current joint is
$\pm\pi$ (the chain reverses), at $T\equiv0\pmod{p-1}$;
\item \emph{bisector} events, where the joint lies on the perpendicular
bisector of the segment from $0$ to $L(s,\chi_p)$, at the remaining residues
modulo $p-1$.
\end{itemize}
For $p=3$ one has $\varphi(3)=2$: even $T$ are folds and odd $T$ are
bisectors. The map $t=I_p(T)$ is defined so that integer $T$ lands at the
$t$-value of its event. Figure~\ref{fig:L3-spirals} draws the two link
chains of $L(s,\chi_3)$ at $T=6.125$.

\begin{figure}[H]
\centering
\includegraphics[width=0.85\textwidth]{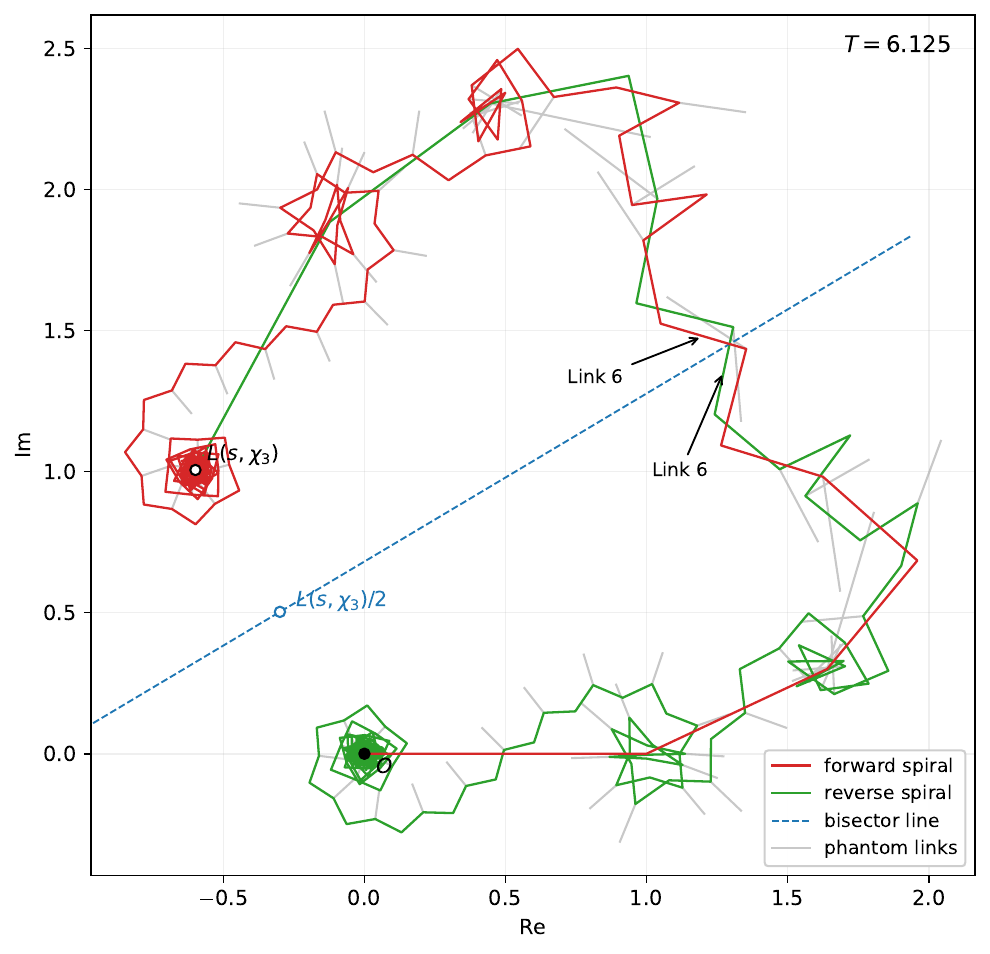}
\caption{The two partial-sum spirals of $L(s,\chi_3)$ at $T=6.125$ on the
critical line ($t=I_3(T)\approx176.13$), drawn with
$\lfloor 2T(T{+}1)p\rfloor=261$ summands. The forward chain (red) starts at
the origin $O$ and winds into its spiral around the analytic value
$L(s,\chi_3)$ (open dot); the reverse chain (green) is its mirror image
across the bisector line (dashed blue), the perpendicular bisector of the
segment from $O$ to $L(s,\chi_3)$ through the labeled midpoint
$L(s,\chi_3)/2$, so the tail of the link chain sits on $L(s,\chi_3)$ and its
head lands at $O$. Every third summand has $\chi_3(n)=0$: these zero-length phantom links
are drawn light gray as the connectors the chain skips. The labeled link $6$
of each chain is the current bisector link: the two chains cross on it,
on the bisector line. Generated by \texttt{fig\_L3\_spirals.py}.}
\label{fig:L3-spirals}
\end{figure}

Because only a fraction $(p-1)/p$ of the summands are links, the smooth
\emph{base} formula generalizing \eqref{eq:IT} is
\begin{equation}\label{eq:Ip-base}
I_p^{\mathrm{base}}(T)
\;=\;
\frac{4\pi\,T}{(p-1)\,\log\dfrac{pT+(p-1)}{pT-(p-1)}}\,.
\end{equation}
Asymptotically $I_p(T)\sim\dfrac{2p}{(p-1)^{2}}\,\pi\,T^{2}$: the leading
coefficient is $\tfrac32\pi$ for $p=3$ and $\tfrac58\pi$ for $p=5$, against
$2\pi$ for zeta's own $I(T)$. (Zeta is not the $p=1$ member of this family,
since the factor $2p/(p-1)^2$ diverges there; for zeta every summand is a
link and there are no phantoms.) The base does not run through the middle of the
events: it touches $I_p$ at the folds and lies below it in between, missing a
one-signed bulge of period $p-1$ imposed by the character. For $p=3$ the
contact is exact, since at even $T$ the gate of \eqref{eq:I3} vanishes and
$c=2$ turns \eqref{eq:I3-geom} into the base; for $p=5$ the fitted correction
returns to zero at $T\equiv0\pmod4$ to within $2\times10^{-3}$. One restores
the bulge by writing
\begin{equation}\label{eq:Ip-split}
I_p(T)
\;=\;
I_p^{\mathrm{base}}(T)+C_p(T).
\end{equation}

For $p=3$ the character is $[1,-1,0]$. The base specializes to
$I_3^{\mathrm{base}}(T)=2\pi T\big/\log\frac{3T+2}{3T-2}$, and the
correction is built from the character-modulated map
\begin{equation}\label{eq:I3-geom}
c=\frac{3+\cos(\pi T)}{2},
\qquad
I_3^{\mathrm{geom}}(T)
=\frac{c\,\pi\,T}{\log\dfrac{3T+c}{3T-c}},
\end{equation}
together with a small odd-$T$ offset
$\alpha(T)=0.390914+0.01712\,T^{-2}$ fitted to event scans:
\begin{equation}\label{eq:I3}
I_3(T)
=
I_3^{\mathrm{geom}}(T)
+\tfrac12\bigl(1-\cos\pi T\bigr)\,\alpha(T).
\end{equation}
(Equivalently $I_3=I_3^{\mathrm{base}}+C_3$ with
$C_3=I_3^{\mathrm{geom}}-I_3^{\mathrm{base}}
+\tfrac12(1-\cos\pi T)\,\alpha$; the gate vanishes on even $T$, so
$C_3=0$ on folds.) On odd $T=1,\ldots,99$ the residual between scanned
event times and \eqref{eq:I3} has RMSE about $8\times10^{-6}$.

For $p=5$ the same pattern gives
$I_5^{\mathrm{base}}(T)=\pi T\big/\log\frac{5T+4}{5T-4}$ and a fitted
period-$4$ correction
\begin{equation}\label{eq:I5}
C_5(T)
=
0.6206-0.6576\cos\tfrac{\pi T}{2}
+0.0369\cos\pi T+\tfrac{0.169}{T^{3.12}}.
\end{equation}
Figure~\ref{fig:I-I3-I5} compares $I(T)$, $I_3(T)$, and $I_5(T)$ on a common
range of $T$.

\begin{figure}[H]
\centering
\includegraphics[width=0.92\textwidth]{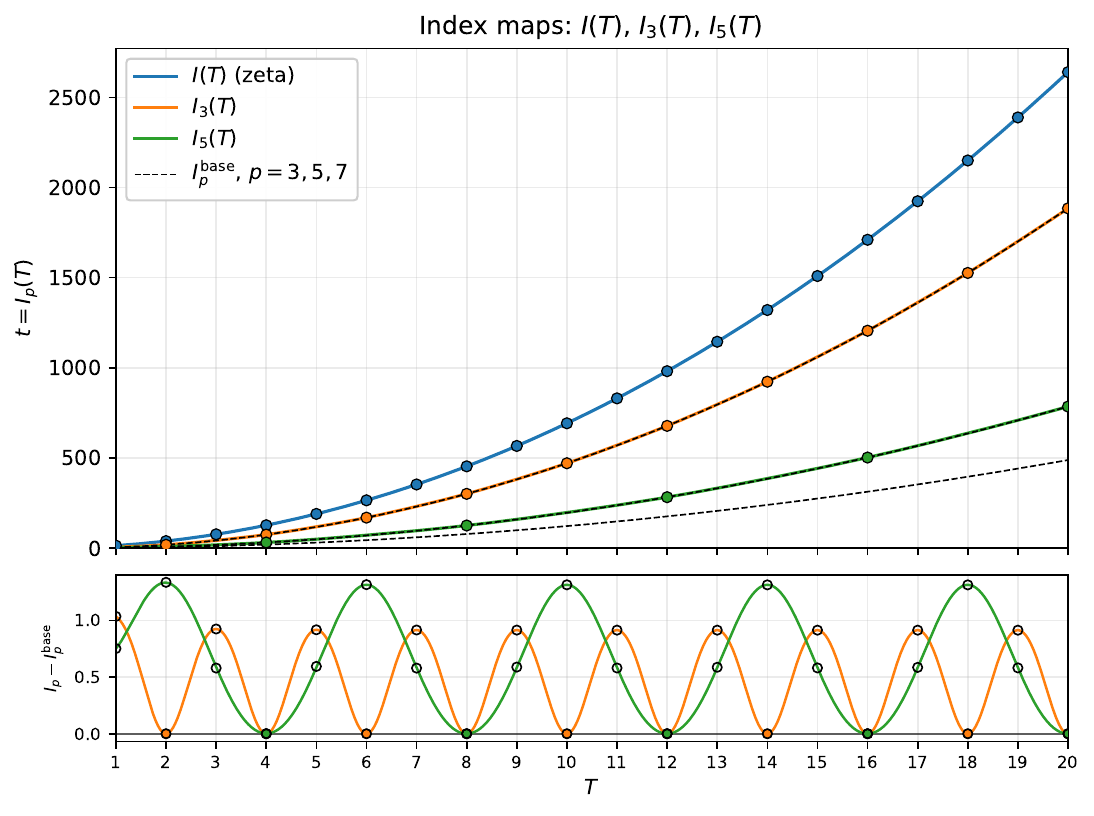}
\caption{Top: index maps on $T\in[1,20]$: zeta's $I(T)$ (blue), $I_3(T)$ for
$p=3$ (orange), and $I_5(T)$ for $p=5$ (green).
Dots mark the fold events, at $T\equiv0\pmod{p-1}$: every integer on the
zeta curve, every other integer on $I_3$, and every fourth integer on
$I_5$. Leading growth is quadratic with coefficients $2\pi$, $\tfrac32\pi$,
and $\tfrac58\pi$ respectively. Thin dashed black: the smooth base maps
\eqref{eq:Ip-base} for $p=3,5,7$. The $p=3$ and $p=5$ bases are hidden under
$I_3$ and $I_5$ at this scale; the $p=7$ base is the lowest curve, drawn for
comparison, no $I_7$ having been fitted here. Bottom: the character bulge on
its own scale, $I_p-I_p^{\mathrm{base}}$, which the top panel cannot resolve
because it is of order $1$ against values in the thousands. It is one-signed,
returning to zero at the folds (dots) and rising to ${\approx}\,0.92$ between
them for $p=3$ and ${\approx}\,1.33$ for $p=5$, and its amplitude barely
decays: the $p=3$ peaks are $0.9253$, $0.9183$, $0.9153$ at $T=3,5,11$. So the
correction is a fixed additive offset, not a fixed relative one. Open black
circles are the numerically scanned events at every integer $T$, measured
against the same base: the folds sit on zero, the intermediate events ride the
bulge, and all of them follow the modeled curves to within
$1.7\times10^{-5}$ for $p=3$ and $1.0\times10^{-2}$ for $p=5$ on $T\le20$.
Generated with the Dirichlet $L$-function companion note.}
\label{fig:I-I3-I5}
\end{figure}

\clearpage
\subsection{Geometric origin of the yin and yang curves}\label{sec:yinyang-origin}

This subsection records the observation that led us to formulate the
equations of the yin and yang curves; only later did we find their
connection to equations in Siegel's 1932 paper~\cite{Siegel1932}.
We attempt to explain what happens when the link chains are viewed in
various coordinate frames, and how they move relative to one another as
$\{T\}$ is varied at various $\lfloor T\rfloor$; this is hard to describe
in words but easy to see with computer visualization assistance using the
Zest program.

The story starts with watching the bisector point ride the chain. As we have
said, as $T$ increases at a steady rate, the bisector point stays on a single link, moving back and forth along
it (at the fraction $\hat d_1$ of its length) all while $T$ is not an
integer. Then at the instant it becomes an integer, the
link the bisector point is on hands off the role of carrying the bisector
point to the next link in the chain. When $\sigma=\tfrac12$, there is no
jump: the handoff happens exactly at the instant when the outgoing
and incoming links fold back onto one another and momentarily occupy the same
line in the plane (when $\sigma=\tfrac12$ the two bisector links are then
coincident; for $\sigma\ne\tfrac12$ they are parallel). When
$\sigma=\tfrac12$, the path of the bisector point is therefore continuous
both as a curve in the complex plane and as a position along the
one-dimensional link structure. Off the line it is continuous only on
regular intervals between parallel-link instants.

\S\ref{sec:IT-origin} watched the two links \emph{neighboring} the bisector
link in the bisector frame. Now watch the \emph{reverse bisector link} in that
same frame: holding the forward bisector link stationary while $T$ increases,
the reverse bisector link revolves around it, and each end of the revolving
link follows a path reminiscent of a teardrop, the pair of paths together
resembling a yin-yang (Figure~\ref{fig:yinyang}). The paths of the two ends
are exactly the yin and yang curves of \S\ref{sec:yinyang-general}: the
bisector frame is the local coordinate frame \eqref{eq:frame-map}, and the two
ends of the revolving link are the points $\mathrm{Yin}(s)$ and
$\mathrm{Yang}(s)$ of \eqref{eq:yin1}--\eqref{eq:yang1}. That is where those
definitions came from, and this subsection reconstructs the rest of the
discovery: how the crossing of the revolving link with the real axis produced
the weights $d_1$ and $d_2$ of \S\ref{sec:decomp}. (The formulas for the
curves, their asymmetry, and their $T\to\infty$ limit were developed in
\S\ref{sec:yinyang-inf}--\ref{sec:yinyang-general} and are not repeated
here.)

\begin{figure}[htbp]
\centering
\includegraphics[width=\textwidth]{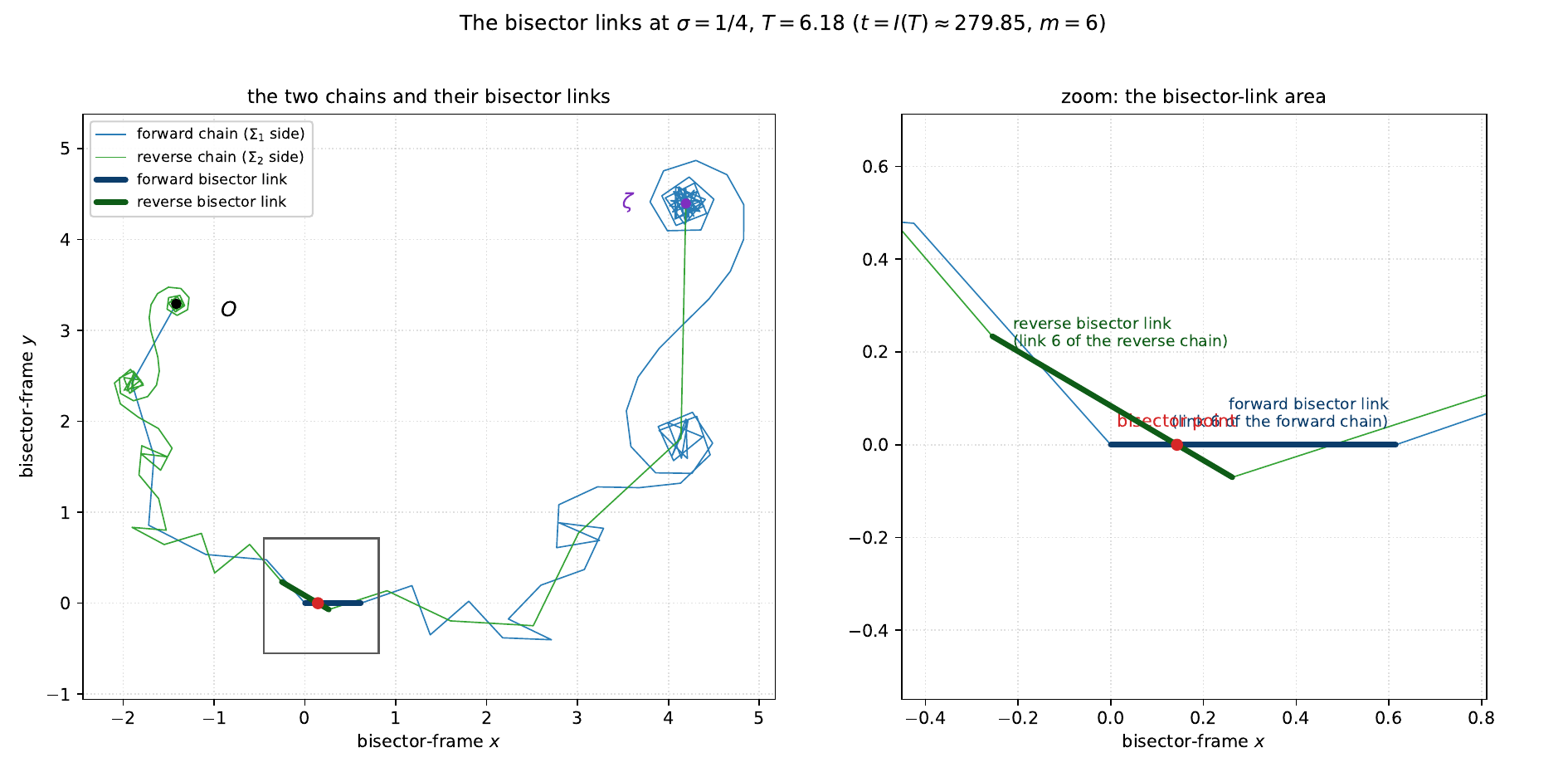}
\caption{The two bisector links at $\sigma=\tfrac14$, $T=6.18$
($t=I(T)\approx 279.85$, $m=6$). Both panels are drawn in the coordinate
frame of the \emph{forward bisector link}: the whole picture is translated
and rotated so that this link lies along the positive $x$-axis, with no
rescaling. All lengths are true, so the link runs from $0$ to
$7^{-1/4}\approx0.615$. \emph{Left:} the forward chain (blue, starting at the
world origin $O$) and the reverse chain (green, ending at $\zeta$), each
drawn out to the link nearest its spiral center; link $6$ of each chain
(the forward and reverse bisector links) is overdrawn with a thick stroke,
and the red dot where they cross is the bisector point.
\emph{Right:} zoom into the gray window: the two bisector links crossing at
the bisector point, which sits on the $x$-axis at distance $d_1$ from the
origin. The bisector frame of this section additionally rescales this link
to the unit interval $[0,1]$ and watches the thick reverse link revolve
around it. Generated by \texttt{fig\_yinyang\_spirals.py}.}
\label{fig:yinyang-spirals}
\end{figure}

\begin{figure}[htbp]
\centering
\includegraphics[width=0.82\textwidth]{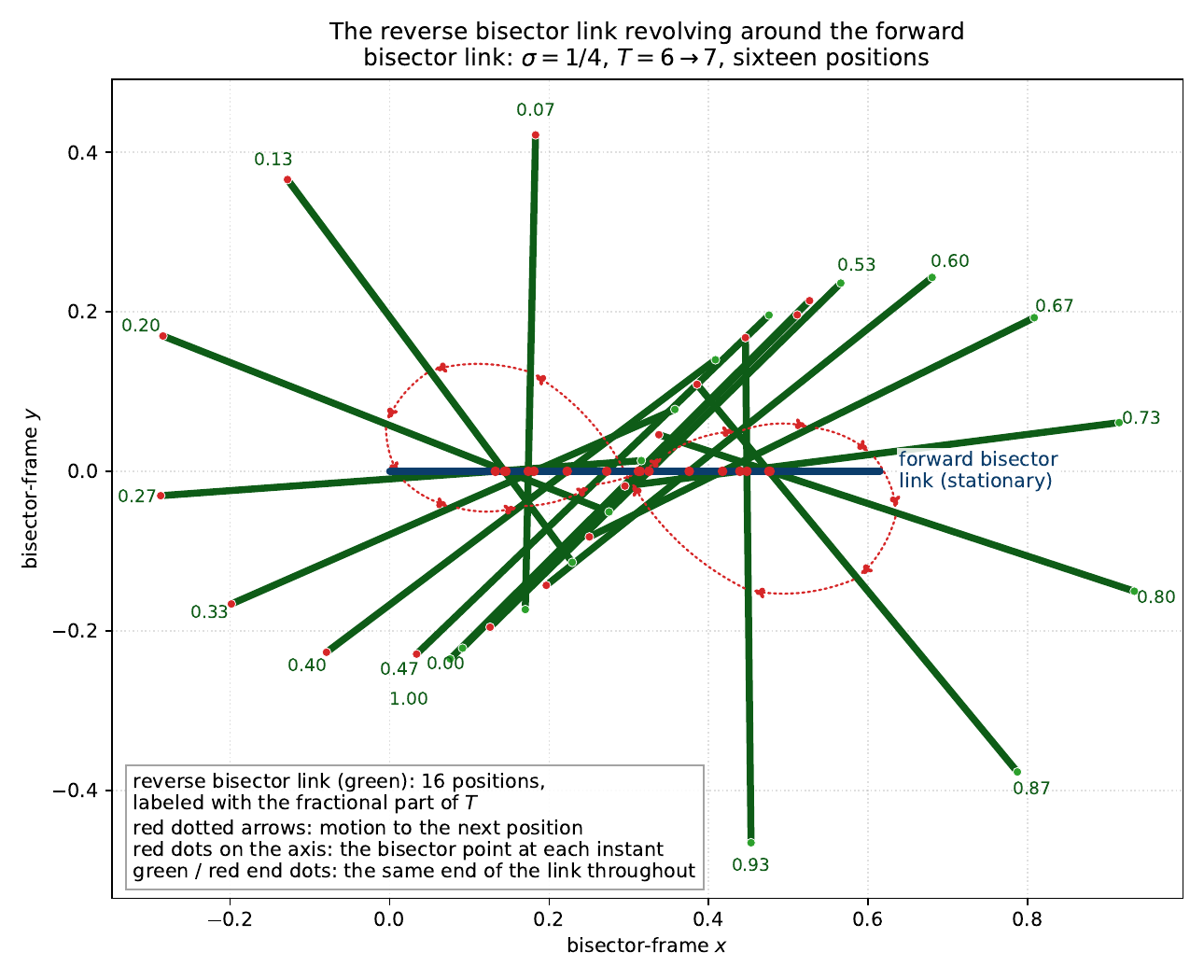}
\caption{The two bisector links in the frame of the forward bisector link as
$T$ sweeps from $6$ to $7$ ($\sigma=\tfrac14$, the frame of
Figure~\ref{fig:yinyang-spirals}, lengths unscaled). The forward bisector
link is stationary (the single thick dark-blue segment on the $x$-axis)
while the reverse bisector link (green, same thickness) is drawn at sixteen
equally spaced values of $T$, each position labeled on its outer end with the
fractional part of $T$; over the sweep it revolves fully around the
stationary link, and at $T=7$ (labeled $1.00$) it returns nearly to its $T=6$
position; the heading turns through a half-revolution and then back, so
the directed segment ends in the same orientation without having rotated
all the way around. The red dotted arrows run from the middle of each position to the
middle of the next, tracing the motion. The small green and red dots mark
the two ends of the revolving link consistently: green is always the
joint-$m$ end and red the joint-$(m{+}1)$ end; they trace the yin and yang
curves of Figure~\ref{fig:yinyang}, in the same colors. Each red dot on the
$x$-axis is the bisector point at one instant: the
crossing of that position with the $x$-axis, at $(d_1,0)$. Near the parallel
instants (fractional part of $T$ near $\tfrac14$ and $\tfrac34$) the
revolving link is nearly parallel to the stationary one and the crossing
races along the axis (the poles of $d_1$ of
Appendix~\ref{sec:pole-locations}). Generated by
\texttt{fig\_bisector\_sweep.py}.}
\label{fig:bisector-sweep}
\end{figure}

The motivating question is the one left open in \S\ref{sec:bisector}: the two
bisector links cross, but \emph{where}? The answer matters because, by
\eqref{eq:bisector-point}, the crossing (the bisector point) sits at
distance $d_1$ from the joint $\Sigma_1$ along the forward bisector link and
at distance $d_2$ from the joint $\Sigma_1+R$ along the reverse bisector link.
If we can \emph{compute} where the crossing is, we obtain $d_1$ and $d_2$, and with
them a new formula for $\zeta$ that uses the fractional summands in place of
the usual remainders. Of course we could do this knowing $\zeta$, but the
trick is to do it without knowing $\zeta$ or either partial sum.

\begin{figure}[htbp]
\centering
\includegraphics[width=0.95\textwidth]{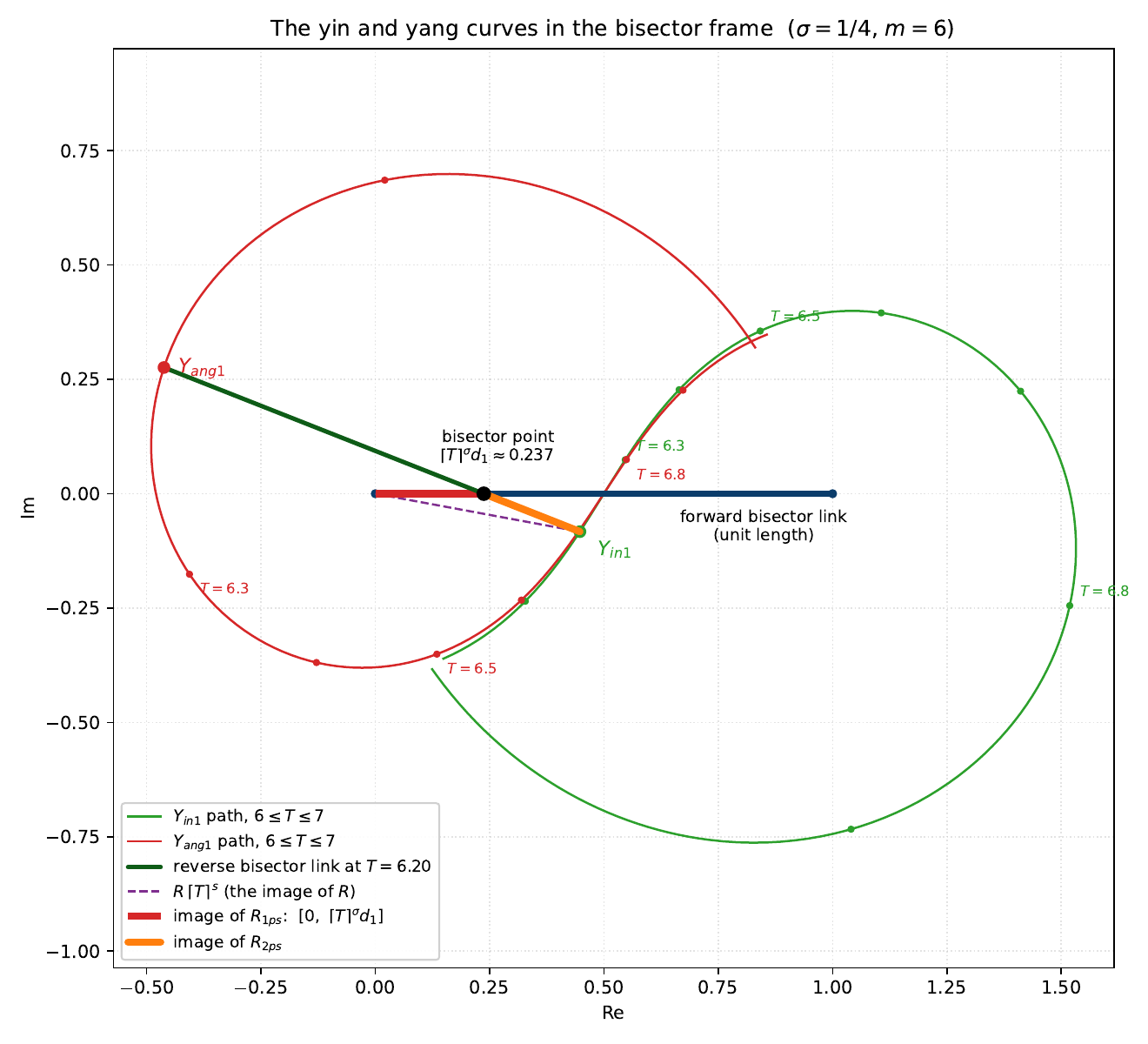}
\caption{The yin and yang curves in the bisector frame, $\sigma=\tfrac14$,
$m=6$. The frame sends the forward bisector link to the unit interval $[0,1]$
(dark blue). As $T$ sweeps $6\le T\le 7$
(Figure~\ref{fig:bisector-sweep}), the near end of
the reverse bisector link traces the yin curve $\mathrm{Yin}$ (green, one dot
every $0.1$ in $T$) and the far end traces the yang curve $\mathrm{Yang}$
(red); the two teardrop paths together resemble a yin-yang. The reverse
bisector link itself is drawn at $T=6.20$ (dark green): it crosses the real
axis at the bisector point (black dot), at the crossing fraction
$\hat d_1=M^{\sigma}d_1\approx 0.237$ of the way along the unit link. The
dashed violet vector from $0$ to $\mathrm{Yin}$ is the in-frame image of
Siegel's remainder $R$; the crossing splits it into the images of $R_{1}$
(red, the piece of the unit link from $0$ to the bisector point) and $R_{2}$
(orange, from the bisector point to $\mathrm{Yin}$). Generated by
\texttt{fig\_yinyang.py}.}
\label{fig:yinyang}
\end{figure}

The frame did not begin as the abstraction of \S\ref{sec:yinyang-general};
it began as a way around a difficulty. In the world frame the reverse chain
is anchored at $\zeta$, which is both moving and precisely the quantity we
do not wish to assume known; worse, the reverse bisector link also moves
relative to its own anchor. The resolution is that the reverse chain's
joint $m$ sits at $\Sigma_1+R$: relative to the forward joint $\Sigma_1$ it
is displaced by exactly $R$, and $R$ has its own integral representation
\eqref{eq:R-T}, independent of any evaluation of $\zeta$. Applying the
frame map \eqref{eq:frame-map} to the two ends of the reverse bisector link
gives exactly $\mathrm{Yin}(s)=R\,M^{s}$ and
$\mathrm{Yang}(s)=\mathrm{Yin}-\chi M^{2s-1}$ of
\eqref{eq:yin1}--\eqref{eq:yang1}. Because the frame map scales the
bisector link to unit length, positions along it are \emph{fractions}: the
crossing we are about to compute is the fraction of the way along the link
at which the bisector point sits, and multiplying that real fraction by the
(unscaled) link vector recovers the fractional summand itself. The actual
distance from the joint, the fraction times the true link length
$M^{-\sigma}$, is the weight $d_1$.

\paragraph{First, the modulus.} From its joint, how far along the forward
bisector link do the two links intersect? This distance is $|R_{1}|$,
because $R_{1}$ is the fractional part of summand $M$ of
$\Sigma_1$, and that summand \emph{is} the forward bisector link. In the
frame the link lies along the real axis, so the intersection is the point
where the segment from $\mathrm{Yin}$ to $\mathrm{Yang}$ crosses the axis. By
the standard segment--axis intersection formula that crossing is
\begin{equation}\label{eq:crossing}
\mathrm{Yin}
\;-\;
\frac{\Im\,\mathrm{Yin}}
     {\Im\,\mathrm{Yang}-\Im\,\mathrm{Yin}}\,
\bigl(\mathrm{Yang}-\mathrm{Yin}\bigr),
\end{equation}
a real number. Substituting \eqref{eq:yin1}--\eqref{eq:yang1} and canceling
$\mathrm{Yang}-\mathrm{Yin}=-\chi M^{2s-1}$ from the ratio,
\begin{equation}\label{eq:yang3lin}
\eqref{eq:crossing}
\;=\;
R\,M^{s}
\;-\;
\frac{\Im\!\bigl(R\,M^{s}\bigr)}
     {\Im\!\bigl(\chi\,M^{2s-1}\bigr)}\,
\chi\,M^{2s-1}.
\end{equation}
Writing $M^{s}=M^{\sigma}e^{i\omega}$, expanding the imaginary parts in
components, and finally writing $\varphi=\arg R$ and $\psi=\arg\chi$ and
collapsing with the sine difference identity
$\sin(2\omega+\psi)\cos(\omega+\varphi)-\cos(2\omega+\psi)\sin(\omega+\varphi)
=\sin(\omega-\varphi+\psi)$, the whole expression reduces to
\begin{equation}\label{eq:frac1}
\eqref{eq:crossing}
\;=\;
M^{\sigma}\,|R|\,
\frac{\sin\!\bigl(\omega-\varphi+\psi\bigr)}{\sin\!\bigl(2\omega+\psi\bigr)}
\;=\; M^{\sigma}\, d_1 .
\end{equation}
The reader will recognize the sine quotient: it is, symbol for symbol, the
definition \eqref{eq:d1} of the weight $d_1$. \emph{This is where that
definition came from.} The crossing \eqref{eq:frac1} is the crossing
fraction (the ``fractional amount'' column of Table~\ref{tab:frac})
along the unit link at which the reverse link crosses, and dividing out the frame's
scaling $M^{\sigma}$ leaves the actual distance from the joint,
$d_1$. (Proposition~\ref{prop:yinyang-chord-finite} is the same identity
run in the other direction: there the definitions are taken as given and
the chord is proved to cross at $\hat d_1=M^{\sigma}d_1$.)

\paragraph{Second, the direction.} The forward bisector link points along the
next summand of $\Sigma_1$: extending the summation cutoff by one adds the
single term $M^{-s}$, and that term is the link vector. Multiplying the
fraction \eqref{eq:frac1} by it undoes the frame map and yields the
fractional summand itself:
\begin{equation}\label{eq:R1-derived}
R_{1}
\;=\;
\bigl(M^{\sigma} d_1\bigr)\cdot M^{-s}
\;=\;
d_1\,M^{-it},
\end{equation}
which is definition \eqref{eq:R1-def}.

\paragraph{The mirror frame gives \texorpdfstring{$d_2$ and $R_{2}$}{d2 and R2}.}
The same procedure in the mirror frame of \S\ref{sec:yinyang-general}
(the footnote there), where the \emph{reverse} bisector link is sent to
$[0,1]$ and the forward one revolves, locates the crossing at the fraction
\begin{equation}\label{eq:frac2}
\frac{M^{1-\sigma}}{|\chi|}\,
|R|\,\frac{\sin\!\bigl(\omega+\varphi\bigr)}{\sin\!\bigl(2\omega+\psi\bigr)}
\;=\;
\frac{M^{1-\sigma}}{|\chi|}\, d_2
\end{equation}
of the way along the unit reverse link, again the ``fractional amount''
column of Table~\ref{tab:frac}, and again the sine quotient is, symbol for
symbol, the definition \eqref{eq:d2} of $d_2$: this is where \emph{it} came
from. (The two normalized fractions are nearly symmetric: since
$|\chi(\sigma+it)|=(t/2\pi)^{\frac12-\sigma}\bigl(1+O(1/t)\bigr)$ and
$\sqrt{t/2\pi}\approx T+\tfrac12$, the prefactor
$M^{1-\sigma}/|\chi|$ is asymptotically
$M^{\sigma}$, matching the forward side.) Multiplying by the
reverse link vector $\chi\,M^{s-1}$ undoes the frame map:
\begin{equation}\label{eq:R2-derived}
R_{2}
\;=\;
\frac{M^{1-\sigma}}{|\chi|}\,d_2
\cdot \chi\,M^{s-1}
\;=\;
d_2\,M^{\,it}\,\frac{\chi}{|\chi|},
\end{equation}
which is definition \eqref{eq:R2-def}.

\begin{figure}[H]
\centering
\includegraphics[width=0.98\textwidth]{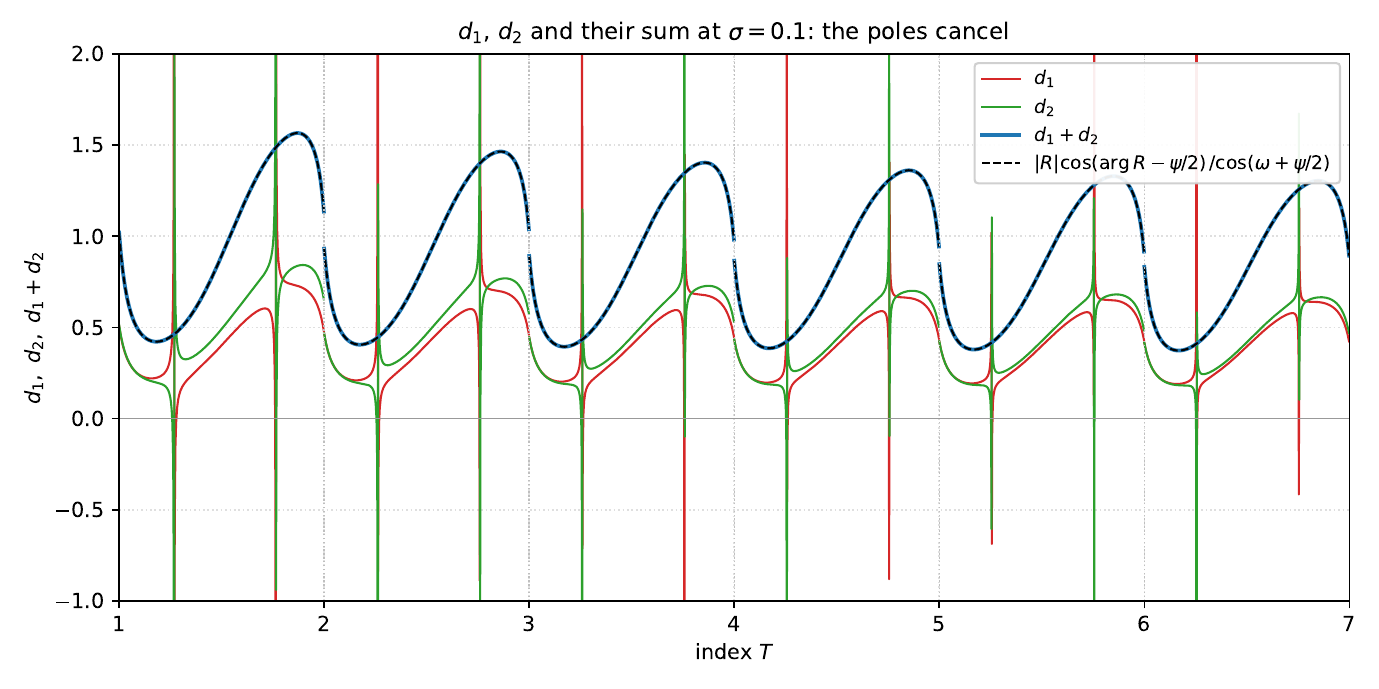}
\caption{$d_1$ (red), $d_2$ (green) and their sum (blue) at $\sigma=0.1$,
$1\le T\le 7$. Twice per unit of $T$ the individual weights blow up at the
parallel-link poles of Appendix~\ref{sec:pole-locations}, but the poles are
equal and opposite and the sum passes smoothly through every one of them; the
black dashed curve is the closed form \eqref{eq:d1-plus-d2}, indistinguishable
from the blue sum (they agree to machine precision on the plotted grid).
Generated by \texttt{fig\_d1\_d2\_sum.py}.}
\label{fig:d1-d2-sum}
\end{figure}

\paragraph{The poles cancel in the sum.} The two weights share the
denominator $\sin(2\omega+\psi)$, and Appendix~\ref{sec:pole-locations}
locates the isolated heights at which it produces genuine poles. In the
\emph{sum} $d_1+d_2$ those poles cancel: adding the numerators with the
sum-to-product identity
$\sin(\omega-\varphi+\psi)+\sin(\omega+\varphi)
=2\sin\bigl(\omega+\tfrac{\psi}{2}\bigr)\cos\bigl(\varphi-\tfrac{\psi}{2}\bigr)$
and writing
$\sin(2\omega+\psi)=2\sin\bigl(\omega+\tfrac{\psi}{2}\bigr)
\cos\bigl(\omega+\tfrac{\psi}{2}\bigr)$, the offending factor divides out:
\begin{equation}\label{eq:d1-plus-d2}
d_1+d_2
\;=\;
|R|\;
\frac{\cos\!\bigl(\varphi-\tfrac{\psi}{2}\bigr)}
     {\cos\!\bigl(\omega+\tfrac{\psi}{2}\bigr)} .
\end{equation}
The zeros of the remaining denominator $\cos(\omega+\tfrac{\psi}{2})$ are
matched by zeros of the numerator (at those instants $R$ is parallel to the
common link direction, forcing $\varphi\equiv-\omega\pmod\pi$), so every
singularity of \eqref{eq:d1-plus-d2} is removable: the sum is smooth across
all the poles of its two terms, which is the closed form behind the
qualitative cancellation remark following Table~\ref{tab:poles}.
Figure~\ref{fig:d1-d2-sum} shows this at $\sigma=0.1$.

This is how the definitions of \S\ref{sec:decomp} were found: the sine
expressions \eqref{eq:d1}--\eqref{eq:d2} are nothing other than the real-axis
crossings of the yang segments, read off in the two bisector frames. Only
afterwards was Theorem~\ref{thm:decomp} proved, confirming that the two
fractional summands so constructed add back to Siegel's $R$ exactly. Note the
economy of the construction: the bisector point, $d_1$, $d_2$, and both
fractional summands are obtained from $R$ and $\chi$ alone. Numerically we
computed $R$ to high accuracy by Kuznetsov's method (\S\ref{sec:kuznetsov}),
which is how the pictures, and then the identity, were checked.

\section{Conjecture: There is a Yin, Yang, \texorpdfstring{$d_1$, $d_2$}{d1, d2} for every link}\label{sec:general-yinyang}

\S\ref{sec:yinyang-origin} watched one pair of links: the reverse bisector link revolving
about the forward bisector link, its ends tracing the yin and yang curves, and
the weights $d_1,d_2$ as the real-axis crossings of that pair. \emph{We
conjecture that the same picture exists at every link of the chain.} On the
critical line which reverse link crosses forward link $k$ is a function of
$T$ alone. Off the line write $C_{\ell}(k;\sigma,T)$, since the segment
intersection may depend on $\sigma$. Once that partner is named, the yin and yang
formulas, the weights, and the remainder split are the same constructions with
that pair in place of the bisector pair. The seven subsections record, in order:
\begin{enumerate}
\item first we need to know which main sum links cross which dual sum links
(\S\ref{sec:links-crossing});
\item then we find the yin and yang curves for each crossing pair
(\S\ref{sec:yinyang-any});
\item then we use those yin-yang formulas to find the weights $d_1$ and
$d_2$ for each link (\S\ref{sec:d1-any});
\item then we use those weights to make the Dirichlet sums $\Sigma_{1x}$
and $\Sigma_{2x}$, whose joints sit at the crossings (\S\ref{sec:sum-x});
\item then we use the crossing points to rearrange the main sum into the
two free-vector sums of the half-links (\S\ref{sec:half-link-sums});
\item then we use the crossing pairing to generalize the partial-summand
formula for $\zeta$ to unequal cutoffs, whose two extreme cuts are example
formulas built on the first link of either chain
(\S\ref{sec:cut-any-link});
\item and finally we list what in this section rests on numerics and is
left unproved (\S\ref{sec:unproved}).
\end{enumerate}

\subsection{Links crossing}\label{sec:links-crossing}

\S\ref{sec:bisector} named one crossing of the two link chains: the forward
bisector link meets the reverse bisector link, and where they meet is the
bisector point $B_1$. All the other links of the forward chain are crossed by some
link of the reverse chain as well, and which one it is depends on $T$. Watching
the link at $m-1$ as $T$, for example, runs from $6$ to $7$, one reverse link
crosses it, and that link is joined by its neighbor that subsequently also
crosses over it, and by the end of the unit interval the crossing links have
advanced by two.

\begin{figure}[H]
\centering
\includegraphics[width=\textwidth]{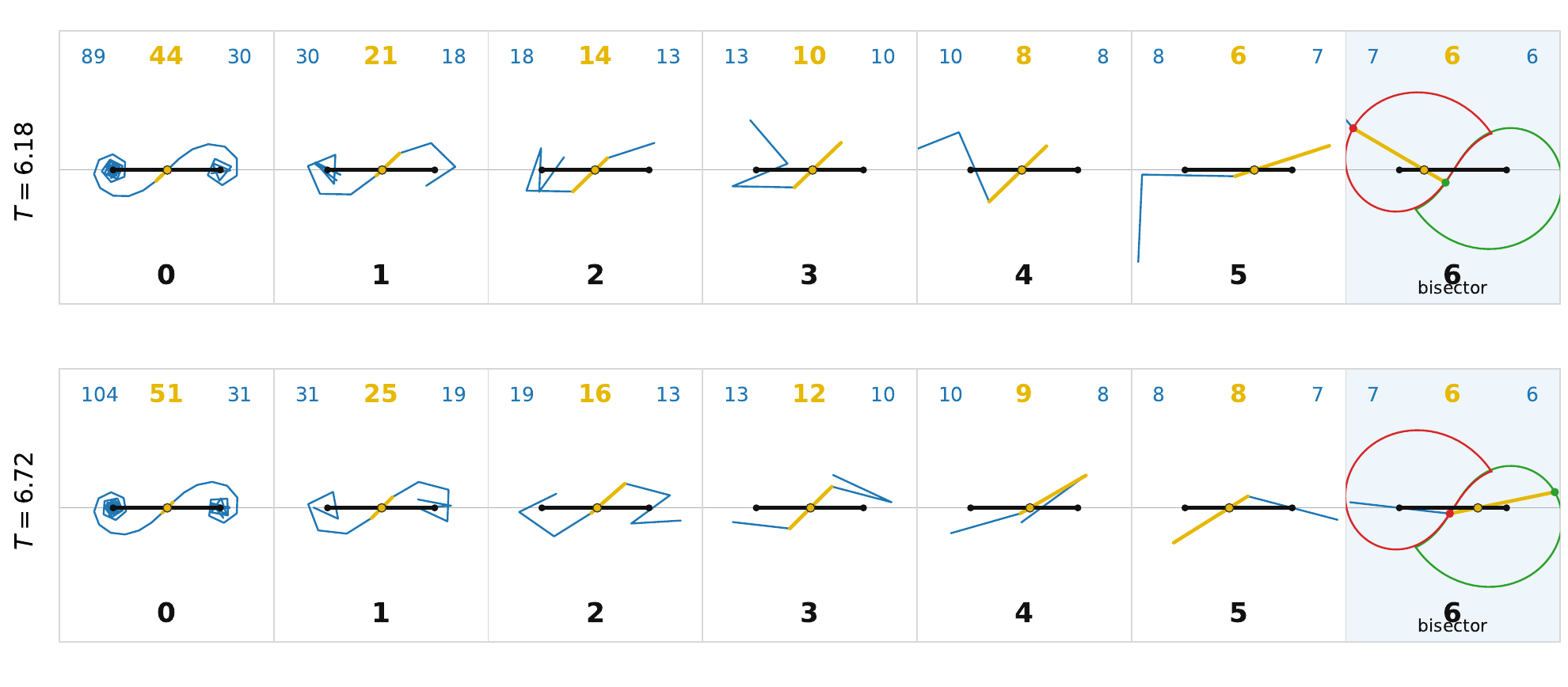}
\caption{Local frames of the forward links at $\sigma=\tfrac12$, $T=6.18$
(top) and $T=6.72$ (bottom), the running example of
Figure~\ref{fig:spiral-summands} and a later height in the same unit
($t=I(T)\approx279.85$ and $327.01$, $m=6$). Each panel is the local coordinate frame for that forward link, sending
one forward link to the unit interval along the $x$-axis;
the number at the bottom is that link. The reflected
reverse chain is drawn in the same frame: the blue links are the part
of it that belongs to the particular forward link, numbered at the two ends with the
outer (higher) link on the left, and the yellow
link is the single reverse link that crosses the forward link, numbered in yellow,
with a yellow dot at the crossing. The rightmost local frame is the bisector; the green and red
curves are the yin and yang loci of the two ends of the reverse
bisector link as $T$ runs through $(6,7)$. Generated by
\texttt{fig\_link\_frames.py}.}
\label{fig:link-frames}
\end{figure}

Numbering summands from one, so that link $k$
carries the summand $n=k+1$ (\S\ref{sec:matrix-product}), forward link $k$ and
reverse link $j$ cross when their two summands are reciprocal about the
Riemann--Siegel cutoff $a=\sqrt{t/2\pi}$ of Appendix~\ref{sec:phase-bound}, as
nearly as two integers can be. Under $t=I(T)$ that cutoff is a function of the
index alone, by \eqref{eq:IT-derived}, so the rule closes in $T$:
\begin{equation}\label{eq:crossing-product}
n\,n'\;=\;a^{2}\;=\;\frac{t}{2\pi}\;=\;\frac{I(T)}{2\pi}
\;=\;\frac{2T+1}{2\log\!\bigl(1+\tfrac1T\bigr)},
\qquad n=k+1.
\end{equation}
Nothing is approximated in $a^{2}$: it is a closed form in $T$ rather than an
asymptotic. But it is not an integer, so $n'=a^{2}/n$ is the real partner, and
the reverse link that actually crosses carries the nearest integer to it. The
product of the two integers therefore misses $a^{2}$, by up to $n/2$; the
rounding is unbiased, its signed mean $-9\times10^{-4}$ over the sample below,
with a typical miss of a quarter of a unit. Write
\begin{equation}\label{eq:crossing-named}
j_0(k,T)\;=\;\Bigl[\frac{a^{2}}{k+1}\Bigr]-1
\;=\;\left[\frac{2T+1}{2(k+1)\log\!\bigl(1+\tfrac1T\bigr)}\right]-1,
\end{equation}
with $[\,\cdot\,]$ the nearest integer. Nearest-integer alone misses at two
families of half-integers. One is the offset
\eqref{eq:siegel-cutoff-index} already recorded: Siegel's cutoff $a$ sits
half a link past $T$, so it steps near the half-integers of $T$ while the
chains themselves step at the integers. The other is the joints of the
hyperbola $nn'=a^{2}$ itself. The \emph{crossing link} is the
rule that names the reverse link integer and removes those misses,
\begin{equation}\label{eq:crossing-rule}
C_{\ell}(k,T)\;=\;
\begin{cases}
m & \text{if }k=m,\\[2pt]
j_0(k,T) & \text{if reverse link }j_0\text{ crosses forward link }k,\\[2pt]
j_0(k,T)-1 & \text{otherwise,}
\end{cases}
\end{equation}
and the crossing summand is $C_{s}(n,T)=C_{\ell}(n-1,T)+1$. The first line
is the self-pair: $a^{2}/(m+1)$ is near a half-integer when $T$ is near an
integer, so $j_0$ names $m\mp1$ at the ends of the unit index, $T$, while the
two bisector links cross at every height. The third line
is the other family, the half-integers of the product $a^{2}/n$ itself:
when $\{a^{2}/n\}\approx\tfrac12$ the geometric crossing has already moved
to the lower neighbor. Replacing $a^{2}$ by $T(T+1)$ or $(T+\tfrac12)^{2}$
does not remove either family; the product is already the exact $I(T)$,
and it only moves the joints. The observed pairing suggests the unproved
involutivity conjecture
\[
C_s(C_s(n,T),T)=n.
\]
At $T=6.18$, where $a^{2}=44.5389$, \eqref{eq:crossing-named} reads
$44,21,14,10,8,6,5$ for $k=0,\dots,6$, and all seven are real crossings
(Figure~\ref{fig:links-crossing}). The sum rule $k+j\approx2T-1$ that the
links beside the bisector obey is this hyperbola's tangent there, and holds
only there.

Two remarks on $a$, since it is easy to read it as a link number. It is not one:
it is a real height, the self-dual summand $n=n'=a$ the law pairs about, the
one at which the forward chain hands over to the reflected one, and by
\eqref{eq:siegel-cutoff-index} it sits half a link past the bisector link, at $a^{2}=(T+\tfrac12)^{2}-\tfrac1{12}+O(T^{-2})$. Nor is it
a fixed number over a unit interval, as $\lfloor T\rfloor$ is: $a^{2}$ climbs
from $42.17$ to $56.17$ as $T$ runs from $6$ to $7$, and it is that climb that
walks each link's partner along inside the interval. Replacing $a^{2}$ by
$\lfloor T\rfloor^{2}=36$ or $(\lfloor T\rfloor+1)^{2}=49$ at $T=6.18$ would
name reverse link $35$ or $48$ as crossing the unit step, and neither does.

\paragraph{The reverse chain is the forward chain reflected in the bisector line.} Number the joints
of both chains from zero as in \S\ref{sec:product-form},
\begin{equation}\label{eq:crossing-joints}
J_n=\sum_{n'\leq n}n'^{-s},
\qquad
K_j=\zeta-\chi(s)\sum_{n\leq j}n^{\,s-1},
\end{equation}
so that $J_0=0$, $K_0=\zeta$, and $K_m=\Sigma_1+R$ is the near end of the
reverse bisector link; link $n$ of either chain runs from its joint $n$ to its
joint $n+1$. On $\sigma=\tfrac12$ we have $n^{\,s-1}=\overline{n^{-s}}$, so
$K_j=\zeta-\chi\overline{J_j}$, and with $\chi=e^{-2i\vartheta}$ and
$Z=e^{i\vartheta}\zeta$ real (Hardy's $Z$, as in \eqref{eq:r-two-cases}),
\begin{equation}\label{eq:chain-mirror}
e^{i\vartheta}K_j \;=\; Z-\overline{e^{i\vartheta}J_j}\,.
\end{equation}
Write $e^{i\vartheta}J_n=X_n+iY_n$. Then \eqref{eq:chain-mirror} puts reverse
joint $j$ at $(Z-X_j,\,Y_j)$: in the $\vartheta$-rotated frame the reverse chain
is the forward chain reflected in the vertical line $X=Z/2$, joint for joint and
link for link. That line is the bisector line of \S\ref{sec:legs}: the
rotation puts $\zeta$ at $Z$ on the real axis, so unrotated it is the
perpendicular bisector of the base running from the origin to $\zeta$, drawn
dotted in
Figures~\ref{fig:full-spirals} and~\ref{fig:legs} through $\zeta/2$ and the
bisector point $B_1$ (Corollary~\ref{cor:bisector-proj}). Doubling-back and the
bisector line are
two different things in the same neighborhood, and worth keeping apart: the
place along the chain is the bisector link, where the chain doubles
back; the bisector line is a line across the plane, the one the two chains
reflect in. The bisector point is where the two meet, the point inside the
bisector link at which the forward chain crosses the bisector line, as
\eqref{eq:crossing-rstar} works out below. This is the
chain-level form of the reflection
$B_2=\chi\overline{B_1}$ behind the equal legs of \S\ref{sec:legs}, and like it,
it is a statement about the critical line only: off the line $|\chi|\neq1$ and
the reverse chain is a scaled copy of the $1-\sigma$ chain instead of a
reflection of this one.

\paragraph{Crossings come in pairs.} Parameterize each chain continuously, the
value $\mu=k+p$ with $0\leq p<1$ meaning the point $p$ of the way along link
$k$, and let $X(\mu),Y(\mu)$ be the rotated coordinates of the forward chain
there and
\begin{equation}\label{eq:crossing-h}
h(\mu)\;=\;X(\mu)-Z/2
\end{equation}
its signed distance to the bisector line. By \eqref{eq:chain-mirror} the reverse
chain at $\nu$ sits at $(Z-X(\nu),Y(\nu))$, so the forward chain at $\mu$ meets
the reverse chain at $\nu$ exactly when
\begin{equation}\label{eq:crossing-conditions}
h(\mu)+h(\nu)=0
\qquad\text{and}\qquad
Y(\mu)=Y(\nu).
\end{equation}
Both conditions are symmetric in $\mu$ and $\nu$, which is why crossings come in
pairs: reverse link $j$ crosses forward link $k$ if and only if reverse link $k$
crosses forward link $j$, and any rule naming the partner must be symmetric
too, as \eqref{eq:crossing-product} is. The pair coincides only when $\mu=\nu$,
and then $h=0$; so a self-paired crossing lies on the bisector line itself, there
is exactly one (the bisector point $B_1$, inside the bisector link of each
chain by Theorem~\ref{thm:d1-bounds}), and every other crossing sits off the
line with its counterpart on the other side.

\paragraph{Sweeps between spiral centers.} Along either chain the turn from one
link to the next is $t\ln(1+1/n)\approx t/n$, decreasing as the chain is walked,
so the joints where that turn is an odd multiple of $\pi$ lie near the
spiral centers \eqref{eq:LN} (the midpoint of the link they bound is closer
still), and between two consecutive centers the chain makes
one full turn.
The middle of a turn is the link where the angle is the even multiple $2\pi c$,
\begin{equation}\label{eq:crossing-sweep}
\frac{t}{n'}=2\pi c
\qquad\Longleftrightarrow\qquad
n'=\frac{t}{2\pi c}=\frac{a^{2}}{c} ,
\end{equation}
and there successive links are parallel rather than antiparallel: their
contributions add in phase, the chain stops circling and runs. Call that run
the $c$-th \emph{sweep}; it carries the chain from the spiral center $S_n=c$ to
the spiral center $S_n=c-1$.

Heuristically, after inserting a smooth compact cutoff, Poisson summation
separates the chain into oscillatory integrals. Stationary phase associates
the $c$-th integral with the saddle $x=a^2/c$ and gives the displayed vector
only as a leading asymptotic:
\begin{equation}\label{eq:crossing-saddle}
\frac{1}{\sqrt c}\,
\exp i\Bigl(\vartheta+t-t\ln\frac{a^{2}}{c}+\frac\pi4\Bigr)
\;=\;\frac{1}{\sqrt c}\,e^{\,i\left(2\Theta-(\vartheta-t\ln c)\right)},
\qquad
2\Theta=2\vartheta+t-t\ln a^{2}+\frac\pi4 .
\end{equation}
No exact Poisson identity is asserted for the unwindowed critical-line
series.
Its modulus is that of the $c$-th summand, $c^{-1/2}$, and its phase is the
$c$-th summand's phase $\vartheta-t\ln c$ reflected in the fixed direction
$\Theta$. Sweep by sweep, the tail of a chain reproduces the head of that chain
reflected: this is the correspondence between joints and spiral centers
recorded in \S\ref{sec:bisector}, Nickel's center-to-center distances and
angles reproducing the initial steps, now carrying an index; the $c$-th sweep
answers to the $c$-th summand.

The reflection \eqref{eq:chain-mirror} undoes the
reflection in $\Theta$. In the frame of the forward chain, therefore, the reverse chain's $c$-th
sweep is the forward chain's own link $c-1$ traversed backwards: over the
sampled heights the center-to-center step and that link agree to a median
$1.4\%$, $2.7\%$, $3.1\%$ and $5.2\%$ of a link length for $c=1,2,3,4$. A copy
of a link laid back along it must cross it, and the quarter turn in
\eqref{eq:crossing-saddle} is what tilts the copy so that it cuts across
instead of lying flat: the two meet near the middle of the link, at some
$43^{\circ}$. Which reverse link carries the $c$-th sweep is
\eqref{eq:crossing-sweep} with $c=n$, and that is
\eqref{eq:crossing-product}. Panel (b) of Figure~\ref{fig:links-crossing} is
the case $c=1$: the whole outermost turn of the reverse chain, links $30$ to
$89$, laid back along the unit step and crossing it at link $44$.

\paragraph{The angle at a crossing.} How the two links lie relative to one
another needs no estimate at all. It is forced, and what forces it is that the
two summands multiply to a whole number.
In the rotated frame forward link $k$ points
along $e^{\,i(\vartheta-t\log n)}$ with $n=k+1$, and by \eqref{eq:chain-mirror}
reverse link $j$ points along $-e^{-i(\vartheta-t\log n')}$ with $n'=j+1$, so
the angle from the first to the second is
\begin{equation}\label{eq:crossing-angle}
\theta(n,n')\;=\;\pi-2\vartheta(t)+t\log(n n')\pmod{2\pi},
\end{equation}
exactly, whether or not the two links meet.

\begin{proposition}\label{prop:crossing-angle}
Fix $C>0$. On $\sigma=\tfrac12$, for positive integers $n,n'$ satisfying
$|nn'-a^2|\leq Ca$, the displayed expansion holds with error $O_C(1/a)$:
\begin{equation}\label{eq:crossing-angle-reduced}
\theta(n,n')
\;\equiv\;
\frac{5\pi}{4}\;-\;\frac{\pi\,(nn'-a^{2})^{2}}{a^{2}}
\;+\;O_C\!\left(\frac1a\right)
\pmod{2\pi}.
\end{equation}
In particular, if $|nn'-a^{2}|\leq a/2$ the two links are never parallel: they
meet at an angle between $\tfrac34\pi$ and $\pi$, that is between $135^{\circ}$
and $180^{\circ}$, at most a quarter turn away from antiparallel.
\end{proposition}

The proof depends on $N=nn'$ being an \emph{integer}. Write $N=a^{2}(1+\epsilon)$.
Since $t=2\pi a^{2}$,
\begin{equation}\label{eq:crossing-angle-integrality}
t\epsilon\;=\;\frac{t\,(N-a^{2})}{a^{2}}\;=\;2\pi N-t\;\equiv\;-t\pmod{2\pi},
\end{equation}
so in
$t\log N=t\log a^{2}+t\epsilon-\tfrac12t\epsilon^{2}+\tfrac13t\epsilon^{3}-\cdots$
the first-order term is $-t$ modulo $2\pi$, whatever $N$ is. The Stirling
expansion of the Riemann--Siegel theta function $\vartheta(t)$
(Appendix~\ref{sec:phase-bound}) is
$2\vartheta(t)=t\log a^{2}-t-\tfrac{\pi}{4}+\tfrac{1}{24t}+O(t^{-3})$, so in
\eqref{eq:crossing-angle} the whole of $t\log a^{2}-t$ cancels against it and
the $\tfrac\pi4$ survives:
\[
\theta\;\equiv\;\pi+\frac{\pi}{4}-\frac{t\epsilon^{2}}{2}
+O\bigl(t|\epsilon|^{3}\bigr)+O\bigl(t^{-1}\bigr).
\]
Here $\tfrac12t\epsilon^{2}=\pi(N-a^{2})^{2}/a^{2}$. Under the proposition's
hypothesis, $|\epsilon|\leq C/a$, so
$t|\epsilon|^3=O_C(1/a)$, which proves
\eqref{eq:crossing-angle-reduced}. When $|N-a^{2}|\leq a/2$ we have
$|\epsilon|\leq1/2a$, so the cubic term is
at most $\pi/12a$ and the quadratic one at most $\pi/4$. Read as an unsigned
angle, $\tfrac54\pi$ is $\tfrac34\pi$, and the correction can only carry it
towards $\pi$.

The naming \eqref{eq:crossing-named} stays inside that hypothesis: it makes
$|nn'-a^{2}|=n\,|n'-a^{2}/n|\leq n/2$, and $n\leq m<a$ for every forward link
below the bisector, so the $135^{\circ}$ is a floor and not an average: no
link-crossing that the rule names can be sharper than that. Only the
self-pair, being the bisector links, handled by the first line of
\eqref{eq:crossing-rule}, where the rule overrides the rounding, can fall
outside it. The quarter turn is the
$-\tfrac\pi8$ of $\vartheta$, which is the Gaussian factor of
\eqref{eq:crossing-saddle} over again: the $43^{\circ}$ by which a crossing
misses antiparallel, the tilt that keeps the sweep from lying flat along the
link it reproduces, is the stationary-phase quarter turn, read off exactly
rather than measured. What the proposition fixes is the direction of the
partner and not its position; that the two segments actually meet is
verified numerically below but is not proved here.

\paragraph{Near the bisector, a sum rule.} Near the self-dual link $n=n'=a$ the
hyperbola \eqref{eq:crossing-product} is indistinguishable from its tangent.
Writing $n=a+x$,
\begin{equation}\label{eq:crossing-tangent}
n'=\frac{a^{2}}{n}=a-x+\frac{x^{2}}{a}-\cdots ,
\qquad\text{so}\qquad
k+j\;=\;n+n'-2\;=\;2a-2+\frac{x^{2}}{a}-\cdots ,
\end{equation}
and since $2a=2T+1-\tfrac1{12T}+O(T^{-2})$ by
\eqref{eq:siegel-cutoff-index}, the sum is $2T-1$ up to the quadratic term and
a drift of order $1/12T$: $k+j$ is one of the two integers straddling $2T-1$,
that is
\begin{equation}\label{eq:crossing-sum}
j\;=\;\lfloor2T\rfloor-1-k
\qquad\text{or}\qquad
j\;=\;\lfloor2T\rfloor-k ,
\end{equation}
a rule needing nothing computed at all. The same rule comes out of the reflection
\eqref{eq:chain-mirror}, which is worth doing because it identifies the fixed
point. Near the bisector the
forward chain sweeps across the bisector line and $h$ passes through a single
zero there, close to straight over the few links either side; writing $r^{*}$
for that zero and linearizing turns the first condition of
\eqref{eq:crossing-conditions} into $\mu+\nu=2r^{*}$, the two arcs straddling
the crossing symmetrically, so $k+j$ straddles $S=2r^{*}-1$. And the crossing
fraction is not a new quantity.
By \eqref{eq:r-two-cases}, $h(m)=X_m-Z/2=-r/2$ with $r=e^{i\vartheta}R$ the real
remainder of Appendix~\ref{sec:phase-bound}, and link $m$ of the forward chain,
being $(m+1)^{-s}$ rotated by $e^{i\vartheta}$, has horizontal extent
$X_{m+1}-X_m=\cos(\omega-\vartheta)/\sqrt{m+1}$. So the zero falls inside link
$m$, at the crossing fraction
\begin{equation}\label{eq:crossing-rstar}
r^{*}=m+\lambda,
\qquad
\lambda=\frac{Z/2-X_m}{X_{m+1}-X_m}
=\sqrt{m+1}\;\frac{r}{2\cos(\omega-\vartheta)}
=M^{1/2}d_1=\hat d_1,
\end{equation}
by the exact reduction \eqref{eq:exact-line} and the normalization
\eqref{eq:frac-vs-weight}. The place where the forward chain crosses the bisector
line is therefore the bisector point $B_1$ itself, sitting at the crossing
fraction $\hat d_1$ along the bisector link exactly as
Figure~\ref{fig:bisector-frame} marks it. So $S=2r^{*}-1=2m+2\hat d_1-1$,
against $2m+2\{T\}-1$ from the tangent. These are not the same number: the limit
of $\hat d_1$ is the waveform $d(x)$ of \S\ref{sec:periodicity}, not $x$ itself,
and the two part by nearly $\tfrac12$ at the ends of the unit interval. But the
straddling pair depends only on which side of $\tfrac12$ each falls, and on that
the two agree. The branch form \eqref{eq:d-profile} gives $d(x)+d(1-x)=1$, and
on $(0,\tfrac12)$ the curve stays below $\tfrac12$: there
$d=\tfrac12+\mathcal W_{\infty}$ with
$\mathcal W_{\infty}=-\tfrac12\tan(2\pi x)\tan\beta$ \eqref{eq:W-infty}, and the
two tangents carry the same sign on each of $(0,\tfrac14)$ and
$(\tfrac14,\tfrac12)$, both positive on the first and both negative on the
second. So $d$ crosses
$\tfrac12$ exactly where $x$ does, and the $O(1/T)$ of \S\ref{sec:periodicity}
can disturb this only within $O(1/T)$ of $\{T\}=0,\tfrac12,1$, which the sample
below does not see it do. With $\tfrac15<\hat d_1<\tfrac45$ from
Theorem~\ref{thm:d1-bounds} holding $S$ within $\tfrac35$ of $2m$, the
straddling integers are therefore $\lfloor2T\rfloor-1$ and $\lfloor2T\rfloor$
whichever side of $\tfrac12$ the fractional part is on, and
\eqref{eq:crossing-sum} follows either way: the tangent to the hyperbola and the
linearization of $h$ are the same line.

How far the tangent can be trusted is \eqref{eq:crossing-tangent} again: it
costs half a link once $x^{2}/a=\tfrac12$, so the sum rule is good only within
\begin{equation}\label{eq:crossing-reach}
|n-a|\;\lesssim\;\sqrt{a/2}\;\approx\;\sqrt{T/2}
\end{equation}
of the bisector (three links either side at $T=17$, fourteen at $T=400$, and just
under two at the $T=6.18$ of the figure), which is why $k=4,5,6$ obey it there
and $k=3$ already does not. Inside that range the two indices trade off one for
one: each step back along the forward chain moves the partner one step forward
along the reverse chain, and over one unit of $T$ the sum moves by two, so each
partner advances by two, in two single steps, one when $\{T\}$ passes
$\tfrac12$ and one when $T$ passes an integer, the two heights at which
$\lfloor 2T\rfloor$ jumps. Outside the range the product takes over and the partners run
away as $a^{2}/n$, reaching the far end of the chain at the unit step.

The bisector link is the one place where the pairing is exact rather than
asymptotic, and that is the first line of \eqref{eq:crossing-rule}. It is
the self-dual link, $n=m+1$ and $x=\tfrac12-\{T\}$, so
\eqref{eq:crossing-named} rounds to $m$ only over the middle half of the unit
interval and to $m\mp1$ at its ends; but the two bisector links still cross, and the
self-paired crossing $B_1$ of \eqref{eq:crossing-conditions} is there at every
height. The bisector links always cross, as \S\ref{sec:bisector} has it, and
the links $j_0$ names near the ends of the interval are
second crossings of the same forward link, not replacements for it.

\begin{figure}[H]
\centering
\includegraphics[width=\textwidth]{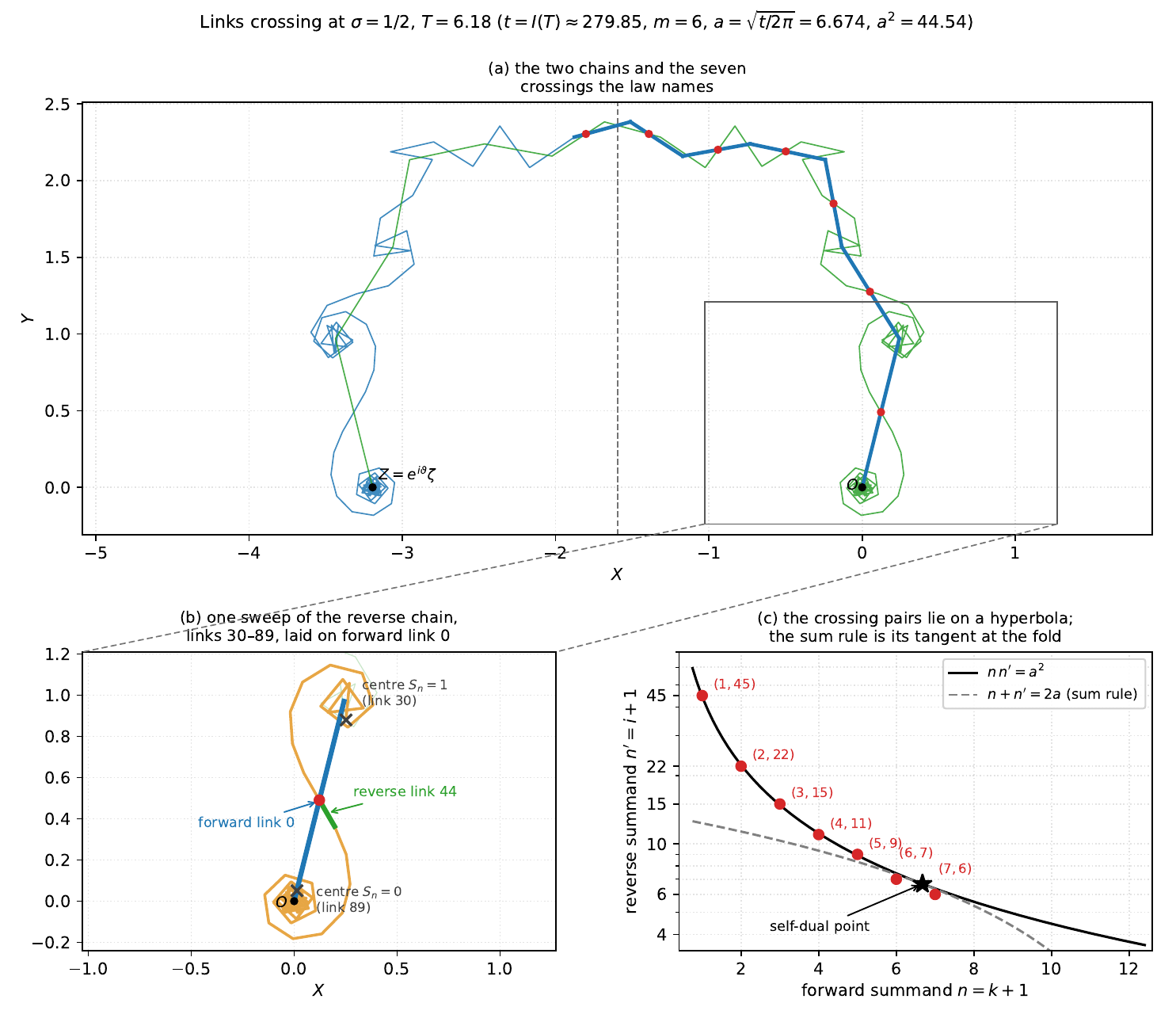}
\caption{Links crossing at $\sigma=\tfrac12$, $T=6.18$, the running example of
Figure~\ref{fig:spiral-summands} ($t=I(T)\approx279.85$, $m=6$,
$a=\sqrt{t/2\pi}=6.674$, $a^{2}=44.54$; the first turn of each chain holds
$90$ links).
(a) Both chains in the $\vartheta$-rotated frame, forward $e^{i\vartheta}J_n$ in
blue from $O$ and reverse $e^{i\vartheta}K_j$ in green from $Z=e^{i\vartheta}\zeta$,
which by \eqref{eq:chain-mirror} is the blue chain reflected in the dashed
bisector line $X=Z/2$. Forward links $0$ through $6$ are drawn heavy and the red dots
are the seven crossings \eqref{eq:crossing-named} names for them,
$j=44,21,14,10,8,6,5$; they run the length of the chain, not just the
neighborhood of the bisector. (b) The boxed region enlarged, showing why. The gold
arc is one turn of the reverse chain, links $30$ to $89$, running from the
spiral center $S_n=1$ to the spiral center $S_n=0$ (crosses, at the links
\eqref{eq:LN} where the turn angle is an odd multiple of $\pi$). By
\eqref{eq:crossing-saddle} that step reproduces forward link $0$ reversed, so
it lies back along the unit step and crosses it near the middle (red dot); the
link carrying the sweep is $44$, the middle of the turn, where the turn angle
is $2\pi$ and $n'=a^{2}/1$. (c) The pairs $(n,n')$ of all seven crossings, on
the hyperbola $nn'=a^{2}$. The self-dual point $n=n'=a$ (star) is the summand
the law pairs about, half a unit of the index past $T$ by
\eqref{eq:siegel-cutoff-index}. These are summand axes, and the star is
plotted here at $(a,a)$ with $a=6.674$, while the bisector link is link $m=6$,
which carries summand $n=7$: the star therefore sits between summands $6$ and
$7$, a third of a summand short of the bisector link's own value rather than
past it.
Near it the hyperbola is indistinguishable from its tangent $n+n'=2a$ (dashed),
which is the sum rule \eqref{eq:crossing-sum}, and away from it the two part
company fast: the tangent gives the unit step a partner near $n'=11$ where the
truth is $45$. Generated by \texttt{fig\_links\_crossing.py}.}
\label{fig:links-crossing}
\end{figure}

\paragraph{How well it holds.} Sampling $m\in\{8,17,40,123,400\}$ at $199$
fractional parts each and testing every forward link $0\leq k\leq m$,
$118{,}007$ cases in all, the link named by \eqref{eq:crossing-rule} crosses
forward link $k$ in every one of them. The named integer
\eqref{eq:crossing-named} alone does so in $117{,}870$, $99.9\%$; the other
$137$ are the two half-integer families of that naming, $13$ of them the
self-pair (first line of \eqref{eq:crossing-rule}) and $124$ the lower
neighbor at $\{a^{2}/n\}\approx\tfrac12$ (third line), never the upper
neighbor and never off by more than one link. The
crossing sits at the middle of the forward link, median position $0.500$ along
it with median deviation $0.009$ four or more links out from the bisector, widening
to $0.148$ at the bisector link, where the crossing is $B_1$ and the position is
$\hat d_1$ instead. The two links meet at a median $137^{\circ}$, tenth to
ninetieth percentile $135^{\circ}$ to $151^{\circ}$, and never below
$135^{\circ}$ in any of them: that floor is
Proposition~\ref{prop:crossing-angle} and not a feature of the sample, and the
exact form \eqref{eq:crossing-angle} holds across the crossings to
$3.8\times10^{-9}$, the accumulated double precision of the chains. The
$43^{\circ}$ of the median is the quarter turn of \eqref{eq:crossing-saddle}
that keeps the sweep from lying flat along the link it reproduces. The center-to-center steps
match those links to a median $1.4\%$, $2.7\%$, $3.1\%$ and $5.2\%$ of a link
length at $c=1,2,3,4$, and the self-pair $(m,m)$ crossed at all $995$ sampled
heights while \eqref{eq:crossing-named} named $m$ at exactly half of them, the
middle half of each unit interval, as the previous paragraph says. The sum rule
\eqref{eq:crossing-sum}, by contrast, names a real crossing in only $10.6\%$ of
the cases overall, but in $99.8\%$ of those inside its range
\eqref{eq:crossing-reach} and $5.1\%$ of those outside it, which is what that
range is worth. The identity $\lambda=\hat d_1$ of \eqref{eq:crossing-rstar}
holds to $6\times10^{-15}$. Computed in \texttt{check\_links\_crossing.py}.

\subsection{General yin and yang form for any link}\label{sec:yinyang-any}

\S\ref{sec:yinyang-origin} watched one link cross another: the reverse bisector link
revolving about the forward bisector link, its two ends tracing the yin and
yang curves \eqref{eq:yin1}--\eqref{eq:yang1}. Now that \eqref{eq:crossing-rule}
says which reverse link crosses forward link $k$, every link of the chain has
such a pair, and they are all one formula. Like the bisector pair it carries no
$\zeta$: the remainder $R$, the factor $\chi$, and $\sigma$ and $T$ are all it
needs.

Send forward link $k$ to $[0,1]$ exactly as \eqref{eq:frame-map} sends the
bisector link: translate its joint $J_k$ to the origin and divide by the link
vector $(k+1)^{-s}$, writing $s=\sigma+iI(T)$ throughout. Reverse joint $j$ of
\eqref{eq:crossing-joints} lands at
\begin{equation}\label{eq:frame-joint-any}
Y_k(j)\;=\;\bigl(K_j-J_k\bigr)(k+1)^{s}
\;=\;(k+1)^{s}\left(\;\sum_{n=k+1}^{m}n^{-s}\;+\;R\;+\;\chi\!\!\sum_{n=j+1}^{m}\!\!n^{\,s-1}\right),
\end{equation}
the second sum read as $-\sum_{n=m+1}^{j}$ once $j>m$. What removes $\zeta$ is
the same fact the bisector frame turned on: $K_m=\Sigma_1+R$, so measuring the
reverse chain from its own joint $m$ replaces the anchor $\zeta$ by the
remainder. The yin and yang points of link $k$ are then the two ends of the
link that crosses it, which are consecutive reverse joints:
\begin{equation}\label{eq:yin-any}
Y_{in}(k,T)=Y_k(j),
\qquad
Y_{ang}(k,T)=Y_k(j+1)=Y_{in}(k,T)-\chi\,(k+1)^{s}(j+1)^{\,s-1},
\qquad j=C_{\ell}(k,T),
\end{equation}
with $j=m$ at the bisector by the first line of \eqref{eq:crossing-rule}.

\begin{remark}
Equation \eqref{eq:frame-joint-any} is the general frame formula, and on the
bisector it \emph{is} the yin formula. It holds there because the crossing
rule sets $j=m$ when $k=m$. Plug $k=m$ and $j=m$ into
\eqref{eq:frame-joint-any}:
\[
Y_m(m)=(m+1)^{s}\left(\sum_{n=m+1}^{m}n^{-s}+R+\chi\sum_{n=m+1}^{m}n^{\,s-1}\right).
\]
The first sum runs from $n=m+1$ to $n=m$ and is empty. The second sum runs
over the same empty range (the joint is $m$, not past it, so the $j>m$
sign-flip is not used). What remains is
\[
Y_m(m)=(m+1)^{s}\bigl(R\bigr)=R\,(m+1)^{s}=\mathrm{Yin}(s),
\]
which is \eqref{eq:yin1}. Geometrically the same thing: $J_m=\Sigma_1$ and
$K_m=\Sigma_1+R$, so
$(K_m-J_m)(m+1)^{s}=R\,(m+1)^{s}$, and the frame map is
\eqref{eq:frame-map} with $M=m+1$. The yang end is the next
reverse joint $j=m+1$. Now $j>m$, so the second sum is
$-\sum_{n=m+1}^{m+1}=-(m+1)^{s-1}$ and the first sum is still empty:
\[
Y_m(m+1)=(m+1)^{s}\bigl(R-\chi\,(m+1)^{s-1}\bigr)
=R\,(m+1)^{s}-\chi\,(m+1)^{2s-1},
\]
which is \eqref{eq:yang1}, since $2s-1=-1+2(\sigma+iI(T))$. So
\eqref{eq:frame-joint-any} at the bisector is not an analogue of yin; it
\emph{is} yin. The closed-form loops of \S\ref{sec:yinyang-general} and
\S\ref{sec:yinyang-origin} are this formula
with the two empty sums already canceled.
\end{remark}

Take $k=m-1$. Its frame factor $m^{s}$ multiplies the summand $m^{-s}$,
so that product is $1$ (forward joint $m$ is the far end of that link), and
\begin{equation}\label{eq:frame-joint-below}
Y_{m-1}(j)\;=\;1+m^{s}\left(R-\chi\!\!\sum_{n=m+1}^{j}\!\!n^{\,s-1}\right).
\end{equation}
Here the crossing changes hands, which is the one thing the bisector's does
not do. By \eqref{eq:crossing-product}, $a^{2}/m=m+2\{T\}+1+O(1/m)$, so the
crossing summand $n'=j+1$ runs $m+1\to m+2\to m+3$ across the unit and the
curves come in three arcs:
\begin{align}
n'=m+1:\quad
Y_{in}&=1+m^{s}R,
&Y_{ang}&=1+m^{s}\bigl(R-\chi(m{+}1)^{s-1}\bigr),
\label{eq:yin-below-a}\\
n'=m+2:\quad
Y_{in}&=1+m^{s}\bigl(R-\chi(m{+}1)^{s-1}\bigr),
&Y_{ang}&=Y_{in}-\chi\,m^{s}(m{+}2)^{s-1},
\label{eq:yin-below-b}\\
n'=m+3:\quad
Y_{in}&=1+m^{s}\Bigl(R-\chi\bigl((m{+}1)^{s-1}+(m{+}2)^{s-1}\bigr)\Bigr),
&Y_{ang}&=Y_{in}-\chi\,m^{s}(m{+}3)^{s-1},
\label{eq:yin-below-c}
\end{align}
the three running over $\{T\}$ in $[0,0.214]$, $[0.215,0.647]$ and
$[0.648,1)$ at $m=6$ (Figure~\ref{fig:yinyang-link5}).

\begin{figure}[H]
\centering
\includegraphics[width=\textwidth]{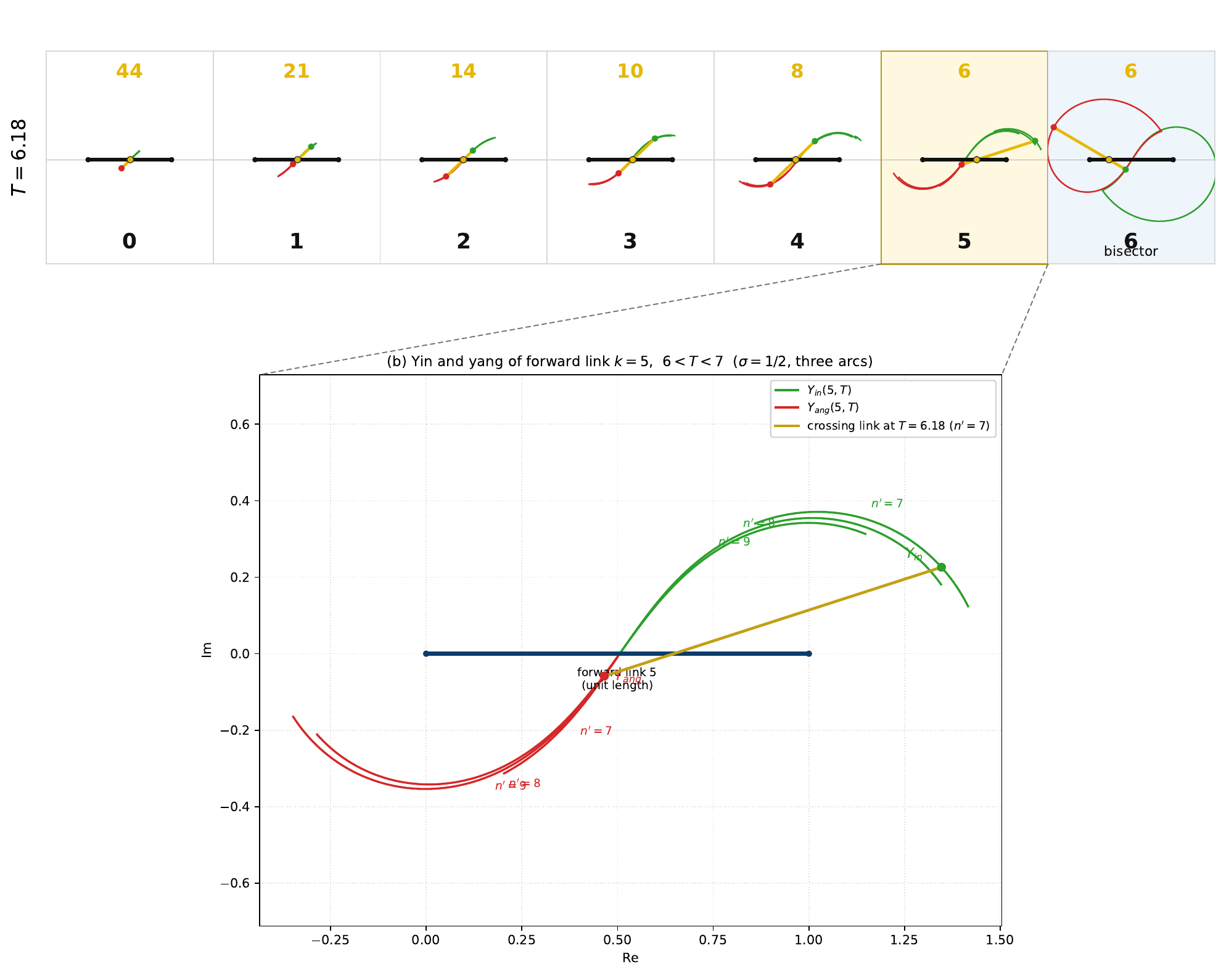}
\caption{(a) Local frames of the forward links at $\sigma=\tfrac12$,
$T=6.18$, as in the top row of Figure~\ref{fig:link-frames}, now with the
yin and yang loci of every link as $T$ runs through $(6,7)$: green
$Y_{in}(k,T)$, red $Y_{ang}(k,T)$. (b) The yin and yang curves of
forward link $k=5$ over $6<T<7$, from
\eqref{eq:frame-joint-below}--\eqref{eq:yin-below-c}.
The frame sends that link to $[0,1]$ (dark blue). Each end traces three arcs,
one per crossing summand $n'=7,8,9$, kept apart where one reverse link
gives way to the next: green is
$Y_{in}(5,T)$, red is $Y_{ang}(5,T)$. The dark-yellow segment is the crossing
link at the running example $T=6.18$, where $C_{\ell}(5,6.18)=6$ so $n'=7$.
Dashed lines carry the local coordinate frame for link $5$ of (a) down to the magnification (b).
Each arc's yin is the previous arc's yang, the same joint passing from one
reverse link to the next.
Generated by \texttt{fig\_yinyang\_link5.py}.}
\label{fig:yinyang-link5}
\end{figure}

The same local frames with each non-bisector arc continued past the unit are
Figure~\ref{fig:yinyang-extended}. The continuation holds the reverse link of
that piece and steps half a unit of $T$ past one end (the end farther from
the local-frame origin for yin, the other end for yang). The bisector is left
alone: its two curves already close.

\begin{figure}[H]
\centering
\includegraphics[width=\textwidth]{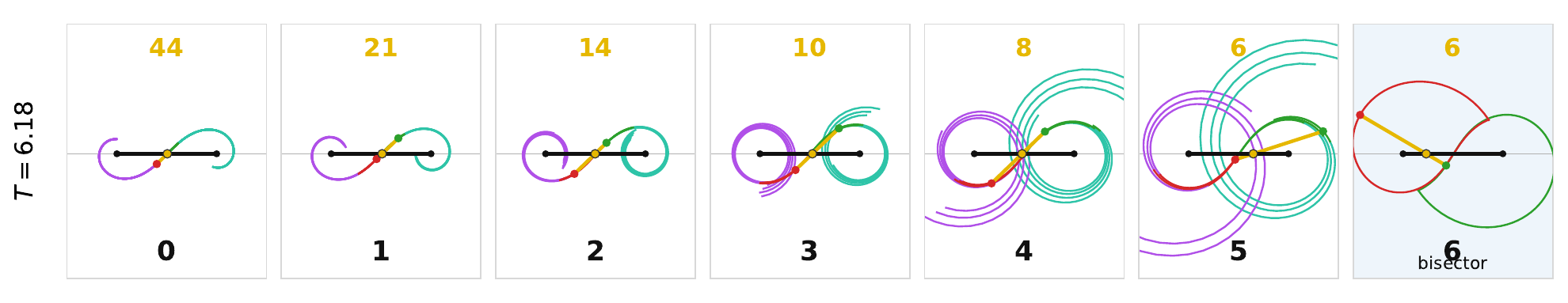}
\caption{Local frames of the forward links at $\sigma=\tfrac12$, $T=6.18$,
as in panel~(a) of Figure~\ref{fig:yinyang-link5}, with each non-bisector
yin and yang arc continued as a solid curve (teal past the green end farther
from the origin, purple past the other red end). The continuation holds the
reverse link of that piece and runs half a unit of $T$. The bisector local frame
is not extended. Generated by \texttt{fig\_yinyang\_extended.py}.}
\label{fig:yinyang-extended}
\end{figure}

Each arc's yin is the preceding arc's yang, the same point at the same $T$:
both are $Y_k(j+1)$, once as the far end of the outgoing crossing link and once
as the near end of the incoming one. That is the crossing of
\S\ref{sec:links-crossing} seen from the ends rather than the middle: one
reverse link is dragged off the forward link and the next is dragged across
it. The joint passing across the forward link belongs to both links at that
instant, so the crossing itself goes through smoothly while the yin curve
steps onto where the yang curve was and the yang steps one joint further out.

The count of reverse links dragged across forward link $k$ in one unit of
$T$ is about $2a/(k+1)$: $n'=a^{2}/(k+1)$ advances by
$2a/(k+1)=(2T+1)/(k+1)+O(T^{-1})$ over a unit of the index, and that is how
many arcs the link gets. On the unit step, $k+1=1$, so about $2a$ reverse
links cross it; at $T=6.18$, $a\approx 6.67$, so roughly a dozen. That is
why the local coordinate frame for link $0$ in Figures~\ref{fig:yinyang-link5}--\ref{fig:yinyang-extended}
is a thicket of short arcs, so close together they are not visible
separately in the figure, while the links beside the bisector only get
two, and the bisector none. Its partner is the other bisector link at
every height, so its yin and yang alone close into single loops, and they
alone have the limit teardrop of \S\ref{sec:yinyang-inf}.

The two ends fix where the crossing reverse link sits along the forward link. In each separate local coordinate frame the link
is $[0,1]$ on the real axis, so the crossing is where the chord from yin to
yang crosses that axis, the computation \eqref{eq:crossing} makes at the bisector.
Writing $v_k=\chi\,(k+1)^{s}(j+1)^{\,s-1}$ for the crossing link's own vector in
the frame, which by \eqref{eq:yin-any} is the offset $Y_{in}-Y_{ang}$, that
crossing fraction is
\begin{equation}\label{eq:crossing-fraction-any}
\lambda_k(T)
\;=\;\operatorname{Re}Y_{in}-\operatorname{Im}Y_{in}\,
\frac{\operatorname{Re}v_k}{\operatorname{Im}v_k}
\;=\;\frac{\operatorname{Im}\bigl(\overline{Y_{in}}\,v_k\bigr)}{\operatorname{Im}v_k}
\;=\;\bigl|Y_{in}\bigr|\,\frac{\sin(\beta-\alpha)}{\sin\beta},
\qquad
\alpha=\arg Y_{in},\quad \beta=\arg v_k,
\end{equation}
Let $u_k$ be the corresponding parameter on the reverse segment. The finite
segments intersect exactly when $0\leq\lambda_k\leq1$ and $0\leq u_k\leq1$.
The distance in the plane from joint $k$ is $\lambda_k\,(k+1)^{-\sigma}$. At
$k=m$ this is \eqref{eq:crossing} itself, $\lambda_m=M^{\sigma}d_1$,
the bisector point's own crossing fraction, which on the critical line is
$\hat d_1$. So the rightmost panel of Figure~\ref{fig:d1-any-link} is the
quantity the rest of the paper is built on, and the panels to its left are that
quantity's counterparts link by link.

Checked against the local frame construction in the Zest web app on the critical
line at $T=6.18,\,6.72,\,12.4,\,17.4,\,40.05$ and $434.37$:
\eqref{eq:frame-joint-any} matches the frame to $1.3\times10^{-15}$,
\eqref{eq:crossing-fraction-any} matches the app's own reading to
$6.9\times10^{-14}$, and $\lambda_m$ matches $\hat d_1$ to $1.0\times10^{-12}$,
in \texttt{probe-yinyang-below-bisector.mjs}.

\subsection{\texorpdfstring{$d_1$ and $d_2$}{d1 and d2} for any link}\label{sec:d1-any}

\S\ref{sec:decomp} split Siegel's remainder along the two bisector links,
and \S\ref{sec:yinyang-origin} found the weights
$d_1,d_2$ as the real-axis crossings of the yin and yang pair in those two
frames. The same two crossings exist for every link: \eqref{eq:yin-any} is
the pair, \eqref{eq:crossing-fraction-any} is the crossing fraction, and the
weights are that fraction read once along the forward link and once along the
reverse one.

Write $n=k+1$, $n'=j+1=C_{s}(n,T)$, and
\begin{equation}\label{eq:W-any}
W_k
\;=\;
\sum_{n=k+1}^{m}n^{-s}
\;+\;R
\;-\;\chi\sum_{n=m+1}^{j}n^{\,s-1}
\end{equation}
for the vector in the plane from forward joint $k$ to reverse joint $j$, so
that $Y_{in}(k,T)=(k+1)^{s}W_k$ by \eqref{eq:frame-joint-any}. At the
bisector both sums are empty and $W_m=R$. Let
\begin{equation}\label{eq:omega-any}
\omega_n=I(T)\log n,
\qquad
\omega_{n'}=I(T)\log n',
\qquad
\varphi_k=\arg W_k,
\qquad
\psi=\arg\chi,
\end{equation}
so $-\omega_n$ is the absolute direction of summand $n$ of $\Sigma_1$, the
same telescoping that made $\omega=\omega_{m+1}$ in \eqref{eq:omega-telescope},
and $\omega_{n'}+\psi$ is the absolute direction of summand $n'$ of
$\Sigma_2$.

\paragraph{Measured along the forward link.} In the frame of link $k$ the crossing fraction
\eqref{eq:crossing-fraction-any} is how far along that link it is from joint
$k$ to the crossing point. The frame scaled the link by $(k+1)^{\sigma}$, so the distance
in the plane is
\begin{equation}\label{eq:d1-any}
d_1(k,T)
\;=\;
(k+1)^{-\sigma}\lambda_k(T)
\;=\;
|W_k|\,
\frac{\sin\!\bigl(\omega_{n'}-\varphi_k+\psi\bigr)}
     {\sin\!\bigl(\omega_n+\omega_{n'}+\psi\bigr)}.
\end{equation}
The sine form is \eqref{eq:crossing-fraction-any} with
$Y_{in}=(k+1)^{s}W_k$ and $v_k=\chi\,(k+1)^{s}(j+1)^{s-1}$, the same
reduction that took \eqref{eq:crossing} to \eqref{eq:d1}:
$\alpha=\omega_n+\varphi_k$, $\beta=\omega_n+\omega_{n'}+\psi$, and
$|Y_{in}|=(k+1)^{\sigma}|W_k|$. At $k=m$ one has $W_m=R$,
$n=n'=\lceil T\rceil$, $\omega_n=\omega_{n'}=\omega$, and
\eqref{eq:d1-any} is \eqref{eq:d1}.

\paragraph{Measured along the reverse link.} The same crossing, read along the reverse
link from joint $j$, is the parameter
$u_k=\Im Y_{in}/\Im v_k$ of the chord in that frame. The reverse link has
length $|\chi|\,(j+1)^{\sigma-1}$, so
\begin{equation}\label{eq:d2-any}
d_2(k,T)
\;=\;
|\chi|\,(j+1)^{\sigma-1}\,u_k
\;=\;
|W_k|\,
\frac{\sin\!\bigl(\omega_n+\varphi_k\bigr)}
     {\sin\!\bigl(\omega_n+\omega_{n'}+\psi\bigr)}.
\end{equation}
$d_1(k,T)$ and $d_2(k,T)$ share a denominator, as $d_1$ and $d_2$ did. At $k=m$ this
is \eqref{eq:d2}. The fractional amounts, the hat of
\eqref{eq:frac-vs-weight}, are the two parameters themselves:
\begin{equation}\label{eq:hat-any}
\hat d_1(k,T)=\lambda_k(T)=(k+1)^{\sigma}\,d_1(k,T),
\qquad
\hat d_2(k,T)=u_k=\frac{d_2(k,T)}{|\chi|\,(j+1)^{\sigma-1}}.
\end{equation}

\paragraph{The direction, and the two-direction form.} Forward link $k$ points along
$e^{-i\omega_n}$ and reverse link $j$ along $e^{i(\omega_{n'}+\psi)}$. The
two fractional summands are those directions, shortened by the two weights,
\begin{equation}\label{eq:R-any}
R_1(k,T)=d_1(k,T)\,e^{-i\omega_n},
\qquad
R_2(k,T)=d_2(k,T)\,e^{\,i(\omega_{n'}+\psi)},
\end{equation}
and they add back to the vector they split:
\begin{equation}\label{eq:W-wedge}
W_k
\;=\;
d_1(k,T)\,e^{-i\omega_n}
\;+\;
d_2(k,T)\,e^{\,i(\omega_{n'}+\psi)}.
\end{equation}
This is the two-direction form \eqref{eq:R-two-dirs} with $R$ replaced by $W_k$ and the two bisector
directions replaced by the two links that actually cross. The sine
expressions \eqref{eq:d1-any}--\eqref{eq:d2-any} are again Cramer's rule for
that $2\times2$ real system, which is why they are real, and why they reduce
to \eqref{eq:d1}--\eqref{eq:d2} at the bisector. On the critical line the
bisector pair still satisfies $d_1(m,T)=d_2(m,T)$, but a general pair need
not: the two links have different lengths, and the crossing sits at
$\hat d_1\approx\tfrac12$ four or more links out from the bisector
(\S\ref{sec:links-crossing}), so $d_1(k)$ and $d_2(k)$ scale as
$n^{-\sigma}$ and $|\chi|n'^{\sigma-1}$.
Figure~\ref{fig:Wk-triangle} is the same decomposition, one local frame per forward
link, in the local frames of the top row of Figure~\ref{fig:yinyang-link5}:
the last local frame is the bisector, $W_6=R$, which is
Figure~\ref{fig:remainder-zoom} seen in that frame.

\begin{figure}[H]
\centering
\includegraphics[width=\textwidth]{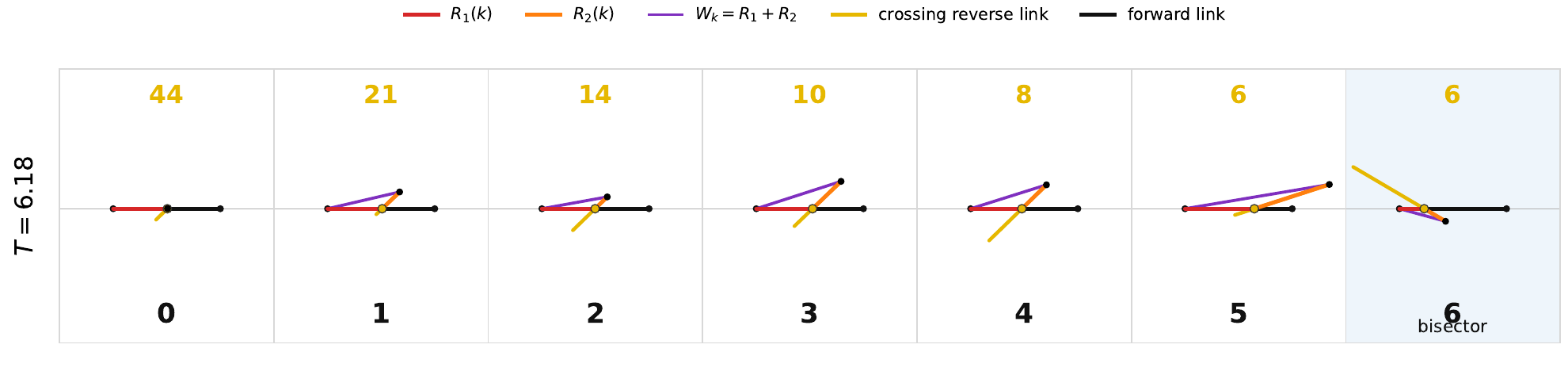}
\caption{The two-direction form \eqref{eq:W-wedge} at $\sigma=\tfrac12$, $T=6.18$, in the
local frame of each forward link $k=0,\dots,6$, the same local frames as the
top row of Figure~\ref{fig:yinyang-link5}. Each local frame sends that link to
$[0,1]$ on the real axis (black); the gold number is the reverse link
that crosses it. Red is $R_1(k)$ from the left joint to the crossing
along the unit link, orange is $R_2(k)$ from the crossing to reverse
joint $K_j$, and purple is the resultant $W_k$. The last local frame is the
bisector, $W_6=R$. Generated by \texttt{fig\_Wk\_triangle.py}.}
\label{fig:Wk-triangle}
\end{figure}

At the running example $\sigma=\tfrac12$, $T=6.18$, the seven pairs are
\begin{center}
\small
\begin{tabular}{crrrr}
\toprule
$k$ & $j$ & $d_1(k)$ & $d_2(k)$ & $\hat d_1(k)$ \\
\midrule
$0$ & $44$ & $0.506$ & $0.00185$ & $0.506$ \\
$1$ & $21$ & $0.360$ & $0.159$ & $0.510$ \\
$2$ & $14$ & $0.286$ & $0.0914$ & $0.496$ \\
$3$ & $10$ & $0.263$ & $0.184$ & $0.525$ \\
$4$ & $8$ & $0.214$ & $0.143$ & $0.478$ \\
$5$ & $6$ & $0.264$ & $0.300$ & $0.647$ \\
$6$ & $6$ & $0.0875$ & $0.0875$ & $0.231$ \\
\bottomrule
\end{tabular}
\end{center}
The last row is Table~\ref{tab:frac} at this $T$: $d_1=d_2=0.0875$ and
$\hat d_1=\sqrt7\,d_1=0.231$. The first six sit near the middle of their
forward link, as \S\ref{sec:links-crossing} measured, and $d_2(0)$ is the
tiny reverse link $44$. Figure~\ref{fig:d1-any-link} draws the same seven
numbers against the whole unit of $T$.

\begin{figure}[H]
\centering
\includegraphics[width=\textwidth]{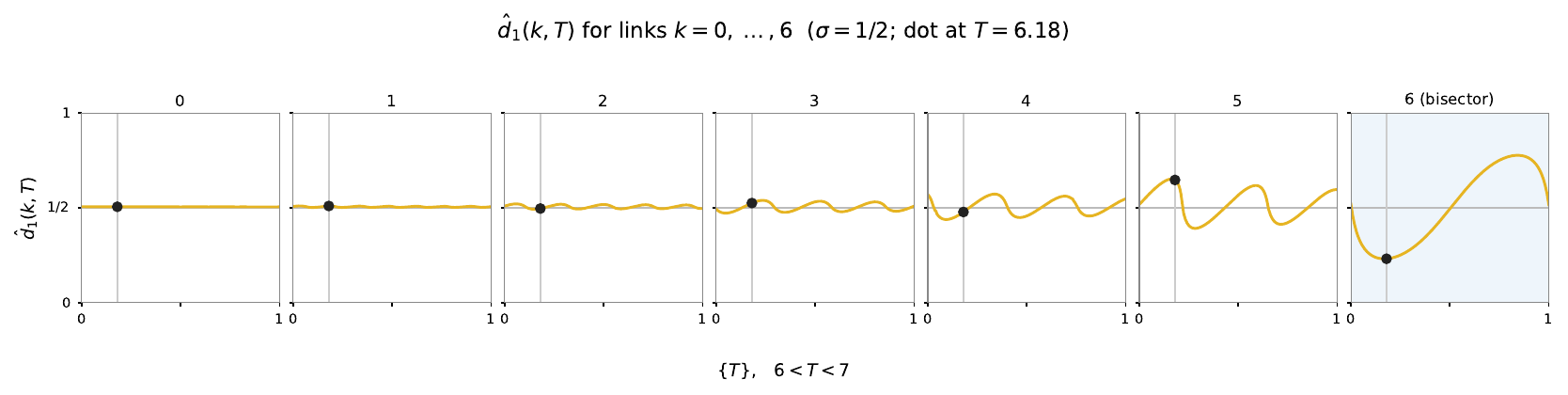}
\caption{$\hat d_1(k,T)=(k+1)^{\sigma}d_1(k,T)$ for forward links
$k=0,\dots,6$ over $6<T<7$, $\sigma=\tfrac12$, one panel per link, $x$ the
fractional part of $T$ and $y$ the crossing fraction along that link from
the left joint. The
light horizontal is $\tfrac12$; the light vertical and the dot are the
running example $T=6.18$, the seven values of the table. Away from the
bisector the track sits on $\tfrac12$ and barely moves; at $k=6$ it is the
normalized $d_1$ of Figure~\ref{fig:d1-periodicity}. Generated by
\texttt{fig\_d1\_any\_link.py}.}
\label{fig:d1-any-link}
\end{figure}

Checked at $\sigma=\tfrac12,\tfrac14,\tfrac7{10}$ and
$T=6.18,\,6.72,\,12.4,\,17.4$, for $k=0,1,m-2,m-1,m$:
\eqref{eq:d1-any}--\eqref{eq:d2-any} at $k=m$ match \eqref{eq:d1}--\eqref{eq:d2}
to $0$, \eqref{eq:hat-any} matches \eqref{eq:crossing-fraction-any} to
$1.5\times10^{-27}$, and \eqref{eq:W-wedge} holds to $4.4\times10^{-28}$, in
\texttt{check\_d1\_any\_link.py}.

\subsection{The crossings: \texorpdfstring{$\Sigma_{1x}$ and $\Sigma_{2x}$}{Sigma\_1x and Sigma\_2x} formula}\label{sec:sum-x}

\S\ref{sec:links-crossing} named the reverse partner of each forward link, and
\S\ref{sec:d1-any} read the crossing fraction \(\hat d_1(k,T)\) along that
link. Those two data locate a point on every forward link \(0,\dots,m\).
Connecting the points in order is the same Dirichlet sum as
\(\Sigma_1+R_{1}\), re-jointed: the joints sit at the crossings instead of
at the integers, and the walk ends at the bisector point \(B_1\) rather than
at \(\Sigma_1\). Call that walk \(\Sigma_{1x}\).\footnote{The subscript \(x\) is for crossing, not a dummy index: \(\Sigma_{1x}\) is the crossing-joint form of \(\Sigma_1\), and \(\Sigma_{2x}\) of \(\Sigma_2\).}

The crossing fraction of link \(k\), from joint \(k\) to the crossing point,
is \(\hat d_1(k,T)\) of \eqref{eq:hat-any}; \(T\) is fixed
through what follows, so we write it \(\hat d_1(k)\). Link \(k\) carries the
summand \(n=k+1\), so the crossing on it is
\begin{equation}\label{eq:sum-x-point}
P_k
\;=\;
J_k+\hat d_1(k)\,(k+1)^{-s}
\;=\;
\sum_{n'\le k}n'^{-s}+\hat d_1(k)\,(k+1)^{-s},
\end{equation}
with \(J_k\) as in \eqref{eq:crossing-joints}. The last crossing is the
bisector point itself, \(P_m=B_1\). The walk
\(0=P_{-1}\to P_0\to\cdots\to P_m\) therefore telescopes to \(B_1\) by
construction. Its first step is \(P_0=\hat d_1(0)\), since \(J_0=0\) and
\(1^{-s}=1\). Taking that step separately and substituting
\eqref{eq:sum-x-point} in the rest, with the empty sum \(\sum_{n'\le0}=0\),
removes the intermediate points and keeps \(0^{-s}\) out of the display:
\begin{equation}\label{eq:sum-1x}
\Sigma_{1x}
\;=\;
\hat d_1(0)+\sum_{k=1}^{m}\Biggl(
\sum_{n'\le k}n'^{-s}+\hat d_1(k)\,(k+1)^{-s}
-\sum_{n'\le k-1}n'^{-s}-\hat d_1(k-1)\,k^{-s}
\Biggr)
\;=\;
\sum_{n'\le m}n'^{-s}+\hat d_1(m)\,(m+1)^{-s}.
\end{equation}
Each full summand \(k^{-s}\) is divided across two consecutive steps,
\(\hat d_1(k-1)\,k^{-s}\) in one and \((1-\hat d_1(k-1))\,k^{-s}\) in the next, and those
two pieces add to one. What remains is the integer partial sum plus the last
fractional stub, which is the ordinary first-half formula
\eqref{eq:B1-sum} once the stub is named.

That stub is \(\hat d_1(m)(m+1)^{-s}\). The bisector link is link
\(m=\lfloor T\rfloor\). It runs from \(J_m=\Sigma_1\) to
\(J_{m+1}=\Sigma_1+(m+1)^{-s}\), so the vector along it is the next Dirichlet
term, \(n=M=m+1\). The factor \((m+1)^{-s}\) is that whole
unused summand, the faint dashed continuation past \(\Sigma_1\) in
Figure~\ref{fig:spiral-summands}. The number \(\hat d_1(m)\) is the crossing
fraction: how far along the link the two bisector links meet. That meeting is \(B_1\), and
\eqref{eq:crossing-rstar} already computed it: the crossing fraction on the
bisector link is the remainder weight itself,
\(\hat d_1(m)=\hat d_1\). Therefore
\begin{equation}\label{eq:sum-1x-remainder}
\hat d_1(m)\,(m+1)^{-s}
\;=\;
\hat d_1\,(m+1)^{-s}
\;=\;
R_{1},
\end{equation}
the ``one more partial summand'' of \S\ref{sec:summand}: not a new term, the
next summand scaled by a real signed coefficient. The unhatted weight \(d_1\)
multiplies a unit vector in the direction of \((m+1)^{-s}\), so
\(|R_{1}|=|d_1|\). The hatted quantity \(\hat d_1\) is a signed normalized
coefficient, and is literally a fraction of the link when
\(0\leq\hat d_1\leq1\); it multiplies the full summand, whose length is
\(M^{-\sigma}\), and
\(\hat d_1\cdot M^{-\sigma}=d_1\) is the same remainder
\eqref{eq:frac-vs-weight}. At the running example \(\sigma=\tfrac12\),
\(T=6.18\), one has \(m=6\), next term \(7^{-s}\) of length
\(1/\sqrt7\approx0.378\), \(d_1\approx0.0875\), and
\(\hat d_1=\sqrt7\,d_1\approx0.231\): the walk uses about \(23\%\) of link
\(6\), and that stub is \(R_{1}\). The first six \(\hat d_1(k)\) in the table of
\S\ref{sec:d1-any} sit near \(\tfrac12\), as the saddle of
\S\ref{sec:links-crossing} said they would; only \(\hat d_1(m)\) is the named
remainder weight. So \eqref{eq:sum-1x} is
\begin{equation}\label{eq:sum-1x-B1}
\Sigma_{1x}
\;=\;
\sum_{n'\le m}n'^{-s}+\hat d_1(m)\,(m+1)^{-s}
\;=\;
\Sigma_1+R_{1}
\;=\;
B_1,
\end{equation}
and not \(\Sigma_1\) itself.

The same construction on the other Dirichlet sum gives \(\Sigma_{2x}\). Write
\(L_k=\chi(s)\sum_{n'\le k}n'^{\,s-1}\) for the joints of \(\Sigma_2\) from
the origin, and \(\hat d_2(k)\) for the crossing fraction of reverse
link \(k\) from \(L_k\) to its crossing point, the \(u_k\) of
\eqref{eq:hat-any}. Then
\begin{equation}\label{eq:sum-2x-point}
Q_k
\;=\;
L_k+\hat d_2(k)\,\chi(s)\,(k+1)^{s-1},
\end{equation}
and the same telescoping, with first step \(Q_0=\hat d_2(0)\,\chi(s)\) and the
empty sum \(L_0=0\), is
\begin{equation}\label{eq:sum-2x}
\begin{aligned}
\Sigma_{2x}
&\;=\;
\hat d_2(0)\,\chi(s)+\sum_{k=1}^{m}\Biggl(
L_k+\hat d_2(k)\,\chi(s)\,(k+1)^{s-1}
-L_{k-1}-\hat d_2(k-1)\,\chi(s)\,k^{s-1}
\Biggr)\\
&\;=\;
\chi(s)\sum_{n'\le m}n'^{\,s-1}+\hat d_2(m)\,\chi(s)\,(m+1)^{s-1}.
\end{aligned}
\end{equation}
The last stub is \(\hat d_2(m)=\hat d_2\), so
\begin{equation}\label{eq:sum-2x-B2}
\Sigma_{2x}
\;=\;
\Sigma_2+R_{2}
\;=\;
B_2.
\end{equation}
Nothing here used \(\sigma=\tfrac12\). The split
\(R=R_{1}+R_{2}\) of \S\ref{sec:decomp} holds at every \(\sigma\), the
weights \(\hat d_1,\hat d_2\) of \eqref{eq:frac-vs-weight} are defined at
every \(\sigma\), and so
\begin{equation}\label{eq:sum-x-zeta}
\Sigma_{1x}+\Sigma_{2x}
\;=\;
B_1+B_2
\;=\;
\zeta
\end{equation}
off the critical line as well as on it. On \(\sigma=\tfrac12\) the two walks
are the two legs of \S\ref{sec:legs}; off it they are just the two summands
of Theorem~\ref{thm:decomp}, re-jointed. Reflecting \(\Sigma_{1x}\) through
\(\zeta/2\) sends the origin to \(\zeta\) and fixes \(B_1\), which is the
other way of seeing the second half on the critical line, not the walk of
\(\Sigma_{2x}\) from the origin. The reverse partners of the forward links
\(0,\dots,m\) are not consecutive, so the points \(P_k\) are not the joints
of a walk along the reverse chain. This re-jointing is not Kapitonets'
iterated-midpoint polygon~\cite{Kapitonets2019}: that construction is the
binomial transform of the partial sums and does not aim at \(B_1\) or
\(B_2\).

\subsection{First-part and second-part sums}\label{sec:half-link-sums}

Here we use the places where each link is crossed to divvy up the main sum
into two separate sums: one adds up the first parts of all the summands, the
other all the second parts. The remainder is handled automatically, because
it is itself the first part of a summand. This subsection is nothing but an
exploration of a rearrangement of the terms of the main sum.

Each forward link $k=0,\dots,m-1$ is cut at its crossing $P_k$ of
\eqref{eq:sum-x-point} into two pieces: a first part
$\hat d_1(k)\,(k+1)^{-s}$ from the joint to the crossing, and a second part
$(1-\hat d_1(k))\,(k+1)^{-s}$ from the crossing to the next joint. The bisector
link contributes only its first part, the stub
$\hat d_1(m)\,(m+1)^{-s}=R_{1}$ from $\Sigma_1$ to $B_1$. Summing those
pieces as free vectors,
\begin{equation}\label{eq:half-link-sums}
V_1
\;=\;
\sum_{k=0}^{m}\hat d_1(k)\,(k+1)^{-s},
\qquad
V_2
\;=\;
\sum_{k=0}^{m-1}\bigl(1-\hat d_1(k)\bigr)\,(k+1)^{-s},
\end{equation}
gives $V_1+V_2=B_1$: the full links $0,\dots,m-1$ plus the bisector stub
are $\Sigma_1+R_{1}$. Figure~\ref{fig:half-link-sums} draws $V_1$ from
the origin and $V_2$ from the tip of $V_1$.

\begin{figure}[H]
\centering
\includegraphics[width=0.82\textwidth]{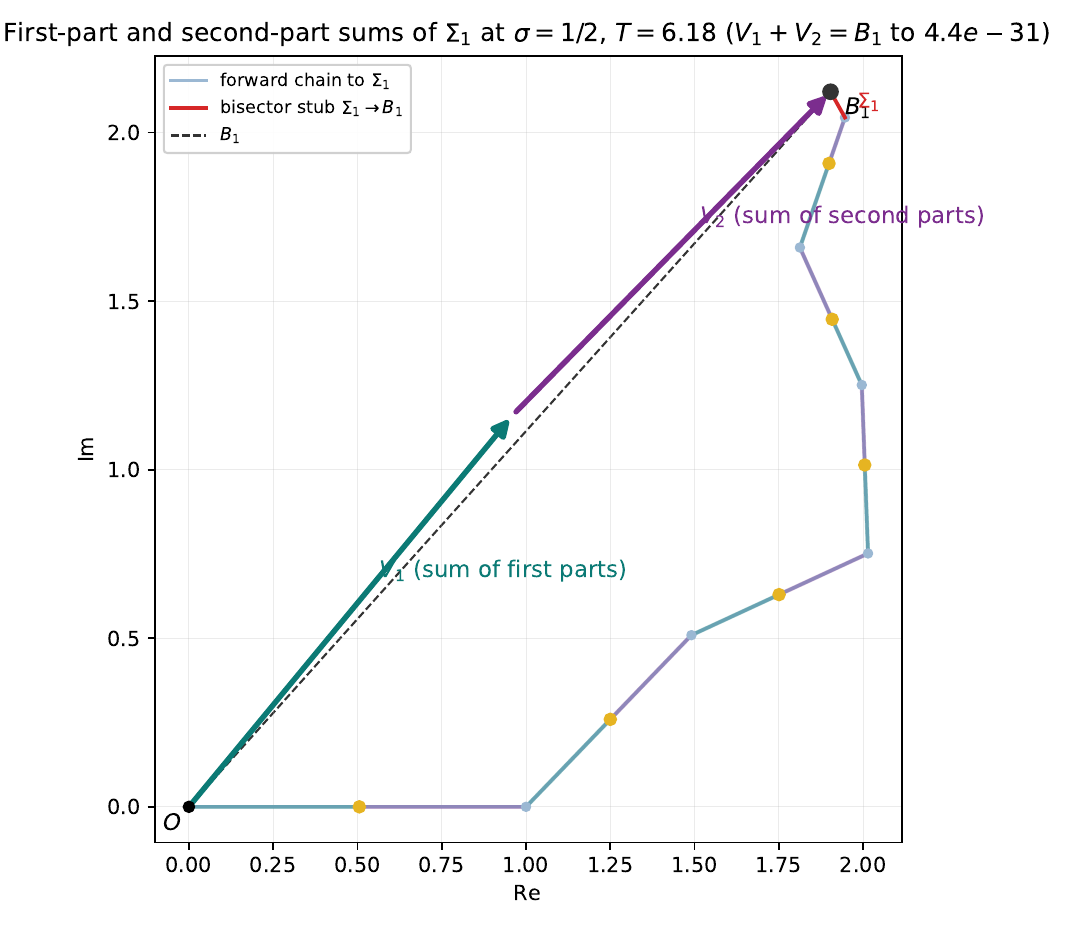}
\caption{Each forward link cut at its yellow crossing: teal the first
part, purple the second, red the bisector stub $\Sigma_1\to B_1$.
The teal arrow is $V_1$, the sum of the first parts including that
stub; the purple arrow is $V_2$, the sum of the second parts, drawn
from the tip of $V_1$. They meet at $B_1$. At
$\sigma=\tfrac12$, $T=6.18$. Generated by
\texttt{fig\_half\_link\_sums.py}.}
\label{fig:half-link-sums}
\end{figure}

\subsection{General partial summand Riemann zeta function}\label{sec:cut-any-link}

The partial-summand Riemann zeta formula \eqref{eq:zeta-clean}, our main
result, presented near the start of this paper, has the same summation upper bound in
both the main sum and the dual sum. But with the knowledge of which links
cross, we can make those two upper bounds different. This subsection presents
that new, more general form of the Riemann zeta function.

For any nonparallel pair of summand directions, Cramer's rule gives unique
real coefficients and the displayed arbitrary-cut identity. The crossing
conjecture is needed only to show that both coefficients lie in $[0,1]$.
For $j=C_{\ell}(k,T)$, the corresponding coefficient relation, written with
the joints of \eqref{eq:crossing-joints}, is
\begin{equation}\label{eq:cut-two-ways}
J_k+\hat d_1(k,T)\,(k+1)^{-s}
\;=\;
K_j-\hat d_2(k,T)\,\chi(s)\,(j+1)^{\,s-1},
\end{equation}
and moving the reverse sum inside $K_j$ across gives
\begin{equation}\label{eq:zeta-any-cut}
\begin{gathered}
\boxed{\;
\zeta(s)
\;=\;
\sum_{n=1}^{k}n^{-s}+\hat d_1(k,T)\,(k+1)^{-s}
\;+\;\chi(s)\Biggl[\;\sum_{n=1}^{j}n^{\,s-1}+\hat d_2(k,T)\,(j+1)^{\,s-1}\Biggr]
\;}\,,\\[4pt]
j=C_{\ell}(k,T).
\end{gathered}
\end{equation}
This is \eqref{eq:zeta-clean} with the two upper summation bounds unequal,
though not independent: $j$ is the crossing partner of $k$ named by the rule
\eqref{eq:crossing-rule}. At $k=m$ the
pairing is the self-pair, $j=m$, both cutoffs are $\lfloor T\rfloor$, and
\eqref{eq:zeta-any-cut} is the boxed formula \eqref{eq:zeta-clean}, hats and all. At any other $k$
the two cutoffs differ: the forward sum stops at $k$ while the reverse one
runs to $j\approx a^{2}/(k+1)$, the two being reciprocal about $a$ by
\eqref{eq:crossing-product}. Reaching that far means carrying the reverse
chain past its own bisector link, the continuation already written into
\eqref{eq:W-any}.

For comparison, equation (4.13.1) of Titchmarsh's
book~\cite[\S4.13]{Titchmarsh1986}, the approximate functional equation,
states: if $h$ is a positive constant and
\[
0<\sigma<1,\qquad 2\pi xy=t,\qquad x>h>0,\qquad y>h>0,
\]
then
\begin{equation}\label{eq:titchmarsh-afe}
\zeta(s)\;=\;\sum_{n\le x}\frac{1}{n^{s}}
\;+\;\chi(s)\sum_{n\le y}\frac{1}{n^{1-s}}
\;+\;O\bigl(x^{-\sigma}\log|t|\bigr)
\;+\;O\bigl(|t|^{\frac12-\sigma}\,y^{\sigma-1}\bigr).
\end{equation}
The two sums are the same, and his constraint $2\pi xy=t$, that is
$xy=a^{2}$, ties the two cutoffs exactly as the crossing product
\eqref{eq:crossing-product} does. Where he pays two error terms for the
freedom to slide the cut, \eqref{eq:zeta-any-cut} pays with the two
fractional terms, exactly.

The extreme cases are the ends of the pairing. Take $j=0$ in
\eqref{eq:zeta-any-cut}: the first reverse link is
$\chi(s)\,1^{\,s-1}=\chi(s)$, so the reverse sum is empty and it is the second
half that collapses to a single term,
\begin{equation}\label{eq:zeta-first-rev-link}
\boxed{\;
\zeta(s)
\;=\;
\sum_{n=1}^{k}n^{-s}+\hat d_1(k,T)\,(k+1)^{-s}+\hat d_2(k,T)\,\chi(s)
\;}\,,
\qquad k=\Bigl[\frac{I(T)}{2\pi}\Bigr]-1\ \text{ or one less},
\end{equation}
the forward cutoff being the partner of $n'=1$, about $(T+\tfrac12)^{2}$
links out. Which of the two it is the crossing test decides, exactly as in the
third line of \eqref{eq:crossing-rule}: the nearest integer names the paired
link at most heights and its lower neighbor at the rest, and off the critical
line the choice can turn on $\sigma$ as well. Here $\Sigma_1$ is walked whole, out to link $k$ and a fraction of the next, and the walk
lands on the segment hanging off $\zeta$ in the direction of $-\chi$. On the
critical line that segment has length one, so the forward chain carried that
far passes within one unit of $\zeta$, and what is left over is
$\hat d_2(k,T)\,\chi(s)$, a real multiple of $\chi$. At $\sigma=\tfrac12$,
$T=6.18$ the cutoff is $k=44$ and the two fractions are $0.012402$ and
$0.505537$, reproducing $\zeta$ to $9.9\times10^{-32}$.

For comparison, equation (4.11.1) of Titchmarsh's
book~\cite[\S4.11]{Titchmarsh1986}, the simplest approximation of
$\zeta(s)$ in the critical strip by a partial sum of its Dirichlet
series, states:
\begin{equation}\label{eq:titchmarsh-partial-sum}
\zeta(s)\;=\;\sum_{n\le x}\frac{1}{n^{s}}
\;-\;\frac{x^{1-s}}{1-s}
\;+\;O\bigl(x^{-\sigma}\bigr),
\end{equation}
uniformly for $\sigma\ge\sigma_0>0$, $|t|<2\pi x/C$, when $C$ is a given
constant greater than $1$. Both close one long partial sum with a small
correction: his is the smooth tail $-x^{1-s}/(1-s)$ plus an error term,
while \eqref{eq:zeta-first-rev-link} closes exactly with the two
fractional terms. And the cutoffs match in scale: his condition puts $x$
at order $t/2\pi$ at least, which is where
$k=\bigl[I(T)/2\pi\bigr]-1$ sits.

The cut at the other end of the pairing is the mirror of this one, $k=0$.
The first forward link runs from $J_0=0$ to
$J_1=1^{-s}=1$: it is the unit segment of the real axis. So the forward sum is
empty, the link is the number $1$, and the entire first half of
\eqref{eq:zeta-any-cut} collapses to a single real number,
\begin{equation}\label{eq:zeta-first-link}
\boxed{\;
\zeta(s)
\;=\;
\hat d_1(0,T)
\;+\;\chi(s)\Biggl[\;\sum_{n=1}^{j}n^{\,s-1}+\hat d_2(0,T)\,(j+1)^{\,s-1}\Biggr]
\;}\,,
\qquad j=C_{\ell}(0,T).
\end{equation}
The reverse cutoff is the partner of the first summand, $n'=a^{2}/1$, so $j$
is $\bigl[\,I(T)/2\pi\,\bigr]-1$ or one less, by \eqref{eq:crossing-rule}
again. On the critical line that is the same integer as the forward cutoff in
\eqref{eq:zeta-first-rev-link}, the chain mirror \eqref{eq:chain-mirror}
carrying the one cut to the other; off the line the two part, as the
measurements at the end of this subsection show. And
$\hat d_1(0,T)=d_1(0,T)$ here: link $0$ has length one, so the fraction of the
link and the distance along it are the same number, and this is the only link
on which the hat of \eqref{eq:frac-vs-weight} does nothing.

Read as a walk, \eqref{eq:zeta-first-link} says this. Start at $\zeta$ and come
in along the reverse chain, $j$ whole links and then a fraction
$\hat d_2(0,T)$ of the next. The walk lands on the real axis, between $0$ and
$1$, and $\hat d_1(0,T)$ is where. Equivalently, $\zeta$ minus that whole
reverse tail is a real number in the unit interval. It is not just any point of
that interval either: it is essentially the midpoint, the $k=0$ panel of
Figure~\ref{fig:d1-any-link} sitting on $\tfrac12$. The crossing fraction has
the closed form of \eqref{eq:d1-any} read at $k=0$, where
$\omega_n=I(T)\log1=0$ and $W_0=K_j$ is the reverse joint itself:
\begin{equation}\label{eq:d1-first-link}
\hat d_1(0,T)
\;=\;
\operatorname{Re}K_j-\operatorname{Im}K_j\,\cot\bigl(\omega_{n'}+\psi\bigr),
\qquad n'=j+1,
\end{equation}
the crossing of the real axis by the line that carries reverse link $j$.

None of this is a shortcut to a value. The reverse sum now has about
$(T+\tfrac12)^{2}$ terms where the balanced cut of \eqref{eq:zeta-clean} has
$2\lfloor T\rfloor$ in all, and the two fractions are read off $\zeta$ just as
$d_1$ and $d_2$ were: \eqref{eq:zeta-first-link} is one complex equation in the
two real unknowns $\hat d_1(0,T)$ and $\hat d_2(0,T)$, so it fixes them from
$\zeta$ and not the other way about. What it fixes is a position. Nor is the
crossing the only one: the continued reverse chain crosses the unit segment many
times, $87$ times before link $400$ at the running example, since both chains
wind. The pairing law picks one of them out, the link paired with the first
summand, and \eqref{eq:d1-first-link} says where that one lands.

At $\sigma=\tfrac12$, $T=6.18$, one has $a^{2}=44.5389$, so $j=44$, the tiny
reverse link of the table in \S\ref{sec:d1-any}; then
$\hat d_1(0)=0.505537$ and $\hat d_2(0)=0.012402$, and the right side of
\eqref{eq:zeta-first-link} reproduces $\zeta$ to $4.6\times10^{-41}$ at forty
digits. The same cut at $k=0,\dots,6$ and that $T$ gives seven formulas for the
one number, with reverse cutoffs $44,21,14,10,8,6,6$: the list of
\S\ref{sec:links-crossing} with the self-pair in last place. Checked on $500$
samples of the critical line spread over $2.05\le T\le60$, and on $100$ more at
$\sigma=0.25,\,0.35,\,0.7,\,0.9$: the crossing lay inside both links at every
sample, and the two sides of \eqref{eq:zeta-first-link} agreed to
$3.2\times10^{-30}$ at thirty digits. The first crossing stayed near the middle
of the unit segment throughout, $\hat d_1(0,T)$ in $[0.5006,0.5233]$ on the line
and in $[0.4858,0.5467]$ off it, and on the line it approaches $\tfrac12$ from
above as $T$ grows: the product $T\bigl(\hat d_1(0,T)-\tfrac12\bigr)$ reads
$0.0354,\,0.0365,\,0.0370,\,0.0373$ at $T=10.5,\,20.5,\,40.5,\,80.5$. The
nearest-integer $j_0$ was the paired link in $481$ of the $500$ line samples;
at the other $19$ it was $j_0-1$, the third line of \eqref{eq:crossing-rule},
and at each of those the crossing sat at $\hat d_2\ge0.98$, at the far end of
the reverse link and about to hand over. In \texttt{check\_first\_link.py}.

The two cuts are mirrors on the critical line: the fractions $0.012402$ and
$0.505537$ of \eqref{eq:zeta-first-rev-link} are the pair of
\eqref{eq:zeta-first-link} exchanged, because the chain mirror
\eqref{eq:chain-mirror} carries forward link $n$ to reverse link $n$ and so
carries the one crossing to the other. Off the critical line the two formulas
are independent, the mirror being what fails there: at $\sigma=\tfrac14$ and
the same height the partner of reverse link $0$ is forward link $43$, with
fractions $0.9721$ and $0.5141$, while the partner of forward link $0$ stays
at reverse link $44$. The two cuts, which name one integer on the line, name
two off it, and which of them drops turns over at $\sigma=\tfrac12$: at
$\sigma=\tfrac34$ and $\sigma=\tfrac9{10}$ it is the reverse cutoff that reads
$43$ and the forward one $44$.

\subsection{Left unproved}\label{sec:unproved}

Section~\ref{sec:general-yinyang} mixes identities that are proved,
identities that are written down and then checked, and readings that are
only measurements. Two
statements about the crossings we would like to prove and have not.

\begin{enumerate}
\item
Every forward link is crossed by a reverse link, and the partner is the
summand $n'=a^{2}/n$ of \eqref{eq:crossing-product}
(\S\ref{sec:links-crossing}).
The product law is motivated by a stationary-phase sweep and then
checked on $118{,}007$ samples.
Proposition~\ref{prop:crossing-angle} settles the direction: the named
partner is never parallel to the link and never more than a quarter turn
from antiparallel, because $nn'$ is an integer. What is missing is the
position. A theorem would have to show that the
two polylines meet, not only that the named integer usually works, and
that needs the sweep \eqref{eq:crossing-saddle} with an effective error
term rather than as a stationary-phase heuristic.

\item
The pairing is an involution: $C_{s}(C_{s}(n,T),T)=n$.
Claimed as the symmetry of \eqref{eq:crossing-product}; if the previous
item is proved, this is immediate.
\end{enumerate}

\section{\texorpdfstring{$\vartheta_1$, $\vartheta_2$}{theta1, theta2} and
the zero-counting function}\label{sec:counting}

We direct our attention back to the legs of \S\ref{sec:legs} and recall the
two leg angles $\vartheta_1$ and $\vartheta_2$. The angle $\vartheta_1$ is
the argument of the bisector point, $B_{1ps}$: it is the angle that Leg~1,
running from the origin to that point,
makes with the real axis. The angle $\vartheta_2$ is the fold angle at the
bisector point: the angle from Leg~1 to Leg~2,
$\vartheta_2=\arg(B_2/B_1)$, recording how far Leg~2 is folded back against
Leg~1. Now isolate $\vartheta_2$ by laying Leg~1 along the $x$-axis, so that
our frame of reference is attached to Leg~1: in that frame we would observe
$\vartheta_2$ swing Leg~2 around while everything else stays stationary. Of
course $\zeta$ sits at the far end of Leg~2, and both legs have the same
length when $\sigma=\tfrac12$, so every time $\vartheta_2$ passes $\pi$,
Leg~2 folds straight back along Leg~1, its far end lands on the origin, and
a zeta zero occurs. Therefore counting the turns of $\vartheta_2$ would give
us a perfect zeta-zero counting formula.

\begin{figure}[H]
\centering
\includegraphics[width=0.9\textwidth]{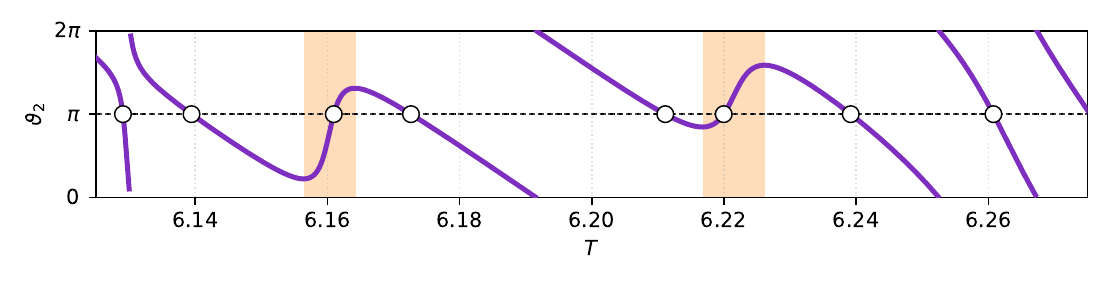}
\caption{The fold angle $\vartheta_2$ of the $ps$ split on the critical line
over $6.125\leq T\leq6.275$, reduced to $[0,2\pi)$.
Open circles mark the zeta zeros, where $\vartheta_2$ passes through $\pi$
(dashed); the shaded stretches are where $\vartheta_2$ runs retrograde, in
the shading of Figure~\ref{fig:count-theta2-ovals}, which returns to this
window in detail. Generated by \texttt{fig\_theta2\_retro\_strip.py}.}
\label{fig:theta2-retro-strip}
\end{figure}

The problem is that $\vartheta_2$
sometimes reverses its motion, and during those times the count would be off
by one. We call the times when $\vartheta_2$ reverses direction
\emph{retrograde}. We chose that name deliberately. When Claudius Ptolemy
observed a planet's orbit from the earth and saw it run backwards in the sky
(retrograde), he kept the earth at the center of the world frame, and so had
to keep the complex epicycles to describe the orbits. We do not want to do
that. We want instead to change our frame of reference to the sun and make
our observation from there, where it is easier to understand that the orbits
are not complicated: they are just ellipses. That is how we fix the
$\vartheta_2$ retrograde problem. We pick a different bisector point, which we call $B_1^{\ast}$ below, with a
different $\vartheta_2$ that never goes retrograde, and from \emph{that}
$\vartheta_2$ we build a new counting formula, which we call $N_{\ast}$.
Unlike $N$, $N_{\ast}$ is continuous and always increasing (the increase is
conditional on RH; Remark~\ref{rem:nstar-mono} makes this precise). And also
unlike $N$, it counts only the zeros on $\sigma=\tfrac12$, while $N$ counts
all the zeros: if there are any off the line, $N$ will count them and
$N_{\ast}$ will not. We think these properties of $N_{\ast}$ will prove
useful. These properties, and the zeta sum's geometry from various frames of
reference, deserve further study.

\subsection{Review: the Riemann--von Mangoldt formula and
\texorpdfstring{$N_{\mathrm{ps}}$}{Nps}}\label{sec:counting-review}

The well-known Riemann--von Mangoldt formula for the number of zeros of $\zeta$ with
imaginary part in $(0,t]$, in its exact form, reads
\begin{equation}\label{eq:N-theta}
N(t) \;=\; \frac{\vartheta(t)}{\pi} \;+\; 1 \;+\; S(t),
\qquad
S(t) \;=\; \frac{1}{\pi}\arg\zeta\!\left(\tfrac12+it\right),
\end{equation}
with $\vartheta$ the Riemann--Siegel theta function of Appendix~\ref{sec:phase-bound}.
The argument here is the one fixed by continuous variation along the two
straight segments from $2$ to $2+it$ and then to $\tfrac12+it$, starting from
$\arg\zeta(2)=0$. The branch has to be named: along any one path $\arg\zeta$ is
determined only up to a multiple of $2\pi$, and $S$ only up to an even integer,
so without the convention \eqref{eq:N-theta} would fix $N$ only modulo $2$.
This is an identity (\cite[Theorem~9.3]{Titchmarsh1986}), obtained by applying the argument
principle to $\xi$ around a rectangle: $\vartheta(t)/\pi+1$ is the contribution of
the $\Gamma$ and $\pi^{-s/2}$ factors, and $S(t)$ is the contribution of $\zeta$
itself along the critical line. Stated in terms of $\vartheta$ the formula is
exact: \eqref{eq:N-theta} carries no error term at all; approximation enters
only when $\vartheta$ is replaced by its expansion, and inexactness only when $S$
is dropped. A word on usage before the counting begins: here and in the rest
of the paper, an \emph{ordinate} always means the ordinate of a zeta zero on
the critical line, the imaginary part $\gamma$ of a zero of
$\zeta(\tfrac12+it)$, equivalently a real zero of $Z$, with
$\gamma_1<\gamma_2<\cdots$ numbering them up the line; under the Riemann
hypothesis these are all the nontrivial zeros, and unconditionally they are
the ones $N_{\ast}$ counts. On the index
the ordinate $\gamma_n$ sits at $T=I^{-1}(\gamma_n)$.
Figure~\ref{fig:staircase-67} shows the two pieces over one unit of
the index.

\begin{figure}[htbp]
\centering
\includegraphics[width=0.92\textwidth]{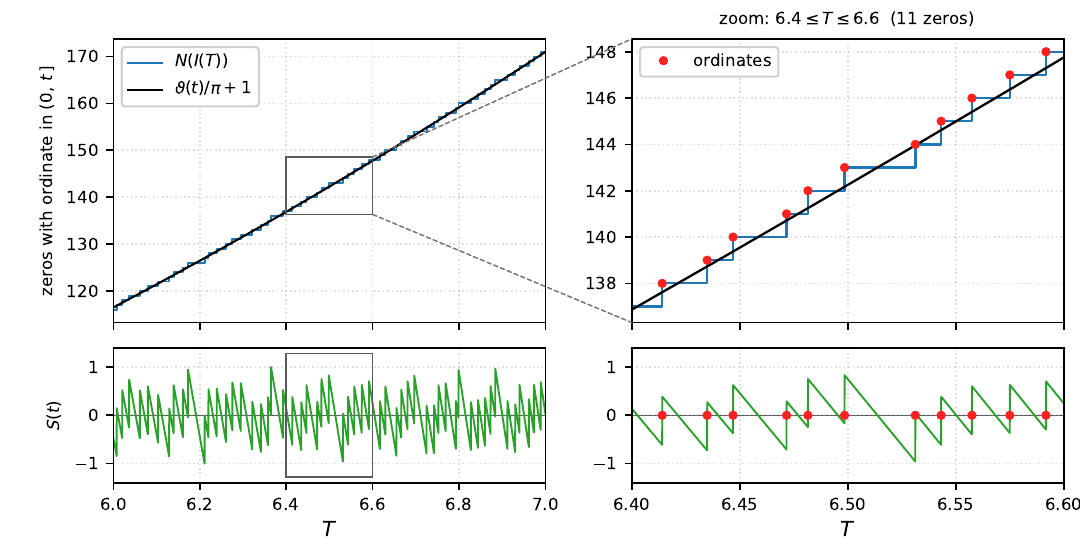}
\caption{\emph{Left:} the staircase $N(I(T))$ (blue) over $6\leq T\leq7$, which
is $264.9\leq t\leq352.9$ and carries the $55$ ordinates $\gamma_{117}$ through
$\gamma_{171}$, against the smooth part $\vartheta(t)/\pi+1$ of
\eqref{eq:N-theta} (black). The vertical gap between them is $S$, which here
stays within $[-1.002,+1.001]$. \emph{Right:} the boxed window
$6.4\leq T\leq6.6$ magnified, with the eleven ordinates it holds,
$\gamma_{138}$ through $\gamma_{148}$, marked in red at the top of each riser
and placed at $T=I^{-1}(\gamma)$. The risers are unevenly spaced while the
smooth part climbs steadily. Below each panel is the gap itself, $S$ (green),
which at this magnification shows plainly as a sawtooth: a unit jump at every
ordinate and a steady slide between them. Generated by \texttt{fig\_staircase\_67.py}.}
\label{fig:staircase-67}
\end{figure}

\paragraph{All of the discontinuity is in $S$.} $N$ is a step function and
$\vartheta(t)/\pi+1$ is analytic, so $S$ carries every jump: at a height
holding zeros it climbs by their total multiplicity, which is one at a simple
ordinate and two at a symmetric off-line pair $\rho,1-\bar\rho$, the two
sharing an imaginary part. In between it falls smoothly, at the rate
\begin{equation}\label{eq:S-slope}
S'(t) \;=\; -\frac{\vartheta'(t)}{\pi} \;=\; -\frac{1}{2\pi}\log\frac{t}{2\pi}
\;+\;O(t^{-2}),
\end{equation}
which is minus the mean zero density. Across an ordinate $S$ jumps up by one (at a
simple zero);
between ordinates it slides back down along that line. The sawtooth is why $S$ has mean zero without being small: each unit jump is
repaid by a decline whose steepness is the density itself. At an ordinate the
definition of $S$ fails, since $\zeta$ vanishes there and its argument is
undefined; one sets $S(\gamma)=\tfrac12[S(\gamma^{-})+S(\gamma^{+})]$. Then
\eqref{eq:N-theta} holds at every $t>0$, reading $N$ the same way at an
ordinate, as $\tfrac12[N(\gamma^{-})+N(\gamma^{+})]$, half a unit below the
count of $(0,\gamma]$; at every other height $N$ is that count outright.

\paragraph{Three counting curves against the staircase.} Expanding $\vartheta$ by
Stirling and dropping $S$ turns \eqref{eq:N-theta} into the Riemann--von
Mangoldt formula in its approximate form,
\begin{equation}\label{eq:N-rvm}
N_{\mathrm{rvm}}(t)
=\frac{t}{2\pi}\log\frac{t}{2\pi}-\frac{t}{2\pi}+\frac78
+\frac{1}{48\pi t}+\frac{7}{5760\pi t^{3}}
\end{equation}
(the $\tfrac78$ collects the $-\tfrac\pi8$ of the expansion, divided by $\pi$,
with the standalone $1$ of \eqref{eq:N-theta}). Nothing of consequence is
discarded in truncating there: at $t=1000$ even the first correction retained,
$1/(48\pi t)$, is only $6.6\times10^{-6}$, while $S(1000)=+0.384$. The whole of
the error that matters is $S$. Set also
\begin{equation}\label{eq:N-ps}
N_{\mathrm{ps}}(T)
=\frac{1}{\pi}\vartheta\bigl(I(T)\bigr)
+\frac{1}{\pi}\vartheta_1\bigl(\tfrac12,T\bigr)+\frac32 ,
\end{equation}
which replaces $S$ by the angle $\vartheta_1$ that leg~1 makes with the real
axis (\S\ref{sec:legs}), the leg running from the origin to the bisector point
$B_1=\Sigma_1+R_{1ps}$. These two behave quite differently against the true
count $N(I(T))$, and a third curve $N_{\ast}$, built below in
\eqref{eq:N-star} from a different splitting of $\zeta$, repairs the one defect
of $N_{\mathrm{ps}}$.

$N_{\mathrm{rvm}}$ is analytic and increasing, so it can only approximate the
staircase in the mean, never landing on an integer at an ordinate, and its
error is exactly the sawtooth $S$.
$N_{\mathrm{ps}}$ instead meets an integer \emph{exactly} at every ordinate,
and only there.

\begin{proposition}\label{prop:N-ps-integer}
On $\sigma=\tfrac12$, at every $T$ with $B_1(T)\neq0$, so that
$\vartheta_1=\arg B_1$ and hence $N_{\mathrm{ps}}$ is defined, the value
$N_{\mathrm{ps}}(T)$ is an integer if and only if $t=I(T)$ is an ordinate.
\end{proposition}

The reason is already Corollary~\ref{cor:bisector-proj}: on the critical line
\begin{equation}\label{eq:B1-rotated}
B_1\,e^{i\vartheta(t)}=\tfrac12 Z(t)+ih,\qquad h\in\mathbb R.
\end{equation}
The real part is
$\lvert B_1\rvert\cos(\vartheta+\vartheta_1)=\tfrac12 Z$.
That cosine is zero exactly when
$\vartheta+\vartheta_1\equiv\tfrac\pi2\pmod\pi$, which is exactly when
$N_{\mathrm{ps}}=(\vartheta+\vartheta_1)/\pi+\tfrac32$ is an integer.
So $N_{\mathrm{ps}}$ is an integer if and only if $Z=0$. (The cosine step
needs $B_1\neq0$: since $|B_1|^{2}=\tfrac14Z^{2}+h^{2}$, that would require
$Z$ and the offset $h$ of \eqref{eq:B1-rotated} to vanish together, which
happens nowhere in the ranges we computed.)
Multiplying by $e^{i\vartheta}$ is the $\vartheta$-frame of the glossary:
it is the rotation $e^{-i\psi/2}$ of Corollary~\ref{cor:bisector-proj}
written with $\vartheta=-\tfrac{\psi}{2}$, since $\chi=e^{-2i\vartheta}$
on the line, and it puts the base $O\to\zeta$ on the real axis so that a
split point reads as $\tfrac12 Z+ih$.
This does not say $N_{\mathrm{ps}}$ equals the staircase $N$: retrograde
ordinates still land on the wrong integer, as the next paragraph
explains. It only says the value is an integer there, and nowhere else.

\paragraph{Where $N_{\mathrm{ps}}$ misses, and why.} Two things keep
\eqref{eq:N-ps} from being a counting function. First, $\vartheta_1$ is an
angle, so \eqref{eq:N-ps} is determined only up to an even integer, and each
passage of $\vartheta_1$ through $\pm\pi$ breaks the curve; there are nine such
wraps over $3<T<5$, and $\tfrac32$ is the constant that puts the curve on the
staircase. Second, and this is the real defect, the curve sometimes meets the
wrong integer. Differentiating \eqref{eq:B1-rotated} at an ordinate, where $Z$
vanishes, gives
\begin{equation}\label{eq:crossing-rate}
\frac{d}{dt}\arg\bigl(B_1e^{i\vartheta}\bigr)\bigg|_{Z=0}=-\frac{Z'(t)}{2h},
\end{equation}
so the crossing carries the angle forward or backward with the sign of $-hZ'$.
Consecutive simple zeros of $Z$ alternate in the sign of $Z'$, so the crossings
all turn the same way exactly when $h$ alternates in sign too. Where $h$ nearly
vanishes it typically keeps its sign across two ordinates instead; the angle
turns back on itself, the count fails to advance, and the second of the pair
reads one too high. Over the first $400$ ordinates $|h|$ has median $1.081$ at the $350$ that
$N_{\mathrm{ps}}$ places correctly and $0.319$ at the $50$ it misses, and those
$50$ are exactly the crossings that run backwards, which we call
\emph{retrograde}. They are also exactly the
wraps: the nine wraps of $\vartheta_1$ over $3<T<5$ counted above sit one apiece
on the nine retrograde ordinates of that range. Figure~\ref{fig:nps-staircase}
shows the cost over a longer stretch: five retrograde ordinates in
$5.9<T<6.7$, and after each one $N_{\mathrm{ps}}$ runs a unit above the
staircase until the following ordinate restores it.
That restoration is the usual local picture, not a law; the companion
note~\cite{NStar} records the scan through $T=100$.
The two ways of describing the
failure, a broken branch and a backward crossing, are one event. Over $3<T<5$, which holds
$55$ zeros, both curves stay within one of the staircase all the same: the
largest deviation is $0.946$ for $N_{\mathrm{rvm}}$ and $0.999$ for
$N_{\mathrm{ps}}$.

Reading $\vartheta_2$ in place of $\vartheta_1$ cannot help. On the critical
line $B_2=\chi\overline{B_1}$ and $\arg\chi=-2\vartheta$, so
\begin{equation}\label{eq:th2-th1}
\vartheta_2=-2\bigl(\vartheta+\vartheta_1\bigr)\pmod{2\pi},
\end{equation}
whichever remainder defines the split. Every sawtooth built from $\vartheta_2$
is therefore the same function as the one built from $\vartheta_1$ and
retrogrades in the same places: among the $62$ ordinates of $3.0\leq T\leq5.2$
the splits at $R_{1ps}$, at $R/2$ and at $R_{1ak}$ run backwards at $10$, $9$
and $12$ of them. The retrograde belongs to the split, not to the branch.

\subsection{Improving \texorpdfstring{$N_{\mathrm{ps}}$}{Nps}: the velocity
split and a curve that lands on every ordinate}\label{sec:counting-star}

The cure is to split $\zeta$ so that the offset $h$ of \eqref{eq:B1-rotated}
alternates by construction. With $\vartheta'$ as in \eqref{eq:S-slope}, exactly
$\tfrac12\operatorname{Re}\tfrac{\Gamma'}{\Gamma}\bigl(\tfrac14+\tfrac{it}{2}\bigr)-\tfrac12\log\pi$,
put
\begin{equation}\label{eq:vel-split}
B_1^{\ast}=\zeta(s)+\frac{\zeta'(s)}{2\vartheta'(t)},\qquad
B_2^{\ast}=\zeta(s)-B_1^{\ast}=-\frac{\zeta'(s)}{2\vartheta'(t)}
=\frac{i}{2\vartheta'(t)}\frac{d\zeta}{dt},
\end{equation}
so leg~2 is the velocity of the chain's endpoint along the line, quarter-turned
and scaled by $1/2\vartheta'$. The companion $B_1^{\ast}$ already appears in
the literature as the function Levinson calls $G$ at his (1.9), where the weight is written in the
gamma-factor form $f'(s)+f'(1-s)$, which on the line is this same real number
$2\vartheta'(t)$~\cite{Levinson1974}; \S\ref{sec:levinson} sets Levinson's use
of $G$ beside ours.

On $\sigma=\tfrac12$ this split is again
reflection-symmetric, $B_2^{\ast}=\chi\overline{B_1^{\ast}}$, so its legs are
exactly equal at every height,
\begin{equation}\label{eq:vel-legs}
\Bigl|\,\zeta+\frac{\zeta'}{2\vartheta'}\Bigr|=\frac{|\zeta'|}{2|\vartheta'|},
\end{equation}
which is the identity $\operatorname{Re}\bigl(\zeta\overline{\zeta'}\bigr)
=-\vartheta'|\zeta|^{2}$, itself no more than the statement that
$Z=e^{i\vartheta}\zeta$ is real. Differentiating that statement gives
$e^{i\vartheta}\zeta'=-iZ'-\vartheta'Z$, and hence \eqref{eq:B1-rotated} holds for the
new split with
\begin{equation}\label{eq:h-star}
e^{i\vartheta}B_1^{\ast}=\frac{Z}{2}-\frac{i\,Z'}{2\vartheta'},
\qquad h^{\ast}=-\frac{Z'}{2\vartheta'} .
\end{equation}
Now $h^{\ast}$ alternates because $Z'$ does, and \eqref{eq:crossing-rate}
returns $-Z'/2h^{\ast}=\vartheta'$: every crossing turns forward, and at exactly
the mean density. Figure~\ref{fig:b1-star-winding} draws the frame of
\eqref{eq:h-star} over one unit of the index: the path of
$e^{i\vartheta}B_1^{\ast}$ loops around the origin once per ordinate, crossing
the imaginary axis exactly at the ordinates, and it is that winding we now
turn into a counting curve. Figure~\ref{fig:b1-star-winding-lehmer} repeats
the drawing on the hardest ground the count meets this low in the strip: the
window $32.87\leq T\leq32.91$ holding the Lehmer pair
$\gamma_{6709}=7005.063$, $\gamma_{6710}=7005.101$, whose two crossings
arrive only $9\times10^{-5}$ index units apart. The path still winds once
per ordinate; between the two crossings of the pair it passes within
$|B_1^{\ast}|=2.0\times10^{-3}$ of the origin, exactly half the famous
shallow extremum $Z=0.00397$ of Lehmer's phenomenon, since $h^{\ast}$
vanishes with $Z'$ right where $Z$ turns.

\begin{figure}[H]
\centering
\includegraphics[width=0.72\textwidth]{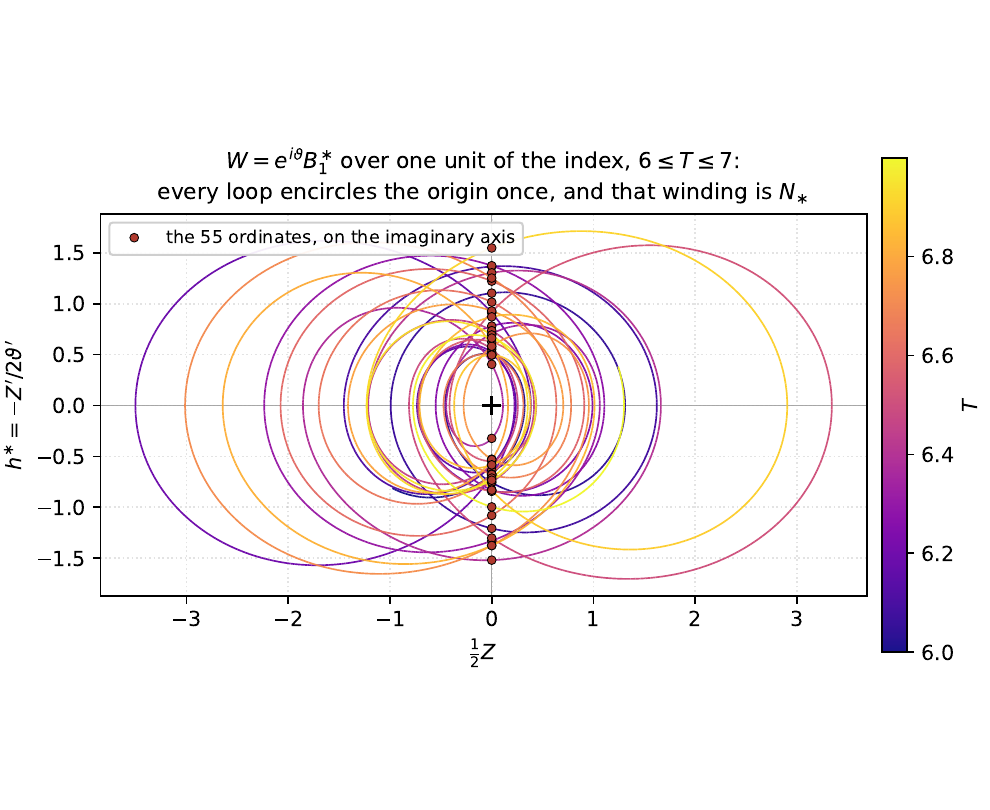}
\caption{The velocity-split companion in the rotated frame of
\eqref{eq:h-star}, $W=e^{i\vartheta}B_1^{\ast}=\tfrac12Z-iZ'/(2\vartheta')$,
over one unit of the index, $6\leq T\leq7$ ($264.9\leq t\leq352.9$), the
color running with $T$. The horizontal coordinate is $\tfrac12Z$ and the
vertical one is the offset $h^{\ast}=-Z'/(2\vartheta')$. Every loop encircles
the origin (black cross) exactly once, and the path crosses the imaginary
axis exactly at the $55$ ordinates of the unit (red dots): that winding is
the counting curve $N_{\ast}$ of \eqref{eq:N-star-prufer}.
Generated by \texttt{fig\_b1\_star\_winding.py}.}
\label{fig:b1-star-winding}
\end{figure}

\begin{figure}[H]
\centering
\includegraphics[width=0.72\textwidth]{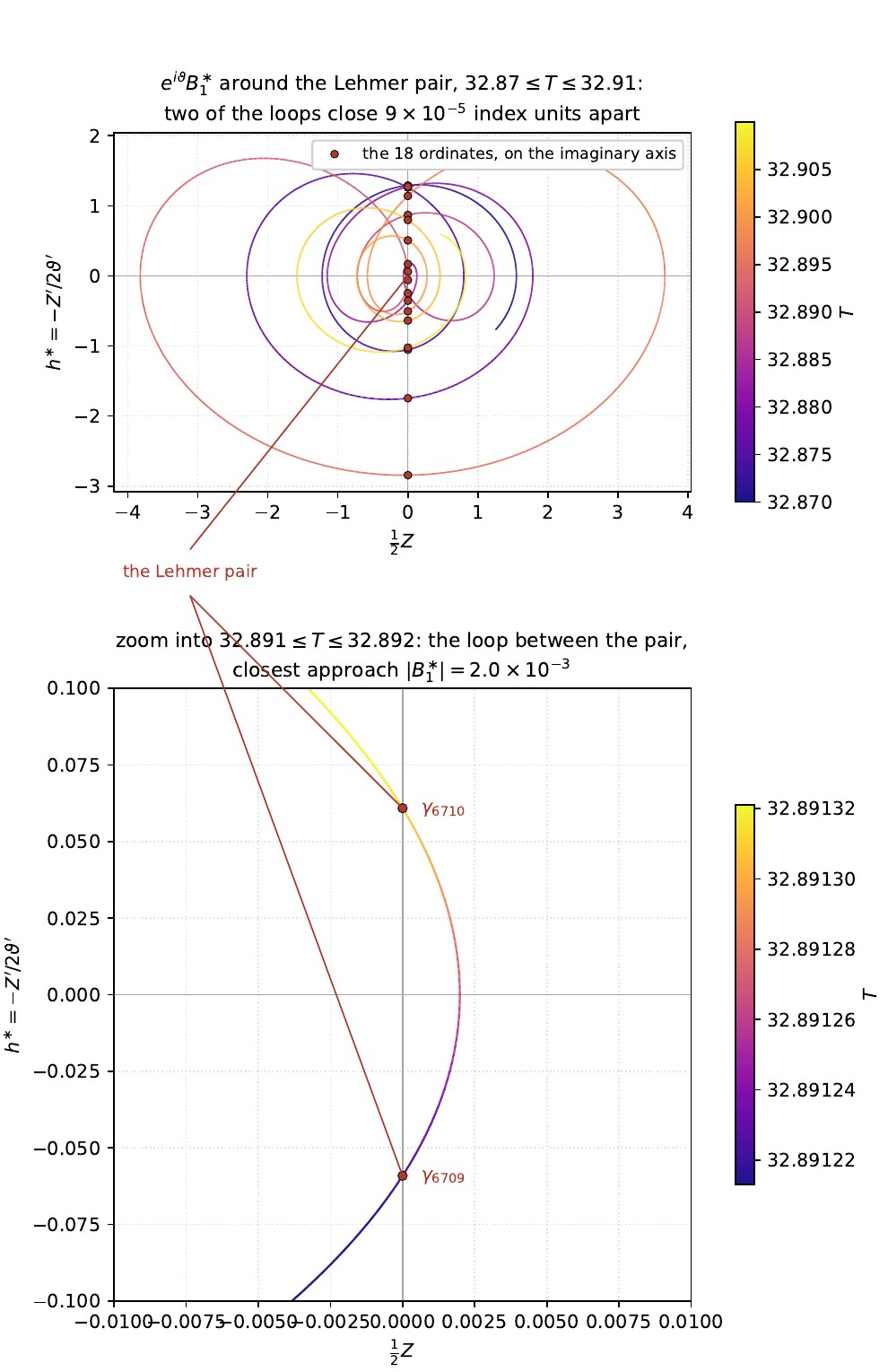}
\caption{The same frame as Figure~\ref{fig:b1-star-winding} over
$32.87\leq T\leq32.91$ ($6996.2\leq t\leq7012.9$), the window around the
Lehmer pair $\gamma_{6709},\gamma_{6710}$. \emph{Top:} the $18$ ordinates of
the window sit on the imaginary axis as before, but the two crossings of the
pair (arrow) land almost on top of one another, $9\times10^{-5}$ index units
apart, and the loop that separates them is nearly invisible at this scale.
\emph{Bottom:} zoom into $32.891\leq T\leq32.892$, the thousandth of a unit
holding the pair, on a horizontal scale stretched ten to one, with the
color scale narrowed to the $9\times10^{-5}$ index units just bracketing
the pair. The two crossings are the red dots at $h^{\ast}=\pm0.06$, and
between them the path bows past the origin at
$|B_1^{\ast}|=2.0\times10^{-3}$, half the shallow extremum $Z=0.00397$ that
makes this pair famous. The winding still advances by exactly one per
ordinate, so $N_{\ast}$ counts the pair correctly.
Generated by \texttt{fig\_b1\_star\_winding\_lehmer.py}.}
\label{fig:b1-star-winding-lehmer}
\end{figure}

We set
\begin{equation}\label{eq:N-star}
N_{\ast}(T)=\frac{1}{\pi}\vartheta\bigl(I(T)\bigr)
+\frac{1}{\pi}\vartheta_1^{\ast}(T)+\frac32,
\qquad \vartheta_1^{\ast}=\arg B_1^{\ast}.
\end{equation}
That is the Riemann--Siegel phase plus the argument of the companion
$B_1^{\ast}=\zeta+\zeta'/(2\vartheta')$, or equivalently the Pr\"ufer angle
of $(Z,\,-Z'/\vartheta')$\footnote{What we have not found in the literature
is the argument of this companion used this way, as an interpolant of the
count. The nearest classical
objects are the exact count $N=\vartheta/\pi+1+S$ of \eqref{eq:N-theta} and the
Pr\"ufer phase of an oscillator~\cite{Prufer1926}.}:
\begin{equation}\label{eq:N-star-prufer}
\boxed{\;
N_{\ast}(T)
=\frac{1}{\pi}\arg\Bigl(Z-\frac{iZ'}{\vartheta'}\Bigr)+\frac32.
\;}
\end{equation}
The argument is the continuous lift, not the principal value: by
\eqref{eq:h-star} it is $\vartheta+\vartheta_1^{\ast}$ of \eqref{eq:N-star},
whose size is carried by $\vartheta$, and it is normalized so that
$N_{\ast}=1$ at $\gamma_1$, which puts the lift at $-\tfrac\pi2$ there.
A principal argument would confine the whole right side to
$(\tfrac12,\tfrac52]$.
The Pr\"ufer angle is Sturm--Liouville theory's device for counting the zeros
of a real oscillator $y(t)$: plot the point $\bigl(y,\,y'/\omega\bigr)$ in the
plane, with $\omega$ the oscillator's local frequency, and track its polar
angle on a continuous branch. The point crosses the vertical axis exactly
where $y$ vanishes, so the angle passes a multiple of $\pi$ at each zero and
its winding is the count; dividing $y'$ by $\omega$ only rounds the orbit,
so the angle advances steadily instead of in bursts. In
\eqref{eq:N-star-prufer} the oscillator is $Z$ and the frequency is
$\vartheta'$, the mean zero density.
The resulting curve is integer-valued at the ordinates and nowhere else, since
$\operatorname{Re}\bigl(e^{i\vartheta}B_1^{\ast}\bigr)=\tfrac12Z$ vanishes only
there, and that advances by exactly one from each ordinate to the next. It lands
on every one of the first $400$ ordinates, $N_{\ast}(I^{-1}(\gamma_n))=n$ to
$2.8\times10^{-13}$, against $50$ misses for $N_{\mathrm{ps}}$; over
$0.6\leq T\leq8$ it is strictly increasing on a grid of $3000$ points, no wrap
occurs ($|\vartheta_1^{\ast}|$ stays below $0.60\pi$), and it rises by $237.64$
across the $237$ zeros of that range. Height costs it nothing and no branch
correction is needed there either: at $\gamma_{10000}=9877.78$ and at
$\gamma_{50000}=40433.69$ the same formula with the same constant $\tfrac32$
still returns exactly $10000$ and $50000$. Where $N_{\mathrm{rvm}}$ threads the
staircase without touching it, and $N_{\mathrm{ps}}$ locates zeros without
counting them, $N_{\ast}$ does both. Like any count read off $Z$ it sees the
zeros on the line, which in the range plotted are all of them.

\begin{figure}[H]
\centering
\includegraphics[width=\textwidth]{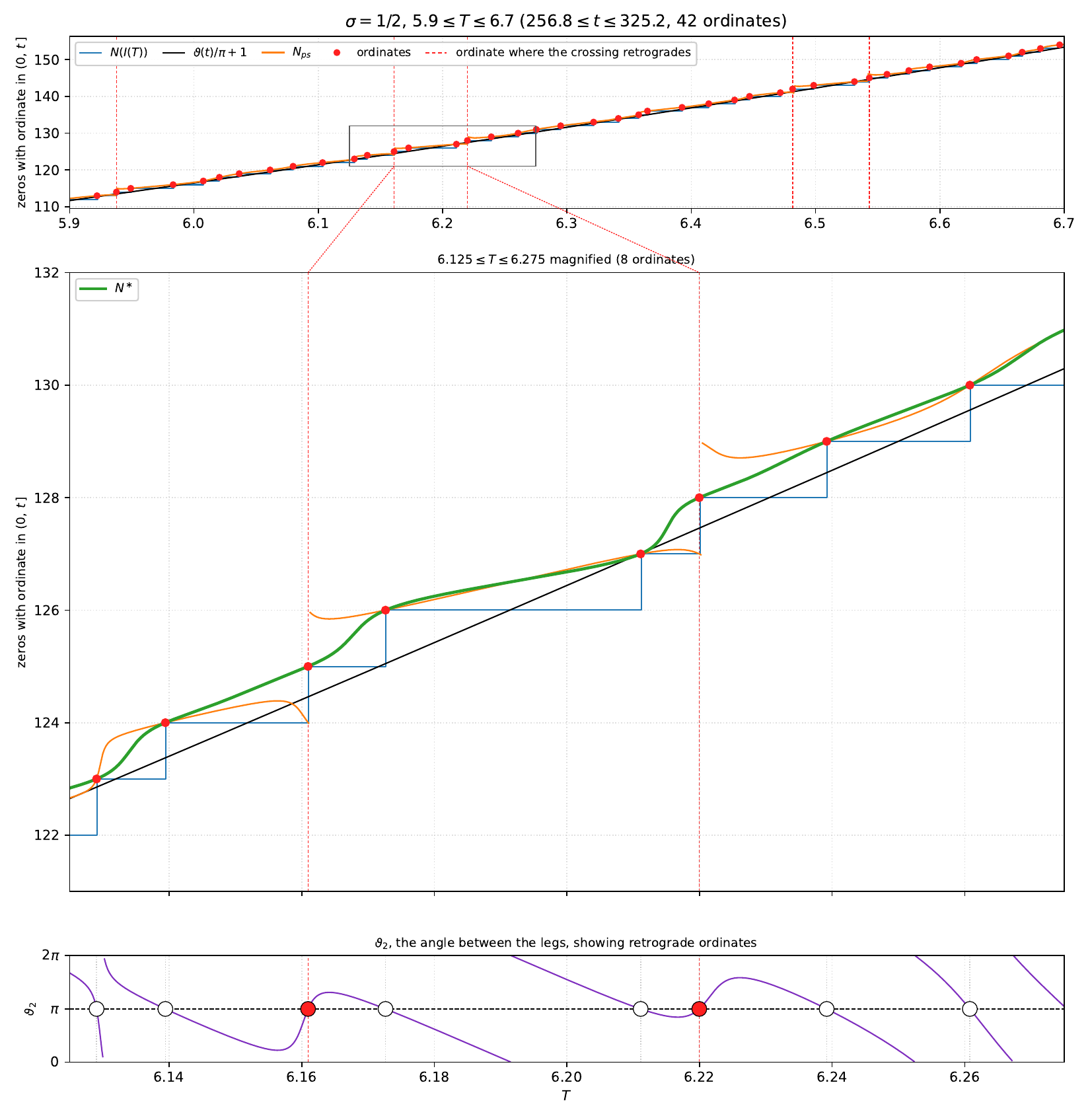}
\caption{Staircase, in \textcolor[HTML]{1F77B4}{blue}, is the exact
Riemann--von Mangoldt formula
\eqref{eq:N-theta}, and the approximate Riemann--von Mangoldt formula
\eqref{eq:N-rvm} is the black curve, that formula with $S$ dropped.
\emph{Top:} $5.9\leq T\leq6.7$, which is $256.8\leq t\leq325.2$ and holds the
$42$ ordinates $\gamma_{113}$ through $\gamma_{154}$, five of them retrograde,
with $N_{\mathrm{ps}}$ of \eqref{eq:N-ps} in
\textcolor[HTML]{FF7F0E}{orange}, broken at its wraps.
\emph{Middle:} the boxed window $6.125\leq T\leq6.275$ magnified, holding $8$
ordinates of which $\gamma_{125}$ and $\gamma_{128}$ are retrograde. The
velocity split $N^{\ast}$ of \eqref{eq:N-star} is drawn in
\textcolor[HTML]{2CA02C}{green}.
\emph{Bottom:} the fold angle of the legs $\vartheta_2$ of \S\ref{sec:legs} on the same
$T$ range, reduced to $[0,2\pi)$. Every passage through $\pi$ (dashed) is at an
ordinate; the two filled dots are the two retrograde ordinates of the panel
above, where $\vartheta_2$ rises through $\pi$, turns around, and comes back
down through it.
Generated by \texttt{fig\_nps\_staircase.py}.}
\label{fig:nps-staircase}
\end{figure}

\begin{remark}[the two offsets $h_{R/2}$ and $h^{\ast}$]\label{rem:taper}
The offsets $h_{R/2}$ and $h^{\ast}$ of the $R/2$ split and the velocity split
are both drawn in Figure~\ref{fig:splits-h}, and they differ only by a
weighting of the links. The velocity offset has the closed form of
\eqref{eq:h-star},
\begin{equation}\label{eq:h-star-slope}
h^{\ast}(t)=-\frac{Z'(t)}{2\,\vartheta'(t)},
\end{equation}
the slope of $Z$ in units of $2\vartheta'$: $h^{\ast}$ vanishes exactly at the
critical points of $Z$ and is largest where $Z$ is steepest. For the $R/2$ split
$e^{i\vartheta}R$ is real, so $h$ is the sine companion of the main sum in the
$\vartheta$-frame,
$h_{R/2}=\operatorname{Im}\bigl(e^{i\vartheta}\Sigma_1\bigr)
=\sum_{n\leq m}n^{-1/2}\sin(\vartheta-t\log n)$, while differentiating the real
part of that rotated main sum and dividing by $-\vartheta'$ gives
\[
h^{\ast}=\sum_{n\leq m}n^{-1/2}
\Bigl(1-\frac{\log n}{\vartheta'}\Bigr)\sin\bigl(\vartheta-t\log n\bigr)+\cdots,
\]
the tail being the remainder's derivative; the displayed sum by itself
reproduces $h^{\ast}$ to $2.3\times10^{-2}$ over the first $400$ ordinates.
The weight is the phase rate $\vartheta'-\log n$ of link $n$ measured in units of
$\vartheta'$, so it runs from $1$ on the first link down to nearly $0$ on the last,
where $\log m\approx\vartheta'$. The counting angle is carried by the same links;
each one contributes in proportion to how much of the rotation rate it retains.
\end{remark}

\begin{remark}[the velocity split off the critical line]\label{rem:vel-offline}
The weight $2\vartheta'(t)$ in \eqref{eq:vel-split} is exactly
$-(\chi'/\chi)(\tfrac12+it)$, since $\chi=e^{-2i\vartheta}$ on the line.
Replacing it by that logarithmic derivative extends the split to the whole
strip:
\begin{equation}\label{eq:vel-split-offline}
B_1^{\ast}(s)=\zeta(s)-\frac{\zeta'(s)}{(\chi'/\chi)(s)},
\qquad
B_2^{\ast}=\zeta-B_1^{\ast},
\qquad
\frac{\chi'}{\chi}(s)=\log 2\pi+\frac{\pi}{2}\cot\frac{\pi s}{2}-\psi(1-s)
\end{equation}
($\psi$ the digamma function, not the $\arg\chi$ of \S\ref{sec:decomp}).
This is again Levinson's $G$ and not an extension of it: his weight
$f'(s)+f'(1-s)$ equals $-(\chi'/\chi)(s)$ at every
$s$~\cite{Levinson1974}, the critical line being the place where it happens to
be real.
Applying the functional equation twice gives $\chi(s)\chi(1-s)=1$, whose
logarithmic derivative forces $(\chi'/\chi)(s)=(\chi'/\chi)(1-s)$;
differentiating $\zeta(s)=\chi(s)\zeta(1-s)$ then yields the exact reflection
\[
B_2^{\ast}(s)=\chi(s)\,B_1^{\ast}(1-s)
\]
at every $s$ in the strip, pairing the split at $(\sigma,T)$ with its mirror
image at $(1-\sigma,T)$. On the critical line the mirror point is the point
itself, the weight is the real number $2\vartheta'$, and the identity
collapses to $B_2^{\ast}=\chi\overline{B_1^{\ast}}$ with its equal legs
\eqref{eq:vel-legs}. Off the line the weight is complex (imaginary part
$\approx\mp7\times10^{-4}$ at $t\approx280$), the offset $h^{\ast}$ ceases to
be a real scalar, and, as for every split, the equal-leg property survives
only on the line. Verified numerically to twenty-four digits at $\sigma=0.3$.

The $\chi$-weight also separates the two chains exactly. Write
$\zeta=F+\chi\widetilde G$ with $F=\Sigma_1+R_1$ and
$\widetilde G=(\Sigma_2+R_2)/\chi$, for any split whose two halves are
holomorphic on the link interval, the $R/2$ and velocity splits among them:
the second half always carries the factor $\chi$, so peeling it off is
natural.
Differentiating, multiplying by $\chi/\chi'$, and subtracting cancels the
$\chi\widetilde G$ term entirely; on the critical line, where
$\chi'/\chi=-2\vartheta'$, the result is
\begin{equation}\label{eq:vel-split-halves}
B_1^{\ast}
=\Sigma_1+R_1+\frac{(\Sigma_1+R_1)'+\chi\,\widetilde G'}{2\vartheta'},
\qquad
B_2^{\ast}
=\Sigma_2+R_2-\frac{(\Sigma_1+R_1)'+\chi\,\widetilde G'}{2\vartheta'}.
\end{equation}
The velocity split therefore hands the whole forward chain $\Sigma_1+R_1$ to
leg~1 and the whole reflected chain $\Sigma_2+R_2$ to leg~2, exactly, with no
mixing at the level of the sums; the two legs trade only the derivative
corrections, damped by $2\vartheta'$. The reflected
half influences $B_1^{\ast}$ only through its rate of change
$\widetilde G'$, never through its value. This is why $B_1^{\ast}$ crowds
against the other split points in Figure~\ref{fig:splits-h}, and the taper
weights $1-\log n/\vartheta'$ of Remark~\ref{rem:taper} are this cancellation
read term by term. Verified to twenty-four digits at $T=6.18$.
\end{remark}

Incidentally, with the off-line definition \eqref{eq:vel-split-offline} in
hand, the zeros of $B_1^{\ast}$ in and around the strip can be pinned down
exactly by the argument principle: the winding number of $B_1^{\ast}$ around
the boundary of a box counts the zeros inside, and bisecting any box that
contains zeros localizes them to arbitrary precision. The winding counts
zeros and nothing else because $B_1^{\ast}$ is analytic on these boxes: its
only poles are at the zeros of $\chi'/\chi$, and that modulus stays above
$3.7$ throughout, while the pole of $\zeta$ at $s=1$ lies at $t=0$, far
below.
Figure~\ref{fig:b1-star-strip-zeros} shows the result. The box
$-1\leq\sigma\leq1$, $6\leq T\leq7$ contains exactly $40$ zeros of
$B_1^{\ast}$, every one strictly to the left of the critical line (the
rightmost at $\sigma\approx0.416$), and each root checks directly,
$|B_1^{\ast}|<10^{-10}$. Two features are worth
recording. First, no zero lies on or to the right of $\sigma=\tfrac12$,
consistent with the reflection $B_2^{\ast}(s)=\chi(s)B_1^{\ast}(1-s)$ placing
the mirror zeros to the right. Second, the window contains $55$ ordinates
but only $40$ zeros, and only $22$ of those sit inside the strip
$\sigma\in[0,1]$: the other $18$ have drifted left of $\sigma=0$, down to
$\sigma\approx-0.71$. Running the same winding count over the box
$-20\leq\sigma\leq-1$ turns up just $5$ more, all between
$\sigma\approx-1.5$ and $\sigma=-1$ (the count confirmed by an independent
densely sampled winding integral around the full boundary), for $45$ zeros
right of $\sigma=-20$. That is ten short of the $55$ ordinates of the
window. We have no theorem saying the two totals should agree, so whether ten
more lie further left, or the tally simply differs from the ordinate count,
we have not determined.

\begin{figure}[H]
\centering
\includegraphics[width=0.7\textwidth]{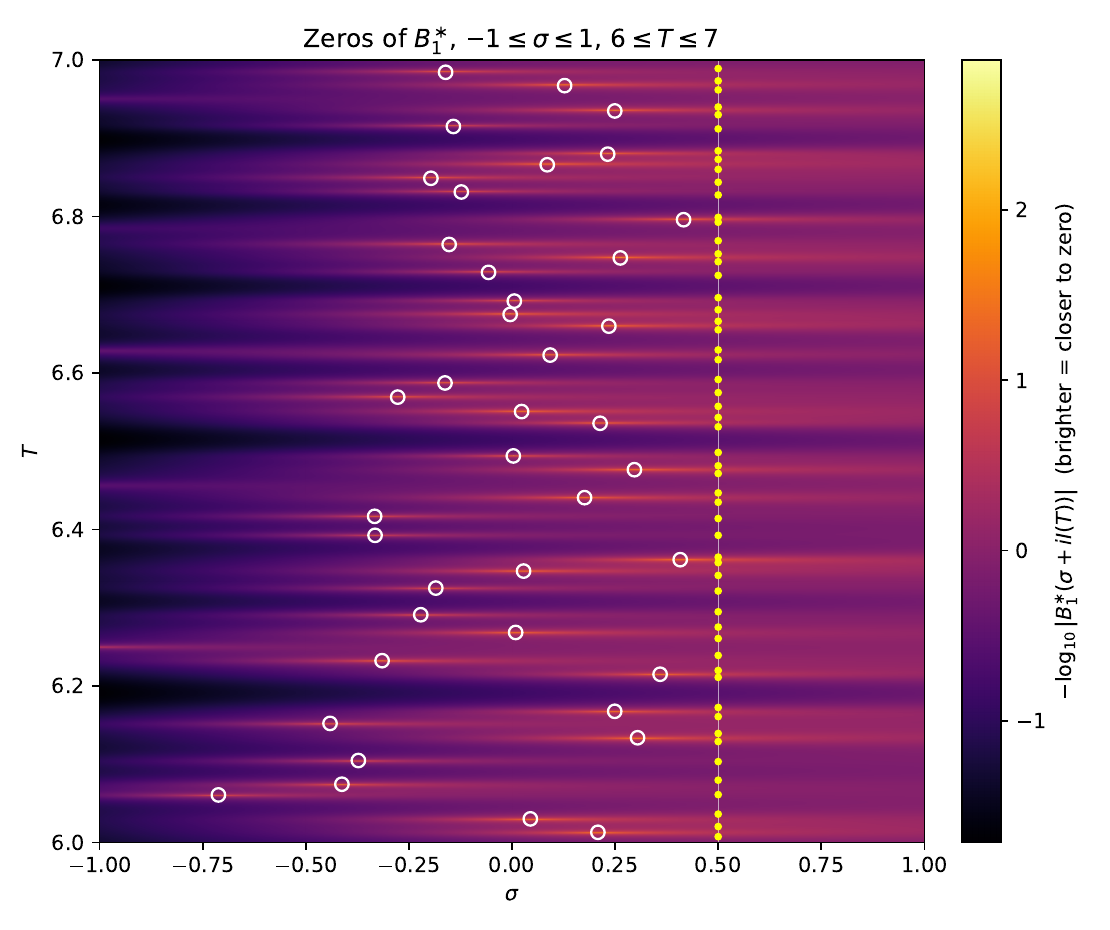}
\caption{The $40$ zeros of $B_1^{\ast}$ in the box $-1\leq\sigma\leq1$,
$6\leq T\leq7$ (white open circles), located by the argument principle:
winding-number counts over recursively bisected boxes, each root then
polished by Muller's method to $|B_1^{\ast}|<10^{-10}$. The background
heatmap, on a grid of $599$ steps in $\sigma\in[-1,1]$ and $5000$ steps in
$T$ with $t=I(T)$ and the off-line definition \eqref{eq:vel-split-offline},
colors each input point by how close the output is to zero,
$-\log_{10}|B_1^{\ast}(\sigma+iI(T))|$, brighter meaning closer; the thin
white vertical line is the critical line, the yellow dots on it are the
$55$ zeta zeros of the window, and the dark horizontal bands are
the ordinates' neighborhoods, where $|B_1^{\ast}|$ stays comparatively
large across the whole region. All $40$ zeros lie strictly left of the
critical line. Generated by \texttt{find\_b1\_star\_zeros.py}.}
\label{fig:b1-star-strip-zeros}
\end{figure}

\begin{figure}[H]
\centering
\includegraphics[width=\textwidth]{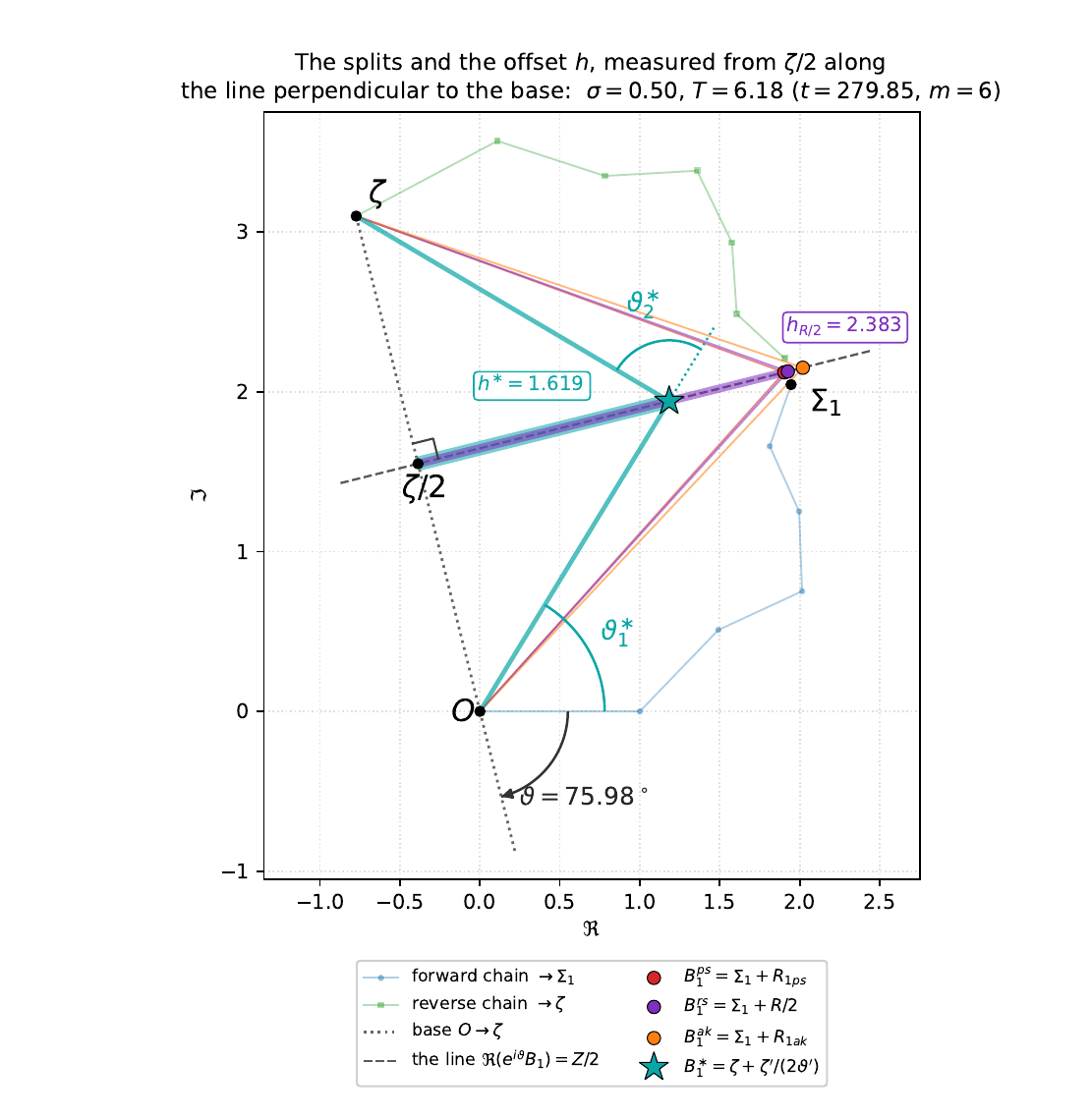}
\caption{The offset $h$ of \eqref{eq:B1-rotated}, drawn at $T=6.18$. Because
$\Re\bigl(e^{i\vartheta}B_1\bigr)=Z/2$ for every split, all of the split points lie
on one line when $\sigma=\tfrac12$, the dashed line through $\zeta/2$
perpendicular to the base $O\to\zeta$ (the right angle is marked). What
distinguishes one split from another is $h$, the signed distance from $\zeta/2$
along that line, drawn as a thick stretch of the line for the $R/2$ split
(purple, $h_{R/2}=2.383$) and for the velocity split (teal, $h^{\ast}=1.619$);
the two legs of each split are drawn in its own color, the velocity split's
legs thickest, and the two chains fainter still. The arc at $O$ is $\vartheta$, swept clockwise off the positive real
axis onto the base line, since $\arg\zeta\equiv-\vartheta\pmod{\pi}$: with
$\vartheta=390.884$ reducing to $75.98^{\circ}$ modulo $\pi$, the base line leaves
the axis at $-75.98^{\circ}$, drawn dotted below $O$, and $\zeta$ itself lies on
the far side of $O$ at $104.02^{\circ}$ because $Z=-3.194$ is negative at this
height. The two teal arcs mark the angles of the velocity split:
$\vartheta_1^{\ast}=\arg B_1^{\ast}$ at $O$, and the fold angle
$\vartheta_2^{\ast}=\arg\bigl(B_2^{\ast}/B_1^{\ast}\bigr)$ at $B_1^{\ast}$,
between the extension of leg~1 (dotted) and leg~2.
Generated by \texttt{fig\_splits\_h.py}.}
\label{fig:splits-h}
\end{figure}

\begin{remark}[$N_{\ast}$ is monotonically increasing if RH is true]
\label{rem:nstar-mono}
Integer at the ordinates and nowhere else, and up by one from each to the next:
what \eqref{eq:N-star} leaves open is whether the curve ever runs backwards in
between. It does not, and the reason is a known consequence of the Riemann
hypothesis. With $W=e^{i\vartheta}B_1^{\ast}=\tfrac12Z-\tfrac{iZ'}{2\vartheta'}$ from
\eqref{eq:h-star} and $\pi N_{\ast}=\arg W+\tfrac{3\pi}{2}$, differentiating
and measuring against the local mean density $\vartheta'$ gives the slope in closed
form,
\begin{equation}\label{eq:nstar-rho}
\rho_{\ast}=\frac{1}{\vartheta'}\,\frac{d}{dt}\arg W
=\frac{\pi}{\vartheta'\,I'(T)}\,\frac{dN_{\ast}}{dT}
=\frac{Z'^2-ZZ''+\dfrac{\vartheta''}{\vartheta'}\,ZZ'}{\vartheta'^2Z^2+Z'^2},
\end{equation}
which returns $1$ at $Z=0$ and $-Z''/\vartheta'^2Z$ at $Z'=0$, as it must. Since
$\vartheta'$ and $I'$ are positive, $\rho_{\ast}$ and $dN_{\ast}/dT$ carry the same sign;
the denominator of \eqref{eq:nstar-rho} is positive and
$\vartheta''/\vartheta'=O(1/t\log t)$, so up to that negligible term the curve
increases exactly where the Laguerre expression $Z'^2-ZZ''$ is positive.

That expression belongs to $\Xi$. Writing
$\Xi(t)=\xi(\tfrac12+it)=-g(t)Z(t)$ with
$g(t)=\tfrac12\bigl(\tfrac14+t^2\bigr)\pi^{-1/4}
\bigl|\Gamma(\tfrac14+\tfrac{it}{2})\bigr|>0$, algebra gives
$\Xi'^2-\Xi\Xi''=g^2\bigl[(Z'^2-ZZ'')-(\log g)''Z^2\bigr]$. If the hypothesis
holds then every zero of $\Xi$ is real, its Hadamard product gives
$-(\log\Xi)''=\sum_k(t-\gamma_k)^{-2}$, and therefore
\begin{equation}\label{eq:laguerre-Z}
Z'^2-ZZ''=Z^2\Bigl[\sum_k\frac{1}{(t-\gamma_k)^2}+(\log g)''\Bigr],
\end{equation}
the sum running over the zeros of $\Xi$ of both signs. The bracket is positive
with room to spare: it is of the order of the squared zero density, carried
mostly by the two ordinates flanking $t$, while $(\log g)''=O(t^{-2})$ takes
away almost nothing. At $t=100.5$ it reads $2.2003$, of which the $600$ zeros
with $|\gamma|\leq542$ supply $2.1972$ and $(\log g)''$ removes
$1.7\times10^{-4}$.

Monotonicity is what makes the curve a counting function outright, rather than a
curve that merely meets the integers in the right places. $N_{\ast}$ is
continuous, it equals $n$ exactly at $\gamma_n$, and it is an integer nowhere
else; if it also increases then between consecutive ordinates it is trapped
strictly between $n$ and $n+1$, and taking the floor recovers
\eqref{eq:N-theta} on the nose,
\begin{equation}\label{eq:floor-N-star}
\bigl\lfloor N_{\ast}(T)\bigr\rfloor=N\bigl(I(T)\bigr),
\end{equation}
with no error term and no sawtooth. We find no exception at $5721$ grid points
across $14\leq t\leq300$ nor at $301$ across $9877\leq t\leq9880$. Three things
are wanted for it, and each is worth naming. The count is read off $Z$, so what
it counts is zeros on the critical line, and it is the hypothesis again that
makes those all of them. The zeros must be simple: at a double zero $Z$ and $Z'$
vanish together, $W=0$, the argument is undefined, and the curve would climb by
one where the count climbs by two; RH does not supply simplicity.
And $\vartheta_1^{\ast}$ must be read on a continuous branch. A fourth is
bookkeeping rather than mathematics: $N$ on the right of
\eqref{eq:floor-N-star} is the plain count of $(0,t]$, so that
$N(\gamma_n)=n$, and not the midpoint value that \S\ref{sec:counting-review}
gives it at an ordinate when reading \eqref{eq:N-theta}. Subtracting
\eqref{eq:N-rvm} from \eqref{eq:N-star} gives
$\vartheta_1^{\ast}=\pi\bigl(N_{\ast}-N_{\mathrm{rvm}}-\tfrac12\bigr)$, so the
principal branch serves precisely while the exact count and the smooth one stay
within one of each other; since $S$ is unbounded that must eventually fail, and
past it the wrap count has to be carried explicitly, which is the bookkeeping
$N_{\mathrm{ps}}$ already needs at every height. There is room to spare at these
heights: over $14\leq t\leq300$ the angle $|\vartheta_1^{\ast}|$ reaches only
$0.506\pi$.

Read the other way, \eqref{eq:nstar-rho} can only fail where $ZZ''>Z'^2$, which
at a critical point of $Z$ means $Z$ and $Z''$ of one sign: a positive local
minimum or a negative local maximum, a pair of zeros that comes to the axis and
fails to reach it. The Laguerre inequality $\Xi'^2\geq\Xi\Xi''$, which Laguerre
proved for every real entire function with only real zeros, is precisely what
forbids that. Being a consequence of the hypothesis and not a theorem about
$\zeta$, it leaves the monotonicity conditional; and being only the first of a
family of inequalities of every order, which Csordas and
Varga~\cite{CsordasVarga1990} show to be equivalent to RH when taken all
together, it is weaker than the hypothesis. A failure of monotonicity would
therefore disprove RH, while a proof of monotonicity would not prove it. Nothing
is close to the boundary: on grids of spacing $0.1$ in windows from $t=10$ up to
$\gamma_{10000}$ the slope never turns negative and never dips below $0.375$.
The near misses do not threaten it either; they push $\rho_{\ast}$ the opposite
way from what one might guess. At the shallow hump inside Lehmer's
pair~\cite{Lehmer1956}, where $Z=3.97\times10^{-3}$, one finds
$\rho_{\ast}=458$, because $-Z''/\vartheta'^2Z$ grows without bound as a hump
flattens onto the axis. What such a pair costs is margin, not sign:
$Z'^2-ZZ''$ reads $0.089$ at that hump against $1.53$ and $1.38$ just outside
the pair. Verified by \texttt{check\_nstar\_mono.py}.
\end{remark}

\begin{remark}[regarding the slope of $N_{\ast}$ and the size of $|Z|$]
\label{rem:nstar-slope}
By \eqref{eq:N-star-prufer},
$\pi N_{\ast}-\tfrac{3\pi}{2}=\arg\bigl(Z-iZ'/\vartheta'\bigr)$ is the polar angle
of the plane point $(Z,-Z'/\vartheta')$, the Pr\"ufer
angle~\cite{Prufer1926} with $\vartheta'$ in the role of the frequency, so the
slope of the counting curve is fixed by $Z$, $Z'$ and $Z''$ at that height
alone; against the local mean density it is the $\rho_{\ast}$ of
\eqref{eq:nstar-rho}. Two values of it are exact. At a simple ordinate,
setting $Z=0$ in \eqref{eq:nstar-rho} leaves $Z'^2/Z'^2$, so $\rho_{\ast}=1$
(also via \eqref{eq:crossing-rate} with $h^{\ast}=-Z'/2\vartheta'$), and the
whole of the variation lives strictly between zeros. At an extremum of $Z$,
where $Z'=0$ leaves $W=\tfrac12 Z$ real,
\begin{equation}\label{eq:rho-peak}
\rho_{\ast}=-\frac{Z''}{\vartheta'^2 Z}\qquad(Z'=0),
\end{equation}
so flatness at a hump is exactly a small curvature-to-height ratio, in units of
$\vartheta'^2$: a broad tall hump gives $\rho_{\ast}$ well below $1$ and a sharp small one
gives $\rho_{\ast}$ well above it. The flatness of $N_{\ast}$ therefore ranks the
size of $|Z|$ between the zeros: over three blocks of $100$ gaps near
$\gamma_{100}$, $\gamma_{1000}$ and $\gamma_{10000}$, $\rho_{\ast}$ at the hump
ranks the height of that hump at $-0.94$ to $-0.99$, and even the single number
$\rho_{\ast}'(\gamma_n)$, read at the zero ahead of the hump, still ranks it at
$-0.66$ to $-0.80$. Figure~\ref{fig:nstar-slope} plots the slope against $|Z|$
in two windows and then the two rankings. What the slope really measures is the
local phase rate of $Z$, not its amplitude: $Z=A\cos\varphi$ with
$\varphi'=\vartheta'$ gives $\rho_{\ast}=1$ whatever $A$ may be, and $\rho_{\ast}<1$
says the phase turns more slowly than $\vartheta'$, which is a wide gap; it is
the sparseness of the zeros that lets $|Z|$ grow ($\rho_{\ast}$ at the hump and
the gap width are collinear to $-0.99$ in rank).
Verified in \texttt{check\_nstar\_slope.py}.
\end{remark}

\begin{table}[H]
\centering
\small
\setlength{\tabcolsep}{4pt}
\renewcommand{\arraystretch}{1.6}
\begin{tabular}{@{}>{\centering\arraybackslash}m{0.14\textwidth}>{\centering\arraybackslash}m{0.13\textwidth}>{\centering\arraybackslash}m{0.22\textwidth}>{\centering\arraybackslash}m{0.16\textwidth}>{\centering\arraybackslash}m{0.20\textwidth}@{}}
\toprule
slope at\ldots & & in $t$ units & normalized by the mean density: the
$\rho_{\ast}$ of \eqref{eq:nstar-rho} & in $T$ units \\
\midrule
an ordinate &
$\dfrac{dN_{\ast}}{dt}\bigg|_{t=\gamma_n}$ &
$\dfrac{\vartheta'(t)}{\pi}=\dfrac{1}{2\pi}\log\dfrac{t}{2\pi}$ &
$\rho_{\ast}=1$ &
$\dfrac{\vartheta'(I(T))\,I'(T)}{\pi}$ \\
\arrayrulecolor{black!20}\midrule
an extremum of $Z$ &
$\dfrac{dN_{\ast}}{dt}\bigg|_{Z'=0}$ &
$-\dfrac{Z''}{\pi\vartheta' Z}$ &
$\rho_{\ast}=-\dfrac{Z''}{\vartheta'^2Z}$ &
$-\dfrac{Z''\,I'(T)}{\pi\vartheta' Z}$ \\
\arrayrulecolor{black}\bottomrule
\end{tabular}
\caption{The exact values of the slope of $N_{\ast}$. At every simple
ordinate the slope is exactly the mean zero density. At a hump of $Z$ the
slope of $N_{\ast}$ is the curvature-to-height ratio of that hump,
\eqref{eq:rho-peak}, so one can use the flatness of $N_{\ast}$ to rank the
heights of the humps of $|Z|$.}
\label{tab:nstar-slope-exact}
\end{table}

\begin{figure}[H]
\centering
\includegraphics[width=\textwidth]{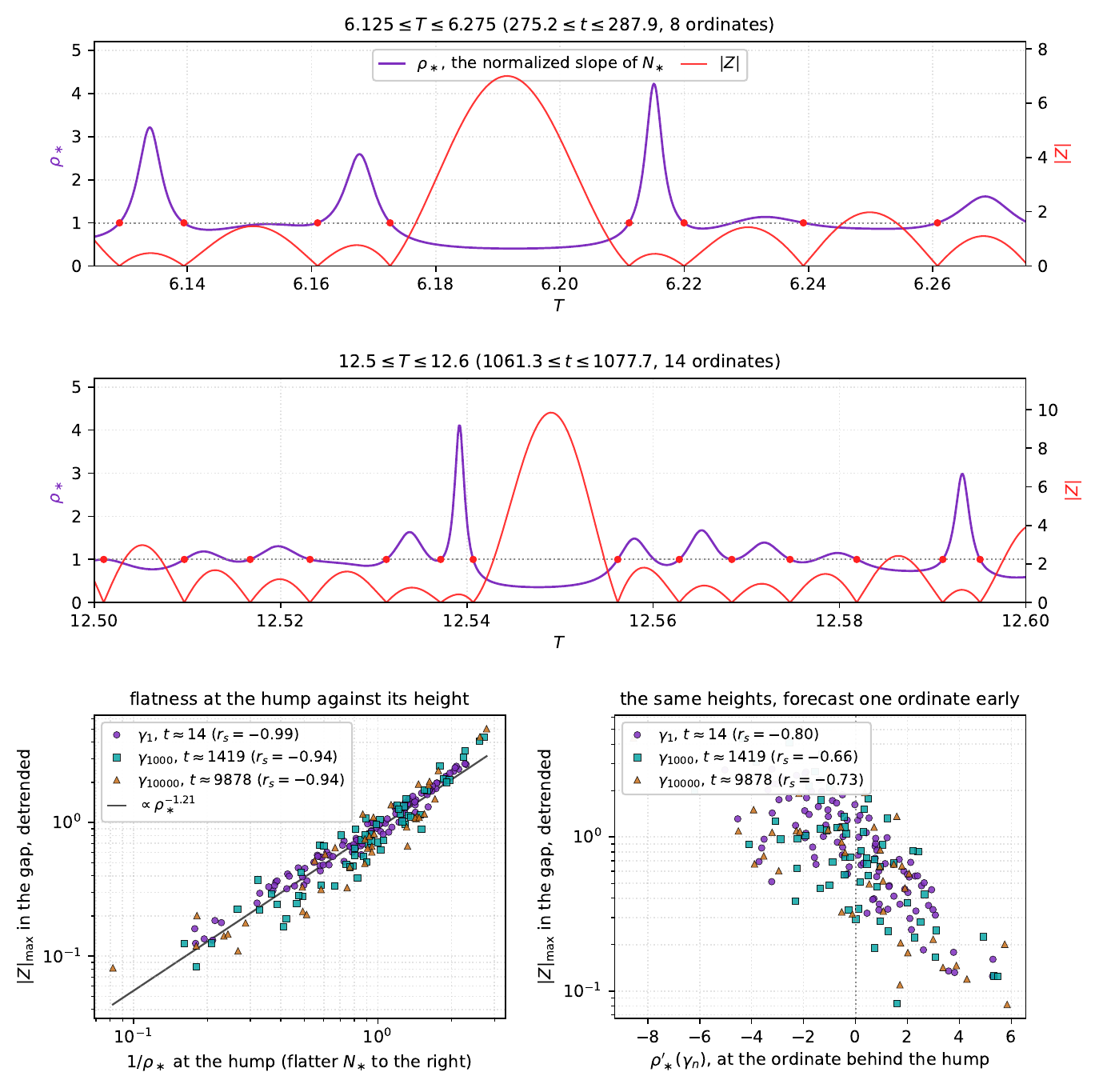}
\caption{The normalized slope $\rho_{\ast}$ of \eqref{eq:nstar-rho} against the size of $|Z|$
between ordinates, everything at $\sigma=\tfrac12$. \emph{Top:} $\rho_{\ast}$ in
purple on the magnified window of Figure~\ref{fig:nps-staircase}, with $|Z|$ in
red on the right-hand scale. The red dots mark the ordinates, where $\rho_{\ast}=1$
exactly; between them $\rho_{\ast}$ dips, and it dips deepest where $|Z|$ rises
highest. \emph{Second:} the same reading twice as high in the index, over
$12.5\leq T\leq12.6$. The hump that reaches $|Z|=9.85$ near $T=12.549$ sits on
the flattest stretch of the window, where $\rho_{\ast}$ bottoms at $0.356$, while each
ordinary hump carries only a shallow arch of $\rho_{\ast}$ just above $1$ and the two
narrowest gaps throw up spikes of $4.11$ and $2.99$.
\emph{Bottom left:} one point per gap for three blocks of ordinates spanning
two decades of $t$, the height of the hump against $1/\rho_{\ast}$ at that hump, both
on log scales, with the heights divided by a $31$-gap moving average so the
slow growth of $|Z|$ with height does not enter. The line is the fit over all
three blocks; the legend quotes the rank correlation against $\rho_{\ast}$ for each
block separately. \emph{Bottom right:} the same hump heights against
$\rho_{\ast}'(\gamma_n)$, read at the ordinate behind the hump rather than at the hump
itself, scaled by the mean spacing $\pi/\vartheta'$. Generated by
\texttt{fig\_nstar\_slope.py}.}
\label{fig:nstar-slope}
\end{figure}

\subsection{Connecting the count, \texorpdfstring{$N_{\ast}$}{N*}, the angle
\texorpdfstring{$\vartheta_2$}{theta2}, and
the equal-leg loci}\label{sec:counting-ovals}

The retrograde ordinates of \S\ref{sec:counting-review} were found by watching
the count; the equal-leg ovals of \S\ref{sec:geometry}
(Figure~\ref{fig:equal-legs-strips}) were found by scanning the strip. They
are the same events. Figure~\ref{fig:count-theta2-ovals} puts the two
side by side. Its top panel repeats the magnified window of
Figure~\ref{fig:nps-staircase}. The bottom panel draws, on the same window,
the fold angles $\vartheta_2$ of the $ps$ split and $\vartheta_2^{\ast}$ of
the velocity split together with the equal-leg
locus $L_1=L_2$ of the $ps$ split in the surrounding strip: the vertical scale
doubles as $2\pi\sigma$, with $\sigma$ on the right-hand axis, so the critical
line, which belongs to the locus at every height, is the horizontal blue line
lying exactly on the dashed line $\vartheta_2=\pi$, and the off-line
components of the locus are the blue ovals straddling it. Shaded are the
stretches of $T$ where $\vartheta_2$ runs \emph{retrograde}: where the
unwrapped angle increases against its prevailing descent.

The correspondence is one to one, three ways. Over $5.9\leq T\leq6.7$ the
angle $\vartheta_2$ runs retrograde on exactly five stretches, each
$0.008$--$0.010$ of $T$ wide. Each stretch contains exactly one ordinate, and
those five are precisely the five retrograde ordinates $\gamma_{114}$,
$\gamma_{125}$, $\gamma_{128}$, $\gamma_{142}$, $\gamma_{145}$ at which
$N_{\mathrm{ps}}$ miscounts in Figure~\ref{fig:nps-staircase}. And the strip
holds exactly five equal-leg ovals over the same window (the two falling in
the zoom window are drawn in Figure~\ref{fig:count-theta2-ovals}; the other
three lie outside it), one per stretch,
coinciding with it end to end to about $10^{-3}$ of $T$. The reading is not
local to this window: twelve windows spread from $T=1$ to $T=50$, holding
$344$ ordinates between them, give $45$ retrograde stretches, $45$ ovals and
$45$ retrograde ordinates, one for one, the stretches narrowing with the
ordinate spacing from $0.013$ of $T$ at $T\approx3$ to $0.0007$ at
$T\approx50$ (\texttt{check\_retro\_ovals.py}). The only other
off-line feature in the window, the thin band at $T\approx6.255$, is not an
oval but the pole band of Figure~\ref{fig:equal-legs-strips}: at
$\{T\}\approx\tfrac14$ the weight $d_1$ passes through its off-line pole
(\S\ref{sec:bisector}), $B_1$ sweeps out and back, and the legs pass through
equality on the way out and on the way back at every $\sigma$ at once, in a
band far narrower than any oval. That last property tells the two apart at
sight: a band still crosses at $\sigma=0.98$, while an oval has closed by
$\sigma\approx0.8$.

The three readings are one event. A retrograde ordinate is one where the
offset $h$ is small (\S\ref{sec:counting-review}: median $|h|$ of $0.319$ at
the missed ordinates against $1.081$ at the rest), which puts the bisector
point near $\zeta/2$; there the two legs are nearly equal \emph{and} nearly
folded, so $\vartheta_2$ sits near $\pi$ and moves slowly, and the slow
passage is what lets the angle turn back on itself. The same near-degeneracy
is what an oval is. On the critical line the leg difference
$\delta=L_1-L_2$ vanishes identically, so an off-line zero of $\delta$ can
open up only where the balance survives a small excursion in $\sigma$, and
that margin exists exactly while the configuration lingers: the oval opens
when the stall begins and closes when prograde motion resumes. During
ordinary prograde motion the legs rotate past each other too quickly for the
balance to hold anywhere off the line, which is why the ovals are confined to
the shaded stretches. Read against \S\ref{sec:counting-star}, the
correspondence also says why the velocity split has no such windows: its
offset $h^{\ast}=-Z'/2\vartheta'$ alternates by construction, no crossing is
retrograde, and $N_{\ast}$ inherits none of the stalls that mark
$N_{\mathrm{ps}}$. What is special about the velocity split is therefore not
the definition of its fold angle
$\vartheta_2^{\ast}=\arg\bigl(B_2^{\ast}/B_1^{\ast}\bigr)$, which is exactly
as for the other splits, but
its behavior: its legs are exactly equal at every height on the line, and
$\vartheta_2^{\ast}$ passes $\pi$ prograde at every ordinate, with no
retrogrades. That
is precisely the property that makes $\vartheta_1^{\ast}$, and hence
$N_{\ast}$, count correctly. Figure~\ref{fig:theta2-star-compare} draws the
two fold angles over the same window.

\begin{figure}[htbp]
\centering
\includegraphics[width=\textwidth]{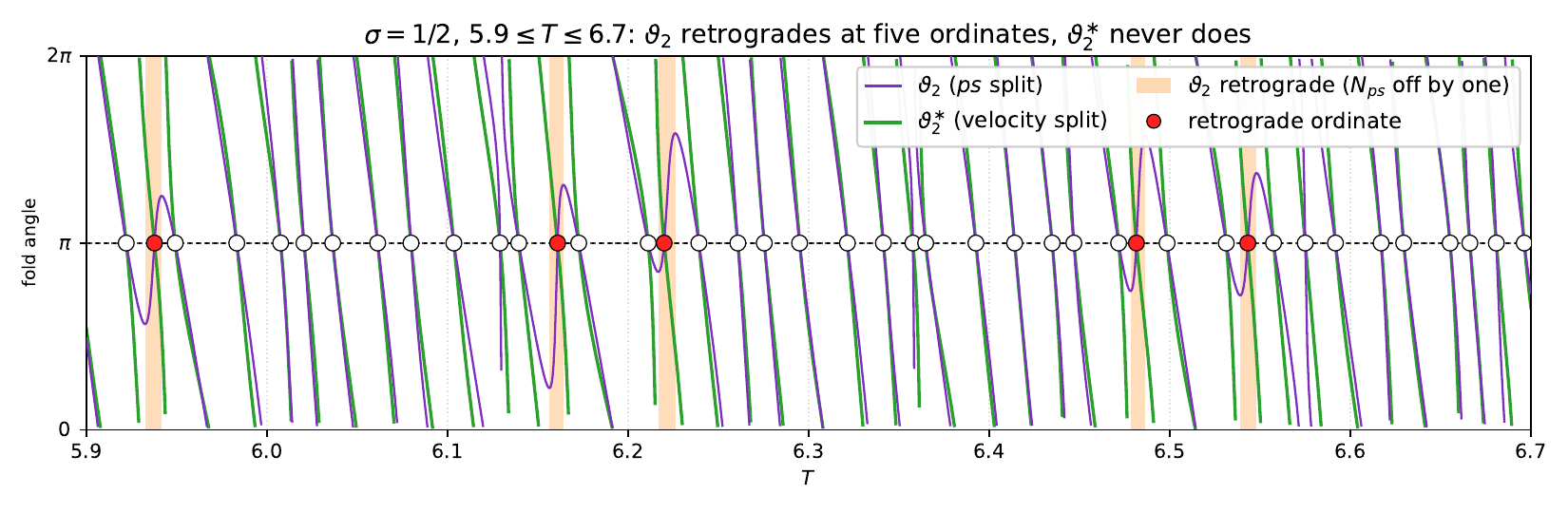}
\caption{The two fold angles compared over $5.9\leq T\leq6.7$ at
$\sigma=\tfrac12$: $\vartheta_2$ of the $ps$ split (purple) and
$\vartheta_2^{\ast}$ of the velocity split (green), both reduced to
$[0,2\pi)$ and broken at their wraps. Dots mark the ordinates, where both
angles pass through the dashed line at $\pi$, filled red at the retrograde
ones. On the shaded stretches $\vartheta_2$ runs retrograde: it stalls at
$\pi$, turns back, and recrosses, and on that stretch $N_{\mathrm{ps}}$ is
off by one. The green curve has no retrograde stretch anywhere: at every
ordinate $\vartheta_2^{\ast}$ passes $\pi$ prograde, and $N_{\ast}$ is never
off the count. Generated by \texttt{fig\_theta2\_star\_compare.py}.}
\label{fig:theta2-star-compare}
\end{figure}

\clearpage
\begin{figure}[H]
\centering
\includegraphics[width=\textwidth]{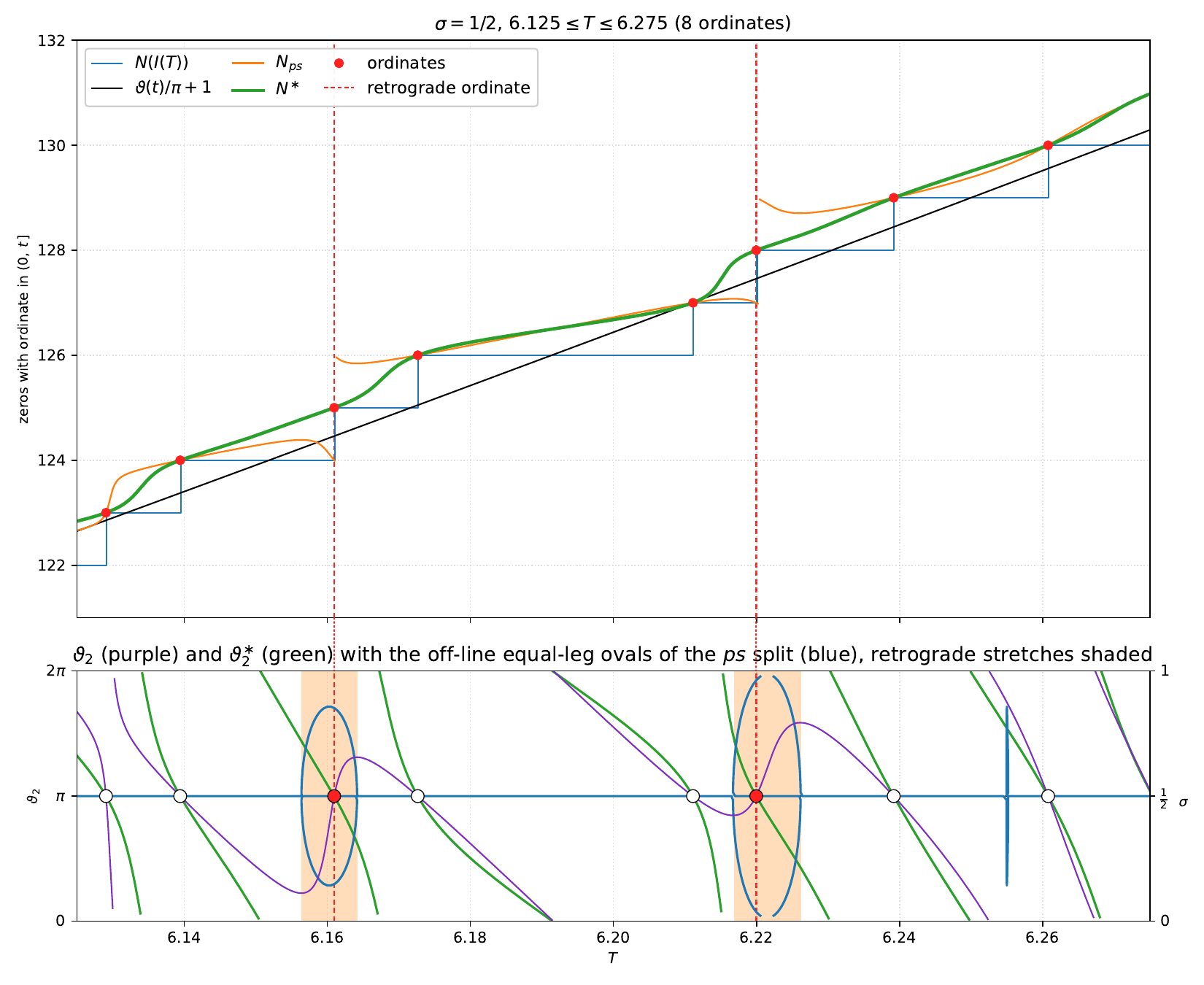}
\caption{The count, the fold angle, and the equal-leg locus of the $ps$ split
read together, everything at $\sigma=\tfrac12$ except the locus.
\emph{Top:} the magnified window $6.125\leq T\leq6.275$ of
Figure~\ref{fig:nps-staircase}: the staircase (blue), the smooth part
(black), $N_{\mathrm{ps}}$ (orange), $N_{\ast}$ (green), and the retrograde
ordinates dashed in red. \emph{Bottom:} $\vartheta_2$ (purple) and
$\vartheta_2^{\ast}$ (green) on the same window, reduced to $[0,2\pi)$,
with the equal-leg
locus $L_1=L_2$ of the $ps$ split drawn in the surrounding strip (blue) on
the shared vertical scale $2\pi\sigma$, $\sigma$ on the right-hand axis. The
horizontal blue line on the dashed $\vartheta_2=\pi$ line is the critical
line, which belongs to the locus at every height; the two blue ovals are the
off-line components in this window. Dots mark the ordinates, where both
angles pass through $\pi$, filled red at the retrograde ones; the
shaded stretches are where $\vartheta_2$ runs retrograde, and there
$\vartheta_2^{\ast}$ passes straight through. Each oval fills a
shaded stretch and each stretch holds one retrograde ordinate. The thin blue
band at $T\approx6.255$ is the pole band of
Figure~\ref{fig:equal-legs-strips} at $\{T\}\approx\tfrac14$, not an oval.
Generated by
\texttt{fig\_count\_theta2\_ovals.py}.}
\label{fig:count-theta2-ovals}
\end{figure}

\subsection{Positivity of the slope under the Riemann hypothesis}
\label{sec:rho-positive}

Remark~\ref{rem:nstar-mono} recorded that the monotonicity of $N_{\ast}$ is a
known consequence of the Riemann hypothesis, through the Laguerre-type
inequality \eqref{eq:laguerre-Z}. Here we give the proof in full, in a form
that shows exactly where the hypothesis enters and how much room the
inequality has.

\begin{theorem}[RH implies $N_{\ast}$ is eventually strictly increasing]
\label{thm:rho-positive}
Assume the Riemann hypothesis. Then there is a $t_0$ such that $\rho_{\ast}(t)>0$
for every $t\geq t_0$ that is not an ordinate; together with $\rho_{\ast}=1$ at the
(simple) ordinates (Remark~\ref{rem:nstar-slope}), $N_{\ast}$ is strictly
increasing there. Grid computations (Remarks~\ref{rem:nstar-mono}
and~\ref{rem:nstar-slope}) support positivity down through the first zero.
\end{theorem}

\begin{proof}
\emph{Step 1: reduction to a monotonicity statement (unconditional).}
Between consecutive zeros $Z\neq0$, so we may factor
$W=\tfrac12Z\,(1-iu)$ with
\[
u(t)=\frac{Z'(t)}{\vartheta'(t)\,Z(t)} .
\]
On such an interval $\arg Z$ is constant ($0$ or $\pi$), hence
$\arg W=\mathrm{const}-\arctan u$ and
\[
\rho_{\ast}=\frac{1}{\vartheta'}\frac{d}{dt}\arg W
=\frac{-u'}{\vartheta'\,(1+u^{2})} .
\]
Since $\vartheta'>0$, proving $\rho_{\ast}>0$ is exactly proving that $u$ is
strictly decreasing on every gap ($u$ falls from $+\infty$ just after a zero
to $-\infty$ just before the next).

\emph{Step 2: the Hadamard product (first use of RH).}
Let $\Xi(t)=\xi(\tfrac12+it)$, where
$\xi(s)=\tfrac12s(s-1)\pi^{-s/2}\Gamma(s/2)\zeta(s)$. Then $\Xi=-g\,Z$ with
$g(t)=\tfrac12(t^{2}+\tfrac14)\,\pi^{-1/4}\,
|\Gamma(\tfrac14+\tfrac{it}{2})|$ the smooth positive factor of
Remark~\ref{rem:nstar-mono}. From the Hadamard
product $\xi'/\xi(s)=\sum_{\rho}(s-\rho)^{-1}$ (zeros paired symmetrically),
the Riemann hypothesis puts every zero at $\rho=\tfrac12+i\gamma_k$, so on
$s=\tfrac12+it$ the differences $s-\rho=i(t-\gamma_k)$ are purely imaginary
and
\[
\frac{\Xi'}{\Xi}(t)=\sum_k\frac{1}{t-\gamma_k},
\qquad
(\log\Xi)''(t)=-\sum_k\frac{1}{(t-\gamma_k)^{2}}=:-S_2(t)<0 .
\]
This is the structural use of the hypothesis: it makes every term of $S_2$ a
positive square. An off-line zero $\beta+i\gamma$ with $\beta\neq\tfrac12$
would contribute
$-\mathrm{Re}\,\bigl[(\tfrac12-\beta)+i(t-\gamma)\bigr]^{-2}$, which turns
\emph{negative} for $|t-\gamma|<|\beta-\tfrac12|$: near an off-line zero the
term loses its sign protection, and it is exactly the hypothesis that
forbids this.

\emph{Step 3: differentiate $u$.}
With $Z=-\Xi/g$ we have $Z'/Z=\Xi'/\Xi-g'/g$, hence
\[
-\vartheta'\,u'
=\underbrace{\sum_k\Bigl[\frac{1}{(t-\gamma_k)^{2}}
+\frac{\vartheta''}{\vartheta'}\cdot\frac{1}{t-\gamma_k}\Bigr]}_{\text{zero
terms}}
\;+\;
\underbrace{(\log g)''-\frac{\vartheta''}{\vartheta'}\cdot\frac{g'}{g}}_{=:E(t)} .
\]

\emph{Step 4: positivity of the zero terms, smallness of $E$.}
Each bracket equals
$\bigl[1+\tfrac{\vartheta''}{\vartheta'}(t-\gamma_k)\bigr]/(t-\gamma_k)^{2}$,
positive whenever
$|t-\gamma_k|<\vartheta'/\vartheta''\approx t\log\tfrac{t}{2\pi}$;
so every zero within distance $t\log t$ of $t$ contributes positively. The
distant zeros can contribute negatively only through their
$\tfrac{\vartheta''}{\vartheta'}\tfrac1{t-\gamma_k}$ tails; pairing $\gamma$
with $-\gamma$ and using the counting density, their total is
$O(\vartheta''/\vartheta')=O\bigl(1/(t\log t)\bigr)$.
Stirling gives $E(t)=O(1/t)$. Meanwhile, on the hypothesis
$S(t)=O(\log t/\log\log t)$~\cite[Theorem~14.13]{Titchmarsh1986}, so
$N(t+1)-N(t-1)=\tfrac{2}{\pi}\vartheta'(t)+o(\log t)>0$ for all large $t$:
the interval $(t-1,t+1)$ contains a zero,
$\operatorname{dist}(t,\{\gamma_k\})\leq1$, and the nearest zero's bracket
alone contributes
$\bigl[1-\vartheta''/\vartheta'\bigr]/\operatorname{dist}^{2}\geq\tfrac12$
once $t$ is large. Hence
\[
-\vartheta'u'\;\geq\;\tfrac12-O(1/t)\;>\;0
\]
for all $t$ beyond a modest $t_0$, so $u'<0$ and $\rho_{\ast}>0$ on every gap.
Below $t_0$ the argument gives nothing. There the computations reported in
Remarks~\ref{rem:nstar-mono} and~\ref{rem:nstar-slope} sample the range and
find grid minima $0.356$--$0.521$, nowhere near $0$; that is evidence, not a
proof of positivity between the sample points.
\end{proof}

\begin{remark}[where the hypothesis is indispensable]\label{rem:rh-role}
Steps~2 and~4 use the Riemann hypothesis, and they use it differently.
Step~2 needs it structurally, to make every term of $S_2$ a positive square;
Step~4 needs only the size bound $S(t)=O(\log t/\log\log t)$, and only to know
that an ordinate lies within distance $1$ of $t$. The unconditional
$S(t)=O(\log t)$ is the same order as the main term $\tfrac2\pi\vartheta'(t)$
and will not serve. Everything else is bookkeeping.
This also shows the converse fails gracefully: $\rho_{\ast}>0$ everywhere is
\emph{weaker} than the hypothesis. It is essentially the first
Laguerre-type inequality \eqref{eq:laguerre-Z}, and the theorem of Csordas
and Varga~\cite{CsordasVarga1990} says one needs the whole infinite family
of higher-order inequalities to recover the hypothesis. Unconditionally,
even $\rho_{\ast}>0$ is open. A negative value would contradict the hypothesis
only above the threshold $t_0$ of Theorem~\ref{thm:rho-positive}, which we do
not make effective, so no single observation at a finite height settles
anything by itself; what the theorem forbids under RH is negative values at
unbounded heights.
\end{remark}

\begin{remark}[the theorem doesn't need simple zeros as an extra
assumption: you get away with RH alone]\label{rem:simplicity}
At a hypothetical multiple zero the reduction of Step~1 degenerates: $u$ is
undefined at the zero itself and $\rho_{\ast}$ there is $0/0$. Between zeros the
argument is unaffected. The hypothesis does not by itself rule out multiple
zeros, so strictly the statement is: $\rho_{\ast}>0$ away from ordinates, and
$\rho_{\ast}=1$ at simple ordinates (Remark~\ref{rem:nstar-slope}). All known zeros
are simple.
\end{remark}

\begin{remark}[what keeps the curve climbing up]\label{rem:floor}
The proof makes quantitative what the figures show: $\rho_{\ast}$ can dip only as
low as the widest local gap forces, through $S_2$ shrinking to its two
nearest terms, but never to zero; under the hypothesis the sum of
inverse-square distances to the zeros is the floor holding the counting
curve up.
\end{remark}

\begin{remark}[sharpness: no uniform bounds are possible]
\label{rem:no-bounds}
Between zeros no bound better than $0<\rho_{\ast}<\infty$ can hold, and the reason
is a conservation law. Across each gap $W$ swings from one crossing of the
imaginary axis to the next, so $\arg W$ advances by exactly $\pi$:
\[
\int_{\gamma_k}^{\gamma_{k+1}}\rho_{\ast}\,\vartheta'\,dt=\pi ,
\]
that is, the density-weighted mean of $\rho_{\ast}$ over a gap is
$\pi/\Delta\vartheta\approx1/(\text{gap length in mean spacings})$. Every gap
delivers exactly one zero's worth of angle; the only freedom is the
distribution inside. A gap much narrower than average therefore forces
$\rho_{\ast}$ large (at the Lehmer pair~\cite{Lehmer1956}
$\gamma_{6709},\gamma_{6710}$ the hump value is $\rho_{\ast}=458$,
Remark~\ref{rem:nstar-mono}), and arbitrarily close pairs, which GUE
statistics predict, force $\rho_{\ast}$ arbitrarily large. The coincident case
is not the limit of that, though. At a zero of multiplicity $r$ the numerator
of \eqref{eq:nstar-rho} is $c^{2}r\,x^{2r-2}$ and the denominator
$c^{2}r^{2}x^{2r-2}$, writing $Z=cx^{r}+\cdots$ with $x=t-\gamma$, so
$\rho_{\ast}$ is undefined at the zero itself and approaches $1/r$ around it:
$1$ at a simple zero, as it must, and $\tfrac12$ at a double one. What
diverges as two simple zeros merge is the peak between them, not the value
where they meet. Symmetrically, a gap
twice the mean forces the mean of $\rho_{\ast}$ over it below $\tfrac12$, and
arbitrarily long gaps force $\rho_{\ast}$ arbitrarily near $0$, so the infimum over
all $t$ is (conjecturally) $0$. The measured minima of
Remark~\ref{rem:nstar-mono} drift accordingly with the widest gap sampled,
down to the $0.356$ at the flat stretch under the tall
$|Z|=9.85$ hump of Figure~\ref{fig:nstar-slope}: big humps sit in wide
gaps, and wide gaps mean slow counting. The one hard bound is the sign
(Remark~\ref{rem:rh-role}). A useful mental model: $\rho_{\ast}$
behaves like the reciprocal of the local normalized gap size, smoothed:
pinned to exactly $1$ at every zero, integrating to exactly one count per
gap, and free to spike or stall in between precisely as much as the zero
spacing deviates from its mean.
\end{remark}

\begin{remark}[recap on what $N$ and $N_{\ast}$ count]\label{rem:count-recap}
$N$ is a contour count (Theorem~9.3 of Titchmarsh~\cite{Titchmarsh1986},
already cited in \S\ref{sec:counting-review}), so its rectangle catches a
zero at any $\sigma$; an off-line pair $\beta+i\gamma$, $1-\beta+i\gamma$
shares one ordinate and would bump $N$ by two. $N_{\ast}$ is built from $Z$
and $Z'$, pure critical-line data; it advances only when
$W=\tfrac12Z-iZ'/2\vartheta'$ of \eqref{eq:h-star} crosses the imaginary
axis, which requires $Z=0$, and only on-line zeros do that; off the line
$|Z|$ merely dips, $W$ bows past the origin without crossing, and the curve
glides through without stepping. That gap is why
$\lfloor N_{\ast}\rfloor=N$ doubles as an off-line-zero detector, which
Table~\ref{tab:N-vs-Nstar} and Remark~\ref{rem:nstar-mono} already assert.
\end{remark}

\subsection{Summary: \texorpdfstring{$N$ and $N_{\ast}$}{N and N*} side by side}
\label{sec:counting-summary}

$N_{\ast}$ contains more and less information than $N$. More because it is
continuous and (on RH) increasing between ordinates, and less because it does not
count off-half-line zeros. Table~\ref{tab:N-vs-Nstar} collects the comparison
of the two counting functions in one place, and
Figure~\ref{fig:N-vs-Nstar-strips} draws them side by side: the top panel is
the full window $6\leq T\leq8.5$, and the three panels below it magnify the
outlined bands.

\begin{table}[H]
\centering
\small
\begin{tabular}{@{}p{0.20\textwidth}p{0.36\textwidth}p{0.36\textwidth}@{}}
\toprule
 & $N$ of \eqref{eq:N-theta} & $N_{\ast}$ of \eqref{eq:N-star-prufer} \\
\midrule
definition &
$\dfrac{\vartheta(t)}{\pi}+1+\dfrac{1}{\pi}\arg\zeta\bigl(\tfrac12+it\bigr)$ &
$\dfrac{1}{\pi}\arg\Bigl(Z-\dfrac{iZ'}{\vartheta'}\Bigr)+\dfrac32$ \\[4pt]
\arrayrulecolor{black!20}\midrule
continuity & step function; jumps at every ordinate & continuous \\
\midrule
at the ordinate $\gamma_n$ & jumps from $n-1$ to $n$ & equals $n$ exactly \\
\midrule
integer values & constant integer between ordinates & integer at the
ordinates and nowhere else \\
\midrule
beside the approximate form & the smooth
$\frac{t}{2\pi}\log\frac{t}{2\pi}-\frac{t}{2\pi}+\frac78+\frac{1}{48\pi t}$
of \eqref{eq:N-rvm} threads the staircase in the mean, never landing on an
integer at an ordinate &
meets the integer exactly at every ordinate, where that approximation never
does \\
\midrule
monotonicity & nondecreasing by construction & increasing, eventually and
conditional on RH (Theorem~\ref{thm:rho-positive}) \\
\midrule
what it counts & all zeros with $0<\operatorname{Im}s\leq t$, on or off the
critical line & only the zeros on the critical line \\
\midrule
how it sees them & the argument principle: $\xi$ winds once around a
rectangle enclosing the whole strip up to height $t$ for every zero inside,
wherever it sits & read off $Z$ alone: it passes an integer exactly where
$Z$ changes sign, and an off-line zero never makes $Z$ vanish \\
\midrule
one from the other & no converse: $N$ does not determine $N_{\ast}$ &
$\lfloor N_{\ast}\rfloor=N$ if $N_{\ast}$
increases, by \eqref{eq:floor-N-star} \\
\arrayrulecolor{black}\bottomrule
\end{tabular}
\caption{The exact staircase $N$ beside the continuous counting curve
$N_{\ast}$. The $N_{\ast}$ column reads the argument of
\eqref{eq:N-star-prufer} on its continuous branch throughout, and its
continuity and its exact integer values at the ordinates both need the zeros
there to be simple: at a multiple zero $Z$ and $Z'$ vanish together,
$B_1^{\ast}=0$, and the angle is undefined.}
\label{tab:N-vs-Nstar}
\end{table}

\begin{figure}[p]
\centering
\includegraphics[width=\textwidth,height=0.92\textheight,keepaspectratio]{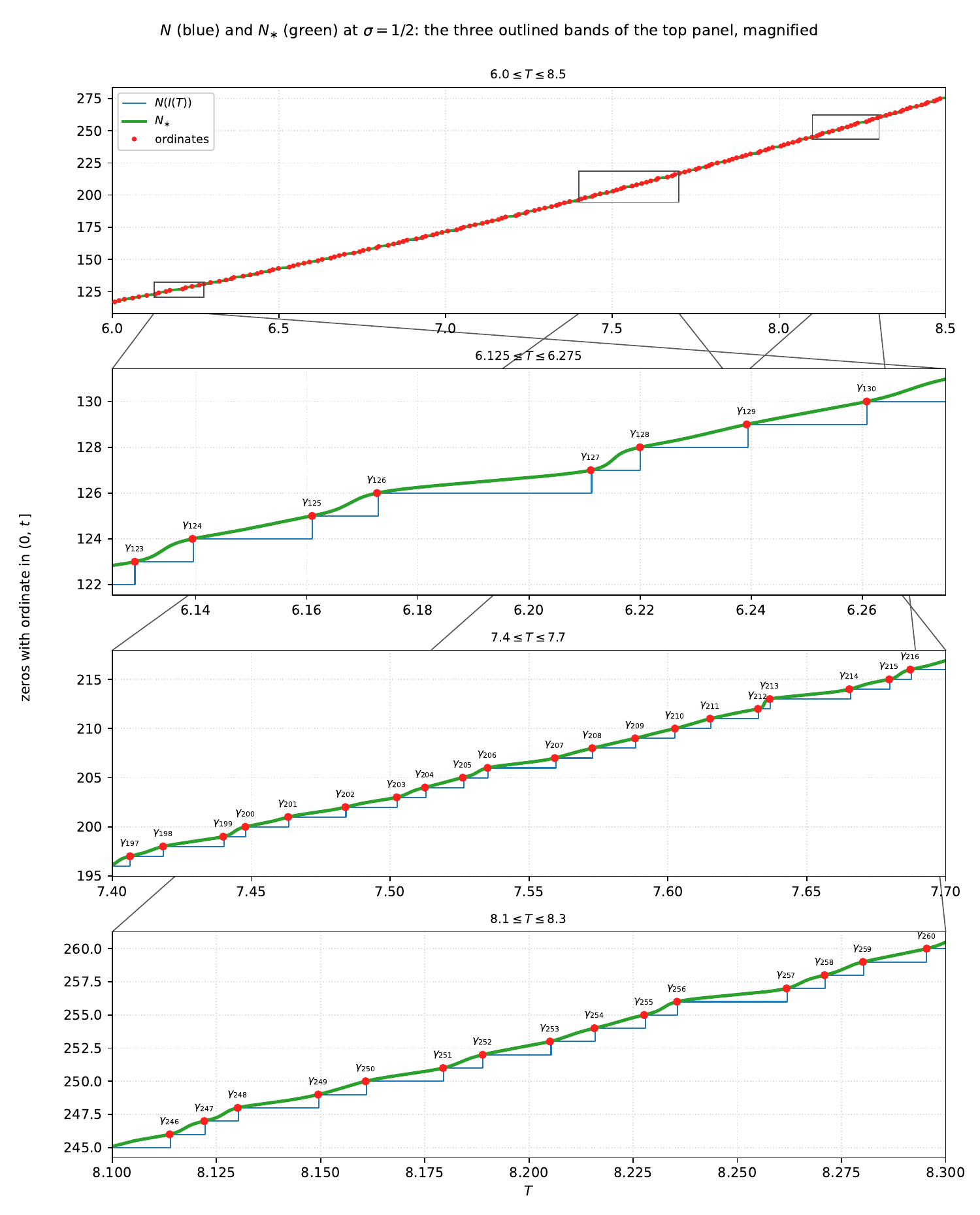}
\caption{Head-to-head comparison of $N$ (blue) and $N_{\ast}$ (thick green)
at $\sigma=\tfrac12$ with the ordinates as red dots. The top panel is the full window
$6\leq T\leq8.5$, holding the $159$ ordinates $\gamma_{117}$ through
$\gamma_{275}$; the three outlined bands are magnified in the panels below
it: $6.125\leq T\leq6.275$ (the window of
Figure~\ref{fig:count-theta2-ovals}), $7.4\leq T\leq7.7$, and
$8.1\leq T\leq8.3$. In the
magnified panels each ordinate carries its label $\gamma_n$.
Generated by \texttt{fig\_N\_vs\_Nstar\_strips.py}.}
\label{fig:N-vs-Nstar-strips}
\end{figure}

\clearpage

\section{Further observations}\label{sec:consequences}

We have made other observations of the zeta function, mostly from a
geometric perspective, and have listed some of them in this final section:
\begin{enumerate}
\item we examine the ratio of the lengths of the two partial sums beside
the ratio of the two partial summand remainders
(\S\ref{sec:length-ratios});
\item we look at the envelope of $\zeta$ on the half-line from a $B_1$ and
a $B_1^{\ast}$ perspective (\S\ref{sec:envelope});
\item since we now have three remainders, we look at their collinearity,
which we could not do when we had just two, because any two points already
lie on a line (\S\ref{sec:colinearity});
\item and since looking at the links separately, each in its own
coordinate frame, was fruitful (it led to the general partial summand
formula for $\zeta$ of \S\ref{sec:cut-any-link}), here we look at the
joint angles (\S\ref{sec:joint-angles}).
\end{enumerate}

\subsection{\texorpdfstring{Length ratios: $\Sigma_1$ to $\Sigma_2$, and $R_{1}$ to $R_{2}$}{Length ratios}}\label{sec:length-ratios}

The four terms of \eqref{eq:zeta-ps} come in two natural pairs: the two
partial sums, and the two fractional links. Within each pair the two lengths
are equal on the critical line, and what follows records what their
ratio does off it. The pairs behave quite differently, and in both cases the
ratio obeys the same reflection law about $\sigma=\tfrac12$.

\begin{remark}[regarding the ratio of the lengths of $\Sigma_1$ and $\Sigma_2$]\label{rem:leg-ratio}
On the critical line $\sigma=\tfrac12$ we have $|\chi|=1$, so the $\Sigma_2$ links
have the same lengths $(n+1)^{-1/2}$ as the $\Sigma_1$ links; each link is then a
rotation composed with a uniform link-length shrink $\sqrt{n/(n+1)}$, and
$\arg\chi=-2\vartheta(t)$, with $\vartheta$ the Riemann--Siegel theta function.
The two chains are thus congruent link for link there, so their end-to-end
displacements agree. Away from the critical line they do not, and the discrepancy has a
closed form. Write $\ell_{\Sigma_1}=|\Sigma_1|$ and $\ell_{\Sigma_2}=|\Sigma_2|$
for the lengths of the net displacements of the two chains,
and put $\Sigma_m(z)=\sum_{n\le m}n^{-z}$. Since
$\sum_{n\le m}n^{\,s-1}=\overline{\Sigma_m\bigl((1-\sigma)+it\bigr)}$, the
second leg is the first one read at the reflected abscissa $1-\sigma$, scaled
by $|\chi|$:
\begin{equation}\label{eq:leg-ratio}
\frac{\ell_{\Sigma_1}}{\ell_{\Sigma_2}}
=
\frac{\bigl|\Sigma_m(\sigma+it)\bigr|}
     {\bigl|\chi(\sigma+it)\bigr|\;
      \bigl|\Sigma_m\bigl((1-\sigma)+it\bigr)\bigr|}.
\end{equation}
\end{remark}

\begin{proposition}\label{prop:leg-ratio-reflection}
At fixed $T$ (hence fixed $t=I(T)$ and $m=\lfloor T\rfloor$), at every
$\sigma$ for which all four lengths are nonzero,
\begin{equation}\label{eq:leg-ratio-symmetry}
\frac{\ell_{\Sigma_1}(\sigma)}{\ell_{\Sigma_2}(\sigma)}
\cdot
\frac{\ell_{\Sigma_1}(1-\sigma)}{\ell_{\Sigma_2}(1-\sigma)}
=
1.
\end{equation}
\end{proposition}

Write \eqref{eq:leg-ratio} at $\sigma$ and at $1-\sigma$ and multiply.
The two partial-sum moduli cancel, leaving
$1\big/\bigl(\lvert\chi(\sigma+it)\rvert\,
\lvert\chi((1-\sigma)+it)\rvert\bigr)$.
From $\chi(s)\chi(1-s)=1$ and the Schwarz reflection
$\chi(\bar s)=\overline{\chi(s)}$ one has
$\chi(1-\sigma+it)=1/\overline{\chi(\sigma+it)}$,
so the product of the moduli is $1$.
The ratio is therefore reciprocal-symmetric about $\sigma=\tfrac12$: whatever
it does on the right half of the strip, it does the reciprocal of on the left,
with the value pinned to $1$ at the center by the link-for-link congruence
of Remark~\ref{rem:leg-ratio}.

\begin{remark}[regarding the ratio of the lengths of $R_{1}$ and $R_{2}$]\label{rem:d-ratio}
The same question for the two fractional links has a shorter answer, because in
the quotient of \eqref{eq:d1} and \eqref{eq:d2} both $|R|$ and the common
denominator $\sin(2\omega+\psi)$ cancel:
\begin{equation}\label{eq:d-ratio}
\frac{d_1}{d_2}
\;=\;
\frac{\sin(\omega+\psi-\arg R)}{\sin(\omega+\arg R)},
\qquad
\frac{|R_{1}|}{|R_{2}|}
\;=\;
\Bigl|\frac{d_1}{d_2}\Bigr| .
\end{equation}
The signed quotient is the one that carries the information; the two moduli
see only its absolute value, and $d_1/d_2$ does go negative off the line, as
the fourth point below records.
This ratio is pure angle data, carrying no length scale at all: it is the law
of sines in the triangle of Figure~\ref{fig:remainder-zoom}, recording how the
direction of $R$ divides the wedge spanned by the two link directions
$e^{-i\omega}$ and $e^{\,i(\omega+\psi)}$. Four consequences are worth
recording. First, on the critical line $\arg R=\tfrac{\psi}{2}$, the two sines
coincide and the ratio is $1$, which is Proposition~\ref{prop:weights}(iii) again.
Second, the ratio at $\sigma$ is the reciprocal of the ratio at $1-\sigma$:

\begin{proposition}\label{prop:d-ratio-reflection}
At fixed $T$, at every $\sigma$ for which $d_2(\sigma)$ and $d_2(1-\sigma)$
are nonzero, so that both ratios are finite,
\begin{equation}\label{eq:d-ratio-symmetry}
\frac{d_1(\sigma)}{d_2(\sigma)}
\cdot
\frac{d_1(1-\sigma)}{d_2(1-\sigma)}
=
1.
\end{equation}
\end{proposition}

The turn $\omega$ depends only on $T$. As above, one has
$\chi(1-\sigma+it)=1/\overline{\chi(\sigma+it)}$, so $\psi$ is the same at
both abscissae. From \eqref{eq:siegel-named} and
$\zeta(s)=\chi(s)\zeta(1-s)$ the remainder satisfies
$R(1-s)=\chi(1-s)R(s)$; Schwarz reflection
$R(\bar z)=\overline{R(z)}$ then gives
\begin{equation}\label{eq:R-reflect}
R(1-\sigma,T)
\;=\;
\chi(1-\sigma+it)\,\overline{R(\sigma,T)},
\end{equation}
hence $\arg R(1-\sigma)=\psi-\arg R(\sigma)$.
The two sines of \eqref{eq:d-ratio} therefore exchange, and the product is $1$.
On the critical line \eqref{eq:R-reflect} is $R=\chi\overline{R}$, which is
the first point again.

Third, because the scale cancels, this ratio is far tamer than
\eqref{eq:leg-ratio}: sampled on a grid over $1<T<20$ and the whole strip it
runs from $0.387$ to $2.583$ (a reciprocal pair, as the second point
requires). Those are not bounds. A grid steps over the isolated loci of the
fourth point below, where the ratio is negative, or $0$, or infinite; what the
sample shows is that away from them the ratio stays within a factor of three
of $1$, whereas
$\ell_{\Sigma_1}/\ell_{\Sigma_2}$ ranges over several orders of magnitude
there. Fourth, the parallel-link poles of \S\ref{sec:pole-locations} cancel out
of it: where $2\omega+\psi=k\pi$, \eqref{eq:d-ratio} gives
$d_1/d_2=(-1)^{k+1}$, exactly $\pm1$, even though $d_1$ and $d_2$ separately
diverge (at $\sigma=\tfrac1{10}$, $T=4.256594$ we compute
$d_1=-d_2\approx-4.07\times10^{36}$). A negative value is off-line news of a
different kind: it says $R$ has left the wedge, which happens in a narrow window
around each pole. The ratio does reach $0$ and $\infty$, but only on the
isolated loci where $R$ is exactly parallel to one of the two link directions;
for instance $d_1=0$ at $\sigma=\tfrac1{10}$, $T=4.2607679$, whose mirror at
$\sigma=\tfrac9{10}$ has $d_2=0$. Off the critical line the locus $d_1=d_2$ is
the family of ovals plotted in Figure~\ref{fig:leg-imbalance}.
\end{remark}

\subsection{\texorpdfstring{The envelope of $\left|\zeta\!\left(\tfrac12+it\right)\right|$}{The envelope of |zeta(1/2+it)|}}\label{sec:envelope}

On the critical line the chain is isosceles for every reflection-symmetric
split, by Corollary~\ref{cor:equal-legs}, so its two legs share one length
$L_1=|B_1|=|B_2|$ and the base is the chord they subtend across the fold
angle of the legs $\vartheta_2$ of \S\ref{sec:legs}. From $\zeta=B_1(1+e^{i\vartheta_2})$,
\begin{equation}\label{eq:zeta-chord}
\bigl|\zeta\bigl(\tfrac12+it\bigr)\bigr|
=2L_1\bigl|\cos\tfrac{\vartheta_2}{2}\bigr| ,
\end{equation}
so the modulus of $\zeta$ splits into a length carried by the links and an angle
read at the joint between the legs. Bounding the angular factor bounds $\zeta$.
Write $u=\vartheta_2/\pi$ reduced to $[0,2)$. Then
$\bigl|\cos\tfrac{\pi u}{2}\bigr|\leq1$ at once, and since
$\cos\tfrac{\pi u}{2}$ is concave on $0\leq u\leq1$ it stays above the chord
joining its endpoints there, giving
$\bigl|\cos\tfrac{\pi u}{2}\bigr|\geq|1-u|$ on the whole of $[0,2)$ by the
reflection $u\mapsto2-u$, with equality only at $u=0,1,2$. Hence
\begin{equation}\label{eq:env-bounds}
2L_1\bigl|1-u\bigr|
\;\leq\;\bigl|\zeta\bigl(\tfrac12+it\bigr)\bigr|\;\leq\;2L_1 ,
\qquad
u\equiv\frac{\vartheta_2}{\pi}\ (\mathrm{mod}\ 2),\quad 0\leq u<2 .
\end{equation}
The reduction matters: $\vartheta_2$ is the continuous lift whose turns do the
counting in \S\ref{sec:counting}, so $\vartheta_2/\pi$ itself grows without
bound and would make the left side exceed the right.
The upper envelope is attained where the two legs point the same way; the
lower one is attained at every ordinate, where
$u=1$ and both sides vanish, and again where $u$ reaches $0$ or $2$ and the two
envelopes meet. The triangular factor $|1-u|$ is the same sawtooth
whose passages through zero do the counting in \S\ref{sec:counting}, appearing
here with an amplitude on it. Over $6\leq T\leq13$, which is
$264.9\leq t\leq1144.6$, we confirm \eqref{eq:zeta-chord} to
$1.7\times10^{-20}$ and both inequalities of \eqref{eq:env-bounds} at every one
of $2001$ sample points, the two legs agreeing to $2.0\times10^{-20}$.

Which split to use is settled by how often the upper envelope is reached.
Equation~\eqref{eq:B1-rotated} gives $2L_1=\sqrt{Z^2+4h^2}$, so $|\zeta|=2L_1$
exactly where the offset $h$ vanishes. For the split at the partial
summand $h$ changes sign inside a gap exactly when its two ends carry opposite
$h$, which
is the alternation that the count of \S\ref{sec:counting} needs and that fails at
a retrograde ordinate, so the envelope is reached in some gaps and missed in
others: $3$ of the $7$ gaps of $6.125\leq T\leq6.275$, and $7$ of the $13$ gaps
of $12.5\leq T\leq12.6$. Those are counts of sign changes; a gap whose two ends
agree could in principle still hold an even number of zeros of $h$, which would
only add touches to the tally. The velocity split has no such gaps. Its offset
$h^{\ast}=-Z'/2\vartheta'$ vanishes at every extremum of $Z$, so
\begin{equation}\label{eq:env-star}
2L_1^{\ast}=\sqrt{Z^2+\Bigl(\frac{Z'}{\vartheta'}\Bigr)^{2}}
\;\geq\;|Z|=\bigl|\zeta\bigl(\tfrac12+it\bigr)\bigr| ,
\end{equation}
with equality on every hump: $7$ touches in $7$ gaps and $13$ in $13$ in those
two windows, matched to $10^{-25}$ (Figure~\ref{fig:envelope}). This is the familiar envelope of an
oscillation of instantaneous frequency $\vartheta'$, with $Z'/\vartheta'$ as the
quadrature companion of $Z$, and it is tighter on average as well, the mean of
$|\zeta|/2L_1$ over $6\leq T\leq13$ being $0.635$ against $0.586$ for the split
at the partial summand.

Both readings are polar coordinates of one function. The rotated bisector point
$W=e^{i\vartheta}B_1^{\ast}=\tfrac12Z-\tfrac{iZ'}{2\vartheta'}$ of \eqref{eq:h-star}
has modulus $L_1^{\ast}$ and argument $\pi\bigl(N_{\ast}-\tfrac32\bigr)$ by
\eqref{eq:N-star}, and $Z=2\operatorname{Re}W$, so
\begin{equation}\label{eq:Z-polar}
Z(t)=-2\,|W|\,\sin\bigl(\pi N_{\ast}(T)\bigr),
\end{equation}
which we confirm to $6\times10^{-13}$, the precision of the counting curve
itself. This is a restatement rather than a new fact, but it puts the two
subsections together: the amplitude of the velocity split is the envelope of
$\zeta$ and its phase is the counting curve, so the zeros are the moments when
the phase passes an integer and the extrema are the moments when the amplitude
is attained. Remark~\ref{rem:nstar-slope} reads the rate of that phase; the
envelope \eqref{eq:env-star} reads its amplitude. Verified in
\texttt{check\_envelope.py}.

\begin{figure}[H]
\centering
\includegraphics[width=\textwidth]{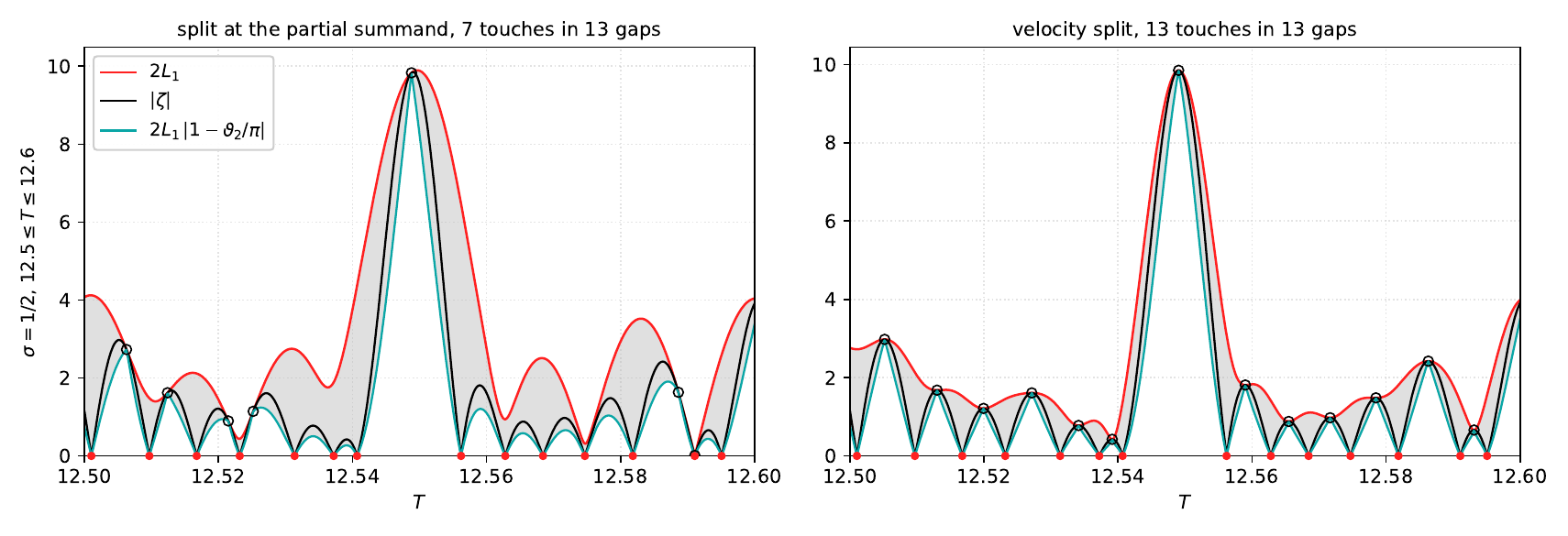}
\caption{The two envelopes of \eqref{eq:env-bounds}, at $\sigma=\tfrac12$, over
$12.5\leq T\leq12.6$, which is $1061.3\leq t\leq1077.7$ and holds $14$
ordinates. In each panel $|\zeta|$ is black,
the upper envelope $2L_1$ red, the lower envelope
$2L_1|1-\vartheta_2/\pi|$ a teal sawtooth, and the gray band between them
is where $|\zeta|$ is confined. Red dots on the axis are the ordinates, where
the lower envelope touches zero along with $|\zeta|$, and the open circles are
the touches of the upper envelope, one for each vanishing of the offset $h$.
\emph{Left:} the split at the partial summand, $B_1=\Sigma_1+R_{1}$, whose $h$
fails to alternate at a retrograde ordinate, so it reaches its bound in only $7$
of the $13$ gaps and several humps rise nowhere near it. \emph{Right:} the
velocity split, whose $h^{\ast}$ vanishes at every extremum of $Z$, so the upper
envelope is touched in every gap. Generated by \texttt{fig\_envelope.py}.}
\label{fig:envelope}
\end{figure}

\begin{remark}[regarding hunting for large values with the first few links]
\label{rem:large-values}
Neither the envelope nor the slope of Remark~\ref{rem:nstar-slope} can be used to
scout, since $2L_1$ needs $R=\zeta-\Sigma_1-\Sigma_2$ and $2L_1^{\ast}$ needs $Z$
and $Z'$, so both are functions of $\zeta$ at a point one has already evaluated.
What can be computed without $\zeta$ is the reason the envelope grows. Leg~1 is
mostly the link sum, so
the size of $2L_1$ is a question of how well the link phases $t\log n$ align, and
the ceiling is the aligned case,
\begin{equation}\label{eq:link-ceiling}
\bigl|\zeta\bigl(\tfrac12+it\bigr)\bigr|\;\leq\;2L_1\;\leq\;
2\sum_{n\leq m}n^{-1/2}+2|R_{1}|\;\leq\;4\sqrt{m}\;\leq\;4\sqrt{T}
\;<\;4\Bigl(\frac{t}{2\pi}\Bigr)^{1/4}
=\frac{4}{(2\pi)^{1/4}}\,t^{1/4}\;<\;2.54\,t^{1/4},
\end{equation}
a bound that costs nothing to evaluate. Every step is an inequality. The
middle one uses $\sum_{n\leq m}n^{-1/2}\leq2\sqrt m-1$, which follows by
induction from $(m+1)^{-1/2}\leq2\sqrt{m+1}-2\sqrt m$, together with
$|R_1|=\hat d_1M^{-1/2}\leq1$, so the fractional link is paid for by the $-1$
and does not disappear. The last inequality in $T$ is $t>2\pi T^2$, which holds
for every $T>0$ because \eqref{eq:IT} makes it $2x+x^2>2\log(1+x)$ at
$x=1/T$. Read in $t$ it is nothing new: the
exponent $\tfrac14$ is the convexity, or trivial, bound $\mu(\tfrac12)\le\tfrac14$
of Lindel\"of~\cite{Lindelof1908}, where
$\mu(\sigma)=\inf\{a:\zeta(\sigma+it)=O(t^{a})\}$,
and \eqref{eq:link-ceiling} recovers it with an
explicit constant. Nor can this path do better. The last step is the triangle
inequality, which keeps only the aligned case and so discards every cancellation
among the phases $t\log n$; in the geometry of this paper $4\sqrt T$ is the arc
length of the chain laid out straight, while $|\zeta|$ is the distance between its
endpoints once it has curled into its spirals. Every improvement past $\tfrac14$
comes from that discarded cancellation, beginning with the $\tfrac16$ that Hardy
and Littlewood obtained by Weyl's method on this same sum and standing today at
$\tfrac{13}{84}$~\cite{Bourgain2017}; the Lindel\"of hypothesis is
$\mu(\tfrac12)=0$. Two things follow, both drawn in
Figure~\ref{fig:large-values}. First, the alignment can be read from a
truncation. Scoring an interval by
$S_K(t)=\bigl|\sum_{n\leq K}n^{-1/2}e^{-it\log n}\bigr|$ and then evaluating $Z$
only at the strongest peaks of that score recovers most of the interval's largest
$|Z|$: averaged over the $19$ unit intervals $5\leq T\leq24$, which is
$190\leq t\leq3771$, the three strongest peaks of the five-link score reach
$0.967$ of it and the single strongest peak of the ten-link score reaches
$0.961$. The point is not the saving at these heights but that $K$ need not grow
with $t$ while $m$ does. Second, the largest sampled values in the intervals
are themselves nearly
regular: $\max|Z|$ over $[T,T+1]$, taken on a grid of eight points per mean
zero spacing and so liable to sit a percent or two under the true crest, fits
$2.93\,T^{0.485}$, its ratio to
$\sqrt{T}$ has mean $2.823$ and standard deviation $0.123$ across those
intervals, and it holds a fraction of the ceiling \eqref{eq:link-ceiling} that
ranges from $0.97$ down to $0.76$, the two extremes falling at $T=9$ and
$T=22$. Most of that spread is the ceiling rather than the maxima: the exact sum
$\sum_{n\leq m}n^{-1/2}=2\sqrt m+\zeta(\tfrac12)+o(1)$ climbs toward its own
asymptote $2\sqrt m$ across the range, from $0.72$ of $4\sqrt T$ at $T=5$ to
$0.86$ at $T=23$, since $\zeta(\tfrac12)=-1.4604$ still weighs heavily at
$m=5$. Against $\sqrt T$ itself the maxima are flat, $2.711$ at $T=5$ and
$2.670$ at $T=23$, which is the same fact as the fitted exponent $0.485$ being
essentially the ceiling's $\tfrac12$: at these heights the maxima simply ride the
trivial bound. In these coordinates the Lindel\"of hypothesis is precisely the
statement that they eventually stop, that $\max|Z|/\sqrt T$ falls like
$T^{-1/2+\varepsilon}$, the chain curling up so efficiently that the distance
between its endpoints grows more slowly than every positive power of $T$ while
its arc length grows like $\sqrt T$. That distance is not bounded, and is not
claimed to be: $|\zeta(\tfrac12+it)|$ is known to be unbounded. Nothing of that is
visible over $19$ intervals with $m\leq23$ links, and nothing here bears on it
either way. So the answer
to whether the next unit of the index will beat the last is yes by default, which
it is in $13$ of the $18$ consecutive comparisons, and the exceptions are what
needs predicting: on the maxima divided by $\sqrt{T}$ the peak of the ten-link
score ranks the intervals at $+0.72$, while the three-link score, whose peak is
nearly the same in every interval, carries nothing at $-0.11$. That last figure
rests on $19$ intervals at modest height and should be read as an indication
only. Verified in \texttt{check\_large\_values.py}.
\end{remark}

\begin{figure}[H]
\centering
\includegraphics[width=\textwidth]{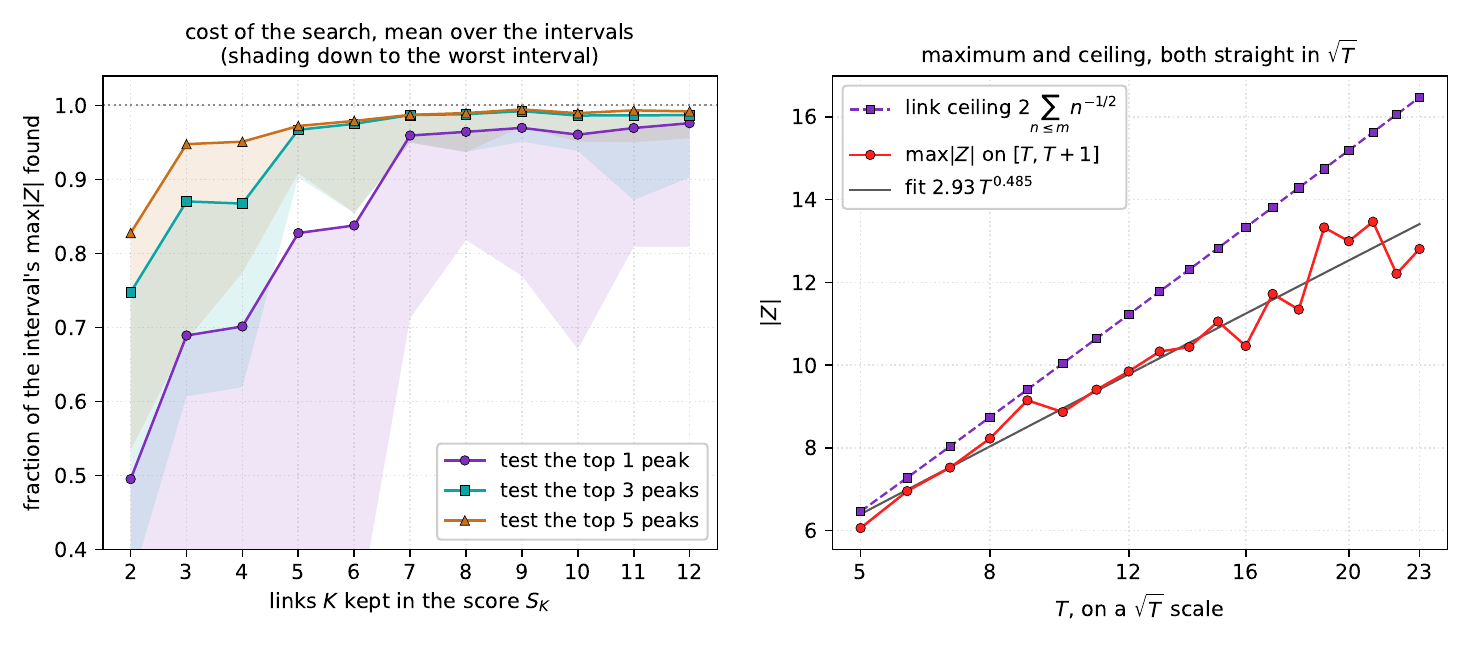}
\caption{Reading large values off the links, over the $19$ unit intervals
$5\leq T\leq24$. \emph{Left:} the fraction of an interval's largest $|Z|$ found by
evaluating $Z$ only at the strongest one, three or five peaks of the $K$-link
score $S_K$, averaged over the intervals, with the shading reaching down to the
worst single interval. With seven links a single candidate captures $95.9\%$ on
average, while the worst sampled interval captures $71.2\%$. \emph{Right:} the interval maximum against $\sqrt{T}$, on which it is
nearly straight, with the link ceiling of \eqref{eq:link-ceiling} above it and
the fitted power law through it. The gap between the two curves is the alignment
the links never quite achieve. Generated by \texttt{fig\_large\_values.py}.}
\label{fig:large-values}
\end{figure}

\subsection{Collinearity of the three first-half remainders}\label{sec:colinearity}

Since all the remainders are relative to the same point, namely $\Sigma_1$,
and since there are three of them, we decided to investigate their
collinearity, which we could not do when there was just one remainder split
(say $R/2$) or two (say $R_{1ak}$), because any two points already lie on a
line.

Collinearity of three points is unchanged by a common translation. The three
points $B_{1ps}$, $B_{1rs}$, and $B_{1ak}$ of
\S\ref{sec:ps-ak-r2} are all relative to the same $\Sigma_1$, so that shared
$\Sigma_1$ cancels, and the
three remainder vectors $R_{1ps}=R_1$, $R_{1rs}=\tfrac12 R$, and $R_{1ak}$ can be
compared directly as points in the plane.
Figure~\ref{fig:colinearity-strip} plots two independent, scale-invariant
measures of how far these three points are from lying on a common line. The
first is the triangle flatness
\begin{equation}\label{eq:flatness}
F
\;=\;
\frac{2\,\lvert\mathrm{area}\rvert}{\mathrm{diam}^2},
\end{equation}
where the area is that of the triangle with the three points as vertices and
the diameter is the largest of the three pairwise distances. The second is
the aspect ratio $\lambda_{\min}/\lambda_{\max}$ of the covariance matrix of
the three points (a principal-component measure). Both quantities are zero
exactly when the points are collinear, and both are invariant under
translation, rotation, and rescaling, so the shrinking of the remainders
with growing $T$ cannot masquerade as collinearity. Each panel is a grid of
$1501$ $\sigma$-samples by $1801$ $T$-rows over $0\le\sigma\le1$,
$3.9\le T\le5.1$, computed from the exact remainder formulas, with the
$\sigma$-sample nearest $\tfrac12$ snapped exactly onto the critical line.
The color scale is logarithmic, and that is what separates residuals at
the double-precision floor, $10^{-16}$ to $10^{-12}$, from the merely
small values, $10^{-4}$ to $10^{-1}$, that fill the rest of the strip.
Exact collinearity on the critical line is the proposition below, not the
heatmap; the heatmap shows that the residual sits at the floor there and
nowhere else in the sampled window.

\begin{figure}[H]
\centering
\includegraphics[width=0.95\textwidth]{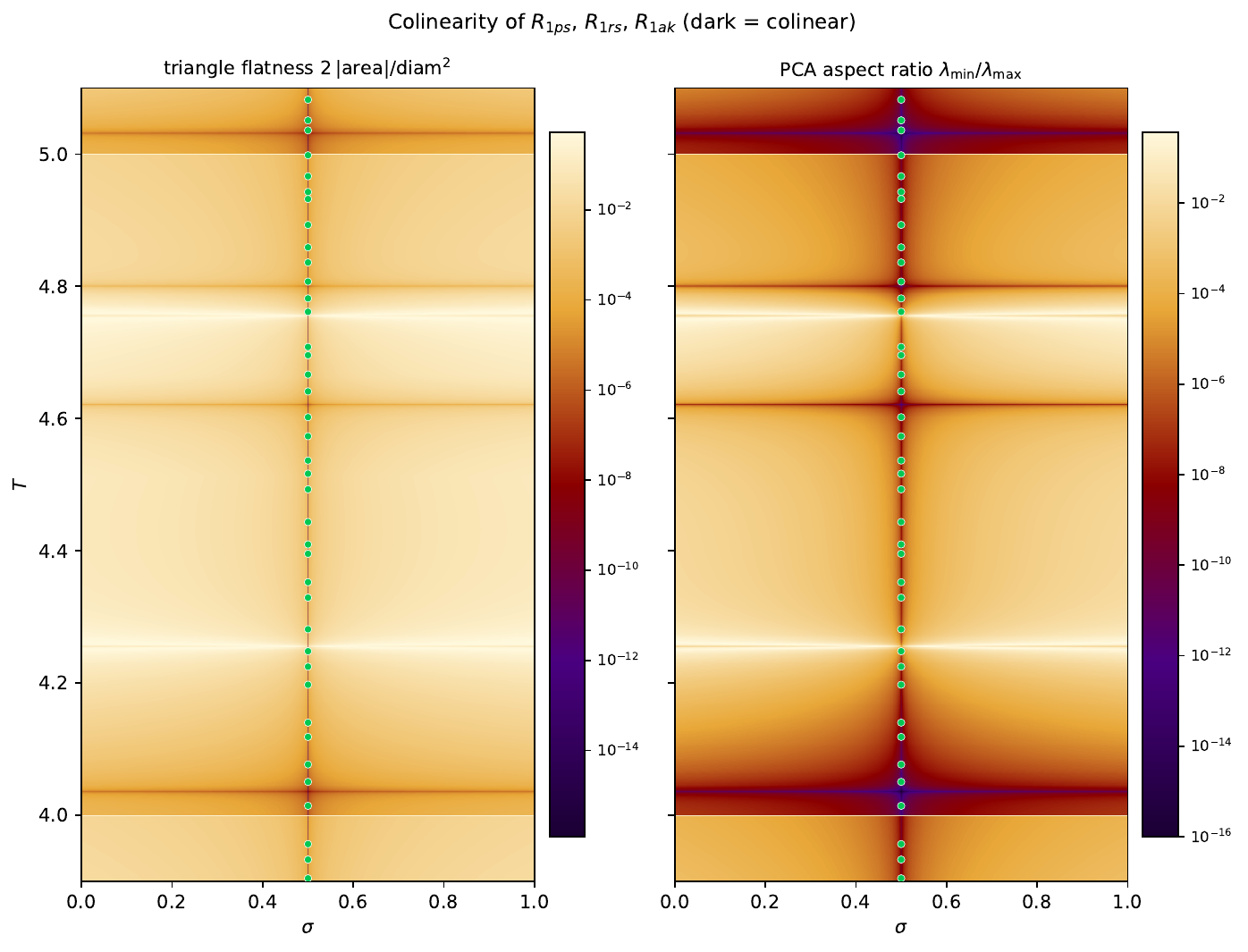}
\caption{Two ways of displaying the same collinearity information for the
three first-half remainders $R_{1ps}$, $R_{1rs}=\tfrac12 R$, $R_{1ak}$
(dark $=$ collinear). The panel on the left uses the flatness measure
\eqref{eq:flatness} of the triangle formed by the three points, and the one
on the right uses the PCA aspect ratio $\lambda_{\min}/\lambda_{\max}$.
Horizontal axis
$\sigma\in[0,1]$; vertical axis the spiral index $T\in[3.9,5.1]$.
Both panels use a log color scale; green
markers are zeta zeros at $\sigma=\tfrac12$. The dark column pinned at the
numerical floor along $\sigma=\tfrac12$ shows the three points are exactly
collinear everywhere on the critical line; off the line collinearity occurs
only along the isolated dark horizontal bands. Generated by
\texttt{fig\_colinearity\_zoom.py} from data produced by
\texttt{remainder-colinearity-heatmap-zoom.mjs}.}
\label{fig:colinearity-strip}
\end{figure}

Both panels show the same two facts. First, the column $\sigma=\tfrac12$
sits at the numerical floor at every $T$: on the critical line the three
first-half remainders are exactly collinear, not just approximately
(Proposition~\ref{prop:colinear-line}). Second,
off the line the measures are many orders of magnitude larger everywhere
except along isolated dark horizontal bands (near $T\approx4.03$, $4.62$,
$4.80$, and $5.03$ in this window), where the triangle momentarily passes
through a flat configuration as its signed area changes sign. The measured
values are symmetric under $\sigma\mapsto1-\sigma$, a reflection of the
functional equation.

The same pattern appears, in the sampled data, across $0\le\sigma\le1$
for $0<T\le20$.
Figure~\ref{fig:colinearity-four-panel} repeats the four-panel layout of
Figure~\ref{fig:leg-imbalance} (two windows spanning the unit cells
$5.95\le T\le7.05$ and $10.95\le T\le12.05$, a third band straddling the
integer $T=17$, and the same full strip
$0\le T\le20$), with the PCA aspect ratio as the plotted color map.
As in the leg-imbalance panels, the dark column on $\sigma=\tfrac12$ runs
unbroken through every window, while off the line the only dark features are
the isolated horizontal collinearity bands. The bands do not sit at integer
$T$: crossing an integer the three remainder formulas jump ($\Sigma_1$ gains
a summand and the sign $(-1)^{\lfloor T\rfloor}$ flips), the signed area of
the triangle changes sign in that jump and lands very close to zero, and the
area then crosses zero a short distance above the integer. The third panel
shows one full unit interval, $16.95\le T\le18.05$: the light-to-dark break
at the integer $T=17$ itself, a nearly flat strip at every $\sigma$ just
above it, and then the five collinearity bands of the cell, at
$T\approx17.011$, $17.252$, $17.612$, $17.752$, and $17.796$, before the
pattern repeats across the next break at $T=18$. Numerically the bands come
five to a unit interval, near $\{T\}\approx0.011$, $0.252$, $0.612$,
$0.752$, $0.796$
at this height, with the second and fourth tracking the pole heights
$\{T\}=\tfrac14,\tfrac34$ of $d_1$ (\S\ref{sec:pole-locations}): near a pole
$R_{1ps}$ runs off to infinity along the fixed direction of the next
summand, so the triangle is dragged through a flat configuration nearby.
The band positions are almost independent of $\sigma$ (they move by less
than $10^{-4}$ between $\sigma=0.2$ and $\sigma=0.35$), which is why the
bands are horizontal.

\begin{figure}[H]
\centering
\includegraphics[width=\textwidth]{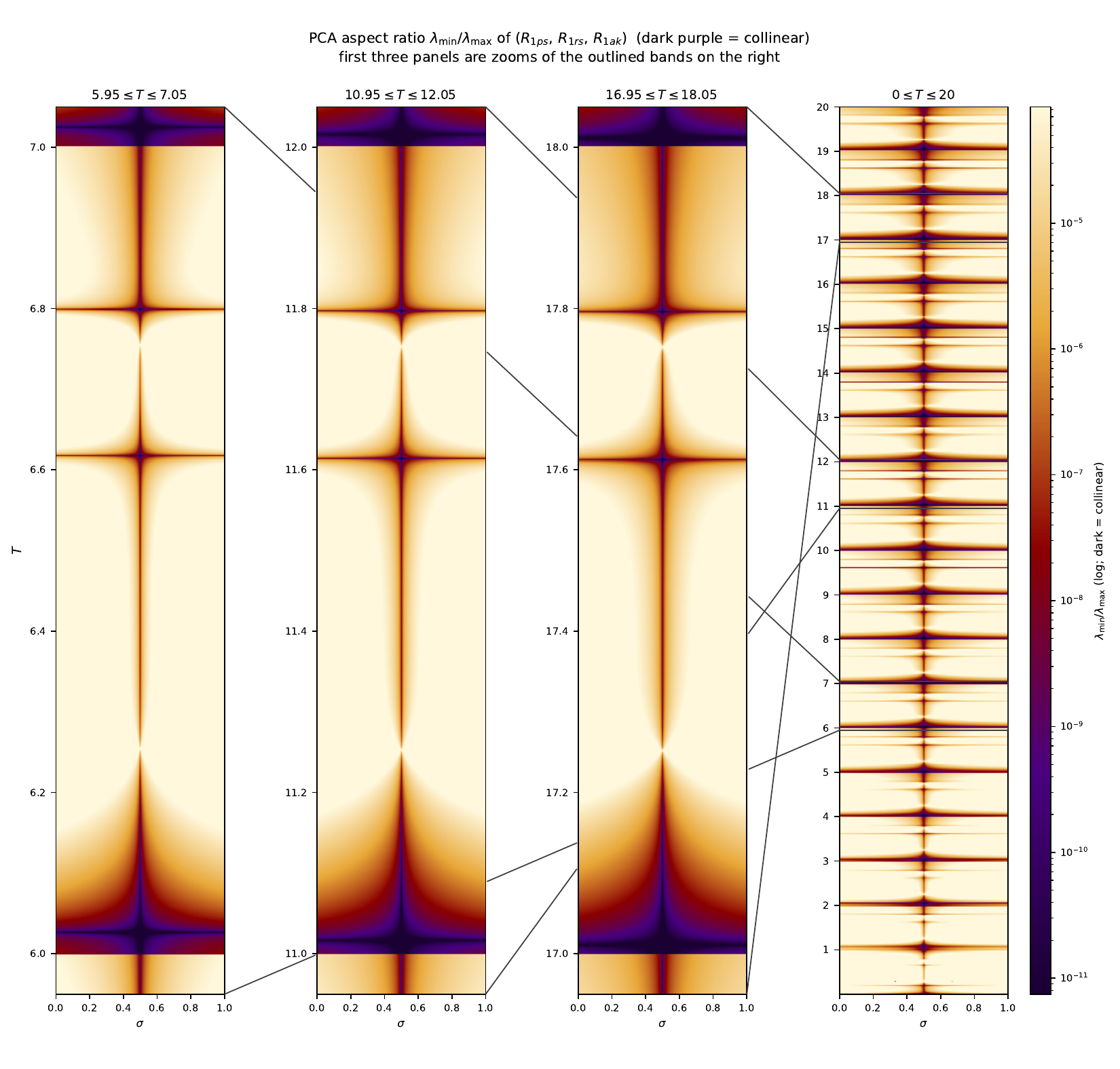}
\caption{PCA aspect ratio $\lambda_{\min}/\lambda_{\max}$ of the three
first-half remainders $(R_{1ps},R_{1rs},R_{1ak})$ (dark purple $=$
collinear), in the four-panel layout of Figure~\ref{fig:leg-imbalance}.
Horizontal axis $\sigma\in[0,1]$; vertical axis the spiral index $T$. Color
is on a log scale. The
rightmost panel is the full strip $0\le T\le20$; the first three panels are
zooms of the outlined bands
($5.95\le T\le7.05$, $10.95\le T\le12.05$, $16.95\le T\le18.05$).
The dark vertical band on $\sigma=\tfrac12$ is the exact collinearity of the
three first-half remainders on the critical line; the dark horizontal bands
are the isolated off-line collinearity events. Generated by
\texttt{plot\_colinearity\_four\_panel.py}.}
\label{fig:colinearity-four-panel}
\end{figure}

Why can this collinearity not be used like the two routes of
\S\ref{sec:toward-proof}? Because collinearity
is \emph{not} forced at a zero. For $\zeta\neq0$ the equal-leg condition for
a split says precisely that $B_1$ is equidistant from $O$ and
$\zeta$, i.e.\ that $B_1$ lies on the perpendicular bisector of the chord
from $O$ to $\zeta$; if all three splits had equal legs at a point where
$\zeta\neq0$, all three bisector points would lie on that single line and the three
first-half remainders would be collinear. But at a zero the chord
degenerates: $\zeta=O$, every point of the plane is equidistant from the two
endpoints, and the (now automatic) equal-leg conditions carry no
collinearity information at all. All three leg pairs are equal at a zero,
as they must be, while the three points can remain in general position.
So a zero implies equal legs for all three splits, but not collinearity, and
the three points need not line up at an off-line zero.

\begin{proposition}\label{prop:colinear-line}
On $\sigma=\tfrac12$ the three first-half remainders $R_{1ps}$,
$R_{1rs}=\tfrac12 R$, and $R_{1ak}$ are collinear at every height $T$,
zeros included.
\end{proposition}

The three points $B_{1ps}$, $B_{1rs}$, and $B_{1ak}$ are those remainders
translated by the common $\Sigma_1$, so it is the same to prove the $B_1$
collinear. On the critical line write $\vartheta$ for the Riemann--Siegel
theta function, so $\chi=e^{-2i\vartheta}$ and $Z=e^{i\vartheta}\zeta$ is
real. Any split $\zeta=B_1+B_2$ that satisfies the reflection
$B_2=\chi\overline{B_1}$ then has
\begin{equation}\label{eq:B1-line}
e^{i\vartheta}B_1
\;=\;
\tfrac12 Z+ih,
\qquad h\in\mathbb R.
\end{equation}
Indeed $e^{i\vartheta}B_2=\overline{e^{i\vartheta}B_1}$, so
$e^{i\vartheta}\zeta=e^{i\vartheta}B_1+\overline{e^{i\vartheta}B_1}
=2\Re(e^{i\vartheta}B_1)$, and the left side is the real number $Z$.
Thus every such $B_1$ has the same real part $Z/2$ in the $\vartheta$-frame:
all of them lie on the single vertical line $\Re(e^{i\vartheta}z)=Z/2$,
which is the perpendicular bisector of the segment from $O$ to $\zeta$ when
$\zeta\neq0$, and the imaginary axis of that frame when $\zeta=0$.
The three named splits are of this kind. For the $ps$ split,
$B_2=\chi\overline{B_1}$ is Corollary~\ref{cor:equal-legs}. For the $rs$
split the same reflection is the identity recorded in
\S\ref{sec:ps-ak-r2}: $\Sigma_2=\chi\overline{\Sigma_1}$ and
$R=\chi\overline{R}$, so
$\chi\overline{B_{1rs}}=\Sigma_2+\tfrac12 R=\zeta-B_{1rs}$.
For the $ak$ split, Siegel's contour integral $f$ of
\eqref{eq:siegel-58} gives the exact pairing
$\zeta(s)=f(s)+\chi(s)f(1-s)$~\cite{Siegel1932}, and on the line
$f(1-s)=\overline{f(s)}$, so $B_{1ak}=f(s)$ satisfies the same
reflection; the numerical $R_{1ak}$ of \eqref{eq:Rak-split} is the
Dirichlet-sum truncation of that identity, which is why the heatmaps
sit at the double-precision floor rather than at a symbolic zero.
Equation~\eqref{eq:B1-line} does not use $\zeta\neq0$, so the three
points remain collinear at a critical-line zero as well: there $Z=0$ and
they lie on $\Re(e^{i\vartheta}z)=0$.

A proof through collinearity therefore still needs a bridge that the routes of
\S\ref{sec:toward-proof} do not. The first step of that bridge is the
identity just proved. What remains is to prove that off the line
collinearity fails outside isolated curves like the dark bands of
Figure~\ref{fig:colinearity-strip}, and to connect a hypothetical
off-line zero to that structure, for instance through nearby off-line
points with $\zeta\neq0$ at which all three splits have equal legs, which
by the perpendicular-bisector argument above would force a collinearity
that the off-line geometry forbids. Those two steps would turn
collinearity into a characterization of the critical line and then into
a statement about zeros; both remain unproved.

\subsection{Joint angles}\label{sec:joint-angles}

In our numerical explorations of observing the links in their own coordinate
frames, which led to \eqref{eq:zeta-any-cut}, we decided to look at the joint
angles too. So we plotted them, and noticed a pattern emerge at the Farey
fractions. This section is the result of that investigation.

The joint angles of the forward chain can be read on their own, without the
chain visualization, using one dot per joint with
the turning angle on the vertical axis, the joint index along the horizontal.
The turning angle at joint $n$ is the continuous joint angle \eqref{eq:jangle}
evaluated at an integer joint,
\begin{equation}\label{eq:theta-joint}
\theta_n \;=\; -I(T)\,\log\tfrac{n}{n-1}
\qquad\text{(reduced to $[-\pi,\pi]$)},
\qquad
\nu(n) \;=\; \frac{I(T)}{2\pi\,n(n-1)},
\end{equation}
where $\nu$, the derivative of
the unwrapped angle divided by $2\pi$, is the local frequency of the strip in
cycles per joint. Two features organize the picture, both consequences of
$I(T)$. First, at the bisector joint $n=\lfloor T\rfloor+1$ the frequency is
$\nu=1+\frac{1}{6T(T+1)}$, one full turn per joint to nine decimals at
$T=13000$, and the angle there is $-\pi$ exactly: that is
\eqref{eq:jangle-count}, the bisector joint of \S\ref{sec:IT-origin}, seen as the right
edge of the strip. The joints just behind it inherit almost the same angle,
$3.141109$, $3.139659$, $3.137242$ at $n=T,T-1,T-2$, which is the same
near-stationarity that makes the bisector frame of
Figure~\ref{fig:bisector-frame} repeat. Second, since
$\nu(n)\approx T^{2}/n^{2}$ falls from $\nu$ of order $10^{8}$ at $n=2$ to $1$ at that
bisector joint, the strip is a sampled chirp, and it is stratified by the joints where
$\nu$ is the reciprocal of a rational. Writing that rational as a Farey
fraction $f=p/q$ in $(0,1]$, its \emph{caustic joint} is the solution of
$n(n-1)=f\,I(T)/2\pi$, that is $\nu=1/f$, so
\begin{equation}\label{eq:caustic-joint}
n_{f}\;=\;\tfrac12\Bigl(1+\sqrt{1+2f\,I(T)/\pi}\,\Bigr)\;\approx\;T\sqrt{f},
\end{equation}
and there the joints fall into exactly $p$ strands, the denominator of
$\nu=q/p$, each a parabola of curvature $4\pi\nu/n$ per joint squared because
the unwrapped angle is locally quadratic (exactly
$|\widetilde\theta''(n_f)|$ of \eqref{eq:caustic-derivs} below). Small denominators give the spines
visible in Figure~\ref{fig:joint-angles}, drawn at the integer index $T=13000$,
where $t=I(T)=1{,}061{,}939{,}999.369541$: there the fractions
$\tfrac14,\tfrac13,\tfrac12,\tfrac23,\tfrac34$ sit at
$n=6501,7506,9193,10615,11259$, and $f=1$ is the bisector joint itself. The parabolic
strands are the same local quadratic phase that produces the Cornu-like arcs of
the curlicue literature of \S\ref{sec:curlicues}, here read off the joints
rather than the partial sums.

\begin{figure}[H]
\centering
\includegraphics[width=\textwidth]{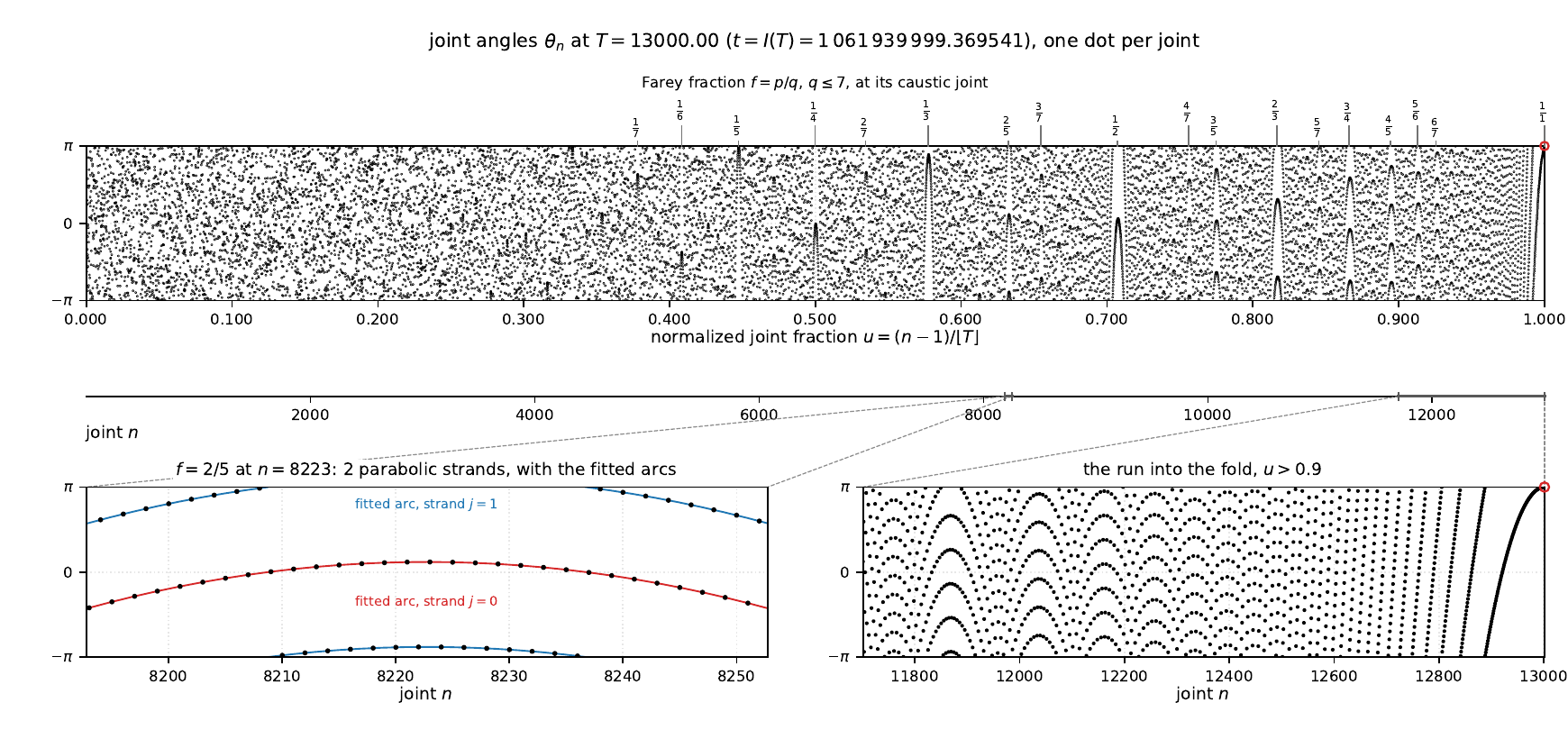}
\caption{The joint-angle strip at $T=13000$, where
$t=I(T)=1{,}061{,}939{,}999.369541$, one dot per joint of the forward
chain, $n=2,\dots,\lfloor T\rfloor+1$, with $\theta_n$ of
\eqref{eq:theta-joint} on the vertical axis. The two edges $\pm\pi$ are
identified, so the strip is a cylinder. \emph{Top:} all $13000$ joints, carrying
three horizontal scales: the Farey fractions $f=p/q$ with $q\le7$ above, each
at its caustic joint \eqref{eq:caustic-joint} and staggered in two rows where
they crowd; the normalized joint fraction
$u=(n-1)/\lfloor T\rfloor$ below, so the bisector joint sits at $u=1$; and the joint index
$n$ below that. The red circle at the right edge is the bisector joint
$n=\lfloor T\rfloor+1$, where $\theta=-\pi$ exactly. \emph{Bottom left:} a
window at $f=\tfrac25$, showing its $p=2$ parabolic strands as dots, each with
its fitted arc \eqref{eq:fitted-arc} drawn through them; the two arcs sit half a
turn apart, so strand $1$ has its vertex wrapped to the opposite edge.
\emph{Bottom right:} the last tenth of the strip, where the
strands steepen and run into that bisector joint on the right in the figure. Dashed lines carry each window back to
the stretch of the joint-$n$ scale it magnifies, marked there by a bracket.
Generated by \texttt{fig\_joint\_angles.py}.}
\label{fig:joint-angles}
\end{figure}

\paragraph{The fitted curve.} The strands are more than parabola-like: they can
be written down. Let $\widetilde\theta(n)=-I(T)\log\frac{n}{n-1}$ be the
unwrapped angle, so $\theta_n$ is $\widetilde\theta(n)$ reduced to
$[-\pi,\pi]$, and
differentiate it three times:
\begin{equation}\label{eq:jangle-derivs}
\widetilde\theta'(n)=\frac{I(T)}{n(n-1)},\qquad
\widetilde\theta''(n)=-\frac{I(T)\,(2n-1)}{[\,n(n-1)\,]^{2}},\qquad
\widetilde\theta'''(n)=\frac{2\,I(T)\,(3n^{2}-3n+1)}{[\,n(n-1)\,]^{3}}.
\end{equation}
At a caustic these have closed forms in the fraction alone, since
\eqref{eq:caustic-joint} fixes $n_f(n_f-1)=f\,I(T)/2\pi$ and
$2n_f-1=\sqrt{1+2f\,I(T)/\pi}$ exactly:
\begin{equation}\label{eq:caustic-derivs}
\begin{aligned}
\widetilde\theta'(n_f)&=\frac{2\pi}{f}=\frac{2\pi q}{p},
\qquad
\widetilde\theta''(n_f)=-\frac{4\pi^{2}}{f^{2}I(T)}
\sqrt{1+\frac{2f\,I(T)}{\pi}}\;\approx\;-\frac{4\pi}{f^{3/2}T},\\[2pt]
\widetilde\theta'''(n_f)&=\frac{16\pi^{3}}{f^{3}I(T)^{2}}
\Bigl(1+\frac{3f\,I(T)}{2\pi}\Bigr)\;\approx\;\frac{12\pi}{f^{2}T^{2}}.
\end{aligned}
\end{equation}
Write $\delta=n-n_f$ for the offset in joints from the caustic. The linear term
$\widetilde\theta'(n_f)\,\delta=(2\pi q/p)\,\delta$ is a carrier sweeping several
radians per joint, and it is what interleaves the dots into strands in the first
place. Along one strand, though, $\delta$ advances by $p$ at a time, and
$(2\pi q/p)\,pm=2\pi q\,m$ drops out: the carrier contributes nothing but a
constant. Absorbing it together with $\widetilde\theta(n_f)$, strand $j$ carries
the phase constant
\begin{equation}\label{eq:caustic-Cj}
C_j\;=\;\widetilde\theta(n_f)+\frac{2\pi q}{p}\bigl([n_f]+j-n_f\bigr),
\qquad j=0,\dots,p-1,
\end{equation}
with $[\,\cdot\,]$ the nearest integer, and the strand itself is that cubic
reduced to $[-\pi,\pi]$
\begin{equation}\label{eq:fitted-arc}
\rho_j(\delta)\;=\; C_j
+\tfrac12\widetilde\theta''(n_f)\,\delta^{2}
+\tfrac16\widetilde\theta'''(n_f)\,\delta^{3}
\qquad\text{(reduced to $[-\pi,\pi]$)}.
\end{equation}
Because $\gcd(p,q)=1$ the constants $C_j$ run through all $p$ residues of
$2\pi/p$, so the $p$ arcs are one parabola stacked at equal spacing around the
cylinder. Third order is the right order: the second derivative supplies the
parabola, the third its visible skew, and the fourth is smaller by a further
$\delta/n_f\sim T^{-1/2}$.

The arcs drawn in Figure~\ref{fig:joint-angles} are \eqref{eq:fitted-arc} at
$f=\tfrac25$ and $T=13000$, where $n_f=8222.74$,
$\widetilde\theta''(n_f)=-3.8208\times10^{-3}$ and
$\widetilde\theta'''(n_f)=1.3941\times10^{-6}$ radians per joint squared and
cubed, and the carrier is $2\pi q/p=5\pi\equiv\pi$, which is why the two strands
sit half a turn apart. Over the $\pm30$ joints shown they track the dots to
$2.2\times10^{-5}$ radians, that being the fourth-order term. Nothing is fitted
in the statistical sense, and \eqref{eq:fitted-arc} is evaluated exactly.
Two limits stand out.
Toward small $f$ the parabola stiffens like $f^{-3/2}$, which is why the
left end of the strip looks like noise at this scale: consecutive joints there
are many turns apart and only the caustics of very small $f$ remain resolved. At
the other end $f=1$ closes the family, with $p=q=1$: a single strand, a carrier of
one full turn per joint, and $C_0=-\pi$. That is the bisector joint of
\eqref{eq:jangle-count} again, so the bisector joint this paper is built on is the last
caustic of the sequence, and the right-hand panel of
Figure~\ref{fig:joint-angles} is the run of the high-denominator caustics into
it.

\paragraph{Incremental change.}\label{sec:incremental}
The strip is a snapshot. Advance the index and every dot moves, though not by
the same amount. Differentiating the unwrapped angle in $T$ at a fixed joint,
\begin{equation}\label{eq:jangle-drift}
\begin{aligned}
\frac{\partial\widetilde\theta}{\partial T}(n)&=-\,I'(T)\log\frac{n}{n-1},
\qquad L=\log\Bigl(1+\tfrac1T\Bigr),\\[2pt]
I'(T)&=\frac{\pi}{L^{2}}\Bigl(2L+\frac{2T+1}{T(T+1)}\Bigr)
=4\pi T+2\pi+O(T^{-1}),
\end{aligned}
\end{equation}
which in cycles is a drift rate
\begin{equation}\label{eq:jangle-drift-cycles}
\mu(n)\;=\;\frac{I'(T)}{2\pi}\log\frac{n}{n-1}
\;\approx\;\frac{2T+1}{n-\tfrac12}\;\approx\;\frac{2}{u},
\qquad u=\frac{n-1}{\lfloor T\rfloor}.
\end{equation}
A value of $\mu$ is easiest to read as a lap count: $\mu(n)=3$ means that dot
laps the strip three times, falling from $+\pi$ to $-\pi$, wrapping, and
repeating, while you advance $T$ by one. Nothing here needs differencing:
$I'(T)$ is available in closed form as
$\pi\bigl(2L+(2T+1)/T(T+1)\bigr)/L^{2}$ with $L=\log(1+1/T)$, which is
$4\pi T+2\pi$ to leading order, about $163{,}369$ at $T=13000$.

The sign is the same everywhere, so the whole strip drifts downward; the
magnitude is not, and that is the point. Since
$\mu(n)\approx(2T+1)/(n-\tfrac12)\approx2/u$, the rate is essentially a $1/u$
hyperbola, equal to exactly $2$ at the bisector joint $u=1$ and growing leftward
to $1.8\times10^{4}$ cycles per unit index at $n=2$. Incidentally, the
$2/u$ form is no longer any good there, reading $26{,}000$, since
$\log\frac{n}{n-1}$ and $1/(n-\tfrac12)$ part company at the first few joints.
The strip does not translate as the index advances, it shears. Read the other
way, $1/\mu(n)$ is the index step that returns joint $n$ to the same angle: a
dot at $u=\tfrac12$ comes back around every $\Delta T=0.25$, one at $u=0.1$
every $\Delta T=0.05$, and one at the bisector joint every half unit.

At the right edge the value is exact and already familiar: the middle expression
of \eqref{eq:jangle-drift-cycles} at $n=\lfloor T\rfloor+1$ is
$(2T+1)/(T+\tfrac12)=2$, and the exact rate is $2.000000001$ at $T=13000$. Two
cycles per unit index is one more than \eqref{eq:jangle-count} reports for the
joint angle, and the discrepancy is the joint itself moving: the bisector joint sits at
$n=\lfloor T\rfloor+1$, which advances one joint per unit index, and one joint
there is worth $\nu=1$ turn, so an observer riding the bisector joint sees $\mu-\nu=1$
cycle, namely $\frac{d}{dT}\pi(2T+1)=2\pi$. The two statements differ only in
whether one rides the bisector joint or stands at a fixed joint.

The pattern the dots drift through, by contrast, hardly moves at all.
Differentiating \eqref{eq:caustic-joint} gives
$dn_f/dT=f\,I'(T)\big/2\pi(2n_f-1)\approx\sqrt f$, so a caustic creeps outward
by a fraction of a joint per unit index, and its normalized position
$u_f=(n_f-1)/\lfloor T\rfloor\approx\sqrt f$ does not depend on $T$ at all: the
spines stand still while the joints stream through them. Their drift rate
follows from pairing \eqref{eq:jangle-drift-cycles} with the cycles per joint
$\nu=I(T)/2\pi n(n-1)$ of \eqref{eq:theta-joint}, which gives
$\mu\approx2\sqrt\nu$: at the caustic of Farey fraction $f$, where $\nu=1/f$,
the dots slide at $2/\sqrt f$ cycles per unit index, four at $f=\tfrac14$ and
two at $f=1$. Only at an integer index does the frame change, when a new link is
appended, a new joint appears at the right edge with $\theta=-\pi$, and $u$
rescales by one part in $\lfloor T\rfloor$.

Figure~\ref{fig:incremental-change} shows one step, $\Delta T=0.02$, which
advances $t$ by $\Delta I=3267.38$. The travel is $0.04/u$ cycles: $14.4^{\circ}$
at the bisector joint, $28.8^{\circ}$ at $u=\tfrac12$, $143.9^{\circ}$ at $u=0.1$, half a
turn at $u=0.0800$, joint $n=1041$, and a full turn at $u=0.0400$, joint
$n=521$. Since the motion is downward everywhere, the arrows in the figure are
drawn downward and wrapped at the lower edge, which is where they land for
joints left of $n=1041$. Left of $n=521$ the joint has gone round at least once,
and rather than draw what is left over the arrow there sweeps the whole strip
from $\pi$ to $-\pi$; at $n=2$ the travel is $2264$ radians, some $360$ turns.
This is the sense in which the left end of the strip is a chirp in the index as
well as in the joint number: it is not only that neighboring joints are many
turns apart at fixed $T$, but that one joint sweeps many turns over a step in
$T$ too small to move the bisector joint by a quarter of a radian. The bottom
panel of Figure~\ref{fig:incremental-change} is $\mu$ on a logarithmic axis, and
the middle panel is a picture of the same quantity at $\Delta T=0.02$.

\begin{figure}[H]
\centering
\includegraphics[width=\textwidth]{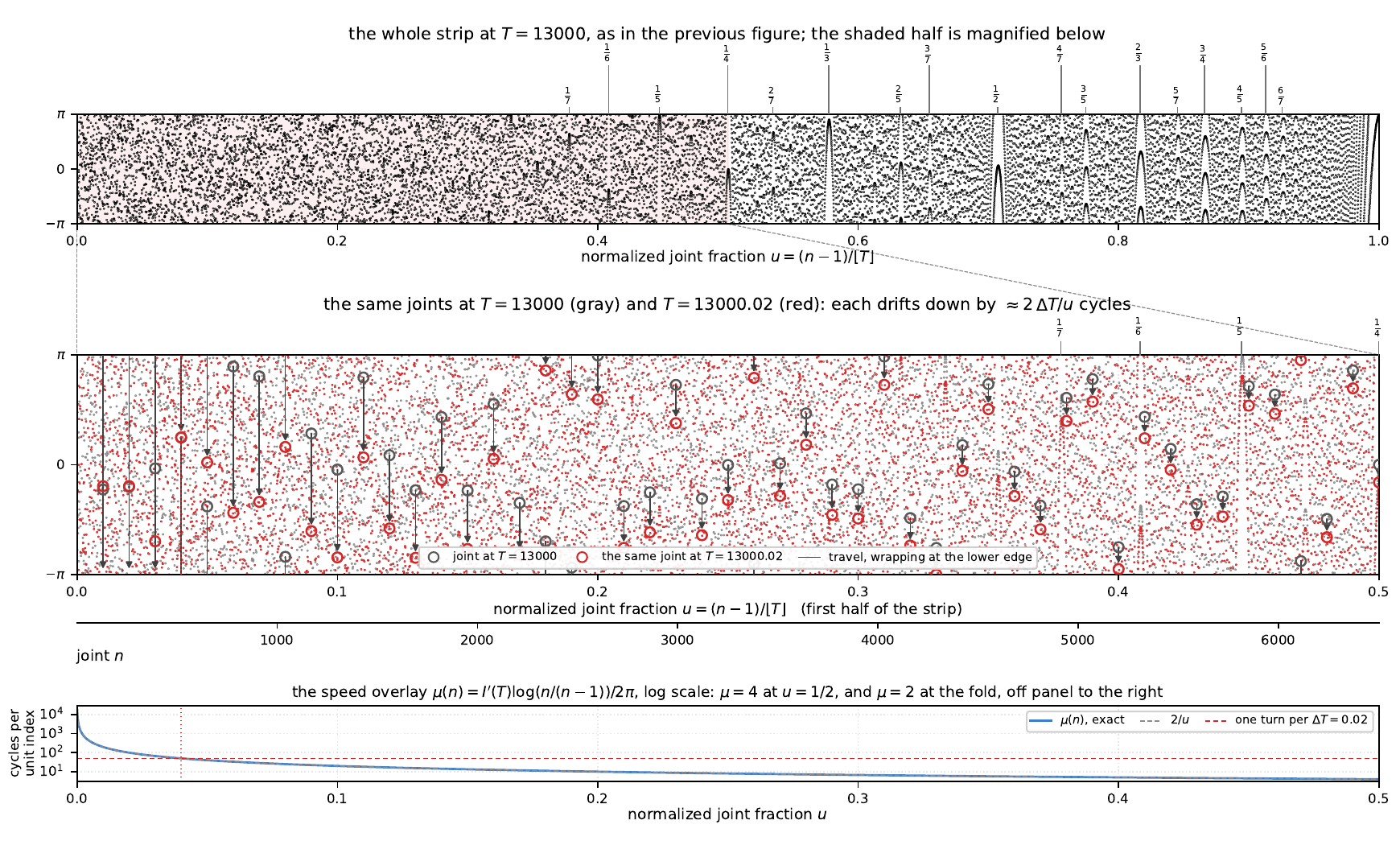}
\caption{One increment of the index, at $T=13000$. \emph{Top:} the whole
joint-angle strip as in Figure~\ref{fig:joint-angles}, with the shaded first
half magnified below. \emph{Middle:} that half twice over, gray at $T=13000$ and
red at $T=13000.02$, every dot having drifted downward by
$\mu(n)\,\Delta T\approx0.04/u$ cycles of \eqref{eq:jangle-drift-cycles}. Fifty
joints spaced evenly across the panel are circled in both colors, with an arrow
along the path traveled; since the drift is downward at every joint the arrows
run downward and wrap at the lower edge, and where the travel exceeds a full turn
the arrow sweeps the whole strip, from $\pi$ to $-\pi$. The travel
lengthens leftward like $1/u$, passing half a turn at $n=1041$ and a full turn
at $n=521$, marked by the dotted vertical line. Both strips carry the Farey
fractions $f=p/q$, $q\le7$, each at its caustic joint
\eqref{eq:caustic-joint}; only $f\le\tfrac14$ falls in the magnified half, since
$u_f\approx\sqrt f$. \emph{Bottom:} the drift rate itself on a logarithmic axis:
exact (blue) against the $2/u$ form (gray), the two separating only at the first
few joints, with the dashed line at one full turn per $\Delta T=0.02$. Generated by
\texttt{fig\_incremental\_change.py}.}
\label{fig:incremental-change}
\end{figure}

\clearpage

\section{Prior literature}\label{sec:prior-lit}

\subsection{Nickel's Argand-diagram geometry}\label{sec:nickel}

Nickel~\cite{Nickel2013}, with a condensed sequel~\cite{Nickel2015}, starts
from the same object this paper starts from,
the Argand diagram of the steps $n^{-s}$, and reaches a construction close
enough to ours that the two are worth setting side by side. They agree in
shape, and they differ in the one respect this paper is about.

\paragraph{Dictionary.} His center of symmetry is the step
$n_p=\bigl\lfloor\sqrt{t/2\pi}\,\bigr\rfloor$, with fractional part $p=\sqrt{t/2\pi}-n_p$. Since
$\sqrt{t/2\pi}\approx T+\tfrac12$ (\S\ref{sec:IT-origin}), his index is
$n_p\approx\lfloor T+\tfrac12\rfloor$ and $p\approx\{T+\tfrac12\}$, so it
nearly agrees with our $m=\lfloor T\rfloor$ on the first half of each unit interval
and exceeds it by one on the second half: the same half-unit offset already
recorded for Siegel's cutoff in \S\ref{sec:kuznetsov}. His factor $Q(s)$,
defined by $\zeta(s)=Q(s)\zeta(1-s)$, is our $\chi(s)$; the displayed form
$Q=n_p^{1-2s}e^{i(t+\pi/4)}$ is the familiar $e^{-2i\vartheta(t)}$ asymptotic,
which reproduces $\chi$ to a relative $1.5\times10^{-5}$ at $T=20.3$ and
$1.6\times10^{-7}$ at $T=200.3$ once the smooth $\sqrt{t/2\pi}$ is used in
place of the rounded $n_p$. His symmetry angle $\Theta$ is our
$\tfrac{\psi}{2}=\tfrac12\arg\chi$.

\paragraph{What matches.} Under that dictionary his Riemann--Siegel equation
$\zeta=P(s)+Q(s)P(1-s)$ is our two-leg decomposition $\zeta=B_1+B_2$ with
$B_2=\chi\,\overline{B_1}$ (Corollary~\ref{cor:equal-legs}, verified in Lean
in Appendix~\ref{app:lean-critical}); his polar form
$\zeta=2\cos(\phi-\Theta)\,\mathsf{P}e^{i\Theta}$ is the two-leg identity
$\zeta=B_1+\chi\overline{B_1}=2|B_1|\cos(\arg B_1-\tfrac{\psi}{2})e^{i\psi/2}$,
not the polar decomposition of $R$; and his observation that the
two magnitudes are equal exactly when $\sigma=\tfrac12$ is the on-line
half of Corollary~\ref{cor:equal-legs} (the present legs can also be equal
on isolated off-line loci). He also writes the classical remainder as
$R=2\mathsf{L}\cos(\theta_L-\Theta)$, where $L=\mathsf{L}e^{i\theta_L}$ runs
from the joint at step $n_p$ to his center: that is exactly the role played
here by the pair $(d_1,-\omega)$. Where his index agrees with ours the two
projected remainders converge, the ratio
$2\mathsf{L}\cos(\theta_L-\Theta)\,/\,2d_1\cos(\omega+\tfrac{\psi}{2})$
falling through $2.24$, $1.34$, $1.11$, $1.03$, $1.01$ at
$T=6.3,\,20.3,\,60.3,\,200.3,\,600.3$.

\paragraph{Two centers, one line.} His $P(s)$ and our $B_1$ are not the same
point, and they are not meant to be. The difference lies almost entirely
along the direction in which $\zeta$ cannot see it: at $T=600.3$ we find
$|P-B_1|=0.064$, of which only $1.2\times10^{-4}$ lies along $e^{i\psi/2}$.
That is his own remark that $\zeta$ is unchanged when $P(s)$ is translated
along the symmetry axis. In the notation of \S\ref{sec:legs} the
admissible centers are the solutions $X$ of $\zeta-X=\chi\overline{X}$, a
line in the plane, and every one of them projects onto $\zeta/2$, which is
Corollary~\ref{cor:bisector-proj}. The freedom he notes and the projection we
prove are the same fact.

\paragraph{The sequel.} Two things in \cite{Nickel2015} touch this paper
further. It traces $P(s)$ and $\zeta(s)$ as two curves over a range of $t$,
joined at regularly spaced values of $t$ to show the two segments, which is the
figure the conveyor belt draws with his
pendant center in place of $B_1$. And it counts zeros: it reads the count off
$\Theta$ through the Gram points and tests that reading against the first
$100{,}000$ ordinates, finding one zero per Gram point on average with the zeros
grouping between successive Gram points. The count there is the classical one,
exact in the mean and never landing on an ordinate; the curves of
\S\ref{sec:counting} land on every one.

\paragraph{What differs.} The correction itself. Nickel's is a lowest-order
determination, carrying a factor $1/(2n_p^{\sigma}\cos 2\pi p)$ that grows
without bound as $\cos2\pi p\to0$, and he is explicit that its error matters
for the distribution of $P(s)$ even though it cancels in $\zeta$; its
direction, $e^{-i(t\ln n_p+2\pi p)}$, is the step direction at $n_p$ turned
by the fractional offset $2\pi p$, not a summand direction of the chain. Our
$R_{1}$ and $R_{2}$ are exact, and their directions are exactly those of
the $(m+1)$st summand of each chain. That is what lets the serial chain of
\S\ref{sec:product-form} close on $\zeta$ identically rather than
asymptotically, and it is the content of the ``one more partial summand''
reading: the correction is not only near the next link but along it.

\subsection{Levinson's \texorpdfstring{$G$}{G}}\label{sec:levinson}

The companion $B_1^{\ast}=\zeta+\zeta'/2\vartheta'$ of \eqref{eq:vel-split},
whose argument carries the counting curve $N_{\ast}$ of
\S\ref{sec:counting-star}, is the function Levinson~\cite{Levinson1974} calls
$G$. That is the paper in which he proves that more than a third of the zeros
of $\zeta$ are on $\sigma=\tfrac12$. Figure~\ref{fig:b1-star-path} draws the
function itself on the critical line, in the plane and in the frame where its
winding is the count.

\begin{figure}[H]
\centering
\includegraphics[width=\textwidth]{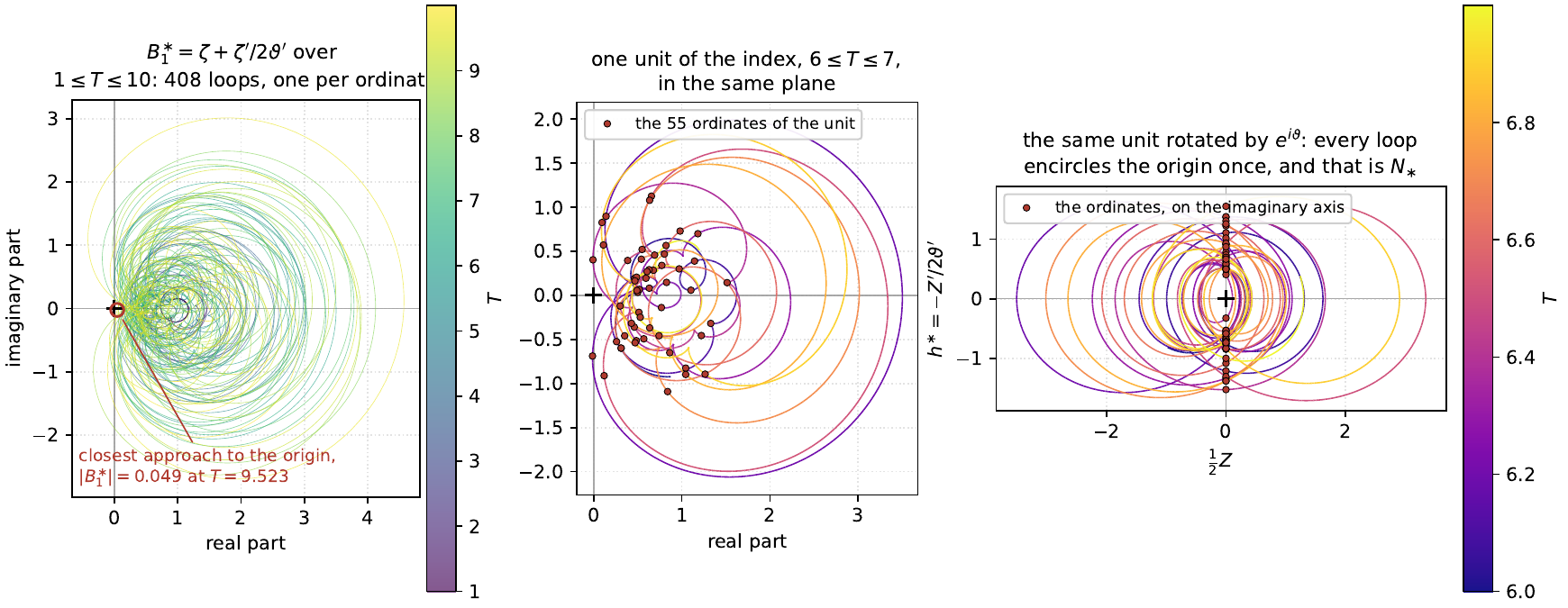}
\caption{Levinson's $G$, which is the companion
$B_1^{\ast}=\zeta+\zeta'/2\vartheta'$ of \eqref{eq:vel-split}, drawn as a path
in the plane on the critical line.
\emph{Left:} the whole of $1\leq T\leq10$, that is $13.6\leq t\leq692.2$, the
color running with $T$: $408$ loops, one for each ordinate of the range. The
path never reaches the origin, which is what lets $\arg G$ carry a count, and
its closest approach, $|G|=0.048$ at $T=9.523$, falls at the tightest pair of
ordinates in the range, $\gamma_{363}$ and $\gamma_{364}$, whose gap is $0.24$
of the mean spacing. There $Z'=0$ and $|G|=\tfrac12|Z|$: it is a Lehmer-style
near-miss, not a small value of $\zeta$ on its own, that brings the curve in.
\emph{Middle:} one unit of the index, $6\leq T\leq7$, in the same plane, with
its $55$ ordinates marked.
\emph{Right:} the same unit rotated by $e^{i\vartheta}$, the frame of
\eqref{eq:h-star} in which the real part is $\tfrac12Z$ and the imaginary part
is the offset $h^{\ast}$. Every loop encircles the origin once and crosses the
imaginary axis exactly at an ordinate. Counting those crossings is Levinson's
reading of the angle; following the turn continuously is $N_{\ast}$.
Computed from $e^{i\vartheta}G=\tfrac12Z-iZ'/2\vartheta'$, which agrees with
$\zeta+\zeta'/2\vartheta'$ to $18$ digits.
Generated by \texttt{fig\_b1\_star\_path.py}.}
\label{fig:b1-star-path}
\end{figure}

\paragraph{Dictionary.} He writes the gamma factor as
$h(s)=\pi^{-s/2}\Gamma(s/2)=e^{f(s)}$, so the functional equation reads
$h(s)\zeta(s)=h(1-s)\zeta(1-s)$. Differentiating it and using it again to
remove $\zeta(1-s)$ leaves him with the argument of
\begin{equation}\label{eq:levinson-G}
G(s)=\zeta(s)+\frac{\zeta'(s)}{f'(s)+f'(1-s)}
\end{equation}
to control, which is his (1.9). His weight is ours, in two steps. From
$f(s)=-\tfrac{s}{2}\log\pi+\log\Gamma(s/2)$,
\[
f'(s)=-\tfrac12\log\pi
+\tfrac12\frac{\Gamma'}{\Gamma}\Bigl(\frac{s}{2}\Bigr),
\]
and $\chi=h(1-s)/h(s)$ has $\log\chi=f(1-s)-f(s)$, so differentiating in $s$
gives
\begin{equation}\label{eq:levinson-weight}
\frac{\chi'}{\chi}(s)=-f'(1-s)-f'(s),
\qquad\text{that is}\qquad
f'(s)+f'(1-s)=-\frac{\chi'}{\chi}(s),
\end{equation}
at every $s$ where the displayed quotients are defined. That is the weight of the off-line form
\eqref{eq:vel-split-offline}, so $G=B_1^{\ast}$ as a meromorphic identity
wherever those quotients are finite and nonzero. On the
critical line the same weight turns real: with $s=\tfrac12+it$ we have
$1-s=\bar{s}$ and $f'(1-s)=\overline{f'(s)}$, so the sum is twice a real part,
\[
f'(s)+f'(1-s)=2\operatorname{Re}f'\bigl(\tfrac12+it\bigr)
=\operatorname{Re}\frac{\Gamma'}{\Gamma}
\Bigl(\tfrac14+\tfrac{it}{2}\Bigr)-\log\pi
=2\vartheta'(t),
\]
the last equality being the exact $\vartheta'$ recorded above
\eqref{eq:vel-split}. Substituting that into \eqref{eq:levinson-G} returns
$\zeta+\zeta'/2\vartheta'$, which is \eqref{eq:vel-split}. So the two are the
same function and not two versions of one: at $t=20$, to take one height, the
two weights are both $1.157750994844834$ and the two companions agree to the
forty digits carried. Levinson himself trades
the weight for $\log(t/2\pi)$ at his (1.4), which costs $O(1/t)$; we keep the
exact $2\vartheta'$. In the language of \S\ref{sec:legs}, $G$ is a leg: the
first leg of the split whose second leg $\zeta-G$ is the velocity of the
endpoint of the chain along the line, quarter-turned and halved by the local
rate of the phase.

\paragraph{The same angle.} His (1.8) locates the ordinates by the condition
$\arg\bigl(h(s)G(s)\bigr)\equiv\tfrac{\pi}{2}\pmod\pi$, and on the critical line
$\arg h$ is exactly the Riemann--Siegel phase,
$\arg h(\tfrac12+it)=\arg\Gamma(\tfrac14+\tfrac{it}{2})-\tfrac{t}{2}\log\pi
=\vartheta(t)$. So the angle he watches is the one \eqref{eq:N-star} is built
from,
\[
\arg\bigl(hG\bigr)=\vartheta+\arg B_1^{\ast}
=\pi\bigl(N_{\ast}-\tfrac32\bigr)\pmod{2\pi},
\]
and his condition on it is precisely the statement that $N_{\ast}$ is an
integer. The two readings of the same angle differ in what is asked of it.
Levinson needs only how often the condition holds, so what he must control is
how far $\arg G$ can move: he counts zeros of $G$ in a rectangle with the
critical line for its left edge and $\sigma=3$ for its right, where
$|G-1|<\tfrac13$ pins the argument down, and a mollifier and Littlewood's
lemma bound the count in between. His remark that $\zeta$ would have
essentially all of its zeros on the line if $\arg G$ never moved at all is the
extreme case of the same reading. Here the angle is wanted between the
ordinates as well as at them. The exact weight gives
$\operatorname{Re}\bigl(e^{i\vartheta}G\bigr)=\tfrac12Z$ by \eqref{eq:h-star},
so the angle advances by exactly $\pi$ from one ordinate to the next, and
$N_{\ast}$ is a continuous curve, integer-valued at the ordinates and nowhere
else, carrying in its slope the size of $|Z|$
(Remark~\ref{rem:nstar-slope}). One further consequence of
\eqref{eq:levinson-G} is common to both readings, the corollary Levinson
credits to Montgomery: a zero of $\zeta$ of multiplicity $m$ leaves $G$ a zero
of multiplicity $m-1$.

%
%
%
%
%

\subsection{Spirals of exponential sums}\label{sec:curlicues}

For quadratic phases this geometry is rigorous ground. The partial sums
$S_N(\tau)=\sum_{n\le N}e^{i\pi\tau n^{2}}$ draw what Berry and
Goldberg~\cite{BerryGoldberg1988} call curlicues: replacing the sum by an
integral produces the Cornu spiral exactly, and a renormalization then carries
the whole pattern onto a magnified and rotated copy of a shorter sum of the same
kind. Coutsias and Kazarinoff~\cite{CoutsiasKazarinoff1987,CoutsiasKazarinoff1998}
work that picture the way this paper works ours. They give the polyline a
discrete radius of curvature
\begin{equation}\label{eq:ck-curvature}
R_N=\tfrac12\bigl|\csc(\psi_N/2)\bigr|,
\qquad \psi_N=\pi\tau(2N+1),
\end{equation}
$\psi_N$ being the turn between consecutive terms, so that inflections of the
pattern sit where $\psi_N$ is nearest a multiple of $2\pi$ and cusps where it
turns around. They then cut the sum \emph{at an inflection}, and replace each
whole spiral of the first sum by a single vector carrying a scale and a phase,
the Fresnel scale $1/\sqrt{|\tau|}$ with signed phase
$e^{i\,\operatorname{sgn}(\tau)\,\pi/4}$ to the spiral's mid-vector. What
comes out is Hardy and Littlewood's approximate functional formula for the theta
function, sharpened to an error uniform in $\tau$ and stated as a Diophantine
inequality that contains Gauss's sum formula. Cutting at the bisector and standing
one link in for a whole spiral is the reading of \S\ref{sec:decomp}; their
$e^{i\pi/4}$ is the Gaussian integral's, the same factor that leaves the
$-\pi/8$ in $\vartheta$; and \eqref{eq:ck-curvature} is the bisector construction of
\S\ref{sec:bisector} for unit links, the distance from a step to the meeting of
the angle bisectors at its two ends, which is Nickel's $L$ once the length
$n^{-\sigma}$ is restored.

The arithmetic is what does not carry over. Their turn $\psi_N$ is linear in
$N$, so the second difference is the constant $2\pi\tau$, the arcs are Cornu
spirals outright, and the theory turns on the continued fraction of $\tau$:
the pattern is a self-similar hierarchy of curlicues within curlicues. Our turn
is $\theta_n=-t\log\bigl(1+\tfrac1n\bigr)$ and its second difference falls like
$t/n^{2}$, passing $2\pi$ once. That is why the chain has a single distinguished
inflection, at $n\approx\sqrt{t/2\pi}$, instead of a hierarchy of them, and why its
arcs are Cornu spirals only locally. The geometric machinery transfers; the
renormalization does not.

One further reading of the same diagram deserves mention.
Kapitonets~\cite{Kapitonets2019} draws his axis of symmetry through the
\emph{midpoint} of the remainder vector at $\sigma=\tfrac12$, observes that the
remainder is then perpendicular to that axis, and concludes both that
away from zeros, where $\arg\zeta$ is defined, $\arg\zeta$ is the normal
direction, and that the two chains project onto the axis
with canceling sums. That is the half-split of
\S\ref{sec:remainders-summary} together with Corollary~\ref{cor:bisector-proj},
read off the picture. His other construction,
recovering $\zeta$ by repeatedly replacing a polygon of partial sums by the
polygon of their midpoints, is unrelated to the bisector point: iterated
midpointing is the binomial transform, and the limit is Ces\`aro summation of the
Dirichlet series, as he says himself.

\clearpage

\section{Glossary}\label{sec:glossary}

The paper uses a geometric language for objects that other writers have
named differently, or not named at all. This section collects those words
in one place. Each entry is a definition in ordinary language, a short
formula when the formula is short, a pointer to the section that
introduces the object, and, where we know one, the name another author
has used for the same thing or a neighboring one. The names are those
already used above; this list is only a map.

The list is not alphabetical. It runs from the parameters and the
polyline itself to the more specialized constructions built on them,
and the entries in each group belong with one another.

\subsection*{The parameters}

\begin{description}
\item[Index $T$, and $I(T)$.]
The real parameter that replaces Siegel's pair $(t,m)$:
$t=I(T)=\pi(2T+1)/\log(1+1/T)$ and $m=\lfloor T\rfloor$
(\S\ref{sec:IT}).
$T$ is called the index even though it is not an integer.
The map is chosen so that integer $T$ lands on a handoff
(\S\ref{sec:IT-origin}).

\item[Unit of the index $T$.]
One stretch of $T$ between consecutive integers: the interval
$(m,m+1)$ with $m=\lfloor T\rfloor$, for example $T$ from $6$ to $7$.
The chains keep the same number of partial-sum links on that stretch.
The ends of the unit are $\{T\}$ near $0$ and near $1$; the middle is
$\{T\}$ around $\tfrac12$
(\S\ref{sec:links-crossing}).

\item[$\chi(s)$.]
The functional-equation factor,
$\zeta(s)=\chi(s)\zeta(1-s)$
(\S\ref{sec:siegel}).
Nickel writes $Q(s)$ for the same factor.

\item[$\omega$.]
The clockwise turn accumulated by the forward summands:
$\omega=I(T)\log\lceil T\rceil$, so that $(m+1)^{-s}$ points along
$e^{-i\omega}$, at absolute direction $-\omega$
(\S\ref{sec:decomp}).

\item[$a$, $a^{2}$.]
Siegel's cutoff $a=\sqrt{t/2\pi}$ and its square
$a^{2}=t/2\pi=I(T)/2\pi$.
The continuous partner prediction is $a^{2}/n$; integer labels
satisfy $n\,n'\approx a^{2}$.
Under $t=I(T)$ one has $a^{2}\approx(T+\tfrac12)^{2}$.
\end{description}

\subsection*{The chain: joints and links}

\begin{description}
\item[Polyline.]
The drawing: a path of straight segments laid end to end, which is
how a partial sum appears in the plane once each summand is drawn as
a vector from the end of the last one
(\S\ref{sec:matrix-product}).
The word is about the picture and not about any new object. The
forward and reverse spirals are the two polylines of this paper, and
statements like ``the two polylines meet'' are statements about the
drawn segments rather than about the sums.

\item[Sum chain, or link chain.]
The partial-sum polyline in the plane: joints connected by links,
one Dirichlet summand per link
(\S\ref{sec:matrix-product}).
``Sum chain'' reads it as a running total; ``link chain'' reads it
as a serial chain of segments. The two names are the same object.
We also call it a spiral, though only the later links wind.
Nickel~\cite{Nickel2013} draws it as the Argand diagram of the
steps $n^{-s}$.
Erickson~\cite{Erickson2005} identifies the later coils as Cornu
spirals and labels their centers $C_n$.
Kapitonets~\cite{Kapitonets2019} calls the polyline the Riemann
spiral.
Berry and Goldberg~\cite{BerryGoldberg1988} call the quadratic-phase
analogue a curlicue; our turn $\theta_n=-t\log(1+1/n)$ is not
quadratic, so the coils are Cornu only locally
(\S\ref{sec:curlicues}).
The $3\times3$ matrices of \S\ref{sec:product-form} are the planar
Denavit--Hartenberg encoding of a serial chain
\cite{DenavitHartenberg1955}.

\item[Joint.]
A vertex of a chain: the value of the partial sum after an integer
number of summands.
Joint $n$ of the forward chain is $J_n=\sum_{r=1}^{n}r^{-s}$ for
$n\ge1$, with $J_0=0$, so joint $0$ is the origin and joint $m$ is $\Sigma_1$ itself
(\S\ref{sec:matrix-product}, \S\ref{sec:links-crossing}).
Erickson's centers $C_n$ stand in one-to-one correspondence with the
joints, though a center is not itself a joint
(\S\ref{sec:bisector}; see \emph{Spiral center}).

\item[Link.]
The segment between two consecutive joints: one Dirichlet summand,
drawn as a vector.
Link $k$ of the forward chain runs from joint $k$ to joint $k+1$ and
carries the summand $n=k+1$
(\S\ref{sec:matrix-product}).
Nickel calls the same object a step.
Coutsias and Kazarinoff~\cite{CoutsiasKazarinoff1987} treat the
quadratic-phase analogue as a discrete curve with a radius of curvature
at each turn.

\item[Spiral center.]
The point a coil of the chain winds around, and the clearing it
leaves at the middle of that coil
(\S\ref{sec:geometry}, \S\ref{sec:links-crossing},
Figure~\ref{fig:last-spiral-zoom}).
The link nearest it is $L_N(T,S_n)$ of \eqref{eq:LN}, where the turn
per link $t\log(1+1/n)$ is the odd multiple $\pi(2S_n+1)$: successive
links are antiparallel there, so the chain sweeps back and forth
across the center and the joints settle onto a ring around it of
radius about half a link length, which is the lens-shaped clearing
the near-diameter links envelope. The center need not lie on any link
or joint, not even on the nearest link.
The centers are not scattered. Numbering them by the spiral number
$S_n=c$ outward from the end of the chain, center $c=0$ is the last
spiral, the largest, and sits at $\zeta$ to within that radius; each
further center is one sweep back: by
\eqref{eq:crossing-saddle} the step from center $c$ to center $c-1$ is
the $c$-th summand reflected in the fixed direction $\Theta$, of
length $c^{-1/2}$. Read outward from $\zeta$ the centers therefore
reproduce the head of the chain reflected, which is Nickel's
center-to-center distances and angles reproducing the initial steps.
In summands, center $c$ is near $n=a^{2}/(c+\tfrac12)$, interleaving
with the sweeps at $n=a^{2}/c$ of \eqref{eq:crossing-sweep}.
Erickson writes $C_n$ for these points.

\item[$\Sigma_1$, the forward chain (or forward spiral).]
The sum chain of the first Dirichlet sum,
$\Sigma_1=\sum_{n=1}^{m}n^{-s}$
(\S\ref{sec:siegel}, \S\ref{sec:matrix-product}).
The chain is a polyline of summands, not a placement: it may be
drawn in any frame.

\item[$\Sigma_2$, the reverse chain (or reverse spiral).]
The sum chain of the reflected Dirichlet sum,
$\Sigma_2=\chi(s)\sum_{n=1}^{m}n^{s-1}$
(\S\ref{sec:siegel}).
In the plane the reverse chain is usually drawn from $\zeta$,
running counter to the forward chain; as a free vector it can sit
anywhere.
On the critical line it is the functional-equation image of the
forward chain: each summand of $\Sigma_2$ is $\chi$ times the
conjugate of the matching summand of $\Sigma_1$.
Nickel's second half $Q(s)P(1-s)$ is the same reflection, with his
$Q$ equal to our $\chi$ (\S\ref{sec:nickel}).
\end{description}

\subsection*{The remainder on the chain}

\begin{description}
\item[$R$.]
Siegel's unsplit remainder integral, so that
$\zeta=\Sigma_1+\Sigma_2+R$ (\S\ref{sec:siegel}).
Every named split below recovers this same $R$, exactly or
approximately.

\item[Split.]
The choice of meeting point: a writing
$R=R_1+R_2$ (or $\zeta=B_1+B_2$) that picks where the two
halves join
(\S\ref{sec:remainders-summary}).
The two pieces leave the same pair of joints: $R_1$ starts at
$\Sigma_1$, and $R_2$ arrives at $\Sigma_1+R$. What a split
chooses is only the point $\Sigma_1+R_1$ between those joints,
its own bisector point. The three named splits are the
partial-summand split $ps$, the half-split $rs$, and the
Kuznetsov / Siegel-$f$ split $ak$. The $rs$ split just halves
$R$ along itself. The $ak$ split does not follow the summand
directions.

\item[Cut.]
To sever a single object, a link, a chain, or a sum, at one point
along it: each forward link is cut at its crossing
(\S\ref{sec:half-link-sums}), the chain is cut at an arbitrary
link in \S\ref{sec:cut-any-link}, and Coutsias and Kazarinoff cut
their sum at an inflection (\S\ref{sec:curlicues}).
The three verbs divide the work: a \emph{split} writes $R$ (or
$\zeta$) as two pieces, a \emph{cut} severs one thing at one
place, and two curves that meet \emph{cross}.

\item[Wedge.]
What the $ps$ split looks like in the plane: the two halves run
along two different link directions,
$R=d_1\,e^{-i\omega}+d_2\,e^{i(\omega+\psi)}$, with real weights
(\S\ref{sec:decomp}, \eqref{eq:R-two-dirs}).
Those two rays form an angle, and $R$ sits in that angle when
both weights are positive. The same picture at a general forward
link, with $W_k$ in place of $R$, is \eqref{eq:W-wedge}.
A positive $ps$ split is drawn as a wedge; a general algebraic split
need not determine one.
In convex analysis the set $\{a\,u+b\,v : a,b\ge0\}$ spanned by
two rays is called a cone in any dimension, including~2, but so
as not to evoke a 3-d mental picture we chose ``wedge'' instead.

\item[$R_{1}$, $R_{2}$ (the $ps$ split; also $R_{1ps}$, $R_{2ps}$).]
The partial-summand split of $R$:
$R=R_{1}+R_{2}$, with
$R_{1}=d_1\,e^{-i\omega}$ and
$R_{2}=d_2\,e^{i(\omega+\arg\chi)}$
(\S\ref{sec:decomp}).
Each piece is the next Dirichlet term of its sum, shortened by a
real weight.
This split is the paper's own, so it is written plainly as
$R_1$, $R_2$; the subscript $ps$ appears only where the three
splits are compared side by side
(\S\ref{sec:remainders-summary}), to distinguish this exact split
from the approximate remainders of other authors.

\item[Fractional summand.]
A remainder that is exactly the next Dirichlet term, scaled by a
real factor: $R_{1}=\hat d_1\,(m+1)^{-s}$ and
$R_{2}=\hat d_2\,\chi\,(m+1)^{s-1}$
(\S\ref{sec:summand}).
The word ``fractional'' means a fraction of one summand, not a
summand with a non-integer index.

\item[$d_1$, $d_2$.]
The signed coefficients of the two fractional summands:
$R_{1}=d_1\,e^{-i\omega}$ and $R_{2}=d_2\,e^{i(\omega+\arg\chi)}$,
so $|R_{1}|=|d_1|$ and $|R_{2}|=|d_2|$
(\S\ref{sec:decomp}, \S\ref{sec:length-ratios}).
On the critical line both are positive and $d_1=d_2$.
Off the line they can change sign, and the signed ratio is
\eqref{eq:d-ratio}.
The sine formulas are \eqref{eq:d1}--\eqref{eq:d2}.

\item[$\hat d_1$, $\hat d_2$ (fractional amounts).]
The same weights read as fractions of a link rather than as
lengths:
$\hat d_1=\lceil T\rceil^{\sigma}\,d_1$ and
$\hat d_2=\lceil T\rceil^{1-\sigma}\,d_2/|\chi|$
(\S\ref{sec:summand}).
$\hat d_1$ is the crossing fraction of the bisector link.
At the running example $\sigma=\tfrac12$, $T=6.18$, one has
$d_1\approx0.0875$ and $\hat d_1\approx0.231$.

\item[$R_{ak}$, $R_{1ak}$, $R_{2ak}$.]
The Kuznetsov / Siegel-$f$ split:
$R\approx R_{ak}=R_{1ak}+R_{2ak}$, where
$R_{1ak}=f(s)-\Sigma_1$
(\S\ref{sec:kuznetsov}, \S\ref{sec:remainders-summary}).
Approximate, and not a single fractional summand.

\item[$R_{rs}$, $R_{1rs}$, $R_{2rs}$.]
The half-split $R_{1rs}=R_{2rs}=\tfrac12 R$
(\S\ref{sec:remainders-summary}).
Kapitonets' midpoint of the remainder vector is this split
(\S\ref{sec:curlicues}).
\end{description}

\subsection*{The bisector}

\begin{description}
\item[Bisector link.]
Link $m=\lfloor T\rfloor$ of a chain: the link that sits at the
boundary between the nearly-straight early chain and the coils
(\S\ref{sec:bisector}).
Each chain has one.
The forward bisector link runs from $\Sigma_1$ along the
$(m+1)$st summand; the reverse bisector link runs from $\Sigma_1+R$
along its $(m+1)$st summand.
On $\sigma=\tfrac12$ the bisector line crosses this link at the
bisector point $B_1$: that is where the forward chain meets the
bisector line, at crossing fraction $\hat d_1$ along the link, not
(except at a handoff) at the bisector joint itself.
This is the neighborhood of Siegel's cutoff
$m=\lfloor\sqrt{t/2\pi}\rfloor$ (\S\ref{sec:siegel}), of Nickel's
center-of-symmetry step $n_p=\bigl\lfloor\sqrt{t/2\pi}\,\bigr\rfloor$ (offset from our $m$
by about half a unit, \S\ref{sec:nickel}), and of the inflection at
which Coutsias and Kazarinoff cut their sum
(\S\ref{sec:curlicues}).

\item[Bisector joint.]
Joint $m$ of the forward chain, the integer vertex at which the
bisector link begins: the point $\Sigma_1$.
The matching joint of the reverse chain is $\Sigma_1+R$.
All three remainder splits of \S\ref{sec:remainders-summary} leave
from this pair of joints.

\item[Bisector point.]
The crossing of the two bisector links,
$B_1=\Sigma_1+R_{1}=\Sigma_1+R-R_{2}$
(\S\ref{sec:bisector}), wherever that crossing is finite.
Off the line it fails at a pole of $d_1$; on the line the
apparent singularity is removable and $B_1$ is defined by
continuity.
On the critical line it is the apex of the isosceles triangle with
base from the origin to $\zeta$.
Nickel's center $P(s)$ plays the same role and is not the same
point: the two differ mostly along the symmetry axis, which $\zeta$
cannot see (\S\ref{sec:nickel}).
Kapitonets draws his axis through the midpoint of the remainder
vector $R$; that midpoint is the half-split $R_{rs}$, not $B_1$
(\S\ref{sec:curlicues}).

\item[Bisector line.]
The line through the bisector point $B_1$ and $\zeta/2$
(\S\ref{sec:legs}, Figure~\ref{fig:full-spirals}),
or a specified limiting direction when those two points coincide
(when $B_1=\zeta/2$).
On $\sigma=\tfrac12$ it is the perpendicular bisector of the
segment from the origin to $\zeta$, and the axis of symmetry of
the isosceles triangle $O$--$B_1$--$\zeta$ whose equal sides are
the two legs. It crosses the bisector link at $B_1$. Off the
critical line the legs are unequal, so there is no isosceles
triangle and the line is no longer a symmetry axis, but we keep
the same name for the line through $B_1$ and $\zeta/2$.

\item[$B_1$, $B_2$; $B_1^{\ast}$, $B_2^{\ast}$.]
The two half-sums that add to $\zeta$:
$B_1=\Sigma_1+R_{1}$ and $B_2=\Sigma_2+R_{2}$, so
$\zeta=B_1+B_2$ (\S\ref{sec:matrix-product}).
$B_1$ is the bisector point.
On the critical line, $B_2=\chi\,\overline{B_1}$.
The starred pair is the velocity-split point: where
$\vartheta'\neq0$,
$B_1^{\ast}=\zeta+\zeta'/(2\vartheta')$ and
$B_2^{\ast}=-\zeta'/(2\vartheta')$ of \eqref{eq:vel-split}, with
$B_2^{\ast}=\chi\,\overline{B_1^{\ast}}$ on the line
(\S\ref{sec:counting-star}).
It is a two-leg meeting point, not another bisector point.
Nickel's Riemann--Siegel equation $\zeta=P(s)+Q(s)P(1-s)$ is this
two-leg split with a different choice of the meeting point
(\S\ref{sec:nickel}).
\end{description}

\subsection*{Frames}

\begin{description}
\item[Link frame.]
The coordinate frame of an arbitrary forward link $k$: translate
joint $k$ to the origin and divide by the link vector $(k+1)^{-s}$,
so that link $k$ becomes the unit interval $[0,1]$
(\S\ref{sec:yinyang-any}, Figure~\ref{fig:link-frames}).

\item[Bisector frame.]
The link frame at $k=m$: it sends the forward bisector link to
$[0,1]$ on the real axis
(\S\ref{sec:IT-origin}, \S\ref{sec:yinyang-origin}).
In that frame the reverse bisector link revolves as $T$ advances,
and the crossing with the axis is the number $d_1$
(or the crossing fraction $\hat d_1$, once the link is read at unit
length).

\item[$\vartheta$-frame.]
The rotation that makes $\zeta$ real on the critical line:
multiply every vector by $e^{i\vartheta(t)}$, where $\vartheta$ is
the Riemann--Siegel theta function, so that
$Z=e^{i\vartheta}\zeta$ lies on the real axis
(\S\ref{sec:counting-review}, Corollary~\ref{cor:bisector-proj}).
The base $O\to\zeta$ is then the real axis, and a split point
reads as $\tfrac12 Z+ih$.
This is the same rotation as $e^{-i\psi/2}$ with
$\vartheta=-\psi/2$, since $\chi=e^{-2i\vartheta}$ on the line.
$Z$ is Hardy's $Z$-function.
\end{description}

\subsection*{Legs}

\begin{description}
\item[Legs.]
The two segments that add to $\zeta$: Leg~1 is the vector $B_1$
from the origin to the bisector point, and Leg~2 is the vector
$B_2=\zeta-B_1$ from the bisector point to $\zeta$
(\S\ref{sec:legs}).
Their lengths are $L_1=|B_1|$ and $L_2=|B_2|$.
Nickel's two vectors $P(s)$ and $Q(s)P(1-s)$ are legs in the same
sense, meeting at his center rather than at $B_1$.

\item[$\vartheta_1$, $\vartheta_2$.]
The two leg angles:
$\vartheta_1=\arg B_1$ is the heading of Leg~1 from the positive
real axis, and
$\vartheta_2=\arg(B_2/B_1)$ is the turn from Leg~1 to Leg~2
(\S\ref{sec:legs}).
Then $\zeta=L_1 e^{i\vartheta_1}+L_2 e^{i(\vartheta_1+\vartheta_2)}$.
These are not the Riemann--Siegel theta function $\vartheta(t)$ of
\eqref{eq:N-theta}. Since $\chi=e^{-2i\vartheta}$, a chosen square
root has phase $-\vartheta\pmod{\pi}$, and $\vartheta/\pi+1$ is only
the smooth part of the zero count; the clash of names is unfortunate
and the two are kept apart by the subscripts.

\item[Half-line.]
The critical line $\sigma=\tfrac12$, read at $t>0$, which is where all the
zeros this paper counts sit. The two names are used interchangeably, as is
``the line''; \emph{off-half-line} and \emph{off-line} both mean
$\sigma\neq\tfrac12$.

\item[Ordinate.]
The imaginary part $\gamma$ of a nontrivial zero $\rho=\beta+i\gamma$.
A critical-line ordinate is one with $\beta=\tfrac12$, equivalently
a real zero of $Z$, and $\gamma_1<\gamma_2<\cdots$ numbers those up
the line (\S\ref{sec:counting}).
On the index, that ordinate sits at $T=I^{-1}(\gamma_n)$.
The word is never used here in its generic sense, the
$y$-coordinate of an arbitrary point.

\item[Retrograde, prograde.]
The two directions of the fold angle $\vartheta_2$
(\S\ref{sec:counting}, Figure~\ref{fig:nps-staircase}).
Prograde is the usual way: $\vartheta_2$ keeps going through
$\pi$ at an ordinate and does not come back.
Retrograde is the other way: $\vartheta_2$ rises through $\pi$,
turns around, and comes back down through it (or the reverse).
The boundary between the two is where the derivative vanishes,
$\vartheta_2'=0$.
A retrograde passage is where $N_{\mathrm{ps}}$ miscounts;
$\vartheta_2^{\ast}$ of the velocity split never retrogrades, which is
why $N_{\ast}$ counts correctly
(\S\ref{sec:counting-ovals}, Figure~\ref{fig:theta2-star-compare}).

\item[Equal legs.]
The condition $L_1=L_2$.
On the critical line it holds identically
(Corollary~\ref{cor:equal-legs}).
Off the line it is a locus of ovals and thin bands in the strip
(Figure~\ref{fig:equal-legs-strips}).
A zero of $\zeta$ requires equal legs \emph{and} $\vartheta_2=\pi$.
Nickel notes that his two magnitudes are equal on
$\sigma=\tfrac12$ (\S\ref{sec:nickel}).

\item[$\mathcal{E}_{\mathrm{ps}}$, $\mathcal{E}_{\mathrm{rs}}$.]
The equal-leg loci of two different splits of the remainder.
Each is the set of $(\sigma,T)$ at which the two legs have the
same length, $|B_1|=|\zeta-B_1|$, for $B_1=B_{1ps}$ and
$B_1=B_{1rs}$ respectively
(\S\ref{sec:ps-ak-r2}).
A zero forces equal legs for every split, so every zero lies in
the intersection
$\mathcal{E}_{\mathrm{ps}}\cap\mathcal{E}_{\mathrm{rs}}$.
The whole critical line is in both.
Off the line, a meeting of the two oval families would be
necessary for an off-line zero (not sufficient).
If they never meet off $\sigma=\tfrac12$, there is no off-line
zero.

\item[Fold, folded.]
Two related uses, both meaning that two segments occupy the same
line and point opposite ways.
At an integer $T$ the neighboring link \emph{folds back} onto the
bisector link, the joint angle is an odd multiple of $\pi$, and
the bisector point hands off to the next link
(\S\ref{sec:IT-origin}, \S\ref{sec:yinyang-origin}): that is a
\emph{handoff}, and for zeta every handoff is of this kind.
Separately, a \emph{folded-leg} point in the strip is a height
where $\vartheta_2=\pi$, so that Leg~2 folds back onto Leg~1
(Figure~\ref{fig:equal-legs-strips}); a zero is an equal-leg point
that is also folded-leg.
\end{description}

\subsection*{How the picture moves with $T$}

\begin{description}
\item[Handoff.]
The instant, at integer $T$, when the bisector point moves from
link $m-1$ onto link $m$
(\S\ref{sec:periodicity}, \S\ref{sec:IT-origin}).
For zeta the two links fold back onto one another and the point
slides across the joint with no jump.

\item[Event.]
Any instant at which the bisector point moves from one link to the
next (\S\ref{sec:IT-Lfunctions}).
For zeta every event is a handoff by fold-back.
Dirichlet $L$-functions have a second kind.

\item[Folded event.]
An event at which the turning angle at the current joint is
$\pm\pi$, so the chain reverses and the bisector point slides
across the joint (\S\ref{sec:IT-Lfunctions}).
For zeta, every handoff is a folded event, once per unit of $T$.
For $L(s,\chi_p)$ these occur at $T\equiv0\pmod{p-1}$.

\item[Bisector event.]
An event at which the bisector \emph{line} (the perpendicular
bisector of the segment from $0$ to the $L$-value) sweeps through
a joint and carries the crossing onto the neighboring link
(\S\ref{sec:IT-Lfunctions}).
Zeta has none.
For $L(s,\chi_p)$ they occupy the remaining residues modulo
$p-1$.

\item[Phantom link.]
A zero summand of $L(s,\chi_p)$, occurring where $\chi_p(n)=0$
(\S\ref{sec:IT-Lfunctions}).
It contributes no segment and receives no index $T$; only the
nonzero steps are links. The name is kept because that is what
\S\ref{sec:IT-Lfunctions} calls them.
\end{description}

\subsection*{Crossings along the chain}

\begin{description}
\item[Cross.]
What two curves do where they meet: the forward chain crosses the
bisector line at $B_1$, each forward link is crossed by a reverse
link, and the chord from yin to yang crosses the real axis at the
crossing fraction. The meeting point itself is the crossing,
next.

\item[Crossing (the point).]
Where a forward link and a reverse link meet in the plane.
The distinguished one is the bisector point $B_1$; in the sampled
range every other forward link $k$ has one as well, written $P_k$
(\S\ref{sec:links-crossing}, \S\ref{sec:sum-x}). That existence is
numerical.

\item[Crossing partner.]
The reverse link that crosses forward link $k$.
The saddle predicts a continuous partner
$n'_{\mathrm{cont}}=a^{2}/n$ with $a=\sqrt{t/2\pi}$ and $n=k+1$;
the integer reverse-link label is the nearest integer to that
target, so $n\,n'\approx a^{2}$
(\S\ref{sec:links-crossing}).
This is the same saddle that Siegel's contour is built around.

\item[Crossing fraction.]
How far along a link the crossing sits, from the left joint toward
the right, as a number in $[0,1]$
(\S\ref{sec:d1-any}, \S\ref{sec:sum-x}).
On forward link $k$ it is $\hat d_1(k)$; on reverse
link $k$ it is $\hat d_2(k)$, the $T$ of $\hat d_1(k,T)$ and
$\hat d_2(k,T)$ being suppressed where it is fixed.
At the bisector, $\hat d_1(m)=\hat d_1$.
Away from the fold the other $\hat d_1(k)$ sit near $\tfrac12$.
This is a fraction of the straight link, not of a curved arc.

\item[$\Sigma_{1x}$, $\Sigma_{2x}$.]
The two Dirichlet sums re-jointed at the crossings instead of at
the integers
(\S\ref{sec:sum-x}).
$\Sigma_{1x}$ walks $P_{-1}=0,P_0,\dots,P_m=B_1$ and totals
$B_1=\Sigma_1+R_{1}$;
$\Sigma_{2x}$ does the same from the origin along $\Sigma_2$ and
totals $B_2$.
The subscript $x$ is for crossing, not a dummy index.
The step lengths are the crossing fractions $\hat d_1(k)$ and
$\hat d_2(k)$.
\end{description}

\subsection*{Yin and yang}

\begin{description}
\item[Yin, yang.]
The two ends of the reverse bisector link, watched in the
bisector frame as $T$ advances
(\S\ref{sec:yinyang-origin}).
Each traces a teardrop; the pair resembles a yin-yang.
Yin is the joint-$m$ end, yang the joint-$(m+1)$ end.
The chord from yin to yang crosses the real axis at the bisector
point, and that crossing is how $d_1$ and $d_2$ were found
(\S\ref{sec:yinyang-origin}).
The same pair exists for every forward link $k$, once the crossing
partner is known (\S\ref{sec:yinyang-any}).
The names are ours.

\item[$\mathrm{Yin}_{\infty}$, $\mathrm{Yang}_{\infty}$.]
The limits of the yin and yang curves as $T\to\infty$ at fixed
fractional part $x=\{T\}$:
$\mathrm{Yin}_{\infty}(x)=1-\Psi(x)\,e^{-2\pi i(x^{2}-1/16)}$ of
\eqref{eq:yin-inf}, with
$\mathrm{Yang}_{\infty}(x)=e^{4\pi ix}\bigl(1-\mathrm{Yin}_{\infty}(x)\bigr)$
(\S\ref{sec:yinyang-inf}).
The two are one curve: yang follows yin half a unit behind
(Lemma~\ref{lem:yinf-polar}).
For $0<T<1$ the yin path is $\zeta$ itself (Erickson's $C_0$).

\item[$\Psi$.]
The classical Riemann--Siegel correction
$\Psi(x)=\cos\bigl(2\pi(x^{2}-x-1/16)\bigr)/\cos(2\pi x)$
of \eqref{eq:psi-fn},
used away from $x=\tfrac14,\tfrac34$, with the continuous values
$\Psi(\tfrac14)=\Psi(\tfrac34)=\tfrac12$,
as in Titchmarsh~\cite[\S15.3]{Titchmarsh1986} and
Pugh~\cite[\S5.4.1]{Pugh1998}.
It enters this paper as the amplitude of $\mathrm{Yin}_{\infty}$,
not as a remainder of its own.

\end{description}

\subsection*{The far chain (links near $\zeta$)}

\begin{description}
\item[Conveyor belt.]
The stretch of the reverse chain between the last two spirals
(spirals $1$ and $0$ of \eqref{eq:LN}), the links near $\zeta$.
The name is ours.

\item[Last link.]
The link nearest the center of the last spiral:
$L_N(T,0)=\lfloor I(T)/\pi\rfloor=\lfloor t/\pi\rfloor\approx 2T^{2}$
(\S\ref{sec:geometry}, \eqref{eq:LN}).
It is the link closest to $\zeta$, far past the bisector; $\zeta$
itself is not typically on it.
The same formula with spiral number $S_n$ names the link nearest the
center of any spiral,
$L_N(T,S_n)=\bigl\lfloor I(T)/\bigl(\pi(2S_n+1)\bigr)\bigr\rfloor$.
Consecutive links there are antiparallel.
\end{description}

\clearpage

\section*{Acknowledgments}
\begin{itemize}
  \item Chris Cowherd for all his extensive help over the last nine years in building many visualization and experimental mathematics tools, including the Zest zeta sum visualization tool
  \item Nicolas Fish for his four years of excellent programming skills in building and tweaking the Zest visualization and measurement tool
  \item Daniel Flores Galiote for his kind tutoring and helpful discussions over the last two years
  \item Brian Conrey for his encouragement and acceptance of this amateur mathematician
  \item Jakob von Raumer for his help with a version of the Lean proof
  \item My family for their support:
    \begin{itemize}
      \item my wife, Flo
      \item Sophie
      \item Anthony
      \item Thomas
    \end{itemize}
  \item Freeman Dyson for his article {\it Birds and Frogs} \\(\texttt {\small https://pdodds.w3.uvm.edu/files/papers/others/2009/dyson2009a.pdf}), which inspired me to begin my journey to investigate one-dimensional quasi-crystals and their possible connection to the Riemann hypothesis nine years ago.
  \item Turingbot, Processing, Unity, OEIS, wolframalpha, ChatGPT, Claude (and Opus and Fable), Cursor, and last but not least: Desmos and Lean (and Leo De Moura, my friend and neighbor, and the creator of LEAN).  Thank you
\end{itemize}

\paragraph{Statement on the use of AI.}
Cursor and various AIs, including Anthropic's Claude (Opus and Fable),
OpenAI's ChatGPT (5.6 Sol), and Grok 4.5, were used in preparing this
paper. AI systems assisted with Python and JavaScript figure-generation
programs, portions of the Zest web app, Lean formalizations, literature
review, glossary, and \LaTeX{} preparation. Figure captions identify their
generating programs. The author checked and edited the resulting text, code,
figures, and formalizations.

Just before this version was finalized, the author set up an adversarial
debate on the paper's content between two competing AIs, with GPT~5.6 Sol Max
acting as the reviewer and Claude Opus~5 Extra High as the defender. The
reviewer raised $171$ issues, each one debated by the defender, with the
author as the final adjudicator of every proposed change. In the author's
opinion about ten of the issues were substantive; some fifty more were acted
upon without being substantive, and the rest were trivial or wrong.

\clearpage

\appendix

\clearpage

\section{Proof of Proposition~\texorpdfstring{\ref{prop:weights}}{4.1}: reality and positivity of \texorpdfstring{$d_1$ and $d_2$}{d1 and d2}}\label{sec:positivity}

This section proves Proposition~\ref{prop:weights}, one subsection per
ingredient. \S\ref{sec:proof-i-ii} proves parts (i) and (ii), which are
short and purely algebraic. \S\ref{sec:phase-bound} sets up the
Riemann--Siegel notation and proves the key phase bound: the factor
$\cos u$ stays away from zero at every fractional part.
\S\ref{sec:proof-iii} proves part (iii), positivity on the critical
line, where no fractional parts need to be excluded at all.
\S\ref{sec:proof-iv} proves part (iv), positivity across the strip
outside two windows of width $O(1/T)$ per unit interval, and
\S\ref{sec:pole-locations} locates those windows, near the
parallel-summand heights $\{T\}\approx\tfrac14,\tfrac34$.
\S\ref{sec:limits} proves the two limit statements used in
\S\ref{sec:periodicity}: the normalized weight converges to the profile
\eqref{eq:d-profile}, and the pinned waves converge to the tangent
waveform \eqref{eq:W-infty}. Finally, \S\ref{sec:d1-bounds} proves the
uniform finite-$T$ bounds $\tfrac15<\hat d_1=\hat d_2<\tfrac45$ on the
critical line. Throughout, the
criterion is one of signs: both weights are positive exactly when the two
numerators of \eqref{eq:d1}--\eqref{eq:d2} carry the same sign as the
common denominator $\sin(2\omega+\psi)$.

\subsection{Proof of parts (i) and (ii): remainders are along the summands, and \texorpdfstring{$d_1$ and $d_2$}{d1 and d2} are real}\label{sec:proof-i-ii}

\begin{proof}
\begin{enumerate}
\item[(i)] For non-integer $T$ we have $M=m+1$, so the next
summand of $\Sigma_1$ factors as
$(m+1)^{-s}=(m+1)^{-\sigma}e^{-it\log(m+1)}=(m+1)^{-\sigma}e^{-i\omega}$,
a positive real multiple of the unit direction $e^{-i\omega}$; by
\eqref{eq:R1-def}, $R_1=d_1e^{-i\omega}$ lies along it. Likewise the next
summand of $\Sigma_2$ factors as $\chi(s)\,(m+1)^{s-1}
=|\chi|\,(m+1)^{\sigma-1}\,e^{\,i(\omega+\psi)}$, a positive real multiple
of $e^{\,i(\omega+\psi)}$, and by \eqref{eq:R2-def},
$R_2=d_2e^{\,i(\omega+\psi)}$ lies along it. Explicitly,
\[
R_1=(m+1)^{\sigma}d_1\cdot(m+1)^{-s},
\qquad
R_2=\frac{(m+1)^{1-\sigma}}{|\chi(s)|}\,d_2\cdot\chi(s)\,(m+1)^{s-1},
\]
with both scale factors real by part (ii).
\item[(ii)] View \eqref{eq:R-two-dirs} as a real $2\times2$ linear
system for the unknowns $d_1,d_2$, one equation each from the real and
imaginary parts. Cramer's rule gives exactly the pair of sine expressions
\eqref{eq:d1}--\eqref{eq:d2}, which are real by inspection: every
ingredient ($|R|$, and sines of real angles) is real. \qedhere
\end{enumerate}
\end{proof}

\subsection{Notation and the key phase bound}\label{sec:phase-bound}

Let $\vartheta(t)=\arg\Gamma\bigl(\tfrac14+\tfrac{it}2\bigr)
-\tfrac t2\log\pi$ be the Riemann--Siegel theta function, with asymptotic
expansion $\vartheta(t)=\tfrac t2\log\tfrac t{2\pi}-\tfrac t2-\tfrac\pi8
+O(1/t)$, and let $Z(t)=e^{i\vartheta(t)}\zeta(\tfrac12+it)$ be the (real)
Riemann--Siegel $Z$ function. On the critical line
$\chi(\tfrac12+it)=e^{-2i\vartheta(t)}$, so $\psi=-2\vartheta$ there. Put
\begin{equation}
a=\sqrt{\tfrac t{2\pi}},\qquad N=\lfloor a\rfloor,\qquad \hat p=a-N,
\qquad u=\omega-\vartheta(t),
\end{equation}
so that $N$ is the Riemann--Siegel cutoff while $m=\lfloor T\rfloor$ is
ours, and $u$ is the phase of the next summand measured against
$-\vartheta$. Expanding $\log(1+\tfrac1T)$ in $t=I(T)$ gives
\begin{equation}\label{eq:a-vs-T}
\frac{t}{2\pi}=\Bigl(T+\tfrac12\Bigr)^{2}-\tfrac1{12}+O\!\Bigl(\tfrac1{T^{2}}\Bigr),
\qquad\text{hence}\qquad
a=T+\tfrac12+O\!\Bigl(\tfrac1{T}\Bigr),
\end{equation}
so, apart from boundary intervals of width $O(1/T)$ in the fractional part
$x=\{T\}$ near $x=0,\tfrac12,1$,
\begin{equation}\label{eq:cut-cases}
x<\tfrac12 \iff N=m,\ \ \hat p=x+\tfrac12+O(1/T),
\qquad
x>\tfrac12 \iff N=m+1,\ \ \hat p=x-\tfrac12+O(1/T).
\end{equation}
Insert $t=2\pi a^{2}$ and $m+1=a+(1-\hat p)$ (case $N=m$) or
$m+1=a-\hat p$ (case $N=m+1$) into $\omega=t\log(m+1)$, expand the
logarithm about $\log a$, and reduce modulo $2\pi$ using $a=N+\hat p$ and
$\pi N^{2}\equiv\pi N \pmod{2\pi}$; with the asymptotic series of
$\vartheta$ this gives
\[
u\equiv\pi(N-1)+4\pi\hat p-2\pi\hat p^{2}+\tfrac\pi8+O(1/a)
\quad(N=m),
\qquad
u\equiv\pi N-2\pi\hat p^{2}+\tfrac\pi8+O(1/a)
\quad(N=m+1),
\]
and substituting $\hat p=x+\tfrac12$, respectively $\hat p=x-\tfrac12$,
turns both cases into one and the same statement:
\begin{equation}\label{eq:cosu}
\cos u \;=\;(-1)^{m+1}\,
\cos\!\bigl(2\pi x(1-x)-\tfrac{3\pi}8\bigr)\;+\;O(1/T).
\end{equation}
Since $2\pi x(1-x)\in[0,\tfrac\pi2]$ for $x\in[0,1]$, the argument of the
cosine on the right lies in $(-\tfrac{3\pi}8,\tfrac\pi8]$, so
\begin{equation}\label{eq:cosu-bound}
|\cos u|\;\geq\;\sin\tfrac\pi8-O(1/T)\;=\;0.3826\ldots-O(1/T):
\end{equation}
the factor $\cos u$ never approaches zero, at any fractional part. That one
inequality is the key ingredient in both halves of the proof. We also need two
elementary facts about the leading Riemann--Siegel correction
$C_0(\hat p)$.

\begin{lemma}\label{lem:C0-trig}
Write $A=2\pi\hat p^{2}-\tfrac\pi8$ and $B=2\pi\hat p$, and let
$C_0(\hat p)=\cos(A-B)/\cos B$ be the leading term of the Riemann--Siegel
correction. Then:
\begin{enumerate}
\item[(i)] $C_0(\hat p)>0$ for $\hat p\in(\tfrac12,1)$, and
$\min_{[1/2,\,1]}C_0=\cos\tfrac{3\pi}8=\sin\tfrac\pi8$, attained at
$\hat p=\tfrac12$;
\item[(ii)] for $\hat p\in[0,\tfrac12]$,
\[
\Lambda(\hat p)\;:=\;\frac{C_0(\hat p)}
{2\cos\bigl(\pi(\tfrac18-2\hat p^{2})\bigr)}
\;\leq\;\tfrac12,
\]
with equality exactly at the endpoints, and $\Lambda(\tfrac14)=\tfrac14$.
\end{enumerate}
\end{lemma}

\begin{proof}
(ii) first, since its argument is the more instructive. Note
$\pi(\tfrac18-2\hat p^{2})=-A$, so
$\Lambda=\cos(A-B)/(2\cos A\cos B)=\tfrac12(1+\tan A\tan B)$ by the product
formula, and the claim is that $\tan A\tan B\leq0$ on $[0,\tfrac12]$. Both
critical transitions happen at the same point: $A=0$ exactly when
$\hat p=\tfrac14$, and $B=\tfrac\pi2$ exactly when $\hat p=\tfrac14$. For
$\hat p<\tfrac14$ we have $A\in(-\tfrac\pi8,0)$, so $\tan A<0$, while
$B\in(0,\tfrac\pi2)$, so $\tan B>0$; for $\hat p>\tfrac14$ the signs
reverse ($A\in(0,\tfrac{3\pi}8)$, $B\in(\tfrac\pi2,\pi)$). The product is
therefore $\leq0$ throughout, vanishing only as $\hat p\to0,\tfrac12$
(where $\tan A\tan B\to0$, so $\Lambda\to\tfrac12$); at $\hat p=\tfrac14$
the two singular factors balance, $\tan A\tan B\to-\tfrac12$, giving the
removable value $\Lambda(\tfrac14)=\tfrac14$.

(i) is a sign inspection of numerator and denominator. On
$(\tfrac12,\tfrac34)$ both $\cos(A-B)$ and $\cos B$ are negative; on
$(\tfrac34,1)$ both are positive (they change sign together at
$\hat p=\tfrac34$, which is the removable singularity of $C_0$), so
$C_0>0$ throughout. That is positivity. The minimum is a separate
claim: $C_0$ extends continuously to the compact interval
$[\tfrac12,1]$ (the value at $\tfrac34$ is removable), and on each of
the two subintervals it is monotone, rising from
$C_0(\tfrac12)=\sin\tfrac\pi8$ through $C_0(\tfrac34)=\tfrac12$
to $C_0(1)=\cos\tfrac\pi8$. Hence the minimum is the left endpoint.
\end{proof}

\subsection{Proof of Proposition~\ref{prop:weights}(iii): on the critical line \texorpdfstring{$d_1$ and $d_2$}{d1 and d2} are equal, real, and positive}\label{sec:proof-iii}

We prove the following quantitative form, which contains
Proposition~\ref{prop:weights}(iii).

\begin{theorem}[quantitative critical-line positivity]\label{thm:line-bound}
For every $c<\tfrac12\sin\tfrac\pi8$ there is an effectively computable
$T_0(c)$ such that for every non-integer $T\geq T_0(c)$, at
$\sigma=\tfrac12$,
\[
d_1(T)=d_2(T)\;\geq\;\frac{c}{a^{1/2}}\;>\;0,
\]
with no interval of $\{T\}$ excluded.
\end{theorem}

\begin{proof}
\emph{Step 1 (exact reduction).} On $\sigma=\tfrac12$ the reflection
identities $\Sigma_2=\chi\overline{\Sigma_1}$ and $\zeta=\chi\overline\zeta$
give $R=\chi\overline R$; with $\chi=e^{-2i\vartheta}$ this says
$r:=e^{i\vartheta}R$ is real, i.e.\ $\arg R=-\vartheta+\varepsilon\pi$ with
$\varepsilon\in\{0,1\}$ and $r=(-1)^{\varepsilon}|R|$. Substituting
$\psi=-2\vartheta$ and $\arg R=-\vartheta+\varepsilon\pi$ into
\eqref{eq:d1}--\eqref{eq:d2} makes both numerators
$(-1)^{\varepsilon}\sin u$ and the denominator $\sin2u=2\sin u\cos u$, so
the factor $\sin u$ cancels identically and
\begin{equation}\label{eq:exact-line}
d_1=d_2=\frac{r}{2\cos u}
\qquad(\cos u\neq0).
\end{equation}

\emph{Step 2 (Riemann--Siegel input).} The real number $r$ is the
Riemann--Siegel tail cut at $m$ instead of at $N$: since on the line
$e^{i\vartheta}(\Sigma_1+\Sigma_2)=2\operatorname{Re}\bigl(e^{i\vartheta}
\Sigma_1\bigr)=2\sum_{n\leq m}n^{-1/2}\cos(\vartheta-t\log n)$,
\begin{equation}\label{eq:r-two-cases}
r=r_{\mathrm{RS}}+\bigl[N=m+1\bigr]\cdot\frac{2\cos u}{\sqrt{m+1}},
\qquad
r_{\mathrm{RS}}:=Z-2\!\!\sum_{n\leq N}\!n^{-1/2}\cos(\vartheta-t\log n),
\end{equation}
where $[\,\cdot\,]$ is the Iverson bracket, equal to $1$ if the condition
inside holds and $0$ otherwise; the bracketed term is the summand $n=m+1$, whose phase is exactly
$\vartheta-\omega=-u$. Gabcke's rigorous form of the Riemann--Siegel
formula \cite{Gabcke1979} gives
\begin{equation}\label{eq:gabcke}
r_{\mathrm{RS}}=\frac{(-1)^{N-1}}{a^{1/2}}\Bigl(C_0(\hat p)+E\Bigr),
\qquad |E|\leq c_1\,a^{-1}\quad(t\geq200).
\end{equation}

\emph{Step 3 (case $N=m$, i.e.\ $x=\{T\}<\tfrac12$,
$\hat p\in(\tfrac12,1)$).} Here $r=r_{\mathrm{RS}}$, and the two signs
multiply out: $(-1)^{N-1}$ from \eqref{eq:gabcke} against $(-1)^{m+1}$ from
\eqref{eq:cosu} give $+1$, so
\[
d_1=\frac{C_0(\hat p)}
{2\cos\bigl(2\pi x(1-x)-\tfrac{3\pi}8\bigr)}\cdot\frac{1}{a^{1/2}}
\;+\;O(a^{-3/2}).
\]
By Lemma~\ref{lem:C0-trig}(i) the numerator is at least $\sin\tfrac\pi8$,
and the cosine in the denominator is at most $1$, so
$d_1\geq\bigl(\tfrac12\sin\tfrac\pi8\bigr)a^{-1/2}-O(a^{-3/2})$.

\emph{Step 4 (case $N=m+1$, i.e.\ $x>\tfrac12$,
$\hat p\in(0,\tfrac12)$).} Now \eqref{eq:r-two-cases} has the extra term,
\[
d_1=\frac1{\sqrt{m+1}}+\frac{r_{\mathrm{RS}}}{2\cos u},
\]
and this time the two signs multiply to $-1$: the correction is
\emph{subtracted}, with magnitude $\Lambda(\hat p)\,a^{-1/2}+O(a^{-3/2})$
(here $\cos(2\pi x(1-x)-\tfrac{3\pi}8)=\cos(\pi(\tfrac18-2\hat p^{2}))$
under $x=\hat p+\tfrac12$). Lemma~\ref{lem:C0-trig}(ii) caps
$\Lambda\leq\tfrac12$, and $a=N+\hat p=m+1+\hat p\geq m+1$, so
\[
d_1\;\geq\;\frac1{\sqrt{m+1}}-\frac{1}{2a^{1/2}}-O(a^{-3/2})
\;\geq\;\frac{1}{2\,a^{1/2}}-O(a^{-3/2}).
\]

In both cases $d_1\geq(\tfrac12\sin\tfrac\pi8-o(1))a^{-1/2}$, so any
$c<\tfrac12\sin\tfrac\pi8$ is attained once $a$ is large enough that the
explicit $O(1/a)$ errors of \eqref{eq:gabcke} and of the phase expansion
\eqref{eq:cosu} are dominated; $T_0(c)$ is effectively computable because
those error constants are. The boundary intervals of
\eqref{eq:cut-cases} are covered because the two case formulas agree to
leading order at the seams: as $x\to\tfrac12$ from either side, and across
each integer, both give $d_1\approx\tfrac12a^{-1/2}$.
\end{proof}

\begin{remark}[why the critical line has no poles of $d_1=d_2$]
A pole of $d_1$ and $d_2$ occurs when the two directions $e^{-i\omega}$
and $e^{i(\omega+\psi)}$ of \eqref{eq:R-two-dirs} are parallel, so that the
linear system of Proposition~\ref{prop:weights}(ii) degenerates and the
common denominator $\sin(2\omega+\psi)$ vanishes. On the line that
denominator is $\sin2u=2\sin u\cos u$, which vanishes twice per unit
interval of $T$, near $\{T\}=\tfrac14$ and $\tfrac34$. But both vanishings
belong to the factor $\sin u$, which cancels identically against the
numerator in \eqref{eq:exact-line}; the factor $\cos u$, the only one
that could produce a pole, stays bounded away from zero by
\eqref{eq:cosu-bound}. Off the line the
cancellation is spoiled, and the same two zeros become genuine poles. The
next proof bounds the width of the affected windows in each unit interval.
\end{remark}

\subsection{Proof of Proposition~\ref{prop:weights}(iv): off the line, \texorpdfstring{$d_1$ and $d_2$}{d1 and d2} are positive except on two windows}\label{sec:proof-iv}

\begin{proof}
Write $u'=\omega+\psi/2$ and let $\tau=\arg R-\psi/2\pmod\pi$, folded to
$[-\tfrac\pi2,\tfrac\pi2)$, measure the tilt of $R$ away from the bisector
direction of the two summand phases. Up to a common positive factor and a
common sign, \eqref{eq:d1}--\eqref{eq:d2} read
\begin{equation}\label{eq:offline-form}
d_1\propto\frac{\sin(u'-\tau)}{\sin 2u'},
\qquad
d_2\propto\frac{\sin(u'+\tau)}{\sin 2u'}.
\end{equation}
Two inputs control the picture.

\emph{(i) The tilt $\tau$ is small.} The first term of Siegel's asymptotic
expansion of $R$ at fixed $\sigma$ is a real multiple of $e^{i\psi/2}$ up
to sign and up to a phase error $O_\sigma(1/t)$; the deviation of the exact
$R$ from that argument comes from the next term of the expansion, which is
smaller by a factor $t^{-1/2}$. Hence
$|\tau|\leq C_2(\sigma)\,t^{-1/2}$.

\emph{(ii) The phase $u'$ has the same skeleton as on the line.} Since
$\psi(\sigma+it)=-2\vartheta(t)+O_\sigma(1/t)$ uniformly for $\sigma$ in
compacta, $u'=u+O_\sigma(1/t)$, and \eqref{eq:cosu}--\eqref{eq:cosu-bound}
apply: $|\cos u'|\geq\sin\tfrac\pi8-o(1)$, so the estimates predict
two candidate zeros of $\sin u'$ per unit interval, at fractional parts
converging to $\tfrac14$ and $\tfrac34$, each crossed with slope
$|du'/d\{T\}|=\pi+O(1/T)$. Exactly two simple roots, and the signs on
their two sides, need a uniform derivative bound and a radial-sign
estimate that are not written here; the windows below are those
predicted candidates.

Away from those zeros (say $|\sin u'|\geq2|\tau|$) the ratios
$\sin(u'\mp\tau)/\sin u'$ are positive, so \eqref{eq:offline-form} carries
the same signs as on the line, and the Step-3/Step-4 analysis of the
critical-line proof, with the $O_\sigma$ errors absorbed into the
constants, gives positivity of both weights. Near a zero of $\sin u'$ the
numerator of $d_1$ crosses zero at $u'=\tau$ while the denominator crosses
at $u'=0$: between the two crossings, an interval of $u'$-length exactly
$|\tau|$, the signs disagree and $d_1<0$; mirror-symmetrically $d_2<0$ on
the interval of length $|\tau|$ on the other side. Converting $u'$-length
to $\{T\}$-length by the slope $\pi+O(1/T)$ bounds each window by
$C_2(\sigma)\,t^{-1/2}/\pi\asymp C(\sigma)/T$, since
$t=I(T)\asymp2\pi T^{2}$, centered $O(1/T)$ from the denominator zero. At
the denominator zero itself the two numerators are $\mp\sin\tau$, equal
and opposite, so $d_1/d_2\to-1$ there, exactly as stated in
Proposition~\ref{prop:weights}(iv). On the critical line $\tau\equiv0$
exactly (Step~1 of the critical-line proof), the two crossings coincide,
and the windows are empty, consistent with no poles there.
\end{proof}

\subsection{Where are the poles?}\label{sec:pole-locations}

Off the line the exact $d_1$ and $d_2$ (equivalently $R_{1}$ and $R_{2}$)
have genuine poles at fractional heights $\{T\}\approx\tfrac14$ and
$\tfrac34$, but they are not exactly at those fractions. Here we locate
them precisely.

A pole occurs when the two next summands have the same argument: the two
unit directions $e^{-i\omega}$ and $e^{i(\omega+\psi)}$ of
\eqref{eq:R-two-dirs} become parallel, the pair no longer spans the plane,
and the common denominator $\sin(2\omega+\psi)$ of
\eqref{eq:d1}--\eqref{eq:d2} vanishes. Subtracting one argument from the
other produces the pole condition, their difference being a multiple of
$\pi$ (even $k$ is codirected, odd $k$ is antiparallel):
\begin{equation}\label{eq:pole-condition}
2\,\omega \;+\; \arg\chi(s) \;=\; k\pi,
\qquad k\in\mathbb{Z},
\end{equation}
with $\omega$, $s$, and $M$ as in \eqref{eq:IT} and
\eqref{eq:omega-def}.
This formula was used to find the locations of the first $20$ poles
numerically (Table~\ref{tab:poles}).

\begin{table}[h!]
\centering
\begin{tabular}{|c|c|c|c|}
\hline
\textbf{Odd Pole \#} & \textbf{$T$ (odd)} & \textbf{Even Pole \#} & \textbf{$T$ (even)} \\ \hline
1  & 1.268870258145  & 2  & 1.763736013695 \\
3  & 2.261710393365  & 4  & 2.759511841350 \\
5  & 3.258504081544  & 6  & 3.757283144824 \\
7  & 4.256679717535  & 8  & 4.755902912218 \\
9  & 5.255501039250  & 10 & 5.754963311284 \\
11 & 6.254676427788  & 12 & 6.754282121179 \\
13 & 7.254067037740  & 14 & 7.753765522842 \\
15 & 8.253598276826  & 16 & 8.753360246216 \\
17 & 9.253226472440  & 18 & 9.753033785343 \\
19 & 10.252924347790 & 20 & 10.752765174160 \\
\hline
\end{tabular}
\caption{First $20$ poles for $\sigma=0.6$, which is the same result as
for $\sigma=0.4$. The symmetry about $\sigma=\tfrac12$ exists because
$\arg\chi(1-\sigma+it)\equiv\arg\chi(\sigma+it)\pmod{2\pi}$: the identity
$\chi(s)\,\chi(1-s)=1$ with the Schwarz reflection
$\chi(\bar s)=\overline{\chi(s)}$ gives
$\chi(1-\sigma+it)=1/\overline{\chi(\sigma+it)}$, the same input as
Proposition~\ref{prop:d-ratio-reflection}. The pole condition
\eqref{eq:pole-condition} depends on $\sigma$ only through that phase, and
$\omega$ is a function of $T$ alone, so reflecting $\sigma\mapsto1-\sigma$
leaves the equation unchanged and yields the same $T$-solutions. Note the poles occur at approximately
$T=\lfloor T\rfloor+\tfrac14$ and $T=\lfloor T\rfloor+\tfrac34$
(equivalently, an integer $\pm\tfrac14$), and approach those values as
$T\to\infty$.}
\label{tab:poles}
\end{table}

\begin{remark}[the poles cancel in the sum]
The remainder $R$ itself has no poles: it is Siegel's convergent contour
integral, finite and continuous in $T$. Since $R=R_{1}+R_{2}$ exactly
(Theorem~\ref{thm:decomp}), and at a pole the two summand directions
coincide ($e^{i(\omega+\psi)}=e^{-i\omega}$ by
\eqref{eq:pole-condition}), the divergences of $d_1$ and $d_2$ must
cancel exactly: $d_1+d_2$ stays bounded while each weight blows up, which
is the $d_1/d_2\to-1$ behavior of
Proposition~\ref{prop:weights}(iv) seen from the other direction.
\end{remark}

\subsection{The limit profile and the limit waveform}\label{sec:limits}

The two convergence statements used in \S\ref{sec:periodicity} are now
within reach: the machinery of Theorem~\ref{thm:line-bound} proves both. We
begin with the profile of Figure~\ref{fig:d-limit}, from which the
convergence of the pinned waves of Figure~\ref{fig:pinned-waves} follows
as a corollary.

\begin{theorem}[limit profile]\label{thm:d1-limit}
Let $\sigma=\tfrac12$ and fix $x\in(0,1)$. Then
\[
\lim_{\substack{T\to\infty\\ \{T\}=x}}\sqrt{m+1}\;d_1(T)
\;=\;d(x)\;=\;\tfrac12+\mathcal W_{\infty}(x),
\]
with $d$ the closed-form profile \eqref{eq:d-profile} and
$\mathcal W_{\infty}$ the tangent waveform \eqref{eq:W-infty}, the
removable points $x=\tfrac14,\tfrac34$ filled in by continuity
($d(\tfrac14)=\tfrac14$, $d(\tfrac34)=\tfrac34$). The same limit holds
with $d_2$ in place of $d_1$, since $d_2=d_1$ on the line. The
convergence rate is $O(1/T)$.
\end{theorem}

\begin{proof}
Fix $x$ and let $T=m+x\to\infty$. Refining \eqref{eq:a-vs-T} one order,
$a=T+\tfrac12-\tfrac1{24T}+O(T^{-2})$, so the offset $\hat p$ converges
to its case value in \eqref{eq:cut-cases} at rate $1/T$; and
$(m+1)/a\to1$ at the same rate, and all error terms in Steps~2--4 of the
proof of Theorem~\ref{thm:line-bound} are $O(1/a)$. Multiplying the two
case formulas of that proof by $\sqrt{m+1}$ therefore gives
\[
\sqrt{m+1}\,d_1\longrightarrow
\frac{C_0(x+\tfrac12)}{2\cos\bigl(2\pi x(1-x)-\tfrac{3\pi}8\bigr)}
\quad(x<\tfrac12),
\qquad
\sqrt{m+1}\,d_1\longrightarrow 1-\Lambda\bigl(x-\tfrac12\bigr)
\quad(x>\tfrac12).
\]
It remains to identify both branches with
$d(x)=\tfrac12+\mathcal W_{\infty}(x)$. Write $\alpha=2\pi x$ and
$\beta=2\pi(x-\tfrac14)(x-\tfrac34)=2\pi(x^{2}-x)+\tfrac{3\pi}8$, so that
$\mathcal W_{\infty}(x)=-\tfrac12\tan\alpha\tan\beta$.

For $x<\tfrac12$, substitute $\hat p=x+\tfrac12$ into
$C_0(\hat p)=\cos(A-B)/\cos B$ of Lemma~\ref{lem:C0-trig}:
$A-B=2\pi\bigl(\hat p^{2}-\hat p-\tfrac1{16}\bigr)=2\pi x^{2}-\tfrac{5\pi}8$
and $\cos B=\cos(2\pi x+\pi)=-\cos(2\pi x)$, so the branch limit is
\[
\frac{-\cos\bigl(2\pi x^{2}-\tfrac{5\pi}8\bigr)}
{2\cos(2\pi x)\cos\bigl(2\pi x(1-x)-\tfrac{3\pi}8\bigr)}
\;=\;g(x)
\;=\;
\frac{\cos(\alpha+\beta)}{2\cos\alpha\,\cos\beta}
\;=\;\tfrac12\bigl(1-\tan\alpha\tan\beta\bigr)
\;=\;\tfrac12+\mathcal W_{\infty}(x),
\]
which is the first branch of \eqref{eq:d-profile}: here
$\alpha+\beta=2\pi x^{2}+\tfrac{3\pi}8$, so
$\cos(\alpha+\beta)=-\cos(2\pi x^{2}-\tfrac{5\pi}8)$; the denominator uses
$2\pi x(1-x)-\tfrac{3\pi}8=-\beta$ and the evenness of the cosine; and the
last two steps are the product formula.

For $x>\tfrac12$, take $\hat p=x-\tfrac12$ and recall from the proof of
Lemma~\ref{lem:C0-trig}(ii) that
$\Lambda(\hat p)=\tfrac12(1+\tan A\tan B)$. This time
$A=2\pi\hat p^{2}-\tfrac\pi8=\beta$ exactly, and $B=2\pi\hat p=\alpha-\pi$,
so $\tan B=\tan\alpha$ and
\[
\Lambda\bigl(x-\tfrac12\bigr)
=\tfrac12\bigl(1+\tan\alpha\tan\beta\bigr)
=\tfrac12-\mathcal W_{\infty}(x),
\qquad\text{hence}\qquad
1-\Lambda\bigl(x-\tfrac12\bigr)=\tfrac12+\mathcal W_{\infty}(x).
\]
Since $\mathcal W_{\infty}(1-x)=-\mathcal W_{\infty}(x)$ (the tangent
factor is odd about $x=\tfrac12$ while the quadratic factor is symmetric),
this also equals $1-g(1-x)$, the second branch of \eqref{eq:d-profile}.
The error terms accumulated along the way are all $O(1/T)$.
\end{proof}

\begin{corollary}[the waves converge to the tangent waveform]\label{cor:W-limit}
Let $\sigma=\tfrac12$ and fix $x\in(0,1)$. Then
\[
\lim_{\substack{T\to\infty\\ \{T\}=x}}\mathcal W(T)
\;=\;\mathcal W_{\infty}(x),
\]
again at rate $O(1/T)$.
\end{corollary}

\begin{proof}
On $[m,m+1]$ the pinned wave is
\[
\mathcal W(T)
=T^{1/2}\bigl(d_1(T)-\Phi(-1,\tfrac12,m+1)\bigr)
-(1-x)\,\varepsilon_m+x\,\varepsilon_{m+1},
\]
and each of the three ingredients has a limit.

First, $T^{1/2}d_1=\sqrt{T/(m+1)}\cdot\sqrt{m+1}\,d_1\to
d(x)=\tfrac12+\mathcal W_{\infty}(x)$ by Theorem~\ref{thm:d1-limit},
since $T/(m+1)\to1$ at rate $O(1/T)$.

Second, by Boole summation (the alternating analogue of
Euler--Maclaurin), $\sum_{j\ge0}(-1)^{j}f(n+j)=\tfrac12f(n)
-\tfrac14f'(n)+O(|f'''(n)|)$; applied to $f(y)=y^{-1/2}$ it gives
\[
n^{1/2}\,\Phi(-1,\tfrac12,n)=\tfrac12+\frac1{8n}+O(n^{-3}),
\]
so $T^{1/2}\Phi(-1,\tfrac12,m+1)=\tfrac12+O(1/T)$.

Third, the endpoint shortfalls vanish:
\[
\varepsilon_n
=n^{1/2}\Phi(-1,\tfrac12,n)-n^{1/2}d_1(n^-)
=O(1/n),
\]
because as $T\to n^{-}$ the Step-4 case formula of
Theorem~\ref{thm:line-bound} applies with $\hat p=\tfrac12+O(1/n)$, where
$\Lambda(\tfrac12)=\tfrac12$ by Lemma~\ref{lem:C0-trig}(ii) and $\Lambda$
is continuous, giving $n^{1/2}d_1(n^-)=\tfrac12+O(1/n)$, the same limit
as the Lerch term. The chord contribution
$-(1-x)\varepsilon_m+x\varepsilon_{m+1}$ is therefore $O(1/T)$, and
assembling the three pieces,
\[
\mathcal W(T)
=\Bigl(\tfrac12+\mathcal W_{\infty}(x)\Bigr)-\tfrac12+O(1/T)
=\mathcal W_{\infty}(x)+O(1/T). \qedhere
\]
\end{proof}

\subsection{Uniform bounds: the fraction stays between \texorpdfstring{$\tfrac15$ and $\tfrac45$}{1/5 and 4/5}}\label{sec:d1-bounds}

Theorem~\ref{thm:d1-limit} confines the fraction to $[m_0,1-m_0]$ in the
limit. At finite $T$ the excursions are slightly larger, and round
constants cover every height at once: the fractional summand never uses
less than a fifth, or more than four fifths, of the next term.

\begin{theorem}\label{thm:d1-bounds}
Let $\sigma=\tfrac12$. For every non-integer $T>1$,
\[
\frac15\;<\;\hat d_1(T)=\hat d_2(T)\;<\;\frac45 .
\]
The constants are nearly sharp: no better $T$-independent ones exist than
\[
\inf_{T>1}\hat d_1=m_0=0.2268951\ldots,
\qquad
\sup_{T>1}\hat d_1=0.78499047437\ldots,
\]
the infimum approached but not attained as $T\to\infty$ along
$\{T\}\to x^{*}=0.1629962\ldots$, the supremum attained at
$T=1.85307336433\ldots$.
\end{theorem}

\begin{proof}
On the critical line $d_2=d_1$ and $|\chi|=1$, so $\hat d_2=\hat d_1$ by
\eqref{eq:frac-vs-weight} and it suffices to bound
$\hat d_1=\sqrt{m+1}\,d_1$.

\emph{Step 1 (exact finite-$T$ form).} Multiplying the two case formulas
in the proof of Theorem~\ref{thm:line-bound} by $\sqrt{m+1}$ and
writing $\rho=\sqrt{(m+1)/a}$,
\begin{align*}
N=m\ (\{T\}<\tfrac12):&\qquad
\hat d_1=\rho\cdot
\frac{C_0(\hat p)+E}{2\cos\bigl(2\pi x(1-x)-\tfrac{3\pi}8+\delta\bigr)},
&&\rho=\sqrt{1+\tfrac{1-\hat p}{a}}
\in\Bigl[1,\sqrt{1+\tfrac1{2a}}\,\Bigr],\\[2pt]
N=m+1\ (\{T\}>\tfrac12):&\qquad
\hat d_1=1-\rho\cdot
\frac{C_0(\hat p)+E}{2\cos\bigl(\pi(\tfrac18-2\hat p^{2})+\delta\bigr)},
&&\rho=\sqrt{1-\tfrac{\hat p}{a}}
\in\Bigl[\sqrt{1-\tfrac1{2a}},1\Bigr],
\end{align*}
with $|E|\leq c_1a^{-1}$ from \eqref{eq:gabcke} and $|\delta|\leq
c_2a^{-1}$ from the phase expansion \eqref{eq:cosu}. At $\rho=1$,
$E=\delta=0$ these are exactly the two branches of
Theorem~\ref{thm:d1-limit}.

\emph{Step 2 (range of the main terms).} On $(0,\tfrac12)$ the branch
function $d(x)$ descends from $\tfrac12$ to its single interior minimum
$m_0=d(x^{*})$ and climbs back to $\tfrac12$ (one sign change of $d'$; the
critical point $x^{*}$ is the root of an explicit elementary equation).
Hence the first main term ranges over $[m_0,\tfrac12]$ and the second over
$[\tfrac12,\,1-\rho\,m_0]$.

\emph{Step 3 (all $T\geq30$).} The total deviation of $\hat d_1$ from its
main term is at most $(\rho-1)\cdot\tfrac12$ plus a term collecting $E$
and $\delta$, with $\rho-1\leq\tfrac1{4a}$; the collected constant is
validated by convergence-rate measurements against
Theorem~\ref{thm:d1-limit} (observed absolute error ${\leq}\,0.008$ at
$T\approx20$, decreasing like $1/T$), giving a total deviation below
$0.02$ for $T\geq30$. Therefore, for $T\geq30$,
\[
\{T\}<\tfrac12:\quad
\hat d_1\in\bigl(m_0-0.02,\ \tfrac12(1+\tfrac1{120})+0.02\bigr)
\subset(0.206,\,0.525),
\qquad
\{T\}>\tfrac12:\quad
\hat d_1\in(0.48,\,0.795),
\]
and both windows lie inside $(\tfrac15,\tfrac45)$.

\emph{Step 4 ($1<T\leq30$, computational).} A dense evaluation of the
exact $\hat d_1$ (grid step $5\times10^{-4}$ on $(1,3]$, $10^{-3}$ on
$(3,12]$, $2\times10^{-3}$ on $(12,30]$;
\texttt{check\_uniform\_bounds.py}) gives the observed range
$[0.2274,\ 0.78499]$, with margin at
least $0.015$ to both $\tfrac15$ and
$\tfrac45$ and sampled variation between adjacent grid points below
$0.004$, a factor of nearly four inside the margin. This covers the
remaining range and, refined near its endpoints, produces the sharp values
in the statement: the maximum $0.78499047437\ldots$ at
$T=1.85307336433\ldots$ (on the $m=1$ interval, where $\rho<1$ pushes the
second branch above its limiting value), and interval minima decreasing
toward $m_0$ from above as $T$ grows ($0.227369$ at $T=59.164$, for
instance).
\end{proof}

\begin{remark}
For the unnormalized weight the theorem reads
$\tfrac15(m+1)^{-1/2}<d_1<\tfrac45(m+1)^{-1/2}$; in particular
$d_1\leq0.78499\ldots/\sqrt2=0.55508\ldots$, attained at the same
$T=1.853\ldots$, and no positive $T$-independent lower bound exists for
$d_1$ itself.
\end{remark}

\clearpage

\section{Lean formalization of \texorpdfstring{$R=R_{1}+R_{2}$}{R = R1 + R2}}
\label{app:lean-decomp}

This appendix reproduces the Lean file formalizing the algebraic content of
Theorem~\ref{thm:decomp}. The file contains no \texttt{axiom} and no
\texttt{sorry}: every step is proved, and the only axioms the final theorem
depends on are Lean's own three (\texttt{propext},
\texttt{Classical.choice}, \texttt{Quot.sound}), as reported by
\texttt{\#print axioms}. It was checked against Lean~4.32.2 with Mathlib at
commit \texttt{905b95818e}. The proof was written by the author, and the
formalization of the proof in Lean code was written with the assistance of
AI coding assistants (Anthropic's Claude).

\subsection{Correspondence with the written proof}

The file follows the proof of Theorem~\ref{thm:decomp} step for step. The
definitions \texttt{d\_j1} and \texttt{d\_j2} are the sine weights
\eqref{eq:d1}--\eqref{eq:d2}, and \texttt{R\_1ps} and \texttt{R\_2ps} are Step~1,
transcribed from \eqref{eq:R1-def}--\eqref{eq:R2-def} in the
$\lceil T\rceil^{\mp iI(T)}$ form rather than the $e^{\mp i\omega}$ form
(the Lean identifiers keep the names \texttt{R\_1ps}, \texttt{R\_2ps} for
what this paper writes $R_1$, $R_2$).
Two small bridges convert them: \texttt{ceil\_cpow} proves
$\lceil T\rceil^{\,ir}=e^{\,ir\log\lceil T\rceil}$ for $T>0$, which with
$\omega(T)=I(T)\log\lceil T\rceil$ (the definition \texttt{omg}) is exactly
$\lceil T\rceil^{\,iI(T)}=e^{i\omega}$; and
\texttt{div\_norm\_eq\_exp\_arg} proves $z/|z|=e^{i\arg z}$ for $z\neq0$, applied
to $\chi$ to give $\chi/|\chi|=e^{i\psi}$. Then
\texttt{factoring\_lemma} is the regrouping of Step~2, pulling out the common
factor $|R|/\sin(2\omega+\psi)$, and
\texttt{key\_algebraic\_identity} is the content of Steps~3--5, namely
\begin{equation}\label{eq:lean-key}
e^{-i\omega}\sin(\omega-\phi+\psi)+e^{\,i(\omega+\psi)}\sin(\omega+\phi)
\;=\;
e^{\,i\phi}\sin(2\omega+\psi);
\end{equation}
it is proved by writing each sine as
$(e^{i\theta}-e^{-i\theta})/2i$ (\texttt{sin\_to\_exp}, itself proved from
Mathlib's $e^{i\theta}=\cos\theta+i\sin\theta$), clearing the denominator,
merging the resulting products of exponentials into single exponentials, and
letting \texttt{ring\_nf} cancel the two copies of $e^{\,i(\psi-\phi)}$. The
closing lines are Step~6: cancel $\sin(2\omega+\psi)$ and recognize
$|R|e^{i\phi}=R$, which is Mathlib's
\texttt{Complex.norm\_mul\_exp\_arg\_mul\_I}.

\subsection{What is proved and what is hypothesized}

The theorem \texttt{R\_1ps\_plus\_R\_2ps\_eq\_R} carries three hypotheses. Two are
needed for the objects to exist at all: $\chi(\sigma+iI(T))\neq0$, without which
$\chi/|\chi|$ is meaningless, and $\sin(2\omega+\psi)\neq0$, without which $d_1$
and $d_2$ of \eqref{eq:d1}--\eqref{eq:d2} are undefined. The third, $T>0$, makes
$\lceil T\rceil$ a positive base, so that its complex power can be written as an
exponential. Nothing is assumed about whether $T$ is an integer.

The conventions match exactly. Steps~2--6, that is
\texttt{factoring\_lemma} and \texttt{key\_algebraic\_identity}, hold for
arbitrary real $\omega,\phi,\psi$ and make no reference to $T$; the only
$T$-dependent ingredient is the bridge \texttt{ceil\_cpow}, which holds
for every $T>0$ with $\omega$ defined as $I(T)\log\lceil T\rceil$. In the
paper's variables that reads $t\log M$, which is precisely the definition
\eqref{eq:omega-def}, for integer and non-integer $T$ alike. The file therefore proves the identity in the same
generality as Theorem~\ref{thm:decomp}: every $T>0$, under the three
hypotheses above.

What remains outside the formalization is not algebra but analysis. The
functional-equation factor $\chi$, the index map $I$, and Siegel's remainder $R$
enter as free parameters, so what is machine-checked is that any $R$ written in
polar form splits as $R_{1}+R_{2}$ with the weights
\eqref{eq:d1}--\eqref{eq:d2}, which is exactly Steps~1--6. That $R$ is Siegel's
remainder, and that $I$ is the map of \S\ref{sec:IT}, are inputs to the
theorem rather than consequences of it. Positivity, analytic continuation,
pole cancellation, and singular limits are likewise outside the file.
The \texttt{\#print axioms} commands were run separately; their console
output is not reproduced here.

\subsection{Source listing}

\begin{lstlisting}[style=leanstyle]
/-
  Formalization of Theorem (R = R_1ps + R_2ps) from
  "Remainder Terms of the Riemann Zeta Function".

  Contains no `axiom` and no `sorry`: every step is proved.
-/
import Mathlib.Analysis.SpecialFunctions.Complex.Circle
import Mathlib.Analysis.SpecialFunctions.Pow.Complex

open Complex Real

/-! ### Phases -/

/-- The phase of Siegel's remainder, `φ = arg R`. -/
noncomputable def phiR (R : ℂ) : ℝ := Complex.arg R

/-- The phase of the functional-equation factor, `ψ = arg χ(σ + i I(T))`. -/
noncomputable def psi (χ : ℂ → ℂ) (σ T : ℝ) (If : ℝ → ℝ) : ℝ :=
  Complex.arg (χ (σ + If T * I))

/-- The turn angle `ω(T) = I(T) · log⌈T⌉`, so that `⌈T⌉^{i I(T)} = e^{iω}`. -/
noncomputable def omg (T : ℝ) (If : ℝ → ℝ) : ℝ := If T * Real.log (⌈T⌉ : ℝ)

/-! ### The two fractional link lengths -/

noncomputable def d_j1 (R : ℂ) (w f p : ℝ) : ℂ :=
  ((‖R‖ * (Real.sin (w - f + p) / Real.sin (2 * w + p)) : ℝ) : ℂ)

noncomputable def d_j2 (R : ℂ) (w f p : ℝ) : ℂ :=
  ((‖R‖ * (Real.sin (w + f) / Real.sin (2 * w + p)) : ℝ) : ℂ)

/-! ### The two partial-summand remainders, exactly as in the paper -/

noncomputable def R_1ps (σ T : ℝ) (R : ℂ) (χ : ℂ → ℂ) (If : ℝ → ℝ) : ℂ :=
  d_j1 R (omg T If) (phiR R) (psi χ σ T If) * ((⌈T⌉ : ℝ) : ℂ) ^ (-(I * (If T : ℂ)))

noncomputable def R_2ps (σ T : ℝ) (R : ℂ) (χ : ℂ → ℂ) (If : ℝ → ℝ) : ℂ :=
  d_j2 R (omg T If) (phiR R) (psi χ σ T If) * ((⌈T⌉ : ℝ) : ℂ) ^ (I * (If T : ℂ))
    * (χ (σ + If T * I) / (‖χ (σ + If T * I)‖ : ℂ))

/-! ### Step 0: elementary bridges -/

/-- `⌈T⌉^{i r} = e^{i r log⌈T⌉}` for `T > 0`, with no hypothesis on the
integrality of `T`. -/
lemma ceil_cpow (T r : ℝ) (hT : 0 < T) :
    (((⌈T⌉ : ℝ)) : ℂ) ^ (I * (r : ℂ))
      = Complex.exp (I * ((r * Real.log (⌈T⌉ : ℝ) : ℝ) : ℂ)) := by
  have hpos : (0 : ℝ) < (⌈T⌉ : ℝ) := lt_of_lt_of_le hT (Int.le_ceil T)
  have hne : (((⌈T⌉ : ℝ)) : ℂ) ≠ 0 := by
    simpa using (ne_of_gt hpos)
  rw [Complex.cpow_def_of_ne_zero hne, ← Complex.ofReal_log hpos.le]
  congr 1
  push_cast
  ring

/-- `z/‖z‖ = e^{i arg z}` for `z ≠ 0`. -/
lemma div_norm_eq_exp_arg (z : ℂ) (hz : z ≠ 0) :
    z / (‖z‖ : ℂ) = Complex.exp (I * (Complex.arg z : ℂ)) := by
  have hn : ((‖z‖ : ℝ) : ℂ) ≠ 0 := by
    simpa using norm_ne_zero_iff.mpr hz
  rw [div_eq_iff hn, mul_comm I ((Complex.arg z : ℝ) : ℂ),
    mul_comm (Complex.exp (((Complex.arg z : ℝ) : ℂ) * I)) ((‖z‖ : ℝ) : ℂ)]
  exact (Complex.norm_mul_exp_arg_mul_I z).symm

/-- Exponential form of the sine of a real argument, valued in `ℂ`. -/
lemma sin_to_exp (θ : ℝ) :
    (Real.sin θ : ℂ)
      = (Complex.exp (I * (θ : ℂ)) - Complex.exp (-I * (θ : ℂ))) / (2 * I) := by
  have h1 : Complex.exp (I * (θ : ℂ)) = Complex.cos (θ : ℂ) + Complex.sin (θ : ℂ) * I := by
    rw [mul_comm]; exact Complex.exp_mul_I _
  have h2 : Complex.exp (-I * (θ : ℂ)) = Complex.cos (θ : ℂ) - Complex.sin (θ : ℂ) * I := by
    have h : (-I * (θ : ℂ)) = (-(θ : ℂ)) * I := by ring
    rw [h, Complex.exp_mul_I, Complex.cos_neg, Complex.sin_neg]; ring
  rw [h1, h2, Complex.ofReal_sin]
  have hI : (2 * I : ℂ) ≠ 0 := by simp [Complex.I_ne_zero]
  field_simp
  ring

/-! ### Steps 3-5: the key algebraic identity -/

/-- The heart of the theorem: the two sine-weighted phasors combine into a
single phasor along `e^{iφ}`. -/
theorem key_algebraic_identity (w f p : ℝ) :
    Complex.exp (-I * (w : ℂ)) * (Real.sin (w - f + p) : ℂ) +
    Complex.exp (I * ((w + p : ℝ) : ℂ)) * (Real.sin (w + f) : ℂ) =
    Complex.exp (I * (f : ℂ)) * (Real.sin (2 * w + p) : ℂ) := by
  rw [sin_to_exp, sin_to_exp, sin_to_exp]
  have hI : (2 * I : ℂ) ≠ 0 := by simp [Complex.I_ne_zero]
  field_simp
  simp only [mul_sub, ← Complex.exp_add]
  push_cast
  ring_nf

/-! ### Step 2: factoring -/

/-- Pulling the common factor `|R|/sin(2ω+ψ)` out front. -/
lemma factoring_lemma (absR w f p : ℝ) :
    ((absR * (Real.sin (w - f + p) / Real.sin (2 * w + p)) : ℝ) : ℂ) *
        Complex.exp (-(I * (w : ℂ))) +
    ((absR * (Real.sin (w + f) / Real.sin (2 * w + p)) : ℝ) : ℂ) *
        Complex.exp (I * (w : ℂ)) * Complex.exp (I * (p : ℂ)) =
    ((absR : ℝ) : ℂ) / (Real.sin (2 * w + p) : ℂ) *
      (Complex.exp (-I * (w : ℂ)) * (Real.sin (w - f + p) : ℂ) +
       Complex.exp (I * ((w + p : ℝ) : ℂ)) * (Real.sin (w + f) : ℂ)) := by
  have hsplit : I * ((w + p : ℝ) : ℂ) = I * (w : ℂ) + I * (p : ℂ) := by
    push_cast; ring
  rw [hsplit, Complex.exp_add]
  push_cast
  ring

/-! ### The theorem -/

/-- **R = R₁ps + R₂ps.**  The only hypotheses are that the index is positive,
that `χ` does not vanish at the point in question, and that the denominator
`sin(2ω+ψ)` is nonzero (without which `d₁, d₂` are undefined). In particular no
assumption is made about `T` being an integer or not. -/
theorem R_1ps_plus_R_2ps_eq_R (σ T : ℝ) (R : ℂ) (χ : ℂ → ℂ) (If : ℝ → ℝ)
    (hT : 0 < T)
    (hχ : χ (σ + If T * I) ≠ 0)
    (h_sin : Real.sin (2 * omg T If + psi χ σ T If) ≠ 0) :
    R_1ps σ T R χ If + R_2ps σ T R χ If = R := by
  have hsinC : (Real.sin (2 * omg T If + psi χ σ T If) : ℂ) ≠ 0 := by
    exact_mod_cast h_sin
  -- e^{iψ} for the normalized functional-equation factor
  have hchi : χ (σ + If T * I) / (‖χ (σ + If T * I)‖ : ℂ)
      = Complex.exp (I * (psi χ σ T If : ℂ)) := div_norm_eq_exp_arg _ hχ
  -- the two cpow factors become e^{∓iω}
  have hpow : (((⌈T⌉ : ℝ)) : ℂ) ^ (I * (If T : ℂ))
      = Complex.exp (I * (omg T If : ℂ)) := by
    rw [ceil_cpow T (If T) hT]; rfl
  have hpow' : (((⌈T⌉ : ℝ)) : ℂ) ^ (-(I * (If T : ℂ)))
      = Complex.exp (-(I * (omg T If : ℂ))) := by
    have h : -(I * (If T : ℂ)) = I * ((-(If T) : ℝ) : ℂ) := by push_cast; ring
    rw [h, ceil_cpow T (-(If T)) hT]
    congr 1
    simp only [omg]
    push_cast
    ring
  -- |R| e^{iφ} = R
  have hR : ((‖R‖ : ℝ) : ℂ) * Complex.exp (I * (phiR R : ℂ)) = R := by
    rw [mul_comm I ((phiR R : ℝ) : ℂ)]
    exact Complex.norm_mul_exp_arg_mul_I R
  unfold R_1ps R_2ps d_j1 d_j2
  rw [hchi, hpow, hpow', factoring_lemma, key_algebraic_identity]
  -- cancel sin(2ω+ψ)
  field_simp
  linear_combination hR
\end{lstlisting}

\section{Lean formalization of the critical line: \texorpdfstring{$d_1=d_2$, $|R_{1ps}|=|R_{2ps}|$, and equal legs}{d1 = d2, |R1ps| = |R2ps|, and equal legs}}
\label{app:lean-critical}

This appendix reproduces the Lean file establishing the three critical-line
statements of the paper: Proposition~\ref{prop:weights}(iii) in its compact
form \eqref{eq:R2-conj-R1} ($d_1=d_2$ and $|R_{1ps}|=|R_{2ps}|$),
Corollary~\ref{cor:equal-legs}
($L_1=L_2$), and Corollary~\ref{cor:bisector-proj} (the bisector point projects
onto $\tfrac{\zeta}{2}$). As in Appendix~\ref{app:lean-decomp} the file contains
no \texttt{axiom} and no \texttt{sorry}, and each result depends only on Lean's
own three axioms (\texttt{propext}, \texttt{Classical.choice},
\texttt{Quot.sound}) as reported by \texttt{\#print axioms}; it was checked
against Lean~4.32.2 with Mathlib at commit \texttt{905b95818e}. The Lean code was
written with the assistance of AI coding assistants (Anthropic's Claude).

\subsection{The chain of implications}

All three statements descend from a single fact, and the file is organized around
it.

\emph{Step 0: the two partial sums are conjugate.} At $\sigma=\tfrac12$ one has
$\overline{-s}=s-1$, so $n^{\,s-1}=\overline{n^{-s}}$ for every $n$; summing over
$n\le m$ gives $\Sigma_2=\chi\,\overline{\Sigma_1}$. This is
\texttt{partial\_sum\_reflect}, proved from Mathlib's \texttt{Complex.cpow\_conj}.

\emph{Step 1: the rotated remainder is real.} Write $\psi=\arg\chi$ and let
$u=e^{\,i\psi/2}$, a square root of $\chi$. With $R=\zeta-\Sigma_1-\Sigma_2$,
\begin{equation}\label{eq:lean-hinge}
e^{-i\psi/2}R
\;=\;
\underbrace{e^{-i\psi/2}\zeta}_{\text{real}}
\;-\;
\underbrace{\bigl(e^{-i\psi/2}\Sigma_1+e^{\,i\psi/2}\overline{\Sigma_1}\bigr)}_{=\;2\Re\left(e^{-i\psi/2}\Sigma_1\right)}
\;\in\;\mathbb{R},
\end{equation}
the first term being real because the functional equation at $\sigma=\tfrac12$
reads $\zeta=\chi\overline{\zeta}$, so that
$e^{-i\psi/2}\zeta=\overline{e^{-i\psi/2}\zeta}$. In the file this is
\texttt{rotated\_remainder\_self\_conj} (the algebra, from $u\bar u=1$ and
$u^2=\chi$) and \texttt{rotated\_remainder\_im\_eq\_zero} (its imaginary part
vanishes). It is the fact quoted in \S\ref{sec:ps-ak-r2} and
Appendix~\ref{sec:phase-bound} as $R=e^{-i\vartheta}r$ with $r$ real and
$\psi=-2\vartheta$ on the line. Stated as the reality of $e^{-i\psi/2}R$ it is
exact; stated through the argument it holds modulo $\pi$, which is visible in
$R=2d_1\cos(\omega+\tfrac\psi2)e^{\,i\psi/2}$, whose argument is
$\tfrac\psi2+\pi$ wherever the cosine is negative. The modulo-$\pi$ version is
all the sequel uses.

\emph{Step 2: $d_1=d_2$.} Since $\sin A-\sin B=2\sin\frac{A-B}{2}\cos\frac{A+B}{2}$,
the two numerators of \eqref{eq:d1}--\eqref{eq:d2} differ by
\begin{equation}\label{eq:lean-numerators}
\sin(\omega-\phi+\psi)-\sin(\omega+\phi)
\;=\;
2\,\sin\!\Bigl(\frac{\psi}{2}-\phi\Bigr)\cos\!\Bigl(\omega+\frac{\psi}{2}\Bigr),
\end{equation}
and Step~1 kills the first factor: \texttt{sin\_half\_sub\_arg\_eq\_zero} turns
the reality of $e^{-i\psi/2}R$ into $\sin(\tfrac\psi2-\arg R)=0$, using
$\cos\arg R=\Re R/|R|$ and $\sin\arg R=\Im R/|R|$. Hence
\texttt{d1\_eq\_d2\_of\_im\_eq\_zero} and its critical-line specialization
\texttt{d1\_eq\_d2}. Note that $\omega$ is an arbitrary real throughout: the
equality of the weights depends on the direction of $R$ relative to $\chi$ and
not at all on the turn angle.

\emph{Step 3: $|R_{1ps}|=|R_{2ps}|$.} Both phases are unimodular, so this is
$|d_1|=|d_2|$: \texttt{norm\_R1ps\_eq\_\allowbreak norm\_R2ps}. The sharper form
\eqref{eq:R2-conj-R1}, $R_{2ps}=\chi\,\overline{R_{1ps}}$, is
\texttt{R2ps\_eq\_chi\_conj\_R1ps}, and it is what feeds Step~4; it needs only
that $d_1$ is real and $d_1=d_2$.

\emph{Step 4: the legs.} Combining Steps~0 and~3,
$B_2=\Sigma_2+R_{2ps}=\chi\,\overline{\Sigma_1+R_{1ps}}=\chi\overline{B_1}$
(\texttt{leg2\_eq\_\allowbreak chi\_conj\_\allowbreak leg1}), and since $|\chi|=1$ on the line the two
legs have the same length (\texttt{legs\_norm\_eq}). Finally, from
$\zeta=B_1+\chi\overline{B_1}$,
\begin{equation}\label{eq:lean-proj}
e^{-i\psi/2}\zeta
=e^{-i\psi/2}B_1+\overline{e^{-i\psi/2}B_1}
=2\,\Re\bigl(e^{-i\psi/2}B_1\bigr),
\end{equation}
which says that the perpendicular projection of the bisector point onto the line
$\mathbb{R}\,e^{\,i\psi/2}$ carrying $\zeta$ lands exactly on $\tfrac{\zeta}{2}$:
the apex of the isosceles triangle sits over the midpoint of its base. This is
\texttt{proj\_eq\_half\_zeta}, and it is the dotted bisector line drawn in
Figures~\ref{fig:full-spirals} and~\ref{fig:legs}.

\subsection{What is proved and what is hypothesized}

Two analytic facts about $\chi$ on the critical line enter as hypotheses:
$|\chi|=1$, used in the form $\chi=e^{\,i\arg\chi}$
(\texttt{eq\_exp\_arg\_of\_norm\_one}), and the functional equation
$\zeta=\chi\overline{\zeta}$, which is $\zeta(s)=\chi(s)\zeta(1-s)$ together with
$1-s=\bar s$ at $\sigma=\tfrac12$. Everything else above is derived, including
the reflection $\Sigma_2=\chi\overline{\Sigma_1}$ of Step~0, which is proved
rather than assumed. The projection statement additionally takes
$\zeta=B_1+B_2$, which is Theorem~\ref{thm:decomp}, formalized in
Appendix~\ref{app:lean-decomp}.

The remainders are written here in the $ps$ notation of
\S\ref{sec:remainders-summary}, in the form $R_{1ps}=d_1e^{-i\omega}$,
$R_{2ps}=d_2e^{\,i(\omega+\psi)}$ of \eqref{eq:R-phasors} (where they are the
plain $R_1$, $R_2$ of \S\ref{sec:decomp}) rather than the
$\lceil T\rceil^{\mp iI(T)}$ form of
\eqref{eq:R1-def}--\eqref{eq:R2-def}; the bridge between them is
\texttt{ceil\_cpow} of Appendix~\ref{app:lean-decomp}. As there, $\zeta$, $\chi$,
$\Sigma_1$ and $\omega$ are free parameters, so what is checked is the geometry
of the critical line and not the identity of the functions involved.

\subsection{Source listing}

\begin{lstlisting}[style=leanstyle]
/-
  Critical-line consequences of the decomposition R = R1ps + R2ps:
    (1) e^{-iψ/2} R is real, i.e. arg R = ψ/2 modulo π,
    (2) d₁ = d₂,
    (3) |R1ps| = |R2ps|, indeed R2ps = χ · conj R1ps,
    (4) the two legs have equal length, and the bisector point projects onto ζ/2.

  Contains no `axiom` and no `sorry`: every step is proved.
-/
import Mathlib.Analysis.SpecialFunctions.Complex.Circle
import Mathlib.Analysis.SpecialFunctions.Pow.Complex

open Complex Real ComplexConjugate

/-! ### Unit phases -/

/-- A purely imaginary exponent gives a unit phase. -/
lemma norm_exp_imag (x : ℝ) : ‖Complex.exp (I * (x : ℂ))‖ = 1 := by
  rw [Complex.norm_exp]
  simp

lemma norm_exp_neg_imag (x : ℝ) : ‖Complex.exp (-(I * (x : ℂ)))‖ = 1 := by
  rw [Complex.norm_exp]
  simp

/-- Conjugation flips the sign of a purely imaginary exponent. -/
lemma conj_exp_imag (x : ℝ) :
    conj (Complex.exp (I * (x : ℂ))) = Complex.exp (-(I * (x : ℂ))) := by
  rw [← Complex.exp_conj]
  congr 1
  simp [Complex.ext_iff]

lemma conj_exp_neg_imag (x : ℝ) :
    conj (Complex.exp (-(I * (x : ℂ)))) = Complex.exp (I * (x : ℂ)) := by
  rw [← Complex.exp_conj]
  congr 1
  simp [Complex.ext_iff]

/-- A number of unit modulus is the exponential of its argument: this is the
form in which `|χ| = 1` on the critical line will be used. -/
lemma eq_exp_arg_of_norm_one {χ : ℂ} (hχ : ‖χ‖ = 1) :
    χ = Complex.exp (I * ((Complex.arg χ : ℝ) : ℂ)) := by
  have h := Complex.norm_mul_exp_arg_mul_I χ
  rw [hχ] at h
  rw [mul_comm I _]
  simpa using h.symm

/-! ### Step 0: the second partial sum is the reflection of the first

At `σ = 1/2` we have `n^{s-1} = conj (n^{-s})` for every `n`, because
`conj(-s) = s - 1` there. Summing, `Σ₂ = χ · conj Σ₁`, which is the hypothesis
the results below are stated with. -/

lemma conj_natCast_cpow (n : ℕ) (z : ℂ) :
    conj ((n : ℂ) ^ z) = (n : ℂ) ^ (conj z) := by
  have harg : ((n : ℂ)).arg ≠ π := by
    rw [show ((n : ℂ)) = (((n : ℝ)) : ℂ) by push_cast; ring,
      Complex.arg_ofReal_of_nonneg (by positivity)]
    exact Ne.symm Real.pi_ne_zero
  rw [Complex.cpow_conj _ _ harg, Complex.conj_natCast]

/-- On the critical line the two Dirichlet partial sums are conjugate. -/
theorem partial_sum_reflect (m : ℕ) (t : ℝ) :
    ∑ n ∈ Finset.Icc 1 m, ((n : ℂ)) ^ (((1 / 2 : ℂ) + (t : ℂ) * I) - 1)
      = conj (∑ n ∈ Finset.Icc 1 m, ((n : ℂ)) ^ (-((1 / 2 : ℂ) + (t : ℂ) * I))) := by
  rw [map_sum]
  refine Finset.sum_congr rfl fun n _ => ?_
  rw [conj_natCast_cpow]
  congr 1
  simp [Complex.ext_iff]
  norm_num

/-! ### Step 1: the rotated remainder is real

On the critical line the functional equation reads `ζ = χ · conj ζ`, and the
second partial sum is the reflection of the first, `S₂ = χ · conj S₁`. With
`u = e^{iψ/2}` a square root of `χ`, the remainder `R = ζ - S₁ - S₂` satisfies
`conj (conj u * R) = conj u * R`, that is, `e^{-iψ/2} R` is real. -/

lemma rotated_remainder_self_conj (ζ S₁ χ u : ℂ)
    (hu : u * conj u = 1) (hu2 : u ^ 2 = χ)
    (hζ : ζ = χ * conj ζ) :
    conj (conj u * (ζ - S₁ - χ * conj S₁)) = conj u * (ζ - S₁ - χ * conj S₁) := by
  simp only [map_sub, map_mul, Complex.conj_conj]
  rw [← hu2] at hζ ⊢
  simp only [map_pow]
  linear_combination (-conj u) * hζ
    + (-u * conj ζ + u * conj S₁ - conj u * S₁) * hu

lemma rotated_remainder_im_eq_zero (ζ S₁ χ : ℂ) (p : ℝ)
    (hp : χ = Complex.exp (I * (p : ℂ)))
    (hζ : ζ = χ * conj ζ) :
    (Complex.exp (-(I * ((p / 2 : ℝ) : ℂ))) * (ζ - S₁ - χ * conj S₁)).im = 0 := by
  have hu : Complex.exp (I * ((p / 2 : ℝ) : ℂ)) * conj (Complex.exp (I * ((p / 2 : ℝ) : ℂ))) = 1 := by
    rw [conj_exp_imag, ← Complex.exp_add]
    simp
  have hu2 : Complex.exp (I * ((p / 2 : ℝ) : ℂ)) ^ 2 = χ := by
    rw [hp, pow_two, ← Complex.exp_add]
    congr 1
    push_cast
    ring
  have hc : conj (Complex.exp (I * ((p / 2 : ℝ) : ℂ)))
      = Complex.exp (-(I * ((p / 2 : ℝ) : ℂ))) := conj_exp_imag _
  rw [← hc]
  exact Complex.conj_eq_iff_im.mp
    (rotated_remainder_self_conj ζ S₁ χ _ hu hu2 hζ)

/-! ### Step 2: reality of `e^{-ip/2} R` says `arg R = p/2` modulo `π` -/

lemma sin_half_sub_arg_eq_zero (R : ℂ) (p : ℝ) (hR : R ≠ 0)
    (him : (Complex.exp (-(I * ((p / 2 : ℝ) : ℂ))) * R).im = 0) :
    Real.sin (p / 2 - Complex.arg R) = 0 := by
  have hexp : Complex.exp (-(I * ((p / 2 : ℝ) : ℂ)))
      = ((Real.cos (p / 2) : ℝ) : ℂ) - ((Real.sin (p / 2) : ℝ) : ℂ) * I := by
    have h : -(I * ((p / 2 : ℝ) : ℂ)) = ((-(p / 2) : ℝ) : ℂ) * I := by
      push_cast; ring
    rw [h, Complex.exp_mul_I, ← Complex.ofReal_cos, ← Complex.ofReal_sin,
      Real.cos_neg, Real.sin_neg]
    push_cast
    ring
  have hnorm : ‖R‖ ≠ 0 := norm_ne_zero_iff.mpr hR
  rw [hexp] at him
  simp only [Complex.mul_im, Complex.sub_re, Complex.sub_im, Complex.mul_re,
    Complex.ofReal_re, Complex.ofReal_im, Complex.I_re, Complex.I_im] at him
  rw [Real.sin_sub, Complex.cos_arg hR, Complex.sin_arg]
  field_simp
  linear_combination -him

/-! ### The weights and the two fractional remainders -/

/-- `d₁ = |R| sin(ω - arg R + ψ)/sin(2ω + ψ)`. -/
noncomputable def d1 (R : ℂ) (w p : ℝ) : ℝ :=
  ‖R‖ * (Real.sin (w - Complex.arg R + p) / Real.sin (2 * w + p))

/-- `d₂ = |R| sin(ω + arg R)/sin(2ω + ψ)`. -/
noncomputable def d2 (R : ℂ) (w p : ℝ) : ℝ :=
  ‖R‖ * (Real.sin (w + Complex.arg R) / Real.sin (2 * w + p))

/-- `R1ps = d₁ e^{-iω}`. -/
noncomputable def R_1ps (R : ℂ) (w p : ℝ) : ℂ :=
  ((d1 R w p : ℝ) : ℂ) * Complex.exp (-(I * ((w : ℝ) : ℂ)))

/-- `R2ps = d₂ e^{i(ω+ψ)}`. -/
noncomputable def R_2ps (R : ℂ) (w p : ℝ) : ℂ :=
  ((d2 R w p : ℝ) : ℂ) * Complex.exp (I * ((w + p : ℝ) : ℂ))

/-! ### Result 1: the two weights are equal

Note that `ω` is arbitrary here: the equality of the weights depends only on the
direction of `R` relative to `χ`, not on the turn angle. -/

theorem d1_eq_d2_of_im_eq_zero (R : ℂ) (w p : ℝ)
    (him : (Complex.exp (-(I * ((p / 2 : ℝ) : ℂ))) * R).im = 0) :
    d1 R w p = d2 R w p := by
  rcases eq_or_ne R 0 with h | h
  · simp [d1, d2, h]
  · have hs : Real.sin (p / 2 - Complex.arg R) = 0 :=
      sin_half_sub_arg_eq_zero R p h him
    have hnum : Real.sin (w - Complex.arg R + p) = Real.sin (w + Complex.arg R) := by
      have hd := Real.sin_sub_sin (w - Complex.arg R + p) (w + Complex.arg R)
      have harg : (w - Complex.arg R + p - (w + Complex.arg R)) / 2
          = p / 2 - Complex.arg R := by ring
      rw [harg, hs, mul_zero, zero_mul] at hd
      linarith
    rw [d1, d2, hnum]

/-- **d₁ = d₂ on the critical line.** -/
theorem d1_eq_d2 (ζ S₁ χ : ℂ) (w p : ℝ)
    (hp : χ = Complex.exp (I * (p : ℂ)))
    (hζ : ζ = χ * conj ζ) :
    d1 (ζ - S₁ - χ * conj S₁) w p = d2 (ζ - S₁ - χ * conj S₁) w p :=
  d1_eq_d2_of_im_eq_zero _ w p (rotated_remainder_im_eq_zero ζ S₁ χ p hp hζ)

/-! ### Result 2: the two fractional remainders have equal length -/

/-- **|R1ps| = |R2ps|.** -/
theorem norm_R1ps_eq_norm_R2ps (R : ℂ) (w p : ℝ) (hd : d1 R w p = d2 R w p) :
    ‖R_1ps R w p‖ = ‖R_2ps R w p‖ := by
  rw [R_1ps, R_2ps, norm_mul, norm_mul, norm_exp_neg_imag, norm_exp_imag, hd]

/-- The sharper statement: `R2ps` is the mirror image of `R1ps` across the real
axis, followed by the rotation `χ`. -/
theorem R2ps_eq_chi_conj_R1ps (R χ : ℂ) (w p : ℝ)
    (hp : χ = Complex.exp (I * (p : ℂ)))
    (hd : d1 R w p = d2 R w p) :
    R_2ps R w p = χ * conj (R_1ps R w p) := by
  have hexp : Complex.exp (I * ((w + p : ℝ) : ℂ))
      = Complex.exp (I * (p : ℂ)) * Complex.exp (I * ((w : ℝ) : ℂ)) := by
    rw [← Complex.exp_add]
    congr 1
    push_cast
    ring
  rw [R_1ps, R_2ps, map_mul, Complex.conj_ofReal, conj_exp_neg_imag, hd, hexp, hp]
  ring

/-! ### Result 3: the two legs have equal length -/

/-- With `B₁ = S₁ + R1ps` and `B₂ = S₂ + R2ps = χ conj S₁ + R2ps`, the second leg
is the reflection of the first. -/
theorem leg2_eq_chi_conj_leg1 (S₁ R χ : ℂ) (w p : ℝ)
    (hp : χ = Complex.exp (I * (p : ℂ)))
    (hd : d1 R w p = d2 R w p) :
    χ * conj S₁ + R_2ps R w p = χ * conj (S₁ + R_1ps R w p) := by
  rw [map_add, mul_add, R2ps_eq_chi_conj_R1ps R χ w p hp hd]

/-- **L₁ = L₂.** -/
theorem legs_norm_eq (S₁ R χ : ℂ) (w p : ℝ) (hχ : ‖χ‖ = 1)
    (hp : χ = Complex.exp (I * (p : ℂ)))
    (hd : d1 R w p = d2 R w p) :
    ‖χ * conj S₁ + R_2ps R w p‖ = ‖S₁ + R_1ps R w p‖ := by
  rw [leg2_eq_chi_conj_leg1 S₁ R χ w p hp hd, norm_mul, hχ, one_mul,
    RCLike.norm_conj]

/-! ### Result 4: the bisector point projects onto ζ/2 -/

/-- If `B₂ = χ conj B₁` and `ζ = B₁ + B₂`, then the orthogonal projection of the
bisector point `B₁` onto the line `ℝ·e^{iψ/2}` through `ζ` is exactly `ζ/2`. The
apex of the isosceles triangle sits over the midpoint of its base. -/
theorem proj_eq_half_zeta (B₁ ζ χ : ℂ) (p : ℝ)
    (hp : χ = Complex.exp (I * (p : ℂ)))
    (hzeta : ζ = B₁ + χ * conj B₁) :
    (((Complex.exp (-(I * ((p / 2 : ℝ) : ℂ))) * B₁).re : ℝ) : ℂ) *
        Complex.exp (I * ((p / 2 : ℝ) : ℂ))
      = ζ / 2 := by
  set em := Complex.exp (-(I * ((p / 2 : ℝ) : ℂ))) with hem
  set ep := Complex.exp (I * ((p / 2 : ℝ) : ℂ)) with hep
  have hu : ep * em = 1 := by
    rw [hep, hem, ← Complex.exp_add]
    simp
  have hu2 : ep ^ 2 = χ := by
    rw [hep, hp, pow_two, ← Complex.exp_add]
    congr 1
    push_cast
    ring
  have hconj : conj em = ep := by rw [hem, hep, conj_exp_neg_imag]
  set z := em * B₁ with hz
  have key : ((z.re : ℝ) : ℂ) * 2 = z + conj z := by
    rw [Complex.add_conj]
    push_cast
    ring
  have hcz : conj z = ep * conj B₁ := by rw [hz, map_mul, hconj]
  rw [eq_div_iff (two_ne_zero : (2 : ℂ) ≠ 0), hzeta]
  linear_combination ep * key + ep * hz + B₁ * hu + ep * hcz + conj B₁ * hu2
\end{lstlisting}


\end{document}